\documentclass[reqno,11pt]{amsart}

\usepackage[top=2.0cm,bottom=2.0cm,left=3cm,right=3cm]{geometry}
\usepackage{amsthm,amsmath,amssymb,dsfont}
\usepackage{mathrsfs,amsfonts,functan,extarrows,mathtools}
\usepackage[colorlinks]{hyperref}

\usepackage{marginnote}
\usepackage{xcolor}
\usepackage{stmaryrd}
\usepackage{esint}
\usepackage{graphicx}
\usepackage{bbm}
\usepackage{wasysym}
\usepackage{textcomp}
\usepackage{float}
\usepackage{wasysym}

\usepackage{tikz}
\usetikzlibrary{calc,decorations.pathreplacing}

\newtheorem{theorem}{Theorem}[section]
\newtheorem{proposition}{Proposition}[section]
\newtheorem{definition}{Definition}[section]
\newtheorem{corollary}{Corollary}[section]
\newtheorem{remark}{Remark}[section]
\newtheorem{lemma}{Lemma}[section]

\numberwithin{equation}{section}
\allowdisplaybreaks

\def\d{\mathrm{d}}
\def\no{\nonumber}
\def\R{\mathbb{R}}

\def\exp{\mathrm{exp}}

\newcounter{wronumber}

\begin{document}
\title[Renormalized Solutions to BSMP]
			{Global renormalized solutions to Boltzmann systems modeling mixture gases of monatomic and polyatomic species}

\author[Yi-Long Luo]{Yi-Long Luo}
\address[Yi-Long Luo]
{\newline School of Mathematics, Hunan University, Changsha, 410082, P. R. China}
\email{luoylmath@hnu.edu.cn}

\author[Jing-Xin Nie]{Jing-Xin Nie${}^*$}
\address[Jing-Xin Nie]{
\newline	School of Mathematics, Hunan University, Changsha, 410082, P. R. China} 
	\email{2359093513@qq.com}

\thanks{${}^*$ Corresponding author \quad \today}

\maketitle

\begin{abstract}
   Inspired by DiPerna-Lions' work \cite{Diperna-Lions}, we study the renormalized solutions to the large-data Cauchy problem of the Boltzmann systems modeling mixture gases of monatomic and polyatomic species, in which the distribution functions $f_\alpha$ characterized the polyatomic species contain the continuous internal energy variable $I \in \R_+$. We first construct the smooth approximated problem and establish the corresponding uniform and physically natural bounds. Then, by employing the averaged velocity (-internal energy) lemma, we can show that the weak $L^1$ limit of the approximated solution is exactly a renormalized solution what we required. Moreover, we also justify that the constructed renormalized solution subjects to the entropy inequality.  \\


\end{abstract}

%
%


\section{Introduction and main results}

\subsection{Model for mixture gases of monatomic and polyatomic species}

In 1975, Borgnakke and Larsen \cite{Borgnakke-Larsen-JCP-1975} established a rarefied flow (a Boltzmann collision model) for a gas mixture of polyatomic molecules possessing internal energy, where the related ideas were followed from the works of Morse \cite{Morse-PF-1964}, Wang Chang-Uhlenbeck \cite{Wang-Uhlenbeck-1951} and Chapman-Cowling \cite{Chapman-Cowling-1960}. The polyatomicity is modeled by a continuous internal energy variable $I \in \R_+$. Compared to the classical mechanical models (i.e., monatomic gases models) such as rough-spheres, loaded spheres and various $C$-bodies, or special intermolecular potentials, which do not possess the continuous internal variables, the Boltzmann collision integral for diatomic or polyatomic gases is of more complexity, which will be illustrated by geometric intuitivity as follows. For instance, while considering the elastic collision for two monatomic molecules with the same unit mass (without internal energy variable), the conservation laws of momentum and energy read
\begin{equation*}
	\begin{aligned}
		\xi + \xi_* = \xi' + \xi_*' \,, \ |\xi|^2 + |\xi_*|^2 = |\xi'|^2 + |\xi_*'|^2 \,,
	\end{aligned}
\end{equation*}
where $(\xi, \xi_*) \in \R^3 \times \R^3$ and $(\xi', \xi_*') \in \R^3 \times \R^3$ are respectively pre-collision and post-collision velocities of two monatomic molecules, see Grad's work \cite{Grad-cut-off} and the references therein. As shown in Figure \ref{Fig1}, in the cross section of the corresponding collision, the four ``points" $(\xi, \xi_*, \xi', \xi_*')$ form a rectangle.
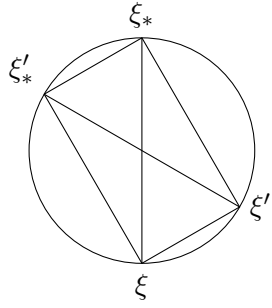
\begin{figure}[h]
	\begin{center}
		\begin{tikzpicture}
			\pgfmathsetmacro\rd{1.5};
			\draw (\rd,0) arc (0: 360: \rd);
			\draw (0, -\rd)--(-30:\rd)node[right]{$\xi'$}--(90:\rd)--(150:\rd)node[above left]{$\xi_*'$}--cycle;
			\draw (0, -\rd) node[below]{$\xi$}--(0,\rd)node[above]{$\xi_*$};
			\draw (-30:\rd)--(150:\rd);
		\end{tikzpicture}
	\end{center}
	\caption{Transformation of velocities for monatomic molecules with same mass.}\label{Fig1}
\end{figure}
However, while considering the collision for two polyatomic molecules with same unit mass, the conservation laws of momentum and total energy are expressed by
\begin{equation*}
	\begin{aligned}
		\xi + \xi_* = \xi' + \xi_*' \,, \ \tfrac{1}{2} |\xi|^2 + \tfrac{1}{2} |\xi_*|^2 + I + I_* = \tfrac{1}{2} |\xi'|^2 + \tfrac{1}{2} |\xi_*'|^2 + I' + I_*' \,,
	\end{aligned}
\end{equation*}
where $(I, I_*)$ and $(I', I_*')$ are respectively the pre-internal energy and post-internal energy variables. In the corresponding cross section, the four ``points" $(\xi, \xi_*, \xi', \xi_*')$ only form a parallelogram (not a rectangle), see Figure \ref{Fig2}.
\begin{figure}[h]
	\begin{center}
		\begin{tikzpicture}
			\draw[dashed] (2,0) arc (0: 360: 2);
			\draw (0,2)--(0,-2);
			\draw (0, -2)--(1.229,-0.75);
			\draw (-1.229,0.75)node[left=3pt,below]{$\xi_*'$}--(1.229,-0.75);
			\draw[dashed] (-1.732,1)--(-1.229,0.75);
			\draw[dashed] (1.229,-0.75)node[right=3pt, above]{$\xi'$}--(1.732,-1);
			\draw (0,2)--(-1.229,0.75)--(0,-2) node[below]{$\xi$};
			\draw (1.229,-0.75)--(0,2) node[above]{$\xi_*$};
		\end{tikzpicture}
	\end{center}
	\caption{Transformation of velocities for polyatomic molecules with same mass.}\label{Fig2}
\end{figure}
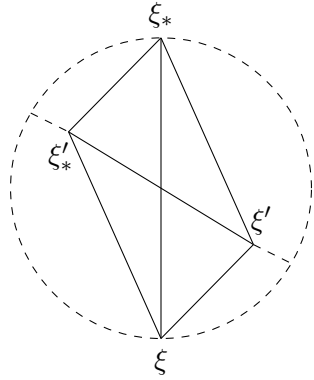
Moreover, if the two collided polyatomic molecules are of different masses, the four ``points" $(\xi, \xi_*, \xi', \xi_*')$ in the collisional cross section will only form a trapezoid, see Figure 1 of \cite{N.Bernhoff-2023} for example. From the geometric intuitivity above, the structure of the Boltzmann collision integral for the polyatomic gases is more complex than that for the monatomic gases.

Very recently, Bernhoff \cite{N.Bernhoff-2023} wrote down a more complicated rarefied gases model captured a multicomponent mixture of $s$ species $a_1, a_2, \cdots, a_s$ with $s_0$ monatomic and $s_1 : = s - s_0$ polyatomic species, and masses $m_1, m_2, \cdots, m_s$, respectively. This model involves the classical Boltzmann collision effect for monatomic molecules, the complex Boltzmann collision interaction possessing continuous internal energy variable $I \in \R_+$ for polyatomic molecules, and also the interaction collisional effect between monatomic and polyatomic molecules. More precisely, the model what we consider in current work is formulated by the following vector-valued Boltzmann equation  
\begin{equation}\label{BE-MP}
	\tfrac{\partial f}{\partial t } + \xi \cdot \nabla_x f = Q(f,f),
\end{equation}
where $ f = ( f_1, \cdots , f_s ) $ is the density distribution function with the components $ f_\alpha = f_\alpha ( t, x, \mathbf{Z} ) $ for $\alpha \in \{1,\cdots , s\}$ at $t > 0$, position $x = ( x_1, x_2, x_3 ) \in \R^3$ associated with mixed microscopic variable $ \mathbf{Z} $ given by 
\begin{equation}\label{eq2.01}
	\mathbf{Z} =\mathbf{Z}_\alpha :=
	\begin{cases}
		 \xi, \qquad \enspace \, \alpha \in \{1,\cdots , s_0 \} \,, \\
		( \xi , I )  , \quad \alpha \in \{s_0 +1,\cdots , s \} \,.
	\end{cases}
\end{equation} 
Here $\xi  = (\xi_1, \xi_2 ,\xi_3 ) \in \R^3$ is the average velocity variable of the microscopic particles and $I > 0$ is the continuous internal energy variable. Namely, $f_\alpha (t, x, \xi)$ ($\alpha \in \{ 1, \cdots, s_0 \}$) represent the density distribution of the monatomic molecules, and $f_\alpha (t, x, \xi, I)$ ($\alpha \in \{ s_0 + 1, \cdots, s \}$) capture that of the polyatomic molecules.

Moreover, the vector-valued collision operator $ Q = ( Q_1 , \cdots , Q_s )$ is a quadratic bilinear operator expressing the change of the velocities and internal energies of particles due to binary collisions. Here $Q_\alpha$ with $\alpha \in \{ 1, \cdots, s \}$ is the collision operator for species $a_\alpha$, the detailed expression of which will be given later. Then the components forms of \eqref{BE-MP} can be represented as
\begin{equation}\label{2.08}
	\tfrac{\partial f_\alpha}{\partial t } + \xi \cdot \nabla_x f_\alpha = Q_\alpha (f,f) \,, \alpha \in \{ 1,\cdots , s \} \,.
\end{equation}
A collision between two particles of species $a_\alpha$ and $a_\beta$ with $\alpha, \beta \in \{ 1, \cdots, s \}$ can be respectively expressed by pre-collisional elements $( \mathbf{Z}, \mathbf{Z}_* ) = ( \mathbf{Z}_\alpha, \mathbf{Z}_{* \beta} )$ and post-collisional elements $( \mathbf{Z}', \mathbf{Z}_*' ) = ( \mathbf{Z}_\alpha', \mathbf{Z}_{* \beta}' )$, where the all collisional elements consist of a microscopic velocity and possibly an internal energy. The notations for pre- and post-collisional pairs may be interchanged as well. Taking the momentum and total energy conservation into consideration, the following relations hold:
\begin{enumerate}
	\item Momentum conservation laws: for $\alpha, \beta \in \{ 1, \cdots, s \}$
	\begin{equation}\label{eq2.02}
		m_\alpha \mathbf{\xi } + m_\beta \mathbf{\xi }_* = m_\alpha \mathbf{\xi }' + m_\beta \mathbf{\xi }'_* \,;
	\end{equation}
    \item Total energy conservation laws:
    \begin{itemize}
    	\item If $ \alpha, \beta \in \{1, \cdots , s_0\}$,
    	\begin{equation}\label{eq2.03}
    		m_\alpha |\mathbf{\xi } |^2 + m_\beta |\xi_* |^2 = m_\alpha |\xi '|^2 +m_\beta |\xi '_*|^2 \,;
    	\end{equation}
        \item If $\alpha \in \{1, \cdots , s_0 \} $ and $\beta \in \{ s_0 + 1, \cdots , s \}$,
        \begin{equation}\label{eq2.04}
        	\tfrac{1}{2} m_\alpha |\xi |^2 + \tfrac{1}{2} m_\beta |\xi_* |^2 + I_* =\tfrac{1}{2} m_\alpha |\xi '|^2 + \tfrac{1}{2} m_\beta |\xi '_*|^2 +I'_* \,;
        \end{equation}
        \item If $\alpha \in \{ s_0 + 1, \cdots , s \} $ and $\beta \in \{ 1, \cdots , s_0 \} $,
        \begin{equation}\label{eq2.05}
        	\tfrac{1}{2} m_\alpha |\xi |^2 + \tfrac{1}{2} m_\beta |\xi_* |^2 + I =\tfrac{1}{2} m_\alpha |\xi '|^2 +\tfrac{1}{2} m_\beta |\xi '_*|^2 +I' \,;
        \end{equation}
        \item If $ \alpha , \beta \in \{ s_0 + 1, \cdots , s \} $,
        \begin{equation}
        	\tfrac{1}{2} m_\alpha |\xi |^2 + \tfrac{1}{2} m_\beta |\xi_* |^2 + I +I_* =\tfrac{1}{2} m_\alpha |\xi '|^2 +\tfrac{1}{2} m_\beta |\xi '_*|^2 +I' +I'_* \,.
        \end{equation}
    \end{itemize}
    Remark that the above energy conservation laws can be written as the following unified form
    \begin{equation}\label{eq2.06}
        \begin{aligned}
            & \tfrac{1}{2} m_\alpha |\xi |^2 + \tfrac{1}{2} m_\beta |\xi_* |^2 + \mathbf{1}_{ \alpha \in \{ s_0 + 1, \cdots, s \} } I + \mathbf{1}_{ \beta \in \{ s_0 + 1, \cdots, s \} }  I_*\\
            = & \tfrac{1}{2} m_\alpha |\xi '|^2 +\tfrac{1}{2} m_\beta |\xi '_*|^2 + \mathbf{1}_{ \alpha \in \{ s_0 + 1, \cdots, s \} } I' + \mathbf{1}_{ \beta \in \{ s_0 + 1, \cdots, s \} } I'_* \,.
        \end{aligned}
    \end{equation}
\end{enumerate}

\subsection{Collision operator and its equivalent transformation}

\subsubsection{Definition of collision operator $Q_\alpha$}

For $\alpha \in \{ 1, \cdots, s \}$, as in \cite{N.Bernhoff-2023}, the component collision operator $Q_\alpha (f,f)$ for species $a_\alpha$ can be introduced by
\begin{equation}\label{2.09}
	\begin{aligned}
		Q_\alpha (f,f)
		= \sum \limits_{\beta =1}^s Q_{\alpha \beta } (f, f) = \sum \limits_{\beta =1}^s  \int_{ \mathcal{Z}_\beta \times \mathcal{Z}_\alpha \times \mathcal{Z}_\beta } W_{\alpha \beta } \Lambda_{\alpha \beta } (f) d \mathbf{Z}_* d \mathbf{Z}' d \mathbf{Z}_*' \,, 
	\end{aligned}
\end{equation}
where
\begin{equation}\label{Z-gamma}
	\begin{aligned}
		\mathcal{Z}_\gamma =
		\left\{
		    \begin{aligned}
		    	\R^3 \,, & \textrm{ if } \gamma \in \{ 1, \cdots, s_0 \} \,, \\
		    	\R^3 \times \R_+ \,, & \textrm{ if } \gamma \in \{ s_0 + 1, \cdots, s \} \,,
		    \end{aligned}
		\right.
	\end{aligned}
\end{equation}
and
\begin{equation}
	\begin{aligned}
		\Lambda_{\alpha \beta } (f) := \tfrac{ f'_\alpha f'_{ \beta_* } }{ (I')^{ \delta ( \alpha )/2 - 1 } ( I'_* )^{ \delta ( \beta )/2 - 1 } } - \tfrac{ f_\alpha f_{ \beta_* } }{ I^{ \delta ( \alpha )/2 - 1 } I_*^{ \delta ( \beta )/2 - 1 } } \,.
	\end{aligned}
\end{equation}
Here $\delta^{(\alpha)} \geq 2 $ ($\alpha \in \{ 1, \cdots, s \}$) with $\delta^{(1)}=\cdots =\delta^{(s_0)} =2$ are the indices of the degeneracies, see Gamba and Pavic-Colic's work \cite{Gamba-Pavic-JMP-2023}, which  denote the number of internal degrees
of freedom of the monatomic or polyatomic species $a_\alpha$. Here and below the abbreviations $ f_{\alpha_*} = f_\alpha ( t, x, \mathbf{Z}_* ), f'_\alpha = f_\alpha ( t, x, \mathbf{Z}' ) $ and $ f'_{\alpha_*} = f_\alpha ( t, x, \mathbf{Z}'_* )$ are employed, where $\mathbf{Z}_*$, $\mathbf{Z}'$ and $\mathbf{Z}'_*$ are defined as the same way of \eqref{eq2.01}.

As in \cite{Bernhoff-AAM-2023,Bernhoff-KRM-2023,N.Bernhoff-2023}, the transition probabilities $W_{\alpha \beta }$ are introduced by
\begin{equation}\label{eq2.10}
	\begin{aligned}
		W_{\alpha \beta } 
		= & W_{\alpha \beta } ( \mathbf{Z} , \mathbf{Z}_* | \mathbf{Z}', \mathbf{Z}_*' ) \\
		= & (m_\alpha + m_\beta )^2 m_\alpha m_\beta (I')^{\delta(\alpha)/2-1} (I'_*)^{\delta(\beta)/2-1} \sigma'_{\alpha \beta } \tfrac{|g'|}{|g|} \hat{\delta}_1 \hat{\delta}_3\\
		= & (m_\alpha + m_\beta )^2 m_\alpha m_\beta I^{\delta(\alpha)/2-1} I_*^{\delta(\beta)/2-1} \sigma_{\alpha \beta } \tfrac{|g|}{|g'|} \hat{\delta}_1 \hat{\delta}_3 \,,
	\end{aligned}
\end{equation}
where
\begin{equation*}
	\begin{aligned}
		& \hat{\delta}_1 = \delta_1 ( \tfrac{1}{2} ( m_\alpha |\xi|^2 + m_\beta |\xi_*|^2 - m_\alpha |\xi'|^2 - m_\beta |\xi'_*|^2 ) -\Delta I ) \,,\\
		& \hat{\delta}_3 = \delta_3 ( m_\alpha \xi + m_\beta \xi_* - m_\alpha \xi' - m_\beta \xi'_* ) \,, \\
		& \cos \theta =\tfrac{g \cdot g'}{|g| |g'|} \,, g= \xi -\xi_* ,g'=\xi '- \xi'_* \,.
	\end{aligned}
\end{equation*}
Here $\delta_3$ and $\delta_1$ denote the Dirac's delta function in $\R^3$ and $\R$, respectively. Moreover, the scattering cross sections $ \sigma_{\alpha \beta } $, $\sigma_{\alpha \beta }' > 0$ and the energy gap $ \Delta I \in \R $ satisfy the following properties:
\begin{enumerate}
	\item If $\alpha, \beta \in \{ 1, \cdots, s_0 \}$, 
	\begin{equation*}
		\begin{aligned}
			\sigma_{\alpha \beta } = \sigma_{\alpha \beta } ( |g|, | \cos \theta | ) \,, \sigma'_{\alpha \beta } = \sigma_{\alpha \beta } ( |g'|, | \cos \theta | ) \,, \Delta I = 0 \,;
		\end{aligned}
	\end{equation*}
    \item If $\alpha \in \{ 1, \cdots, s_0 \}$ and $ \beta \in \{ s_0 + 1, \cdots, s \} $,
    \begin{equation*}
    	\begin{aligned}
    		\sigma_{\alpha \beta } = \sigma_{\alpha \beta } ( |g|, | \cos \theta |, I_*, I_*' ) \,, \sigma'_{\alpha \beta } = \sigma_{\alpha \beta } ( |g'|, | \cos \theta |, I_*', I_* ) \,, \Delta I = I_*' - I_* \,;
    	\end{aligned}
    \end{equation*}
    \item If $\alpha \in \{ s_0 + 1, \cdots, s \}$ and $ \beta \in \{ 1, \cdots, s_0 \} $,
    \begin{equation*}
    	\begin{aligned}
    		\sigma_{\alpha \beta } = \sigma_{\alpha \beta } ( |g|, | \cos \theta |, I, I' ) \,, \sigma'_{\alpha \beta } = \sigma_{\alpha \beta } ( |g'|, | \cos \theta |, I', I ) \,, \Delta I = I' - I \,;
    	\end{aligned}
    \end{equation*}
    \item If $\alpha, \beta \in \{ s_0 + 1, \cdots, s \}$, 
    \begin{equation*}
    	\begin{aligned}
    		& \sigma_{ \alpha \beta } = \sigma_{\alpha \beta } ( |g|, | \cos \theta |, I , I_* , I' ,I'_*) \,, \sigma'_{ \alpha \beta } = \sigma_{ \alpha \beta } ( |g'|, | \cos \theta |, I' , I'_*, I , I_* ) \,, \\
    		& \Delta I = I'+ I'_* - I - I_* \,.
    	\end{aligned}
    \end{equation*}
\end{enumerate}
Here $\theta \in [0, 2 \pi]$ represents the velocities deflection angular between pre- and post-collisions. By summarizing the above relations, the internal energy gap $\Delta I$ can be unified as
\begin{equation*}
    \begin{aligned}
        \Delta I = \mathbf{1}_{ \alpha \in \{ s_0 + 1, \cdots, s \} } I' + \mathbf{1}_{ \beta \in \{ s_0 + 1, \cdots, s \} }  I_*' - \mathbf{1}_{ \alpha \in \{ s_0 + 1, \cdots, s \} } I - \mathbf{1}_{ \beta \in \{ s_0 + 1, \cdots, s \} }  I_* \,.
    \end{aligned}
\end{equation*}

Moreover, the scattering cross sections $\sigma_{\alpha \beta } $ for $ \alpha, \beta \in \{ 1,\cdots , s\}$ are assumed to satisfy the microreversibility conditions
\begin{equation}\label{eq2.11}
	(I')^{\delta(\alpha)/2-1} (I'_*)^{\delta(\beta)/2-1} |g'|^2 \sigma'_{\alpha \beta } 
	= I^{\delta(\alpha)/2-1} I_*^{\delta(\beta)/2-1} |g|^2 \sigma_{\alpha \beta }
\end{equation}
and the symmetry relations
\begin{equation}\label{eq2.12}
    \begin{aligned}
        \sigma_{\alpha \beta } (&|g|,|\cos \theta |, I , I_* , I' ,I'_*) =\sigma_{\beta \alpha }(|g|,|\cos \theta |, I_*, I  ,I'_*, I' ) \,, \\
        \sigma_{\alpha \alpha} = & \sigma_{\alpha \alpha} (|g|,|\cos \theta |, I , I_* , I' ,I'_*) = \sigma_{\alpha \alpha} (|g|,|\cos \theta |, I_* , I , I' ,I'_*) \\
        = & \sigma_{\alpha \alpha} (|g|,|\cos \theta |, I_* , I , I'_* ,I') \,.
    \end{aligned}
\end{equation}
Moreover, we shall adopt in this paper an assumption of weak angular cut-off (see Grad \cite{Grad-cut-off}), namely, $\sigma_{\alpha \beta }$ are integrable associated with the angular variable $\theta$, i.e.,
\begin{equation}\label{Asum-angular-cutoff}
    \begin{aligned}
        \sigma_{\alpha \beta} (\mathbf{Z}, \mathbf{Z}_*, \mathbf{Z}', \mathbf{Z}_*', \theta) \in L^1_{loc} ( d \mathbf{Z}_*;  L^1 (d \theta) ) \,,
    \end{aligned}
\end{equation}
where the variables $\mathbf{Z}, \mathbf{Z}_*, \mathbf{Z}', \mathbf{Z}_*'$ subject to the conservation laws \eqref{eq2.02}-\eqref{eq2.06}, and the deflection angular $\theta \in [0, 2 \pi]$. As a result, together with the above relations and the known properties of Dirac’s delta function, the transition probabilities $W_{\alpha \beta }$ can be transformed to
\begin{equation*}
	\begin{aligned}
		W_{\alpha \beta } 
		&= (I')^{\delta(\alpha)/2-1} (I'_*)^{\delta(\beta )/2-1} \mu_{\alpha \beta } \ \sigma'_{\alpha \beta } \tfrac{|g'|}{|g|} 1_{\mu _{\alpha \beta }|g|^2>2\Delta I} \delta_{3} (G_{\alpha \beta }- G'_{\alpha \beta }) \delta_{1} (\sqrt{|g|^2-\tfrac{2\Delta I}{\mu_{\alpha \beta } }}-|g'|)\\
		&=I^{\delta(\alpha)/2-1} I_*^{\delta(\beta )/2-1} \mu_{\alpha \beta } \ \sigma_{\alpha \beta }\tfrac{|g|}{|g'|} 1_{\mu _{\alpha \beta }|g|^2>2\Delta I} \delta_{3} (G_{\alpha \beta }- G'_{\alpha \beta }) \delta_{1} (\sqrt{|g|^2-\tfrac{2\Delta I}{\mu_{\alpha \beta } }}-|g'|) \,,
	\end{aligned}
\end{equation*}
where 
\begin{equation*}
    \begin{aligned}
        G_{\alpha \beta } =\tfrac{m_\alpha \mathbf{\xi } + m_\beta \mathbf{\xi }_*}{m_\alpha + m_\beta}\,, G'_{\alpha \beta } =\tfrac{m_\alpha \mathbf{\xi' } + m_\beta \mathbf{\xi' }_*}{m_\alpha + m_\beta} \,, \mu _{\alpha \beta }= \tfrac{m_\alpha m_\beta }{(m_\alpha + m_\beta )} \,.
    \end{aligned}
\end{equation*}

Remark that the collision operator $Q_{\alpha \beta }(f,f)$ can be split into the gain term $Q^+_{\alpha \beta } (f,f)$ and loss term $Q^-_{\alpha \beta } (f,f)$, namely,
\begin{equation*}
	Q_{\alpha \beta } (f,f) =Q^+_{\alpha \beta } (f,f) -Q^-_{\alpha \beta } (f,f),
\end{equation*}
where
\begin{equation*}
	\begin{aligned}
		& Q^+_{\alpha \beta } (f,f) = \int_{\mathcal{Z}_\beta \times \mathcal{Z}_\alpha \times \mathcal{Z}_\beta} W_{\alpha \beta } \tfrac{f'_\alpha f'_{\beta_*}}{(I')^{\delta(\alpha)/2-1} (I'_* )^{\delta(\beta)/2-1}} d \mathbf{Z}_* d \mathbf{Z}' d \mathbf{Z}_*', \\
		& Q^-_{\alpha \beta } (f,f) =\int_{\mathcal{Z}_\beta \times \mathcal{Z}_\alpha \times \mathcal{Z}_\beta} W_{\alpha \beta }
		\tfrac{f_\alpha f_{\beta_*}}{I^{\delta(\alpha)/2-1} I_* ^{\delta(\beta)/2-1}} d \mathbf{Z}_* d \mathbf{Z}' d \mathbf{Z}_*' .
	\end{aligned}
\end{equation*}

\subsubsection{Equivalent transformation of the collision operator}

Notice that the total molecular energy conservation law from \eqref{eq2.06} can be simply written as
\begin{equation}\label{Eab}
    \begin{aligned}
        E_{\alpha \beta } = & \tfrac{\mu_{\alpha \beta } }{2} |g|^2+ \mathbf{1}_{ \alpha \in \{ s_0 + 1, \cdots, s \} } I + \mathbf{1}_{ \beta \in \{ s_0 + 1, \cdots, s \} }  I_* \\
        = & \tfrac{\mu_{\alpha \beta } }{2} |g'|^2+ \mathbf{1}_{ \alpha \in \{ s_0 + 1, \cdots, s \} } I' + \mathbf{1}_{ \beta \in \{ s_0 + 1, \cdots, s \} }  I_*'.
    \end{aligned}
\end{equation}
If $\beta \in \{s_0 + 1, \cdots , s\}$, one introduces a parameter $\mathfrak{R} \in [0,1]$ such that 
\begin{equation}\label{IR}
	\mathfrak{R} E_{\alpha \beta } =\tfrac{\mu _{\alpha \beta }}{2} |g'|^2, \quad (1- \mathfrak{R}) E_{\alpha \beta } = \mathbf{1}_{ \alpha \in \{ s_0 + 1, \cdots, s \} } I' + I'_*.
\end{equation}
Moreover, if $\alpha \in \{s_0 + 1, \cdots , s\} $, one introduces a parameter $\mathfrak{r} \in [0,1]$ such that 
\begin{equation}\label{Ir}
	I'=\mathfrak{r}(1-\mathfrak{R})E_{\alpha \beta }, \quad I'_* =(1-\mathfrak{r})(1-\mathfrak{R})E_{\alpha \beta }.
\end{equation}
Finally, for the sake of symmetry,we introduce the extra parameters
\begin{equation*}
	\mathfrak{R}' =\tfrac{\mu _{\alpha \beta }}{2 E_{\alpha \beta }} |g|^2,\quad  \mathfrak{r}'=\tfrac{I}{I+I_*}.
\end{equation*}
By a series of change of variables \cite{N.Bernhoff-2023}, there holds
\begin{equation*}
	\begin{aligned}
		d\xi' d\xi'_*dI'dI'_* 
		=&\ dG'_{\alpha \beta } dg' dI'dI'_* = \ |g'|^2 dG'_{\alpha \beta } d|g'|d\omega dI'dI'_* \\
		=& \ \sqrt{2} (\tfrac{E_{\alpha \beta }}{\mu_{\alpha \beta } })^{3/2} \mathfrak{R}^{1/2} d\mathfrak{R}d\omega dG'_{\alpha \beta } dI'dE'_{\alpha \beta } \\
		=& \tfrac{\sqrt{2} }{\mu_{\alpha \beta }^{3/2}} E_{\alpha \beta }^{5/2} (1-\mathfrak{R})\mathfrak{R}^{1/2} d\mathfrak{r}d\mathfrak{R}d\omega dG'_{\alpha \beta } dE'_{\alpha \beta }.
	\end{aligned}
\end{equation*}
Then the collision operators $Q_{\alpha \beta}$ can be transformed to the following forms.
\begin{enumerate}
    \item If $ \alpha ,\beta \in \{1,\cdots, s_0 \}$,
    \begin{equation} \label{eq2.14}
        Q_{\alpha \beta } (f,f) = \iint_{\R^3 \times \mathbb{S}^2 } B_{0\alpha \beta } (\xi-\xi_*,\omega )(f'_\alpha f'_{\beta_*}- f_\alpha f_{\beta_*})	d\xi_* d\omega , 
    \end{equation}
    where
    \begin{equation}\label{B0-ab}
        \begin{aligned}
            B_{0\alpha \beta } (\xi-\xi_*,\omega )=\tilde{\sigma}_{\alpha \beta }|g|,\quad \tilde{\sigma}_{\alpha \beta }(|g|,\omega)=\sigma_{\alpha \beta }(|g|,|\cos \theta|) \,;
        \end{aligned}
    \end{equation}
    \item If $\alpha \in \{1,\cdots, s_0 \}$ and $\beta \in \{s_0+1,\cdots,s \}$,
    \begin{equation}\label{eq2.15}
	\begin{aligned}
		 Q_{\alpha \beta } (f,f) =& \iiiint_{ \R^3 \times \R_+ \times [0,1] \times \mathbb{S}^2} B_{1\alpha \beta }(\xi-\xi_*,I_*, \mathfrak{R} ,\omega )(\tfrac{f'_\alpha f'_{\beta_*}}{(I'_* )^{\delta(\beta)/2-1}} - \tfrac{f_\alpha f_{\beta_*}}{ I_* ^{\delta(\beta)/2-1}}) \\
		& \times (1-\mathfrak{R})^{\delta(\beta )/2-1}  \mathfrak{R}^{1/2} I_*^{\delta(\beta)/2-1} d\xi_* d\mathfrak{R}d\omega dI_*, 
	\end{aligned}
    \end{equation}
    where
    \begin{equation}\label{B1-ab}{\small
        \begin{aligned}
            B_{1\alpha \beta }(\xi-\xi_*,I_*,\mathfrak{R},\omega )  = \tfrac{\tilde{\sigma}_{\alpha \beta }|g|E_{\alpha \beta }}{\mathfrak{R}^{1/2} (1-\mathfrak{R})^{\delta(\beta)/2-1}},\  \tilde{\sigma}_{\alpha \beta } (|g|,I_* ,\mathfrak{R},\omega )=\sigma_{\alpha \beta }(|g|,|\cos \theta|,I_*,I'_*) \,;
        \end{aligned}}
    \end{equation}
    \item If $\alpha \in \{s_0+1,\cdots,s \}$ and $\beta \in \{1,\cdots, s_0 \}$,
    \begin{equation}\label{eq2.16}
	\begin{aligned}
		  Q_{\alpha \beta} (f,f) = & \iiint_{\R^3  \times [0,1] \times \mathbb{S}^2} B_{1\beta \alpha }(\xi-\xi_*,I,\mathfrak{R},\omega ) (\tfrac{f'_\alpha f'_{\beta_*}}{(I')^{\delta(\alpha)/2-1} } - \tfrac{f_\alpha f_{\beta_*}}{I^{\delta(\alpha)/2-1}}) \\
		&\times (1-\mathfrak{R})^{\delta(\alpha  )/2-1} \mathfrak{R}^{1/2} I^{\delta(\alpha )/2-1} d\xi_* d\mathfrak{R}d\omega \,,
	\end{aligned}
    \end{equation}
    where 
    \begin{equation}\label{B1-ba}
        \begin{aligned}
            B_{1 \beta \alpha }(\xi-\xi_*,I,\mathfrak{R},\omega )  = \tfrac{\tilde{\sigma}_{\alpha \beta }|g|E_{\alpha \beta }}{\mathfrak{R}^{1/2} (1-\mathfrak{R})^{\delta(\alpha)/2-1}},\  \tilde{\sigma}_{\alpha \beta } (|g|,I ,\mathfrak{R},\omega )=\sigma_{\alpha \beta }(|g|,|\cos \theta|,I,I') \,;
        \end{aligned}
    \end{equation}
    \item If $ \alpha ,\beta \in \{s_0+ 1,\cdots, s \}$,
    \begin{equation}\label{eq2.17}
	\begin{aligned}
		Q_{\alpha \beta } (f,f) 
		=& \iiiint_{\R^3 \times \R_+ \times [0,1]^2 \times \mathbb{S}^2 } (\tfrac{f'_\alpha f'_{\beta_*}}{(I')^{\delta(\alpha)/2-1} (I'_* )^{\delta(\beta)/2-1}} - \tfrac{f_\alpha f_{\beta_*}}{I^{\delta(\alpha)/2-1} I_* ^{\delta(\beta)/2-1}})\\
		& \times B_{2 \alpha \beta } \ \mathfrak{r}^{\delta(\alpha)/2-1} (1-\mathfrak{r})^{\delta (\beta )/2-1 } (1-\mathfrak{R} ) ^{(\delta (\alpha) +\delta (\beta ))/2-1} \mathfrak{R}^{1/2} \\
		&\times I^{\delta(\alpha)/2-1 } I_*^{\delta (\beta )/2-1} d\xi_* d\mathfrak{r}d\mathfrak{R}d\omega dI_* \,,
        \end{aligned}
    \end{equation}
    where
    \begin{equation}\label{B2-ab}
        \begin{aligned}
            B_{2\alpha \beta }(\xi-\xi_*,I+I_*,\mathfrak{R},\mathfrak{r},\omega)&= \tfrac{\tilde{\sigma}_{\alpha \beta } |g|E^2_{\alpha \beta }}{\mathfrak{r}^{\delta(\alpha)/2-1} (1-\mathfrak{r})^{\delta (\beta )/2-1 }(1-\mathfrak{R} ) ^{(\delta (\alpha) +\delta (\beta ))/2-2} \mathfrak{R}^{1/2} } \,,\\
            \tilde{\sigma}_{\alpha \beta } (|g|,I+I_*,\mathfrak{R},\mathfrak{r},\omega) &= \sigma_{\alpha \beta } (|g|,|\cos \theta|,I+I_*,I'+I'_*) 
        \end{aligned}
    \end{equation}
\end{enumerate}
In the above collision integrals, the post-velocities $(\xi', \xi_*')$ can be represented by
\begin{equation}
    \begin{aligned}
        \left\{
             \begin{aligned}
                 \xi ' = & \tfrac{m_\alpha \xi +m_\beta \xi_* }{m_\alpha + m_\beta } + \omega \tfrac{m_\beta  }{m_\alpha + m_\beta } \sqrt{|g|^2 -\tfrac{2 \Delta I}{\mu_{\alpha \beta }}} \,, \\
                 \xi_* ' = & \tfrac{m_\alpha \xi +m_\beta \xi_* }{m_\alpha + m_\beta } - \omega \tfrac{m_\alpha }{m_\alpha + m_\beta } \sqrt{|g|^2 -\tfrac{2 \Delta I}{\mu_{\alpha \beta }}}
             \end{aligned}
        \right.
    \end{aligned}
\end{equation}
with parameter $\omega \in \mathbb{S}^2$. Together with the symmetric properties in \eqref{eq2.11}-\eqref{eq2.12}, the new forms of the collision integral kernels $B_{i \alpha \beta }$ ($i=0,1,2$) defined in \eqref{B0-ab}, \eqref{B1-ab} and \eqref{B2-ab} are invariant under the following change of variables:
\begin{equation}\label{eq2.18}
	\begin{aligned}
		(\xi , \xi_* , I , I_*, \mathfrak{R}, \mathfrak{r} , \omega ) &\leftrightarrow (\xi' , \xi'_* , I' , I'_*, \mathfrak{R}', \mathfrak{r}' , \omega' ),\\
		(\xi , \xi_* , I , I_*, \mathfrak{R}, \mathfrak{r} , \omega ) &\leftrightarrow (\xi_* , \xi , I_* , I, \mathfrak{R}, 1-\mathfrak{r} , -\omega ).
	\end{aligned}
\end{equation}
Namely, 
\begin{itemize}
    	\item If $ \alpha, \beta \in \{1, \cdots , s_0\}$,
    	\begin{equation}\label{eq2.03}
    	B_{0 \alpha \beta } (\xi - \xi_* ,  \omega ) = B_{0 \alpha \beta } (\xi' - \xi'_* ,\omega' ) =B_{0 \alpha \beta } (\xi_* - \xi , -\omega )  \,;
    	\end{equation}
        \item If $\alpha \in \{1, \cdots , s_0 \} $ and $\beta \in \{ s_0 + 1, \cdots , s \}$,
        \begin{equation}\label{eq2.04}
        	B_{1 \alpha \beta } (\xi -\xi_* ,  I_*, \mathfrak{R}, \omega ) = B_{1 \alpha \beta } (\xi' - \xi'_* , I'_*, \mathfrak{R}', \omega' ) =B_{1 \alpha \beta } (\xi_* - \xi , I , \mathfrak{R}, -\omega )  \,;
        \end{equation}
        \item If $\alpha \in \{ s_0 + 1, \cdots , s \} $ and $\beta \in \{ 1, \cdots , s_0 \} $,
        \begin{equation}\label{eq2.05}
        	B_{1  \beta  \alpha} (\xi - \xi_* , I ,  \mathfrak{R}, \omega ) = B_{1 \beta  \alpha } (\xi' - \xi'_* , I' , \mathfrak{R}',  \omega' ) =B_{1 \beta \alpha } (\xi_* - \xi ,  I, \mathfrak{R}, -\omega ) . \,;
        \end{equation}
        \item If $ \alpha , \beta \in \{ s_0 + 1, \cdots , s \} $,
        \begin{equation}\label{B-inv}
        \begin{aligned}
            B_{2 \alpha \beta } (\xi - \xi_* , I + I_*, \mathfrak{R}, \mathfrak{r} , \omega ) &= B_{2 \alpha \beta } (\xi' - \xi'_* , I' + I'_*, \mathfrak{R}', \mathfrak{r}' , \omega' )\\
            &=B_{2 \alpha \beta } (\xi_*- \xi , I_* + I, \mathfrak{R}, 1-\mathfrak{r} , -\omega ) \,.
        \end{aligned}
        \end{equation}
    \end{itemize}

Together with the angular cutoff assumptions \eqref{Asum-angular-cutoff}, one can further define the following functions
\begin{equation}\label{A}{\small
    \begin{aligned}
		A_{0 \alpha \beta } (\xi-\xi_*)=&\int_{\mathbb{S}^2 } B_{0 \alpha \beta }(\xi-\xi_*,\omega ) d\omega, \\
		A_{1\alpha \beta }(\xi-\xi_*,I_*)= &\iint_{[0,1] \times \mathbb{S}^2 } B_{1 \alpha \beta } (\xi-\xi_*,I_*,\mathfrak{R}, \omega )(1-\mathfrak{R})^{\delta(\beta )/2-1}\mathfrak{R}^{1/2} d\mathfrak{R}d\omega, \\
		A_{1 \beta \alpha }(\xi-\xi_*,I)=&\iint_{[0,1] \times \mathbb{S}^2 } B_{1 \beta \alpha }(\xi-\xi_*,I,\mathfrak{R},\omega) (1-\mathfrak{R})^{\delta(\alpha )/2-1}\mathfrak{R}^{1/2} d\mathfrak{R}d\omega, \\
		A_{2\alpha \beta } (\xi-\xi_*,I+I_*)=&\iiint_{[0,1]^2 \times \mathbb{S}^2 } B_{2\alpha \beta }(\xi-\xi_*,I+I_*,\mathfrak{R},\mathfrak{r},\omega ) \\
     & \times \mathfrak{r}^{\delta(\alpha )/2-1} (1-\mathfrak{r})^{\delta(\beta )/2-1} (1-\mathfrak{R})^{(\delta(\alpha)+\delta (\beta ))/2-1} \mathfrak{R}^{1/2} d\mathfrak{r}d\mathfrak{R}d\omega .
    \end{aligned}}
\end{equation}
Employing the above kernel functions, one introduces the following linear operators
{\small
    \begin{align}\label{L-def}
		\no & L_{0\alpha \beta }(f)= \int_{ \R^3} f_{\beta_*} A_{0\alpha \beta }(\xi-\xi_*)d\xi_*, \ L_{1\alpha \beta }(f)= \iint_{ \R^3\times \R_+} f_{\beta_*} A_{1\alpha \beta }(\xi-\xi_*,I_*)d\xi_* dI_*,\\
		& L_{1\beta \alpha }(f)= \int_{ \R^3} f_{\beta_*} A_{1\beta \alpha }(\xi-\xi_*,I)d\xi_*,\ L_{2\alpha \beta }(f)= \iint_{ \R^3\times \R_+} f_{\beta_*} A_{2\alpha \beta }(\xi-\xi_*,I+I_*)d\xi_* dI_* \,.
	\end{align}}
Actually, the above definitions can represent the loss term $Q_{\alpha \beta}^- (f,f)$ as follows:
  \begin{align}\label{Qn-alpha-beta}
    \no & Q_{\alpha \beta}^- (f,f) = f_\alpha L_{0 \alpha \beta} (f) \,, \quad \alpha, \beta \in \{ 1, \cdots, s_0 \} \,, \\
    \no & Q_{\alpha \beta}^- (f,f) = f_\alpha L_{1 \alpha \beta} (f) \,, \quad \alpha \in \{ 1, \cdots, s_0 \} \,, \beta \in \{ s_0 + 1, \cdots, s \} \,, \\
    \no & Q_{\alpha \beta}^- (f,f) = f_\alpha L_{1 \beta \alpha} (f) \,, \quad \alpha \in \{ s_0 + 1, \cdots, s \} \,, \beta \in \{ 1, \cdots, s_0 \} \,, \\
    & Q_{\alpha \beta}^- (f,f) = f_\alpha L_{2 \alpha \beta} (f) \,, \quad \alpha, \beta \in \{ s_0 + 1, \cdots, s \} \,.
  \end{align}
Furthermore, we also introduce
\begin{equation}\label{L-alpha-def}
  \begin{aligned}
    & L_\alpha (f) = \sum_{\beta = 1}^{s_0} L_{0 \alpha \beta} (f) + \sum_{\beta = s_0 + 1}^{s} L_{1 \alpha \beta} (f) \,, \quad \alpha \in \{ 1, \cdots, s_0 \} \,, \\
    & L_\alpha (f) = \sum_{\beta = 1}^{s_0} L_{1 \beta \alpha} (f) + \sum_{\beta = s_0 + 1}^{s} L_{2 \alpha \beta} (f) \,, \quad \alpha \in \{s_0 + 1, \cdots, s \} \,,
  \end{aligned}
\end{equation}
which implies
\begin{equation}\label{Qn-alpha}
  \begin{aligned}
    Q_\alpha^- (f, f) = \sum_{\beta = 1}^{s} Q_{\alpha \beta}^- (f, f) = f_\alpha L_\alpha (f) \,, \quad \alpha \in \{ 1, \cdots, s \} \,.
  \end{aligned}
\end{equation}

\subsection{Notations and main results}
 
\subsubsection{Notations}

In this paper, the symbol $\R^d$ means the d-dimensional real Euclidean space for integer $d \geq 1$. $\R_+$ stands for the positive half real line, i.e., $\R_+ = \{ x \in \R^1; x > 0 \}$. $\mathbb{S}^2 = \{ x \in \R^3; |x| =1 \}$ is the unit sphere in $\R^3$. $\mathbb{N}$ denotes by the nonnegative integers set, and $\mathbb{N}^k = \{ \boldsymbol{n} = (n_1, \cdots, n_k) | n_i \in \mathbb{N}, 1 \leq i \leq k \}$ for any integer $k \geq 1$. For any $\boldsymbol{n} \in \mathbb{N}^k$, we define its $1$-norm $| \boldsymbol{n} | = \sum_{i=1}^{k} n_i$. The brief notation ``a.e." means ``almost everywhere". $L^p (d \mu)$ represents the standard $L^p$ space associated with the measure $\mu$. For any functional space $F$, the symbol $F_+$ stands for $\{ f \in F; f \geq 0 \textrm{ a.e.} \}$. Moreover, $B_R$ is the ball in $\R^3$ centered at $0$ with radius $R \in (0, \infty)$.
 
We also consider the real Hilbert space
\begin{equation*}
	X:=(L^2(d\mathbf{\xi }))^{s_0}\times (L^2(d\mathbf{\xi }dI))^{s_1},
\end{equation*}
with inner product
\begin{equation}\label{X-innerP}
	(f,g)= \sum \limits_{\alpha =1}^{s_0} \int_{ \R^3}f_\alpha g_\alpha d\xi + \sum \limits_{\alpha =s_0+1}^{s} \iint_{\R^3\times \R_+}f_\alpha g_\alpha d\xi dI,\ f,g \in X.
\end{equation}

For any smooth function $f (t,x,\xi)$ over $ (t,x,\xi) \in \R_+ \times \R^3 \times \R^3 $ and $\boldsymbol{\zeta} = ( \zeta_0, \zeta_1, \cdots, \zeta_6 ) \in \mathbb{N}^7$, we define the partial derivative operator $D^{ \boldsymbol{ \zeta } }$ as
\begin{equation*}
  \begin{aligned}
    D^{ \boldsymbol{ \zeta } } f = \frac{ \partial^{ | \boldsymbol{\zeta} | } f }{  \partial t^{\zeta_0} \partial x_1^{\zeta_1} \partial x_2^{\zeta_2} \partial x_3^{\zeta_3} \partial \xi_1^{\zeta_4} \partial \xi_2^{\zeta_5} \partial \xi_3^{\zeta_6} } \,.
  \end{aligned}
\end{equation*}
Moreover, for any smooth function $g (t,x,\xi, I)$ over $ (t,x,\xi, I) \in \R_+ \times \R^3 \times \R^3 \times \R_+$ and $\boldsymbol{\gamma} = ( \gamma_0, \gamma_1, \cdots, \gamma_6, \gamma_7 ) \in \mathbb{N}^8$, we introduce the partial derivative operator $D^{ \boldsymbol{ \gamma } }$ as
\begin{equation*}
  \begin{aligned}
    D^{ \boldsymbol{ \gamma } } g = \frac{ \partial^{ | \boldsymbol{\gamma} | } g }{  \partial t^{\gamma_0} \partial x_1^{\gamma_1} \partial x_2^{\gamma_2} \partial x_3^{\gamma_3} \partial \xi_1^{\gamma_4} \partial \xi_2^{\gamma_5} \partial \xi_3^{\gamma_6} \partial I^{\gamma_7} } \,.
  \end{aligned}
\end{equation*}
In this paper, all limits $\lim\limits_{n \to + \infty} f^n$ are in the sense of subsequence if necessary.

\subsubsection{Initial data}

In this paper, we focus on the Cauchy problem of the equation \eqref{BE-MP}, whose initial data is 
\begin{equation}\label{f0}
		f ( t, x, \mathbf{Z} ) |_{t=0} = f_0 (x, \mathbf{Z}) = (f_{1,0},\cdots,f_{s,0}) (x, \mathbf{Z})
\end{equation}
with the following assumptions:
\begin{enumerate}
    \item For all $\alpha \in \{ 1, \cdots, s \}$,
    \begin{equation}\label{f0-p}
        \begin{aligned}
            f_{\alpha, 0} \geq 0 \quad a.e. \ (x, \mathbf{Z}) \in \R^3 \times \mathcal{Z_\alpha} \,,
        \end{aligned}
    \end{equation}
    where $\mathcal{Z_\alpha}$ is given in \eqref{Z-gamma};
    \item The components $f_{\alpha, 0}$ of the initial data $f_0$ satisfy:
    \begin{itemize}
        \item If $\alpha \in \{1,\cdots ,s_0\}$,
        \begin{equation}\label{fs0}
            \iint_{\R^3 \times \R^3} f_{\alpha,0} (1+|x|^2 +|\xi|^2 +|\log f_{\alpha,0}|)dxd\xi <\infty \,.
        \end{equation}
        \item If $\alpha \in \{s_0+1,\cdots ,s\} $,
        \begin{equation}\label{fs}
            \iiint_{\R^3 \times \R^3 \times \R_+} f_{\alpha,0} (1+|x|^2 +|\xi|^2+I +|\log(I^{1-\delta(\alpha)/2} f_{\alpha,0})|)dxd\xi dI<\infty.
    \end{equation}
    \end{itemize}
\end{enumerate}

\subsubsection{Hypotheses on collision kernels}

Based on the angular cutoff assumption \eqref{Asum-angular-cutoff}, we have introduced the average collision kernel functions $A_{i \alpha \beta}$ ($i=0,1,2$) and $A_{1 \beta \alpha }$ in \eqref{A}. The following hypotheses about $A_{i \alpha \beta}$ ($i=0,1,2$) and $A_{1 \beta \alpha }$ will be imposed: for any fixed $0 < R < + \infty$,
\begin{equation}\label{collision}
    \begin{aligned}
        & ( 1 + | \xi |^2 )^{-1} \int_{ | z - \xi | \leq R } A_{ 0 \alpha \beta } (z) d z \rightarrow 0, \quad | \xi | \rightarrow \infty, \\
        & ( 1 + | \xi |^2 +  \eta )^{-1} \int_{ | z -\xi | \leq R } A_{ 1 \alpha \beta } ( z, \eta ) d z \rightarrow 0,  \quad | \xi |,\eta  \rightarrow \infty, \\
        & ( 1 + | \xi |^2 )^{-1} \int_{ | z - \xi | \leq R } \int_{ 0 < \eta < R } A_{ 1 \beta \alpha } ( z, \eta ) d z d \eta \rightarrow 0, \quad | \xi | \rightarrow \infty , \\
        & ( 1 + | \xi |^2 +  I )^{-1} \int_{ | z - \xi | \leq R } \int_{ 0 < \eta - I < R } A_{ 2 \alpha \beta } ( z , \eta ) d z d \eta \rightarrow 0, \quad | \xi |, I \rightarrow \infty .
    \end{aligned}
\end{equation}

\subsubsection{Definition of renormalized solutions}

\begin{definition}\label{Def-RNS}
	We say that $f=(f_1,\cdots,f_s)$ is a renormalized solution of \eqref{BE-MP} if it satisfies 
	\begin{enumerate}
		\item {For $\alpha \in \{1,\cdots,s_0\}$, $g_\alpha = \beta (f_\alpha) = \log (1+f_\alpha)$ solves
			\begin{equation*}
				\tfrac{\partial g_\alpha}{\partial t } + \xi \cdot \nabla_{x} g_\alpha =\tfrac{1}{1+f_\alpha} Q_\alpha (f,f) \quad \text{in } \mathscr{D}'((0,\infty)\times \R^3 \times \R^3),
		\end{equation*} 
and $\frac{1}{1 + f_\alpha} Q_\alpha^\pm (f,f) \in L_{loc}^1 ( (0, \infty) \times \R^3 \times \R^3 )$;}
		\item {For $\alpha \in \{s_0+1,\cdots ,s\}$, $g_\alpha$ solves
			\begin{equation*}
				\tfrac{\partial g_\alpha}{\partial t } + \xi \cdot \nabla_{x} g_\alpha =\tfrac{1}{1+f_\alpha} Q_\alpha (f,f) \quad \text{in } \mathscr{D}'((0,\infty)\times \R^3 \times \R^3\times \R_+),
		\end{equation*}
and $\frac{1}{1 + f_\alpha} Q_\alpha^\pm (f,f) \in L_{loc}^1 ( (0, \infty) \times \R^3 \times \R^3 \times \R_+ )$.}
	\end{enumerate}
\end{definition}

\begin{remark}
  As shown in Lemma \ref{Lemma 3.1} later, the definition of the renormalized solution above is equivalent to that for all $\beta (\cdot) \in C^1 [0, + \infty )$ such that $|\beta' (t)| \leq \tfrac{C}{1 + t}$,
  \begin{equation*}
    \begin{aligned}
      \tfrac{\partial \beta (f_\alpha)}{\partial t} + \xi \cdot \nabla_x \beta (f_\alpha) = \beta' (f_\alpha) Q_\alpha (f,f) \ (\alpha \in \{ 1,\
       \cdots, s \})
    \end{aligned}
  \end{equation*}
  hold in the sense of distribution.
\end{remark}

Next we introduce Boltzmann-type H-functional by
\begin{equation}\label{Hf}
	\begin{aligned}
		H(f)(t)& = (f,\log (\varphi^{-1} f) ) \\
		&=\sum \limits_{\alpha=1}^{s_0} \iint_{\R^3 \times \R^3 } f_\alpha \log f_\alpha  dxd\xi + \sum \limits_{\alpha = s_0+1}^{s}\iiint_{\R^3 \times \R^3 \times \R_+ } f_\alpha \log (I^{1-\delta(\alpha)/2} f_\alpha ) dxd\xi dI.
	\end{aligned}
\end{equation}

Now we state the main result of current work.

\begin{theorem}\label{MainThm}
Assume that the collision kernels satisfy \eqref{Asum-angular-cutoff} and \eqref{collision}. Let the initial data $f_0 = (f_{1,0}, \cdots, f_{s,0})$ obeys \eqref{f0-p}, \eqref{fs0} and \eqref{fs}. Then the Cauchy problem \eqref{BE-MP}-\eqref{f0} admits a renormalized solution $f=(f_1,\cdots,f_s)$ enjoying the following properties: For any given $T<\infty$ and $0 < R < \infty$, there is a $C_T > 0$ such that
\begin{enumerate}
	\item {for $\alpha \in \{1,\cdots,s_0\},f_\alpha \in C([0,\infty); L^1 (\R^3 \times \R^3))_+,$
	\begin{equation}\label{Qa}
		\begin{aligned}
			&\tfrac{1}{1+f_\alpha} Q^-_{\alpha} (f,f) \in L^\infty ([0,\infty);L^1 (\R^3 \times B_R)), \ \tfrac{1}{1+f_\alpha} Q^+_{\alpha} (f,f) \in L^1 ([0,T]\times \R^3 \times B_R),
		\end{aligned}
	\end{equation}
	\begin{equation}\label{fa}
	\iint_{\R^3 \times \R^3 } f_\alpha (1+|x|^2 +|\xi|^2 +|\log f_\alpha|)dxd\xi \leq C_T;
	\end{equation}	}
	\item{for $\alpha \in \{s_0+1,\cdots,s\},f_\alpha \in C([0,\infty); L^1 (\R^3 \times \R^3\times \R_+))_+,$
		\begin{equation}\label{Qb}
			\begin{aligned}
				&\tfrac{1}{1+f_\alpha} Q^-_{\alpha} (f,f) \in L^\infty ([0,\infty);L^1 (\R^3 \times B_R \times (0,R))), \\
				&\tfrac{1}{1+f_\alpha} Q^+_{\alpha} (f,f) \in L^1 ([0,T] \times \R^3 \times B_R \times (0,R)),
			\end{aligned}
		\end{equation}
	\begin{equation}\label{fb}
\iiint_{\R^3 \times \R^3 \times \R_+} f_\alpha (1+|x|^2 +|\xi|^2 +I + |\log (I^{1-\delta(\alpha)/2}f_\alpha)|)dxd\xi dI<C_T;
		\end{equation}}
	\item{ The following entropy inequality holds:
	\begin{equation}\label{entropy-th}
	\begin{aligned}
		&H(f)(t)+ \tfrac{1}{4}\sum \limits_{\alpha =1}^{s_0} \int_{0}^{t} d \tau \iint_{ \R^3 \times \R^3 } e_\alpha(\tau,x,\xi ) dxd\xi\\
		& +\tfrac{1}{4}\sum \limits_{\alpha =s_0+1}^{s}\int_{0}^{t} d \tau \iiint_{ \R^3 \times \R^3 \times \R_+}  e_\alpha(\tau,x,\xi ,I) dxd\xi dI\leq H(f_0),
	\end{aligned}
	\end{equation}
	where 
	\begin{equation}\label{ee}
		\begin{aligned}
			e_\alpha&=\sum \limits_{\beta =1}^s \int_{(\R^3\times \R_+)^3} (\tfrac{I^{\delta(\alpha)/2-1} I_* ^{\delta(\beta)/2-1}f'_\alpha f'_{\beta_*}  } {f_\alpha f_{\beta_* } (I')^{\delta(\alpha)/2-1} (I'_* )^{\delta(\beta)/2-1}} - 1)\\
			&\times \log (\tfrac{I^{\delta(\alpha)/2-1} I_* ^{\delta(\beta)/2-1} f'_\alpha f'_{\beta_*}}{f_\alpha f_{\beta_* } (I')^{\delta(\alpha)/2-1} (I'_* )^{\delta(\beta)/2-1} } ) \tfrac{f_\alpha f_{\beta_* } }{I^{\delta(\alpha)/2-1} I_* ^{\delta(\beta)/2-1} }W_{\alpha\beta }d\xi_* d\xi'd\xi'_* dI_* dI' dI'_*.
		\end{aligned}
	\end{equation}	}
\end{enumerate}
\end{theorem}

\subsection{Sketch of proofs}

In this subsection, we mainly sketch the ideas of proving the main theorem. Since the Boltzmann collision operator $Q (f,f)$ is not integrable, it is impossible to prove the weak solution of the problem \eqref{BE-MP}. In 1989, DiPerna-Lions \cite{DiPerna-Lions-Inv} introduced the so-called renormalized solution to the transported equations. Then they also defined the renormalized solution to the classical Boltzmann equation and proved its global existence with large initial data in \cite{Diperna-Lions}. Inspired by their works, we first give the definition of the renormalized solution to the Boltzmann system \eqref{BE-MP} in Definition \ref{Def-RNS}, in which the factor $\frac{1}{1 + f_\alpha}$ in the $f_\alpha$-equation is such that $ \frac{1}{1 + f_\alpha} Q_\alpha (f, f) $ is locally integrable. Thus the weak solutions of the renormalized form can be introduced. We remark that the equivalent expressions of renormalized solution are given in Lemma \ref{Lemma 3.1}. The main goal of current paper is to verify the global existence of the renormalized solution with large initial data.

{\bf Step 1. Construction of smooth approximation problem \eqref{approximation}.} While constructing the smoothly approximated problem \eqref{approximation} of the original problem \eqref{BE-MP}, there are four key ingredients: 1) smooth approximation of initial data; 2) smooth approximation of collision kernels; 3) existence, uniqueness and positivity of smooth approximated problem \eqref{approximation}; 4) uniform bounds of the approximated problem in the a priori spaces.

We initially construct the smoothly approximated initial data $f^n_0 = (f^n_{\alpha, 0})_{ \alpha \{ 1, \cdots, s \} }$ in the Schwartz space with positive components as in Lemma \ref{Lemma 4.8} and \ref{Lemma 4.9}. The constructed approximated initial data subject to the integrability \eqref{fs0}-\eqref{fs} uniformly in $n \geq 1$. Moreover, by using the weak lower semicontinuity of the entropy product $H(f)$ given in Lemma \ref{Lemma 4.6}, the approximated initial data constructed above also converge to the initial data $f_0$ in the sense of entropy product, see Lemma \ref{Lmm-Ini-Entpy}.

We next construct the smoothly approximated collision operator $\tilde{Q}_{n} (f^n, f^n)$ beginning with the collision kernels approximation. Note that the collision kernels $B_{2 \alpha \beta}$, $B_{1 \alpha \beta}$, $B_{1 \beta \alpha}$ and $B_{0 \alpha \beta}$ may possess the singularity at $|z|, \eta = 0, \infty$ and the deflection angular $\theta = 0$, i.e., $z \cdot \omega = 0$. As in Lemma \ref{Lemma 4.10}-\ref{Lemma 4.11}-\ref{Lemma 4.12}, we take zero-approximation near the underlying singular points and use the smooth modifiers, so that the approximated collision kernels $B_{2 \alpha \beta, n}$, $B_{1 \alpha \beta, n}$, $B_{1 \beta \alpha, n}$ and $B_{0 \alpha \beta, n}$ pointwisely converge to the original collision kernels as $n \to + \infty$. Moreover, together with the angular cutoff assumption \eqref{Asum-angular-cutoff}, the reduced collision kernels $A_{2 \alpha \beta, n}$, $A_{1 \alpha \beta, n}$, $A_{1 \beta \alpha, n}$ and $A_{0 \alpha \beta, n}$ obey the assumptions \eqref{collision} uniformly in $n \geq 1$. Consequently, one can construct the approximated collision operator $\tilde{Q}_n (f^n, f^n) = ( \tilde{Q}_{\alpha, n} (f^n, f^n) )_{\alpha \in \{ 1, \cdots, s \}}$ in \eqref{Q L}. Remark that the factor $N_n (f^n)^{-1}$ given in \eqref{Nnf} is such that the nonlinear approximated collision operator $\tilde{Q}_n (f^n, f^n)$ admits linear upper bounds in $L^1 \cap L^\infty$ as in Lemma \ref{Lemma 4.13} and \ref{Lemma 4.13 of infty}. Moreover, Lemma \ref{Lemma 4.13} shows that $\tilde{Q}_n (f^n, f^n)$ is Lipschitz continuous $L^1$ space, which will be used to prove the well-posedness of the approximated problem \eqref{approximation}. On the $L^1$ estimates, the $L^1$ norm of the loss term $\tilde{Q}^-_n (f^n, f^n) = f^n \tilde{L}_{n} (f^n)$ can easily be bounded by the $L^1$ norm of $f^n$ due to $\tilde{L}_{n} (f^n) \in L^\infty$. In order to estimate the $L^1$ norm of the gain term $\tilde{Q}^+_n (f^n, f^n)$, we shall utilize the relation $ \iint_{ \R^3\times \R_+}Q^+_{\alpha \beta ,n}(f,f) \varphi d \mathbf{Z}_\alpha = \iint_{ \R^3\times \R_+} Q^-_{ \alpha \beta,n }(f,f) \varphi' d \mathbf{Z}_\alpha $ to shift the gain term to the loss term, so that the gain term can be bounded by using the same arguments of the loss term. About the $L^\infty$ estimates, the $L^\infty$ norm of the loss term $\tilde{Q}^-_n (f^n, f^n)$ can be obtained by the analogous arguments of $L^1$ norm of loss term. For the gain term, the key is to obtain the bound \eqref{Q+-Linfty}, in which the quantity $I_{2 \alpha \beta, n}$ is dominated by \eqref{I-2abn-1}. For the sake of controlling $I_{2 \alpha \beta, n}$ by $L^1$ norm of $f^n_\beta$, we should discuss the internal energy gap $\Delta I = \mathbf{1}_{ \alpha \in \{ s_0 + 1, \cdots, s \} } I' + \mathbf{1}_{ \beta \in \{ s_0 + 1, \cdots, s \} }  I_*' - \mathbf{1}_{ \alpha \in \{ s_0 + 1, \cdots, s \} } I - \mathbf{1}_{ \beta \in \{ s_0 + 1, \cdots, s \} }  I_* $ by the cases $\Delta I \leq 0$ and $\Delta I > 0$ as in Lemma \ref{Lemma 4.13 of infty}. Together with $N_n (f^n)^{-1}$, the $L^\infty$ norm of the gain term can be bounded by that of $f^n$.

We then prove the well-posedness, positivity and regularity of the approximated problem \eqref{approximation}. Together with Lemma \ref{Lemma 4.13}, the operator semigroup theory shows the existence and uniqueness of \eqref{approximation} in $ C ( [0, + \infty); \prod_{\alpha = 1}^{s} L^1 (\R^3 \times \mathcal{Z}_\alpha; d x d \mathbf{Z}_\alpha) ) $. For the positivity, we first consider the auxiliary Cauchy problem \eqref{g}, i.e.,
\begin{equation*}
  \begin{aligned}
    \tfrac{\partial}{\partial t} g^n + \xi \cdot \nabla_x g^n = F (g^n) \,, \ g^n |_{t = 0} = f_0 \,.
  \end{aligned}
\end{equation*}
Following the same arguments of well-posedness of \eqref{approximation}, the above auxiliary problem for $g^n$ is also well-posed in $ C ( [0, + \infty); \prod_{\alpha = 1}^{s} L^1 (\R^3 \times \mathcal{Z}_\alpha; d x d \mathbf{Z}_\alpha) ) $. Then mild form of $g^n$ shows that the components of $g^n$ are nonnegative under the positive components of initial data $f_0$. As a result, the $g^n$ equation is exactly the approximated $f^n$ equation. Thus, the Gr\"onwall inequality for the approximated $f^n$ equation can prove the positive lower bounds of the components of $f^n$. Furthermore, as in Lemma \ref{Lemma 4.16}, the mild form \eqref{fn-alpha} of $f^n_\alpha$ and Lemma \ref{Lemma 4.13}-\ref{Lemma 4.13 of infty} can imply the regularity of the approximated solution $f^n$ by employing the Induction Principle for the orders of derivatives.

We finally derive the uniform bounds of the approximated solution $f^n$ as in Lemma \ref{Lemma 4.17}. Note that the collision kernels $B_{2 \alpha \beta ,n}$, $B_{1 \alpha \beta, n}$, $B_{1 \beta \alpha, n}$ and $B_{0 \alpha \beta, n}$ constructed in Lemmas \ref{Lemma 4.10}-\ref{Lemma 4.11}-\ref{Lemma 4.12} enjoy the same symmetry properties of $B_{2 \alpha \beta }$, $B_{1 \alpha \beta}$, $B_{1 \beta \alpha}$ and $B_{0 \alpha \beta}$ in \eqref{B0-ab}-\eqref{B1-ab}-\eqref{B1-ba}-\eqref{B2-ab}. Then the all structural properties of the collision operator $Q$ will also be valid for the approximated collision operator $\tilde{Q}_n = ( \tilde{Q}_{\alpha, n} )_{ \alpha \in \{ 1, \cdots, s \} }$ constructed in \eqref{Q L}. Then the collision invariant form of $Q$
	\begin{equation*}
		\varphi_i= ae_i + b \mathbf{m} \xi +c (\mathbf{m} |\xi|^2+2\mathbf{I}),\quad i=1,\cdots,s 
	\end{equation*}
is also that of $\tilde{Q}_n$, where $a,c\in \R $ and $ b\in \R^3$ are arbitrary, see Proposition \ref{prop2.3}. Consequently, one can easily obtain the required uniform bounds.

{\bf Step 2. Convergence of nonlinear approximated collision operator.} By the uniform bounds given in Lemma \ref{Lemma 4.17} and the sufficient conditions of relatively weak $L^1$ compactness in Lemma \ref{rwc} and Corollary \ref{rwc1}, we know that every component $f^n_\alpha$ of $f^n$ is weakly convergent to a nonnegative function $f_\alpha$ in $L^1$. The difficulty is to prove the convergence of the nonlinear collision operator $\tilde{Q}_n (f^n, f^n)$. 

First, we prove the weak $L^1$ compactness of the nonlinear collision operator $\tilde{Q}_n (f^n, f^n)$. Note that the loss term $\tilde{Q}_{\alpha, n}^- (f^n, f^n) = f^n_\alpha \tilde{L}_{\alpha, n} (f^n)$ has simpler structure than the gain term $\tilde{Q}_{\alpha, n}^+ (f^n, f^n)$, so that we initially study the weak compactness of the loss term as in Lemma \ref{theo-4.1}. Due to $0 \leq \frac{1}{1 + f^n_\alpha} \tilde{Q}^-_{\alpha, n} (f^n, f^n) \leq L_{\alpha, n} (f^n) = \sum_{\beta = 1}^s L_{\alpha \beta, n} (f^n)$, we only need to show the relatively weak compactness of $ L_{\alpha \beta, n} (f^n) $ in $L^1$. Instead, we consider the cutoff operator $ L^m_{\alpha \beta, n} (f^n) = \iint_{\R^3 \times \R_+} A_{\alpha \beta, n} (\xi - \xi_*, I + I_*) \mathbf{1}_{|\xi - \xi_*| \leq m} \mathbf{1}_{I+I_* < m} f^n_{\beta *} d \xi_* d I_* $ for any fixed $m > 0$. One can prove that $ L^m_{\alpha \beta, n} (f^n) $ satisfy \eqref{Psi-Lnab}, i.e.,
\begin{equation*}
  \iiint_{\R^3 \times B_R \times(0,R)} \Psi (L^m_{ \alpha \beta ,n}(f^n)) (t) dxd\xi dI \leq C(T,m) 
\end{equation*}
with $\Psi (t) = t (\log t)^+$. Then De La Vall\'ee-Poussin Criterion (Lemma \ref{theo-DE La}) tells us that $ L^m_{\alpha \beta, n} (f^n) $ is equi-integrable in $L^1$. The corresponding tightness is easy to be verified. The Dunford-Pettis Theorem (Lemma \ref{theorem-dunford}) implies the relatively weak $L^1$ compactness of $L^m_{ \alpha \beta ,n}(f^n)$. Moreover, by using the assumptions \eqref{collision}, one gains that $L^m_{ \alpha \beta ,n}(f^n)$ strongly converges to $L_{ \alpha \beta ,n}(f^n)$ as $m \to + \infty$, which finishes the proof of Lemma \ref{theo-4.1}. While proving the relatively weak $L^1$ compactness of $\frac{1}{1 + f^n_\alpha} \tilde{Q}^+_{\alpha, n} (f^n, f^n)$ as in Lemma \ref{theo-4.2}, one mainly needs to show the Arkeryd-type inequality
\begin{equation*}
  \tilde{Q}^+_{\alpha, n} (f^n, f^n) \leq \tfrac{1}{\log K} \tilde{e}^n_\alpha + K \tilde{Q}^-_{\alpha, n} (f^n, f^n)
\end{equation*}
for any fixed $K > 1$. Then, together with Lemma \ref{Lemma 4.17} and Lemma \ref{theo-4.1}, the conclusions in Lemma \ref{theo-4.2} can be obtained.

Second, we study the properties of limits of $f^n$. The key is to establish the averaged velocity (-internal energy) Lemma \ref{theo-6.2}-\ref{theo-6.2.1}-\ref{theo-6.4}-\ref{theo-6.4.1}, whose proofs come from the works \cite{Diperna-Lions,Golse-velocity,Golse-perthame-velocity}. By now, we do not prove that the approximated solution $f^n$ satisfies the conditions in the averaged velocity (-internal energy) lemmas. However, by Lemma \ref{theo-4.1}-\ref{theo-4.2}, the renormalized form $g^n_{\delta \alpha} = \frac{1}{\delta} \log (1 + \delta f^n_\alpha)$ of $f^n_\alpha$ exactly subjects to the required conditions, so that $ \int g^n_{\delta \alpha} \varphi $ strongly converges to $\int \tilde{g}_{\delta \alpha} \varphi$ in $L^1 ((0, T) \times \R^3)$. Fortunately, the strongly convergent relations between $ g^n_{\delta \alpha} $ and $f^n_\alpha$ are established in Lemma \ref{converge f,g}. Then the enhanced averaged velocity (-internal energy) lemma are derived in Lemma \ref{theo-6.5} and Corollary \ref{cor6.6}, hence, $\int f^n_\alpha \varphi$ strongly converges to $\int f_\alpha \varphi$ in $L^p ( (0, T); L^1 (\R^3) )$ for $1 \leq p < \infty$. Moreover, together Lemma \ref{converge f,g} and the renormalized approximated solution $g^n_{\delta \alpha}$, as in Lemma \ref{cor6.12}, we can prove that the limit function $f_\alpha \in C([0,\infty), L^1 (\R^3 \times \mathcal{Z}_\alpha))_+$ enjoys the bounds \eqref{fa} and \eqref{fb} in Theorem \ref{MainThm}. The core is to employ the mild form of $g^n_{\delta \alpha}$.

Third, we investigate the properties of limits of $\tilde{Q}_n (f^n, f^n )$. As stated before, the loss term $\tilde{Q}_{\alpha, n}^- (f^n, f^n) = f^n_\alpha \tilde{L}_{\alpha, n} (f^n)$ admits a simpler structure than the gain term. So we initially consider the loss term whose main ingredient is the operator $\tilde{L}_{\alpha, n} (f^n)$. As in Lemma \ref{converge L}, we can show that $ \tilde{L}_{\alpha, n} (f^n) $ strongly converges to $ L_\alpha (f) $ in $L^1$ by applying the averaged velocity (-internal energy) lemmas and assumptions \eqref{collision}. Together with the enhanced averaged velocity (-internal energy) lemmas (Lemma \ref{theo-6.5} and Corollary \ref{cor6.6}) and the Arkeryd-type inequality, Lemma \ref{theo-6.7} and Corollary \ref{cor6.8} indicate that $\int \tilde{Q}^\pm_{\alpha, n} (f^n, f^n) \varphi \mathbf{1}_{E}$ strongly converges to $\int Q^\pm_\alpha (f,f) \varphi \mathbf{1}_{E}$ in $L^1 ((0, T) \times B_R)$ for some finite measured Borel subset $E$, in which $ \tilde{L}_{\alpha, n} (f^n) $ is actually uniformly converges to $L_\alpha (f)$. Moreover, based on Lemma \ref{theo-6.7} and Corollary \ref{cor6.8}, we can prove the following two convergent results: 1) For $\beta_\delta (t) = t \wedge \frac{1}{\delta}$ or $\beta_\delta (t) = \frac{t}{1 + \delta t}$, $\tilde{Q}^\pm_{\alpha, n} (\beta_\delta (f^n), \beta_\delta (f^n))$ weakly converges to $Q^\pm (\tilde{\beta}_\delta, \tilde{\beta}_\delta)$ in $L^1$, where w-$\lim_{n \to + \infty} \beta_\delta (f^n) = \tilde{\beta}_\delta$ (see Lemma \ref{theo-6.9} and Corollary \ref{cor6.9}). 2) The sequence $\int \frac{ \tilde{Q}^\pm_{\alpha, n} (f^n, f^n) }{1 + \delta \tilde{L}_{\alpha, n} (f^n) } \varphi$ strongly converges to $ \int \frac{ Q^\pm_{\alpha} (f, f) }{1 + \delta L_{\alpha} (f) } \varphi $ in $ L^1 ( (0, T) \times \R^3 ) $, see Lemma \ref{QL L1} and Corollary \ref{corQLL1}. By employing the all above properties of limits, in Lemma \ref{Q in L1} and Corollary \ref{Q-L-loc} we can show that $\frac{Q_\alpha^- (f, f) }{1 + f_\alpha} \in L^\infty_t L^1_{loc}$ and $\frac{Q^+_\alpha (f, f)}{1 + f_\alpha} \in L^1_t L^1_{loc}$, hence, the integrability \eqref{Qa} and \eqref{Qb} in Theorem \ref{MainThm} is valid.

{\bf Step 3. Existence of renormalized solution to \eqref{BE-MP}.} As in Lemma \ref{distri-mild}, $f$ is a renormalized solution to \eqref{BE-MP} if and only if $f$ is a mild solution with $\frac{1}{1 + f_\alpha} Q^\pm_\alpha (f, f) \in L^1_{loc}$. Lemma \ref{Q in L1} and Corollary \ref{Q-L-loc} have been shown that $\frac{1}{1 + f_\alpha} Q^\pm_\alpha (f, f) \in L^1_{loc}$. We thus only need to prove that the limit $f$ is a mild solution. By Lemma \ref{converge L}, $L_\alpha (f) \in L^1 ( (0, T) \times \R^3 \times B_R )$ for $\alpha \in \{ 1, \cdots, s_0 \}$ and $L_\alpha (f) \in L^1 ( (0, T) \times \R^3 \times B_R \times (0, R) )$ for $\alpha \in \{ s_0 + 1, \cdots, s \}$. Then Lemma \ref{Lemma3.3} tells us that we only need to prove that $f$ satisfies the identities \eqref{3.5} and \eqref{3.6}, namely,
\begin{equation*}
  \begin{aligned}
    &\mathcal{T} (f_\alpha^\sharp) : = f^\sharp_\alpha (t,x,\mathbf{Z}_\alpha) -f^\sharp_\alpha (s,x,\mathbf{Z}_\alpha) \exp(-(F^\sharp _{\alpha} (t)-F^\sharp _{\alpha} (s)))\\
	=&\int_{s}^{t} Q^+_{\alpha } (f,f) ^\sharp (s,x,\mathbf{Z}_\alpha) \exp(-(F^\sharp _{\alpha } (t)- F^\sharp _{\alpha } (\sigma))) d\sigma : = \mathcal{Q} (f_\alpha^\sharp) \quad  a.e.\ x,\mathbf{Z}_\alpha.
  \end{aligned}
\end{equation*}

We first prove the limit $f$ is a supersolution to \eqref{BE-MP} given in Definition \ref{SupperSolution}, i.e., $f$ obeys the inequality $ \mathcal{Q} (f_\alpha^\sharp) \leq \mathcal{T} (f_\alpha^\sharp) $. We start with the approximated identities \eqref{eq up3}-\eqref{eq up4}, i.e.,
\begin{equation*}
  \begin{aligned}
    &f^{n\sharp}_\alpha (t,x,\mathbf{Z}_\alpha) -f^n_{\alpha ,0} (x,\mathbf{Z}_\alpha )\exp(-\tilde{F}^\sharp _{\alpha ,n } (t,x,\mathbf{Z}_\alpha))\\
		=&\int_{0}^{t}\tilde{Q}^+_{\alpha ,n} (f^n,f^n)^\sharp (s ,x,\mathbf{Z}_\alpha) \exp(-(\tilde{F}^\sharp _{\alpha ,n} (t,x,\mathbf{Z}_\alpha )-\tilde{F}^\sharp _{\alpha ,n} (s,x,\mathbf{Z}_\alpha))) ds \,.
  \end{aligned}
\end{equation*}
Lemma \ref{Lemma 7.1} easily implies that $f^{n\sharp}_\alpha (t,x,\mathbf{Z}_\alpha) -f^n_{\alpha ,0} (x,\mathbf{Z}_\alpha )\exp(-\tilde{F}^\sharp _{\alpha ,n } (t,x,\mathbf{Z}_\alpha))$ weakly converges to $\mathcal{T} (f_\alpha^\sharp)$ in $L^1$. However, there is no direct convergence on the gain term $\tilde{Q}^+_{\alpha ,n} (f^n,f^n)$ (all results are associated with various constraints). We then consider $h^n_M = ( h^n_{\alpha M} )_{ \alpha \in \{ 1, \cdots, s \} }$ with $h^n_{\alpha M} = f^n_\alpha \wedge M$, which means that $ \tilde{Q}^+_{\alpha ,n} (h^n_M,h^n_M) \leq \tilde{Q}^+_{\alpha ,n} (f^n,f^n) $ for $M > 0$. Lemma \ref{Lemma 7.2} infers that $ \int_{0}^{t}\tilde{Q}^+_{\alpha ,n} (h^n_M,h^n_M)^\sharp (s ,x,\mathbf{Z}_\alpha) \exp(-(\tilde{F}^\sharp _{\alpha ,n} (t,x,\mathbf{Z}_\alpha )-\tilde{F}^\sharp _{\alpha ,n} (s,x,\mathbf{Z}_\alpha))) ds $ weakly converges to $ \mathcal{Q} (h_{\alpha M}^\sharp) $ in $L^1$ by using Lemma \ref{theo-6.9} and Corollary \ref{cor6.9}. Here $h_{ \alpha M }$ is the weak $L^1$ limit of $h^n_{\alpha M}$. Then one has $\mathcal{Q} (h_{\alpha M}^\sharp) \leq \mathcal{T} (f_{\alpha}^\sharp)$. Moreover, \eqref{hM to f} shows that $h_{\alpha M}$ strongly converges to $f_\alpha$ in $L^1$ as $M \to + \infty$. It then implies that $ \mathcal{Q} (f_\alpha^\sharp) \leq \mathcal{T} (f_\alpha^\sharp) $, namely, $f$ is a suppersolution of \eqref{BE-MP}, see Lemma \ref{theo-7.3}.

We second show the limit $f$ is a subsolution to \eqref{BE-MP} given in Definition \ref{SubSolution}, hence, $ \mathcal{Q} (f_\alpha^\sharp) \geq \mathcal{T} (f_\alpha^\sharp) $. We mainly focus on the approximated renormalized equation \eqref{eq6.75}, which can be equivalently expressed by \eqref{eq8.93}, i.e.,
\begin{equation*}
  \begin{aligned}
    g^{n\sharp}_{\delta \alpha } (t) -g^n_{\delta \alpha,0}  \exp(-\tilde{F}^\sharp _{\alpha ,n } (t)) = \int_{0}^{t} \{\tfrac{\tilde{Q}^+_{\alpha ,n} (f^n,f^n)^\sharp }{1+\delta f^{n\sharp }_\alpha} +j^{n\sharp}_{\delta \alpha }\tilde{L}_{\alpha ,n}(f^n)^\sharp\} (\tau) \exp(-(\tilde{F}^\sharp _{\alpha ,n } (t)- \tilde{F}^\sharp _{\alpha ,n } (\tau))) d \tau ,
  \end{aligned}
\end{equation*}
where $j^{n}_{\delta \alpha } =g^{n}_{\delta \alpha } -\frac{f^n_\alpha}{1+\delta f^n_\alpha}$. By the similar arguments in Lemma \ref{Lemma 7.1}, one sees that $ g^{n\sharp}_{\delta \alpha } (t) -g^n_{\delta \alpha,0}  \exp(-\tilde{F}^\sharp _{\alpha, n } (t)) $ weakly converges to $\tilde{g}^{\sharp}_{\delta \alpha } (t) -\tilde{g}_{\delta \alpha,0}  \exp(- F^\sharp _{\alpha } (t))$ in $L^1$, where $\tilde{g}_{\delta \alpha }$ is the weak limit of $g^{n}_{\delta \alpha }$. Lemma \ref{jnw} then shows that $ \int_{0}^{t} j^{n\sharp}_{\delta \alpha }\tilde{L}_{\alpha ,n}(f^n)^\sharp (\tau) \exp(-(\tilde{F}^\sharp _{\alpha ,n } (t)- \tilde{F}^\sharp _{\alpha ,n } (\tau))) d \tau $ weakly converges to $ \int_{0}^{t} j^{\sharp}_{\delta \alpha } L_{\alpha}(f)^\sharp (\tau) \exp(-({F}^\sharp _{\alpha} (t)- {F}^\sharp _{\alpha} (\tau))) d \tau $ in $L^1$, where $ j_{\delta \alpha } $ is the weak limit of $j^{n}_{\delta \alpha }$. Lemma \ref{theo-4.2} implies that $ \tfrac{\tilde{Q}^+_{\alpha ,n} (f^n,f^n)^\sharp }{1+\delta f^{n\sharp }_\alpha} $ weakly converges to some function $ Q^+_{\delta \alpha} $ in $L^1$. Together with Lemma \ref{theo-6.7} and Corollary \ref{cor6.8}, we can prove that $ Q^+_{\delta \alpha} \leq Q^+_\alpha (f, f) $ almost everywhere. As a result, one has
\begin{equation*}
	\begin{aligned}
		&\tilde{g}^{\sharp}_{\delta \alpha } (t,x,\mathbf{Z}_\alpha) -\tilde{g}_{\delta \alpha,0} (x,\mathbf{Z}_\alpha) \exp(-F^\sharp _{\alpha} (t,x,\mathbf{Z}_\alpha)) \\
		\leq& \int_{0}^{t} \{j^{\sharp}_{\delta \alpha }(\tau) L_\alpha (f)^\sharp(\tau)+Q^+_{\alpha } (f,f) ^\sharp (\tau)\} \exp(-(F^\sharp _{\alpha } (t)- F^\sharp _{\alpha } (\tau))) d \tau \ a.e.
	\end{aligned}
\end{equation*}
Lemma \ref{converge f,g} tells us that $ \tilde{g}^{\sharp}_{\delta \alpha } (t,x,\mathbf{Z}_\alpha) -\tilde{g}_{\delta \alpha,0} (x,\mathbf{Z}_\alpha) \exp(-F^\sharp _{\alpha} (t,x,\mathbf{Z}_\alpha)) $ pointwisely converges to $f^{\sharp}_{\alpha } (t,x,\mathbf{Z}_\alpha) - f_{\alpha,0} (x,\mathbf{Z}_\alpha) \exp(-F^\sharp _{\alpha} (t,x,\mathbf{Z}_\alpha))$ as $\delta \to 0^+$. Moreover, the relation \eqref{j-0} reads that $ \int_{0}^{t} j^{\sharp}_{\delta \alpha }(\tau) L_\alpha (f)^\sharp(\tau) \exp(-(F^\sharp _{\alpha } (t)- F^\sharp _{\alpha } (\tau))) d \tau \to 0 $ pointwisely as $\delta \to 0^+$. Consequently, $ \mathcal{Q} (f_\alpha^\sharp) \geq \mathcal{T} (f_\alpha^\sharp) $, hence, $f$ is the subsolution to \eqref{BE-MP}. Namely, we prove that $f$ is the renormalized solution to \eqref{BE-MP}.

{\bf Step 4. Entropy inequality.} In order to prove the entropy inequality \eqref{entropy-th} from the approximated entropy identity \eqref{e na}, it suffices to verify \eqref{HH-limit}, i.e., $ \liminf_{n \to + \infty} \tilde{H}_{\alpha \beta }^n (f^n) \geq H_{\alpha \beta }(f) $. We initially justify that the sequences $ P^n_{\alpha \beta} = N_n (f^n)^{-1} \frac{f^n_\alpha f^n_{\beta_* }}{I^{\delta(\alpha)/2-1} I_* ^{\delta(\beta)/2-1}} B^R_{2 \alpha \beta, n}$ and $ \tilde{P}^n_{\alpha \beta} = N_n (f^n)^{-1} \frac{ {f_\alpha^n}' {f_{\beta *}^n}' }{(I')^{\delta(\alpha)/2-1} (I'_* )^{\delta(\beta)/2-1}} B^R_{2 \alpha \beta, n}$ are relatively weakly compact in $L^1(E_R,d \Theta_{\alpha \beta} )$ for almost all $(t,x) \in (0,R) \times \in B_R$ by employing the Dunford-Pettis Theorem (Lemma \ref{theorem-dunford}). Following the analogous arguments of Lemmas \ref{converge L}-\ref{theo-6.7}-\ref{theo-6.9}-\ref{QL L1} in Subsection \ref{Subsec:CFNCO}, we then can verify that $ \int_{E_R} N (f^n)^{-1} P^n_{\alpha \beta}\varphi d \Theta_{\alpha \beta} $ and $ \int_{E_R} N (f^n)^{-1} \tilde{P}^n_{\alpha \beta}\varphi d\Theta_{\alpha \beta} $ strongly converge to $ \int_{E_R} N (f)^{-1} P_{\alpha \beta}\varphi d \Theta_{\alpha \beta} $ and $ \int_{E_R} N (f)^{-1} \tilde{P}_{\alpha \beta}\varphi d \Theta_{\alpha \beta} $, respectively, in $ L^1((0,R)\times B_R) $, where $ N (f) $ is given in \eqref{Nf}. Together with the fact $N (f^n)$ strongly converges to $N (f)$ in $ L^1((0,R)\times B_R) $, we obtain that for almost all $(t,x)\in (0,R)\times B_R$, $P^n_{\alpha \beta} \rightarrow P_{\alpha \beta}$ and $\tilde{P}^n_{\alpha \beta} \rightarrow \tilde{P}_{\alpha \beta}$ weakly in $ L^1(E_R, d \Theta_{\alpha \beta})$. At the end, by using the convexity of the function
\begin{equation*}
j(a,b)=\left\{\begin{array}{rl}
	(a-b) \log \frac{a}{b}, & \quad \text{for } a, b>0 \,, \\
	+\infty, & \quad \text{for } a \text { or } b \leq 0 \,,
\end{array}\right.
\end{equation*}
there holds $\liminf\limits_{n\rightarrow \infty} \int_{E_R}j(P^n_{\alpha \beta},\tilde{P}^n_{\alpha \beta})d \Theta_{\alpha \beta } \geq \int_{E_R}j(P_{\alpha \beta} ,\tilde{P}_{\alpha \beta}) d \Theta_{\alpha \beta }$, which means that 
$$ \liminf\limits_{n \to + \infty} \tilde{H}_{\alpha \beta }^n (f^n) \geq H_{\alpha \beta }(f) \,. $$ 
So the entropy inequality \eqref{entropy-th} in Theorem \ref{MainThm} is obtained.

\subsection{Historical remarks}

Extensive research works have been carried out on the Boltzmann equation and related models, which can be classified into two main categories: 1) the classical solutions regime; 2) the renormalized solutions regime. In the framework of classical solutions, all works are imposed on the assumptions of the small initial data. For instance, the readers can be referred to the works \cite{Guo-CPAM-2006,Guo-ARMA-2010,GJJ-CPAM-2010,JL-APDE-2022,JLT-TAMS-2024,JLZ-ARMA-2023,JXZ-IUMJ-2018,SG-ARMA-2008} and the references therein. We here focus on the renormalized solutions regime.

We first review the existence of renormalized solutions for Boltzmann equation or related models. In 1989, DiPerna-Lions \cite{DiPerna-Lions-Inv} introduced the so-called renormalized solution to the transported equations. Then they also defined the renormalized solution to the classical Boltzmann equation on angular cutoff model and proved its global existence over whole space with large initial data in \cite{Diperna-Lions}. In 1991, DiPerna-Lions \cite{DL-ARMA-1991} proved the entropy inequality of renormalized solution to the Boltzmann equation over whole space. In 1994, Villani \cite{Villani-ADE-1966} justified the existence of renormalized solution with a nonnegative measure to the Landau equation over whole spatial space. In 2002, Alexandre-Villani \cite{AV-CPAM-2002} verified the existence of renormalized solution with a nonnegative measure to the Boltzmann equation for angular noncutoff model over whole space. In 2010, Mischler \cite{Mischler-ASENS-2010} proved the existence of renormalized solution to the cutoff Boltzmann equation, Vlasov-Poisson and Fokker-Planck type models over bounded domain with Maxwell reflection boundary conditions. In 2019, Ars\'enio and Saint-Raymond \cite{DSR-2019-BOOK} studied the renormalized solution to the Vlasov-Maxwell-Boltzmann system over period domain. Moreover, in 2019, Jiang-Zhang investigated the renormalized solution of the Boltzmann equation for angular cutoff model in \cite{JZ-SIMA-2019} and the renormalized solution with a nonnegative measure of the Boltzmann equation for angular noncutoff model in \cite{JZ-JDE-2019} over the bounded domain with incoming boundary condition.

We also remark that there were many works on the hydrodynamic limits in the renormalized solutions regime, which was initiated by Bardos-Golse-Levermore \cite{BGL-JSP-1991,BGL-CPAM-1993} in 1990s. Since then there has been many contributions to this program \cite{BGL-ARMA-2000,GL-CPAM-2002,GSR-IM-2004,GSR-JMPA-2009,JLM-CPDE-2010,LM-ARMA-2010,LM-ARMA-2001-I,LM-ARMA-2001-II,SR-ARMA-2003}. In particular, the work of Golse and Saint-Raymond \cite{GSR-IM-2004} is the ﬁrst completed rigorous justiﬁcation of the NSF limit from the Boltzmann equation in a class of bounded collision kernels without making any nonlinear weak compactness hypotheses. They have recently extended their result to the case of hard potentials \cite{GSR-JMPA-2009}. With some new nonlinear estimates, Levermore and Masmoudi \cite{LM-ARMA-2010} treated a broader class of collision kernels that includes all hard potential cases and, for the ﬁrst time in this program, soft potential cases. All works mentioned above were focused on either the periodic spatial domain or the whole space. For the bounded spatial domain issues, there also has been serval related works in the renormalized solutions regime. In \cite{MSR-CPAM-2003}, Masmoudi and Saint-Raymond proved the linear Stokes limit from the fluctuations of renormalized solutions to the Boltzmann equation in a bounded domain. In \cite{SR-2009-BOOK}, Saint-Raymond rigorously justified the NSF limit in the renormalized solutions regime over a bounded domain. In \cite{MSR-CPAM-2003,SR-2009-BOOK}, the fluid boundary conditions were either the Dirichlet or the Navier slip boundary condition depending on the relative sizes of the accommodation coefﬁcient and the Knudsen number. In \cite{JM-CPAM-2017}, Jiang-Masmoudi rigorously proved the NSF limit with Neumann boundary condition from the renormalized solutions to the Boltzmann equation in a bounded domain with Maxwell reflection boundary condition for small accommodation coefficient $\sqrt{Kn}$. Moreover, Jiang-Zhang \cite{JZ-SIMA-2019,JZ-JDE-2019} rigorously verified the NSF limit from the renormalized solutions to the Boltzmann equation (including cutoff and noncutoff models) in a bounded domain with incoming boundary condition.

\subsection{Organization of current paper}

In Section \ref{Sec:Prel}, we give some preliminaries which will be frequently used in this paper. We construct the approximated problem \eqref{approximation} and establish its uniform bounds in Section \ref{Sec:Apr}. The weak $L^1$ compactness for nonlinear collision operator are given in Section \ref{sec weak com}. Moreover, in Section \ref{Sec:Con}, we mainly study the convergence of the approximated solution $f^n$ and the nonlinear collision operator $\tilde{Q}_n (f^n, f^n)$. The integrable properties of the corresponding limits are also studied in this section. Section \ref{Sec:Ext} gives the proof of existence of renormalized solutions to \eqref{BE-MP}. Section \ref{Sec:Entropy} gives the justification of the entropy inequality \eqref{entropy-th} in Theorem \ref{MainThm}. Finally, we verify the weak lower semicontinuity of entropy, i.e. prove Lemma \ref{Lemma 4.6} in Section \ref{subsec-4.6}.

\section{Preliminaries}\label{Sec:Prel}

\subsection{Properties of weakly relatively compactness in $L^1$}
From the works \cite{ARKERYD-Boltz,R-de-la-Vallée-Poussin,Dunford-Schwartz-1958,Evans-1990-AMS}, we obtain the following properties.

\begin{definition}\label{def-equi}
	Let $\mathscr{F} $ be a subset of $L^1(X, \mathcal{M},\mu )$. We say that $\mathscr{F}$ is equi-integrable if for every $\epsilon >0 $ there exists $\delta >0 $ such that for any measurable set $E \subset X $ with $\mu (E ) < \delta $ and for all $f\in \mathscr{F}$,
	\begin{equation*}
		\int_{E} |f|dx < \epsilon .
	\end{equation*}
\end{definition}

It is a fact that $\mathscr{F} $ being equi-integrable is equivalent to 
	\begin{equation*}
		\lim \limits_{C \rightarrow \infty } \sup \limits_{f\in \mathscr{F }} \int_{|f|>C } |f|d \mu = 0 .
	\end{equation*}

	\begin{lemma}[Dunford-Pettis] \label{theorem-dunford}
	Suppose that $\mathscr{F}$ is a Bounded subset of $ L^1(X, \mathcal{M},\mu )$. $\mathscr{F}$ is weakly relatively compact if and only if 
	\begin{enumerate}
		\item {$\mathscr{F} $ is equi-integrable};
		\item {for any $\epsilon>0, $ there exists compact set $K \subset X$ such that 
			\begin{equation*}
				\int_{K^c} |f(x) | dx <\epsilon , \quad \forall f \in \mathscr{F}.
		\end{equation*}}
	\end{enumerate}
\end{lemma}

\begin{lemma}[De La Vall\'ee-Poussin Criterion]\label{theo-DE La}
	A subset $\mathscr{F}$ of $L^1(X, \mathcal{M},\mu )$ is equi-integrable if and only if there exists a function $H : \R_+ \rightarrow \R_+ $ satisfying 
	\begin{equation*}
		\tfrac{H(r)}{r} \rightarrow +\infty \ as \ r \rightarrow +\infty
	\end{equation*}
	and such that 
	\begin{equation*}
		\sup \limits_{f \in \mathscr{F}} \int_{X} H(f) d \mu  <+\infty .
	\end{equation*}
\end{lemma}

\begin{lemma}\label{rwc}
	Let $J$ be a discrete or continuous set of indicators. Given $M_1, M_2 >0$, assume that $\{ f_j (x,\xi,I):j\in J \} \subset L^1(\R^3 \times \R^3 \times \R_+)_+$ satisfies
	\begin{equation*}
    \begin{aligned}
		\iiint_{\R^3\times \R^3 \times \R_+} (1+|x|^2+|\xi|^2+I)f_j dx d\xi dI\leq M_1 \,, \ \iiint_{\R^3\times \R^3 \times \R_+} f_j |\log f_j| dxd\xi dI \leq M_2
    \end{aligned}
	\end{equation*}
	for all $j \in J$. Then $\{f_j(x,\xi,I):j\in J\}$ is relatively weakly compact in $L^1(\R^3\times \R^3 \times \R_+)$.
\end{lemma}

	\begin{proof}
	It's not hard to find $f_j$ is bounded in $L^1$. Moreover, for any $\epsilon >0$, there must exist $R > e^{ \frac{2 M_2}{\epsilon} } >0$ such that  $\log f_j >\frac{2M_2}{\epsilon }$, when $ f_j>R$. Let $\delta=\frac{\epsilon }{2R}$. Then for any $j\in J$ and measurable set $E$ with $\mu (E)<\delta$, one has
	\begin{equation*}
		\begin{aligned}
			\iiint_{E} f_j dxd\xi dI =& \iiint_{E\cap \{f_j\leq R\}} f_j dx d\xi dI +\iiint_{E\cap \{f_j> R\}} f_j dxd\xi dI \\
			\leq & R \mu (E) +\iiint_{E\cap \{f_j> R\}} \tfrac{1}{|\log f_j|}  f_j|\log f_j| dxd\xi dI 
			\leq \epsilon.
		\end{aligned}
	\end{equation*}
	Besides, for any $\epsilon>0$, there must exist $R>0$ such that $(1+|x|^2+|\xi|^2+I)^{-1} <\frac{\epsilon}{M_1}$ when $|x|>R$. Let $K=\{x:|x|\leq R\}$. Then for any $j\in J$, one has
	\begin{equation*}
		\iiint_{K^c} f_j dx d\xi dI 
		=\iiint_{K^c} (1+|x|^2+|\xi|^2+I)^{-1} (1+|x|^2+|\xi|^2+I) f_j  dx d\xi dI\leq \epsilon.
	\end{equation*}
	By Dunford-Pettis Theorem (Lemma \ref{theorem-dunford}), we know $\{f_j\}_{j \in J}$ is relatively weakly compact in $L^1(\R^3\times \R^3 \times \R_+)$. Then the proof of Lemma \ref{rwc} is finished.
\end{proof}

\begin{corollary}\label{rwc1}
	Let $J$ be a discrete or continuous set of indicators. Given $M_1, M_2 >0$, assume that $\{ f_j (x,\xi):j\in J \} \subset L^1(\R^3 \times \R^3)_+$ satisfies
	\begin{equation*}
		\iint_{\R^3\times \R^3 } (1+|x|^2+|\xi|^2)f_j dx d\xi \leq M_1, \ \iint_{\R^3\times \R^3 } f_j |\log f_j| dxd\xi  \leq M_2
	\end{equation*}
for all $j \in J$. Then $\{f_j(x,\xi):j\in J\}$ is relatively weakly compact in $L^1(\R^3\times \R^3 )$.
\end{corollary}

We now introduce the following weighted $L^1$ spaces.

	\begin{definition}[DiPerna-Lions' work \cite{Diperna-Lions}] \label{def 4.1}
If $f(x,\xi)$ is a measurable function on $\R^3 \times \R^3$, denote by
	\begin{equation*}
		\Vert f \Vert_{1,2} = \iint_{\R^3 \times \R^3 } (1+|x|^2 + |\xi|^2 ) |f (x,\xi)|  dxd\xi \,.
	\end{equation*} 
Define the space $L^1_2(\R^3 \times \R^3 )$ by 
\begin{equation*}
		L^1_2(\R^3 \times \R^3 ) = \{ f (x, \xi) \textrm{ measurable on } \R^3 \times \R^3; \Vert f \Vert_{1,2} < \infty \}
	\end{equation*} 
endowed with the $L^1_2$-norm $ \Vert f \Vert_{1,2} $.
\end{definition}

	\begin{definition} \label{def 4.2}
If $f(x,\xi,I)$ is a measurable function on $\R^3 \times \R^3 \times \R_+$, denote by
	\begin{equation*}
		\Vert f \Vert_{1,3} = \iiint_{\R^3 \times \R^3 \times \R_+ } (1+|x|^2 + |\xi|^2+ I ) |f (x,\xi,I)|  dxd\xi dI \,.
	\end{equation*} 
Define the space $L^1_3(\R^3 \times \R^3 \times \R_+)$ by 
\begin{equation*}
		L^1_3(\R^3 \times \R^3 \times \R_+) = \{ f (x, \xi, I) \textrm{ measurable on } \R^3 \times \R^3 \times \R_+; \Vert f \Vert_{1,3} < \infty \}
	\end{equation*} 
endowed with the $L^1_3$-norm $ \Vert f \Vert_{1,3} $.
\end{definition}

\begin{lemma} \label{Lemma 4.4}
Assume $f_\alpha(x,\xi,I) \in L^1_3 (\R^3 \times \R^3 \times \R_+)_+ $ and $\alpha \in \{s_0+1,\cdots,s\}$. Then
\begin{equation*}
	\iiint_{\R^3 \times \R^3\times \R_+ } f_\alpha \log (I^{1-\delta(\alpha)/2}f_\alpha) dxd\xi dI<+\infty
\end{equation*} 
if and only if 
\begin{equation*}
	f_\alpha \log ^+ (I^{1-\delta(\alpha)/2}f_\alpha) \in L^1(\R^3 \times \R^3 \times \R_+) .
\end{equation*}
\end{lemma}

\begin{proof}
	Firstly, the necessity is obvious. On the contrary, if $f_\alpha \log ^+ (I^{1-\delta(\alpha)/2}f_\alpha) \in L^1(\R^3 \times \R^3 \times \R_+) $, we need to prove $\iiint_{\R^3 \times \R^3\times \R_+ } f_\alpha \log (I^{1-\delta(\alpha)/2}f_\alpha) dxd\xi dI <+\infty $. It suffices to prove $f_\alpha \log ^- (I^{1-\delta(\alpha)/2}f_\alpha) \in L^1(\R^3 \times \R^3\times \R_+ )$. In fact, using the inequality $u^{\frac{1}{2}} \log u^{-1} <1 (0<u<1)$, one has
	\begin{equation*}{\small
		\begin{aligned}
		&\iiint_{\R^3 \times \R^3\times \R_+ } f_\alpha \log^- (I^{1-\delta(\alpha)/2}f_\alpha) dxd\xi dI\\
			=& \iiint_{\R^3 \times \R^3\times \R_+} I^{\delta(\alpha)/2-1} I^{1-\delta(\alpha)/2} f_\alpha \log \tfrac{1}{I^{1-\delta(\alpha)/2}f_\alpha} \cdot \mathbf{1}_{0<I^{1-\delta(\alpha)/2}f_\alpha< \exp-(|x|^2 + |\xi|^2 +I)} dxd\xi dI\\
		+& \iiint_{\R^3 \times \R^3\times \R_+  }   f_\alpha \log \tfrac{1}{I^{1-\delta(\alpha)/2}f_\alpha} \cdot \mathbf{1}_{\exp-(|x|^2 + |\xi|^2 +I ) <I^{1-\delta(\alpha)/2}f_\alpha <1} dxd\xi dI\\
		\leq & \iiint_{\R^3 \times \R^3 \times \R_+} I^{\delta(\alpha)/2-1}  \exp(-\tfrac{1}{2}(|x|^2 + |\xi|^2+I )) dxd\xi dI+ \iiint_{\R^3 \times \R^3\times \R_+} (|x|^2 + |\xi|^2+I ) f_\alpha dx d\xi dI < +\infty. 
		\end{aligned}}
	\end{equation*}
	Therefore, the proof of Lemma \ref{Lemma 4.4} is finished.
\end{proof}

\begin{corollary}\label{cor4.5}
		Assume $f_\alpha(x,\xi) \in L^1_2 (\R^3 \times \R^3 )_+ $ and $\alpha \in \{1,\cdots,s_0\}$. Then
	\begin{equation*}
		\iint_{\R^3 \times \R^3 } f_\alpha \log f_\alpha dxd\xi <+\infty \textrm{ if and only if } f_\alpha \log ^+ f_\alpha \in L^1(\R^3 \times \R^3 ) \,.
	\end{equation*} 
\end{corollary}

\begin{lemma}[Product Limit Theorem, \cite{BGL-1993-CPAM,Diperna-Lions}]\label{product limit} 
Let $\mu$ be a finite, positive Borel measure on a Borel subset $X$ of $\R^3$. Assume that $f^n,f\in L^1(d \mu)$ and $g^n,g\in L^\infty(d \mu)$ satisfy
\begin{enumerate}
\item {$w-\lim \limits_{n\rightarrow \infty} f^n = f$ in $L^1(d \mu)$;}
\item {$\{g^n\}_{n\geq 1}$ is bounded in $L^\infty (d \mu)$ and $\lim \limits_{n\rightarrow \infty} g^n=g, a.e.$}
\end{enumerate}
Then $w-\lim \limits_{n \rightarrow \infty} f^n g^n= fg$ in $L^1(d \mu)$, where “$w-\lim$” means weak limit.
\end{lemma}

Now we give the next two lemmas, which will be used to construct the approximate initial data.

\begin{lemma}\label{Lemma 4.6}
	Assume that $f^n_\alpha(x,\xi ) \in L^1(\R^3\times \R^3)_+$ with $\alpha \in \{1,\cdots,s_0\}$ satisfy
	\begin{enumerate}
		\item {there exists a constant $M >0$ such that 
			$$\Vert f^n_\alpha\Vert_{1,2} \leq M ,\ f^n_\alpha \log f^n_\alpha \in L^1(\R^3\times \R^3); $$ }
		\item {$\{ f^n_\alpha (x,\xi ) \} $ converges to $f_\alpha (x,\xi ) $ weakly in $ L^1(\R^3 \times \R^3 ) $. }
	\end{enumerate}
Moreover, assume that $f^n_\alpha(x,\xi ,I) \in L^1(\R^3\times \R^3 \times \R_+)_+$ with $\alpha \in \{s_0+1,\cdots,s\}$ satisfy
	\begin{enumerate}
		\item[(3)] {there exists a constant $M >0$ such that 
			$$\Vert f^n_\alpha \Vert_{1,3} \leq M ,\ f^n_\alpha \log (I^{1-\delta(\alpha)/2}f^n_\alpha) \in L^1(\R^3 \times \R^3\times \R_+ ); $$ }
		\item[(4)] {$\{ f^n_\alpha(x,\xi ,I) \} $  converges  to $f_\alpha(x,\xi,I ) $ weakly in $  L^1(\R^3 \times \R^3\times \R_+) $. }
	\end{enumerate}
	Then 
	\begin{equation}\label{eq4.28}{\small
    \begin{aligned}
		& \sum_{\alpha = 1}^{s_0} \iint_{\R^3 \times \R^3} f_\alpha \log f_\alpha dxd\xi + \sum_{\alpha = s_0 + 1}^{s} \iiint_{\R^3 \times \R^3\times \R_+ } f_\alpha \log (I^{1-\delta(\alpha )/2 } f_\alpha) dxd\xi dI \\
\leq & \liminf \limits_{n\rightarrow \infty } \Big\{ \sum_{\alpha = 1}^{s_0} \iint_{\R^3 \times \R^3 } f^n_\alpha \log f^n_\alpha dxd\xi + \sum_{\alpha = s_0 + 1}^{s} \iiint_{\R^3 \times \R^3\times \R_+ } f^n_\alpha \log (I^{1-\delta(\alpha)/2}f^n_\alpha) dxd\xi dI \Big\} .
\end{aligned}}
	\end{equation} 
\end{lemma}

The proof of Lemma \ref{Lemma 4.6} will be given in Section \ref{subsec-4.6} later.

\subsection{Collision invariants}

The results in this subsection follow directly from the Section 2.2 of \cite{N.Bernhoff-2023}.

  \begin{lemma}\label{measure}
	For all $\{ \alpha, \beta \} \subset \{1, \cdots , s\}$, the following measures are invariant under the changes \eqref{eq2.18}:
	\begin{equation*}
\begin{aligned}
		d\Gamma_{\alpha \beta }= & W_{\alpha \beta } (\xi, \xi_*, I, I_* ) d\xi d\xi_*d\xi'd\xi'_* dI dI_*dI'dI'_* \\
= & 
\left\{
  \begin{aligned}
    & B_{0 \alpha \beta} (\xi - \xi_*, \omega) d \omega d \xi_* d \xi \,, \qquad \qquad \qquad \qquad \qquad \qquad \qquad \qquad \quad \alpha, \beta \in \mathcal{I}_m \,, \\
    & B_{1 \alpha \beta} ( \xi - \xi_*, I_*, \mathfrak{R}, \omega ) ( 1 - \mathfrak{R} )^{ \delta (\beta) / 2 - 1 } \mathfrak{R}^\frac{1}{2} I_*^{ \delta (\beta) / 2 - 1 } d \omega d \mathfrak{R} d \xi_* d I_* d \xi \,, \ \alpha \in \mathcal{I}_m, \beta \in \mathcal{I}_p \,, \\
    & B_{1 \beta \alpha} ( \xi - \xi_*, I, \mathfrak{R}, \omega ) ( 1 - \mathfrak{R} )^{ \delta (\alpha) / 2 - 1 } \mathfrak{R}^\frac{1}{2} I^{ \delta (\alpha) / 2 - 1 } d \omega d \mathfrak{R} d \xi_* d \xi d I \,, \quad \alpha \in \mathcal{I}_p, \beta \in \mathcal{I}_m \,, \\
    & B (\xi - \xi_*, I + I_*, \mathfrak{R}, \mathfrak{r}, \omega) \mathfrak{r}^{  \delta (\alpha) / 2 - 1 } ( 1 - \mathfrak{r} )^{  \delta (\beta) / 2 - 1 } \mathfrak{R}^\frac{1}{2} \\
    & \ \times ( 1 - \mathfrak{R} )^{ ( \delta (\alpha) + \delta (\beta) ) / 2 - 1 } I^{  \delta (\alpha) / 2 - 1 } I_*^{  \delta (\beta) / 2 - 1 } d \omega d \mathfrak{r} d \mathfrak{R} d \xi_* d I_* d \xi d I \,, \quad \alpha, \beta \in \mathcal{I}_p \,,
  \end{aligned}
\right.
\end{aligned}
	\end{equation*}
where $\mathcal{I}_m = \{ 1, \cdots, s_0 \}$ and $\mathcal{I}_p = \{ s_0 + 1, \cdots, s \}$.
\end{lemma}

The weak form of the collision operator $Q(f,f)$ reads 
\begin{equation*}
(Q(f,f),g) = \sum \limits_{\alpha , \beta =1}^s \int_{(\R^3\times \R_+)^4} (\tfrac{f'_\alpha f'_{\beta_*}}{(I')^{\delta(\alpha)/2-1} (I'_* )^{\delta(\beta)/2-1}} - \tfrac{f_\alpha f_{\beta_*}}{I^{\delta(\alpha)/2-1} I_* ^{\delta(\beta)/2-1}}) g_\alpha d\Gamma_{\alpha \beta }
\end{equation*}
for any function $g= (g_1,\cdots , g_s)$, with $g_\alpha =g_\alpha (\mathbf{Z} )$. Then by the Lemma \ref{measure}, we have the following proposition.

\begin{proposition}\label{prop 2.2}
	Let $g= (g_1,\cdots , g_s)$, with $g_\alpha =g_\alpha (\mathbf{Z} )$, it follows that 
	\begin{equation*}
		(Q(f,f),g) 
		= \tfrac{1}{4}\sum \limits_{\alpha , \beta =1}^s \int_{(\R^3\times \R_+)^4} (\tfrac{f'_\alpha f'_{\beta_*}}{(I')^{\delta(\alpha)/2-1} (I'_* )^{\delta(\beta)/2-1}} - \tfrac{f_\alpha f_{\beta_*}}{I^{\delta(\alpha)/2-1} I_* ^{\delta(\beta)/2-1}}) \Delta_{\alpha \beta } (g) d \Gamma_{\alpha \beta },
	\end{equation*}
	where $\Delta_{\alpha \beta } (g) := g_\alpha +g_{\beta_*}-g_\alpha ' - g'_{\beta _*} $.
\end{proposition}

\begin{definition}
	A vector function $g= (g_1,\cdots , g_s)$, with $g_\alpha = (\mathbf{Z} )$ is a collision invariant if 
	\begin{equation*}
		\int_{(\R^3 \times \R_+)^4} (\tfrac{f'_\alpha f'_{\beta_*}}{(I')^{\delta(\alpha)/2-1} (I'_* )^{\delta(\beta)/2-1}} - \tfrac{f_\alpha f_{\beta_*}}{I^{\delta(\alpha)/2-1} I_* ^{\delta(\beta)/2-1}}) \Delta_{\alpha \beta } (g) d\Gamma_{\alpha \beta } =0,
	\end{equation*}
	for all $\{ \alpha, \beta \} \subset \{1, \cdots , s\}$.
\end{definition}

\begin{proposition}\label{prop2.3}
	Let $\mathbf{m}=(m_1,\cdots, m_s), \mathbf{I} = (\underbrace{0,\cdots ,0}_{s_0} ,\underbrace{I,\cdots , I}_{s_1} )$, and $\{e_1,\cdots , e_s\}$ be the standard basis of $\R^s$. Then the vector space of collision invariants is generated by 
	\begin{equation*}
		\{e_1, \cdots ,e_s, \mathbf{m} \xi_x, \mathbf{m} \xi_y, \mathbf{m} \xi_z ,\mathbf{m} |\xi|^2+2\mathbf{I}\}.
	\end{equation*}
\end{proposition}

Define 
\begin{equation}\label{W[f]}
	W[f]:=(Q(f,f),\log (\varphi ^{-1} f )) ,
\end{equation}
where $\varphi = diag (I^{\delta(1)/2-1} , \cdots, I^{\delta(s)/2-1})$. It follows by Proposition \ref{prop 2.2} that
\begin{equation}
	\begin{aligned}
		W[f]=&-\tfrac{1}{4}\sum \limits_{\alpha , \beta =1}^s \int_{(\R^3\times \R_+)^4} 
		 (\tfrac{I^{\delta(\alpha)/2-1} I_* ^{\delta(\beta)/2-1}f'_\alpha f'_{\beta_*}  }
		{f_\alpha f_{\beta_* } (I')^{\delta(\alpha)/2-1} (I'_* )^{\delta(\beta)/2-1}} - 1) \\
		&\times \log (\tfrac{I^{\delta(\alpha)/2-1} I_* ^{\delta(\beta)/2-1} f'_\alpha f'_{\beta_*}}{f_\alpha f_{\beta_* } (I')^{\delta(\alpha)/2-1} (I'_* )^{\delta(\beta)/2-1} } ) \tfrac{f_\alpha f_{\beta_* } }{I^{\delta(\alpha)/2-1} I_* ^{\delta(\beta)/2-1} }d\Gamma_{\alpha \beta  } .
	\end{aligned}
\end{equation}
It is known that $(x-1)\log x \geq 0$ for $x\geq 0$, which means that $W[f]$ is non-positive, that is $W[f]\leq 0.$ Then, by means of the weak formulation, the entropy identity is obtained by multiplying \eqref{BE-MP} by $1+\log (\varphi^{-1} f)$, integrating over $ \mathbf{Z}_\alpha \in \mathcal{Z}_\alpha $ and using \eqref{eq2.18} that
\begin{equation}\label{EntropyId}
	\begin{aligned}
		&\tfrac{d}{dt} H(f)(t) +
		\tfrac{1}{4} \sum \limits_{\alpha , \beta =1}^s  \int_{(\R^3\times \R_+)^4} 
		(\tfrac{I^{\delta(\alpha)/2-1} I_* ^{\delta(\beta)/2-1}f'_\alpha f'_{\beta_*}  } {f_\alpha f_{\beta_* } (I')^{\delta(\alpha)/2-1} (I'_* )^{\delta(\beta)/2-1}} - 1) \\
		& \times \log (\tfrac{I^{\delta(\alpha)/2-1} I_* ^{\delta(\beta)/2-1} f'_\alpha f'_{\beta_*}}{f_\alpha f_{\beta_* } (I')^{\delta(\alpha)/2-1} (I'_* )^{\delta(\beta)/2-1} } ) \tfrac{f_\alpha f_{\beta_* } }{I^{\delta(\alpha)/2-1} I_* ^{\delta(\beta)/2-1} }d\Gamma_{\alpha \beta  } = 0 \,,
	\end{aligned}
\end{equation}
where $H(f)$ is defined in \eqref{Hf}.

\subsection{Equivalent description of renormalized solutions}

In this subsection, we will give the definitions of distributional, mild and renormalized solutions and the connection among them. The main results and methods of proof are taken from \cite{Diperna-Lions}. Besides, throughout this section,
we will denote by $f_\alpha$ a nonnegative element of $L^1_{loc} ((0,+\infty)\times \R^3 \times \R^3)$ for all $\alpha \in \{1,\cdots,s_0\}$ and a nonnegative element of $L^1_{loc} ((0,+\infty)\times \R^3 \times \R^3\times \R_+)$ for all $\alpha \in \{s_0+1,\cdots,s\}$.
 
\begin{definition}[Distributional solution]\label{def-distributed}
		We say that $f=(f_1,\cdots,f_s)$ is a distributional solution of \eqref{BE-MP} if it satisfies 
	\begin{enumerate}
		\item {for $\alpha \in \{1,\cdots,s_0\}$, $f_\alpha$ solves
			\begin{equation*}
				\tfrac{\partial f_\alpha}{\partial t } + \xi \cdot \nabla_{x} f_\alpha = Q_\alpha (f,f), \quad \text{in } \mathscr{D}'((0,\infty)\times \R^3 \times \R^3);
		\end{equation*} }
		\item {for $\alpha \in \{s_0+1,\cdots ,s\}$, $f_\alpha$ solves
			\begin{equation}\label{distribution solution}
				\tfrac{\partial f_\alpha}{\partial t } + \xi \cdot \nabla_{x} f_\alpha = Q_\alpha (f,f), \quad \text{in } \mathscr{D}'((0,\infty)\times \R^3 \times \R^3\times \R_+).
		\end{equation}}
	\end{enumerate}
\end{definition}

\begin{definition}[Mild solution]\label{def-mild}
	We say that $f=(f_1,\cdots,f_s)$ is a mild solution of \eqref{BE-MP} if it satisfies 
\begin{enumerate}
	\item {for $\alpha \in \{1,\cdots,s_0\}$ and for almost all $(x,\xi)\in \R^3 \times \R^3,$ one has $Q^\pm_\alpha (f,f)^\sharp (t,x,\xi ) \in L^1 (0,T)$ for all $T<\infty$, and
		\begin{equation*}
			\begin{aligned}
					f^\sharp_\alpha (t,x,\xi )-f^\sharp_\alpha(s,x,\xi ) &=\int_{s}^{t} Q_\alpha(f,f)^\sharp (\sigma ,x,\xi )d\sigma \quad (\forall 0<s<t<\infty),
			\end{aligned}
		\end{equation*}	
	where $f^\sharp_\alpha (t,x,\xi )= f_\alpha(t,x+t\xi ,\xi )$;}
	\item {for $\alpha \in \{s_0+1,\cdots ,s\}$ and for almost all $(x,\xi,I) \in \R^3 \times \R^3\times \R_+,$ there hold $Q^\pm_\alpha (f,f)^\sharp (t,x,\xi ,I) \in L^1(0,T)$ for all $T<\infty$, and
			\begin{equation}\label{mild}
			\begin{aligned}
				f^\sharp_\alpha (t,x,\xi ,I)-f^\sharp_\alpha(s,x,\xi ,I) &=\int_{s}^{t} Q_\alpha(f,f)^\sharp (\sigma ,x,\xi ,I)d\sigma \quad (\forall 0<s<t<\infty),
			\end{aligned}
		\end{equation}	
		where $f^\sharp_\alpha(t,x,\xi ,I)= f_\alpha(t,x+t\xi ,\xi ,I)$.}
\end{enumerate}
\end{definition}

Next we study the connection among the distributional, mild and renormalized solutions. 

\begin{lemma}\label{Lemma 3.1}
The following statements hold:
	\begin{enumerate}
		\item {If $Q^\pm_\alpha(f,f) \in L^1_{loc} $, then $f$ is a distributional solution of \eqref{BE-MP} if and only if $f$ is a renormalized solution of \eqref{BE-MP}; }

		\item {$f$ is a renormalized solution of \eqref{BE-MP} if and only if for all $\beta \in C^1[0,+\infty)$ such that $|\beta '(t)|\leq \frac{C}{1+t}$ for some $C > 0$, $\beta(f_\alpha)$ enjoys
		\begin{equation*}
		\tfrac{\partial }{\partial t }\beta (f_\alpha) + \xi \cdot \nabla_x \beta (f_\alpha)= \beta'(f_\alpha) Q_\alpha(f,f) \quad \text{in } \mathscr{D}'((0,\infty)\times \R^3 \times \R^3) 
		\end{equation*}	
for all $\alpha \in \{1,\cdots,s_0\}$ and $\beta'(f_\alpha) Q_\alpha(f,f) \in L^1_{loc} ((0, \infty) \times \R^3 \times \R^3) $, and 
		\begin{equation*}
			\tfrac{\partial }{\partial t }\beta (f_\alpha) + \xi \cdot \nabla_x \beta (f_\alpha)= \beta'(f_\alpha) Q_\alpha(f,f),\quad \text{in } \mathscr{D}'((0,\infty)\times \R^3 \times \R^3\times \R_+).
		\end{equation*}		
for all $\alpha \in \{s_0+1,\cdots,s\}$ and $\beta'(f_\alpha) Q_\alpha(f,f) \in L^1_{loc} ((0, \infty) \times \R^3 \times \R^3 \times \R_+)$.}
	\end{enumerate}
\end{lemma}

\begin{proof}
The proof is almost the same as Lemma II.1 of \cite{Diperna-Lions}. We omit the details here for simplicity.
\end{proof}

    \begin{lemma}\label{distri-mild}
    For $\alpha \in \{ 1, \cdots, s \}$, the following statements hold:
	\begin{enumerate}
		\item {If $Q^\pm(f,f) \in L^1_{loc}$, then $f$ is a distributional solution of \eqref{BE-MP} if and only if $f$ is a mild solution of \eqref{BE-MP}; }
		\item {$f$ is a renormalized solution of \eqref{BE-MP} if and only if $f$ is a mild solution of \eqref{BE-MP} and $(1+f_\alpha )^{-1} Q^\pm_\alpha(f,f) \in L^1_{loc}$. }
	\end{enumerate}
\end{lemma}
\begin{proof}

The proof is almost the same as Lemma II.2 of \cite{Diperna-Lions}. We omit the details here for simplicity.
\end{proof}

Denote by
\begin{equation}\label{F-sharp}{\small
\begin{aligned}
F^\sharp_\alpha (t,x,\xi)&= \int_{0}^{t} L_\alpha (f)^\sharp(s ,x,\xi )ds,\quad \alpha \in \{1,\cdots,s_0\},\\
F^\sharp_\alpha (t,x,\xi ,I)&= \int_{0}^{t} L_\alpha (f)^\sharp(s ,x,\xi,I)ds,\quad \alpha \in \{s_0+1,\cdots,s\} \,,
\end{aligned}}
\end{equation}
where the operator $L_\alpha (f)$ is defined in \eqref{L-alpha-def}.

\begin{lemma}\label{Lemma3.3}
	Assume that 
	\begin{equation*}
	\begin{aligned}
	L_\alpha (f)&\in L^1((0,T)\times B_R \times B_R ),\quad \alpha \in \{1,\cdots,s_0\},\\
	L_\alpha (f)&\in L^1((0,T)\times B_R \times B_R \times (0,R)), \quad \alpha \in \{s_0+1,\cdots,s\},
	\end{aligned}
	\end{equation*}
	for all $R,T<\infty$. $f$ is a mild solution of \eqref{BE-MP} if and only if the following conclusions hold:
	\begin{enumerate}
	\item {For all $\alpha \in \{1,\cdots,s_0\},\ 0<s<t<\infty$, one has
	\begin{equation}\label{3.5}
	\begin{aligned}
	&f^\sharp_\alpha (t,x,\xi) -f^\sharp_\alpha (s,x,\xi ) \exp(-(F^\sharp _{\alpha} (t)-F^\sharp _{\alpha} (s)))\\
	=&\int_{s}^{t} Q^+_{\alpha } (f,f) ^\sharp (s,x,\xi ) \exp(-(F^\sharp _{\alpha } (t)- F^\sharp _{\alpha } (\sigma))) d\sigma \quad  a.e.\ x,\xi;
	\end{aligned}
	\end{equation}	}		
	\item {For all $\alpha \in \{s_0+1,\cdots,s\},\ 0<s<t<\infty$, one has
	\begin{equation}\label{3.6}
	\begin{aligned}
	&f^\sharp_\alpha (t,x,\xi,I) -f^\sharp_\alpha (s,x,\xi ,I) \exp(-(F^\sharp _{\alpha} (t)-F^\sharp _{\alpha} (s)))\\
	=&\int_{s}^{t} Q^+_{\alpha } (f,f) ^\sharp (s,x,\xi ,I) \exp(-(F^\sharp _{\alpha } (t)- F^\sharp _{\alpha } (\sigma))) d\sigma \quad  a.e.\ x,\xi,I.
	\end{aligned}
	\end{equation}	}		
	\end{enumerate}
\end{lemma}

\begin{proof}

Without loss of generality, we only consider the case $\alpha \in \{ s_0 + 1, \cdots, s \}$. Assume that $f$ is a mild solution of \eqref{BE-MP}. Then $Q^\pm_\alpha (f,f)^\sharp (t,x,\xi ,I) \in L^1_{loc}(0,+\infty )$ for almost all $(x,\xi ,I) \in \R^3 \times \R^3\times \R_+ $ and
	\begin{equation*}
		f^\sharp_\alpha (t,x,\xi ,I) - f^\sharp_\alpha (s,x,\xi ,I) =\int_{s}^{t} Q_\alpha (f,f)^\sharp (\sigma ,x,\xi ,I)d\sigma \quad (\forall 0<s<t<\infty).
	\end{equation*}
	Therefore, for such $(x,\xi ,I)$, one has
	\begin{equation}\label{3.27}
		\tfrac{\partial}{\partial t} f^\sharp_\alpha (t,x,\xi ,I) + f^\sharp_\alpha L_\alpha (f)^\sharp (t,x,\xi ,I) = Q^+_\alpha (f,f)^\sharp (t,x,\xi ,I) \quad a.e.\ t>0.
	\end{equation}
Note that for almost all $(x,\xi ,I)\in \R^3 \times \R^3\times \R_+$, $ F^\sharp_\alpha(t,x,\xi ,I)$ is absolutely continuous with respect to $t$. From multiplying \eqref{3.27} by $\exp (F^\sharp_\alpha(t,x,\xi ,I))$, we then find 
	\begin{equation*}
		\begin{aligned}
			&f^\sharp_\alpha (t,x,\xi ,I) -f^\sharp_\alpha (s,x,\xi ,I) \exp(-(F^\sharp _{\alpha} (t)-F^\sharp _{\alpha} (s)))\\
			=&\int_{s}^{t} Q^+_{\alpha } (f,f) ^\sharp (\sigma,x,\xi ,I) \exp(-(F^\sharp _{\alpha } (t)- F^\sharp _{\alpha } (\sigma))) d\sigma \quad a.e.\ x,\xi ,I.
		\end{aligned}
	\end{equation*} 

Conversely, by \eqref{3.6}, $f^\sharp_\alpha (t,x,\xi ,I)$ is absolutely continuous with respect to $t$ for almost all $(x,\xi ,I)$ and
	\begin{equation*}
		\begin{aligned}
			&f^\sharp_\alpha (t,x,\xi ,I)\exp (F^\sharp _{\alpha}(t,x,\xi ,I)) -f^\sharp_\alpha (s,x,\xi ,I) \exp(F^\sharp _{\alpha} (s,x,\xi ,I))\\
			=&\int_{s}^{t} Q^+_{\alpha } (f,f) ^\sharp \exp( F^\sharp _{\alpha } (\sigma,x,\xi ,I)) d\sigma \quad a.e.\ x,\xi ,I.
		\end{aligned}
	\end{equation*} 
	Fixing $s$ and taking derivatives for $t$, one has
	\begin{equation*}
		\tfrac{\partial}{\partial t}f^\sharp (t,x,\xi ,I)+f^\sharp_\alpha L_\alpha (f)^\sharp (t,x,\xi ,I) = Q^+_\alpha (f,f)^\sharp (t,x,\xi ,I) \quad a.e.\ t,x,\xi ,I.
	\end{equation*}
	According to the above proof procedure, we have that $f$ is a mild solution of \eqref{BE-MP}. The proof of Lemma \ref{Lemma3.3} is finished.
\end{proof}


\section{Approximate problem and global uniform bounds}\label{Sec:Apr}

In this section, we mainly construct the following approximate problem:
\begin{equation} \label{approximation}
		\tfrac{\partial f^n}{\partial t} +\xi \cdot \nabla_x f^n =\tilde{Q}_n(f^n, f^n) \,, \ f^n|_{t=0}=f^n_{0} \,,
\end{equation} 
where $f^n_{0} = ( f_{1,0}^n, \cdots, f^n_{s,0} )$ and $\tilde{Q}_n=(\tilde{Q}_{1,n},\cdots,\tilde{Q}_{s,n})$ are respectively the approximated initial and nonlinear collision operator that will be explicitly constructed later. Then we will prove the existence, positivity and uniqueness of the approximate problem. Moreover, the regularity of the solution to \eqref{approximation} will also be verified. At the end, we shall derive the conservation laws, entropy inequality and the uniform-in-$n$ bounds of the approximate problem \eqref{approximation}.

\subsection{Approximation of initial data $f_0$ }

In this subsection, we focus on constructing the approximation of the initial data $f_0$ given in \eqref{f0}-\eqref{fs}.

\begin{lemma}\label{Lemma 4.8}
	Assume that $f_{\alpha,0 } \in L^1_2 ( \R^3 \times \R^3)_+ $ for $\alpha \in \{1,\cdots ,s_0\}$ satisfy \eqref{fs0}, hence,
	\begin{equation*}
		\iint_{\R^3 \times \R^3 } f_{\alpha,0 } (x,\xi ) \cdot (1+|x|^2 +|\xi|^2 + |\log  f_{\alpha,0 } |)  dxd\xi <+ \infty .
	\end{equation*}
	Then there exists a sequence $\{f^n_{\alpha,0 }\}_{n\geq 1} \subset \mathscr{S} (\R^3 \times \R^3 )$ (the Schwartz space over $\R^3 \times \R^3$) and a constant $C > 0$ such that 
\begin{equation}\label{Ap-fs0}
  \begin{aligned}
    & \lim \limits_{n\rightarrow \infty}  \iint_{\R^3 \times \R^3 } (1+|x|^2 +|\xi|^2) |f^n_{\alpha,0 } - f_{\alpha,0 }| dxd\xi=0 \,, \ f^n_{\alpha,0 } \geq \tfrac{1}{n} \exp (- \tfrac{|x|^2+ |\xi|^2 }{2}) \,, \\
    & \iint_{\R^3 \times \R^3} f^n_{\alpha,0 }  (1+|x|^2 +|\xi|^2 + |\log f^n_{\alpha,0 }|)dxd\xi  \leq C \,, \\
    & \limsup \limits_{n\rightarrow \infty} \iint_{\R^3 \times \R^3 } f^n_{\alpha,0 } \log f^n_{\alpha,0 }  dxd\xi  \leq \iint_{\R^3 \times \R^3 } f_{\alpha,0 } \log  f_{\alpha,0 } dxd\xi \,.
  \end{aligned}
\end{equation}
\end{lemma}

\begin{proof}
 Let $h^n_\alpha (x,\xi ) = f_{\alpha,0 } (x,\xi )  \chi_n (x,\xi )$ for $ n \geq 1$. Here $\chi_n = \mathbf{1}_{ \{ (x,\xi) : |x|^2+ |\xi|^2 \leq n^2 \} }$. Obviously, by the Dominated Convergence Theorem,
	\begin{equation}\label{eq4.44}
		\lim \limits_{n\rightarrow \infty}  \iint_{\R^3 \times \R^3 } (1+|x|^2 +|\xi|^2) |h^n_\alpha - f_{\alpha,0 }| d x d \xi = 0 \,.
	\end{equation}
	Assume that $j_{\delta_n} $ is a standard mollification. Let 
	\begin{equation*}
		\tilde{f}^n_{\alpha,0 } (x,\xi ) = \iint_{\R^3 \times \R^3 } j_{\delta_n } (x-y, \xi - \zeta ) h^n_\alpha (y,\zeta)dyd\zeta .
	\end{equation*}
	As a result, $\tilde{f}^n_{\alpha,0 } \in \mathscr{S} (x,\xi )$.
	For any $ n \geq 1$, there exists $\delta_n^* \in (0,1)$ such that for all $\delta_n \in (0, \delta_n^*)$,
	\begin{equation*}
		\iint_{\R^3 \times \R^3 } |\tilde{f}^n_{\alpha,0 } -h^n_\alpha | dxd\xi \leq \tfrac{1}{n(n+2)^2 }.
	\end{equation*}
Due to $\mathrm{supp} \tilde{f}^n_{\alpha ,0} , \mathrm{supp} h^n_\alpha \subseteq
 \{ (x,\xi ) : |x|^2 + |\xi|^2 \leq (n+1)^2 \}$, one then has
	\begin{equation}\label{eq4.45}
		\iint_{\R^3 \times \R^3 } |\tilde{f}^n_{\alpha,0 } -h^n_\alpha |(1+|x|^2 + |\xi|^2  ) dxd\xi \leq \tfrac{1}{n} \,.
	\end{equation}
Combining \eqref{eq4.44} and \eqref{eq4.45}, we get
	\begin{equation} \label{eq4.46}
		\lim \limits_{n\rightarrow \infty } \iint_{\R^3 \times \R^3 } |\tilde{f}^n_{\alpha,0 } -f_{\alpha,0 } | (1+|x|^2 + |\xi|^2  ) dxd\xi =0.
	\end{equation}
As a consequence, there exists a constant $C_1 > 0$ such that  
	\begin{equation} \label{eq4.47}
		\sup \limits_{n\geq 1 } \iint_{\R^3 \times \R^3 } \tilde{f}^n_{\alpha,0 } (1+|x|^2 + |\xi|^2) dxd\xi \leq C_1 .
	\end{equation}
Note that $\iint_{\R^3 \times \R^3 } j_{\delta_n} dxd\xi =1$ and $ t\log t $ is a continuous convex function over $\R_+$. The Jensen's inequality then implies
	\begin{equation*}
		\begin{aligned}
			&\tilde{f}^n_{\alpha,0 } \log \tilde{f}^n_{\alpha,0 } \leq \iint_{\R^3 \times \R^3 } j_{\delta_n } (x-y, \xi - \zeta ) ( h^n_\alpha \log h^n_\alpha (y,\zeta ))dyd\zeta, \\
			&\tilde{f}^n_{\alpha,0 } \log ^+ \tilde{f}^n_{\alpha,0 } \leq \iint_{\R^3 \times \R^3 } j_{\delta_n } (x-y, \xi - \zeta ) | h^n_\alpha \log h^n_\alpha (y,\zeta )|dyd\zeta.
		\end{aligned}
	\end{equation*}
Integrating above two inequalities with respect to $(x,\xi ) $ and combining the definition of $h^n_\alpha$, one obtains
	\begin{equation}\label{eq4.48}
		\begin{aligned}
            \iint_{\R^3 \times \R^3 } \tilde{f}^n_{\alpha,0 } \log \tilde{f}^n_{\alpha,0 } dxd\xi \leq & \iint_{\R^3 \times \R^3 } f_{\alpha,0 } \log f_{\alpha,0 } dxd\xi, \\
            \iint_{\R^3 \times \R^3 } \tilde{f}^n_{\alpha,0 } \log ^+ \tilde{f}^n_{\alpha,0 } dxd\xi \leq &  \iint_{\R^3 \times \R^3 } f_{\alpha,0 } | \log f_{\alpha,0 } | dxd \xi.
	\end{aligned}
	\end{equation}
	By \eqref{eq4.46} and \eqref{eq4.48}, we know $\tilde{f}^n_{\alpha,0 }$ converges strongly to $f_{\alpha,0 }$ in $L^1 (\R^3 \times \R^3 )$ and
	\begin{equation*}
		\begin{aligned}
   \iint_{\R^3 \times \R^3 } f_{\alpha,0 } \log f_{\alpha,0 } dxd\xi
  & \geq \limsup \limits_{n\rightarrow \infty }  \iint_{\R^3 \times \R^3 } \tilde{f}^n_{\alpha,0 } \log \tilde{f}^n_{\alpha,0 } dxd\xi .
		\end{aligned}
	\end{equation*}

Besides, by Corollary \ref{cor4.5}, the assumption \eqref{fs0} and the first inequality in \eqref{eq4.48}, one concludes that there exists a constant $C_0 > 0$ such that for all $n \geq 1$
	\begin{equation*}
		\iint_{\R^3 \times \R^3  } \tilde{f}^n_{\alpha,0 } \log^- \tilde{f}^n_{\alpha,0 } dxd\xi \leq C_0.
	\end{equation*}
Combining the second inequality in \eqref{eq4.48}, one sees
	\begin{equation*}
		\sup \limits_{n} \iint_{\R^3 \times \R^3  } \tilde{f}^n_{\alpha,0 }| \log \tilde{f}^n_{\alpha,0 }| dxd\xi \leq \iint_{\R^3 \times \R^3 } f_{\alpha,0 } |\log f_{\alpha,0 }| dxd\xi+ C_0.
	\end{equation*}
In summary, the above inequalities show that such $\tilde{f}^n_{\alpha,0 }$ satisfies the first, the third and the last relations in \eqref{Ap-fs0} of Lemma \ref{Lemma 4.8}. 

In order to obtain a sequence of functions that satisfies all the conditions of the Lemma \ref{Lemma 4.8}, let
\begin{equation*}
	f^n_{\alpha,0 } = \tilde{f}^n_{\alpha,0 } + \tfrac{1}{n} \exp (-\tfrac{|x|^2 +|\xi|^2 } {2}) \triangleq \tilde{f}^n_{\alpha,0 } + \varphi_n \in \mathscr{S} (\R^3 \times \R^3 ) .
\end{equation*}
Obviously, $f^n_{\alpha,0 }$ satisfies the first two conditions in \eqref{Ap-fs0}. Using the elementary inequality $(u+v) \log (u+v ) \leq u \log u + v \log v + u+ v$ and letting $u=\tilde{f}^n_0, v =\varphi_n,$ one has
	\begin{align}\label{eq4.49}
		\no \iint_{\R^3 \times \R^3 } f^n_{\alpha,0 } \log ^+ f^n_{\alpha,0 } d x d \xi \leq & \iint_{\R^3 \times \R^3 } \tilde{f}^n_{\alpha,0 } | \log \tilde{f}^n_{\alpha,0 } | dxd\xi + \iint_{\R^3 \times \R^3 } \tilde{f}^n_{\alpha,0 } dxd\xi  \\
		+& \iint_{\R^3 \times \R^3 } (\tfrac{1+\log n}{n} +\tfrac{|x|^2 +|\xi|^2 } {2n} ) \exp (-\tfrac{|x|^2 +|\xi|^2 } {2}) dxd\xi \leq C .
	\end{align}
Since $\varphi_n \to 0 $ and $\tilde{f}^n_{\alpha,0 } \to f_{\alpha ,0}$ strongly in $L^1(\R^3\times \R^3)$, one knows that $ f^n_{\alpha ,0}$ converges strongly to $f_{\alpha ,0}$ in $L^1(\R^3\times \R^3)$. Similarly in arguments about $\tilde{f}^n_{\alpha,0 }$ above, we can conclude that 
\begin{equation*}
	\begin{aligned}
		\iint_{\R^3 \times \R^3 } f^n_{\alpha,0 } \log ^- f^n_{\alpha,0 } dxd\xi \leq C_2
	\end{aligned}
\end{equation*}
for some $C_2 > 0$. Then the third inequality in \eqref{Ap-fs0} holds. Besides, it is derived from the Differential Mean Value Theorem and $\log x < x$ for $x > 0$ that 
\begin{equation*}
	f^n_{\alpha,0 } \log f^n_{\alpha,0 } \leq \tilde{f}^n_{\alpha,0 } \log \tilde{f}^n_{\alpha,0 } +\varphi_n + \varphi_n ^2 +\tilde{f}^n_{\alpha,0 } \varphi_n \,,
\end{equation*}
which concludes that $ \limsup \limits_{n\rightarrow \infty }  \iint_{\R^3 \times \R^3 } f^n_{\alpha,0 } \log f^n_{\alpha,0 } dxd\xi \leq \iint_{\R^3 \times \R^3 } f_{\alpha,0 } \log f_{\alpha,0 } dxd\xi $. Namely, the forth relation in \eqref{Ap-fs0} holds. The proof of Lemma \ref{Lemma 4.8} is therefore finished.
\end{proof}

\begin{lemma}\label{Lemma 4.9}
	Assume that $f_{\alpha,0 } \in L^1_3( \R^3 \times \R^3 \times \R_+)_+ $ for $\alpha \in \{s_0+1,\cdots ,s\}$ satisfy
	\begin{equation}\label{eq 4.51}
		\iiint_{\R^3 \times \R^3 \times \R_+ } f_{\alpha,0 } (1+|x|^2 +|\xi|^2+I + |\log (I ^{1-\frac{\delta(\alpha )}{2} } f_{\alpha,0 } )|)  dxd\xi dI<+ \infty .
	\end{equation}
	Then there exists a sequence $\{f^n_{\alpha,0 }\}_{n\geq 1} \subset \mathscr{S} (\R^3 \times \R^3 \times \R_+ ) $ (the Schwartz space over $\R^3 \times \R^3 \times \R_+$) and a constant $C > 0$ such that 
	\begin{enumerate}
		\item {$\lim \limits_{n\rightarrow \infty}  \iiint_{\R^3 \times \R^3 \times \R_+ } (1+|x|^2 +|\xi|^2+I) |f^n_{\alpha,0 } - f_{\alpha,0 }| dxd\xi dI=0 $;}\label{1 of Lemma 4.9}
		\item {$f^n_{\alpha,0 } \geq \frac{1}{n} \exp (- \frac{|x|^2+ |\xi|^2 + I }{2})$; }\label{2 of Lemma 4.9}
		\item {$\iiint_{\R^3 \times \R^3 \times \R_+ } f^n_{\alpha,0 } (1+|x|^2 +|\xi|^2 + I + |\log (I ^{1-\frac{\delta(\alpha )}{2} }f^n_{\alpha,0 })|)dxd\xi dI \leq C $;}\label{3 of Lemma 4.9}
		\item {\small $\limsup \limits_{n\rightarrow \infty} \iiint_{\R^3 \times \R^3 \times \R_+ } f^n_{\alpha,0 } \log (I ^{1-\frac{\delta(\alpha )}{2} } f^n_{\alpha,0 } ) dxd\xi dI
			\leq \iiint_{\R^3 \times \R^3 \times \R_+ } f_{\alpha,0 } \log (I ^{1-\frac{\delta(\alpha )}{2} } f_{\alpha,0 } ) dxd\xi dI$.}\label{4 of Lemma 4.9}
	\end{enumerate}
\end{lemma}

\begin{proof}
 Let $h^n_\alpha (x,\xi ,I ) = f_{\alpha,0 } (x,\xi ,I )  \chi_n (x,\xi ) \lambda_n (I)$ for all $n \geq 1$. Here $\lambda_n (I) = \mathbf{1}_{ \{ I \in \R_+ : \frac{1}{n} \leq I \leq n \} } $ and $\chi_n (x, \xi) = \mathbf{1}_{ \{ (x,\xi) : |x|^2+ |\xi|^2 \leq n^2 \} } $. It is obvious that
	\begin{equation}\label{hn-f}
		\lim \limits_{n\rightarrow \infty}  \iiint_{\R^3 \times \R^3 \times \R_+ } (1+|x|^2 +|\xi|^2+I) \cdot |h^n_\alpha - f_{\alpha,0 }| dxd\xi dI=0.
	\end{equation}
	Let $j_{\delta_n} (x, \xi) $ and $ \varphi_{\epsilon_n} (I)$ be the standard mollification. Define
	\begin{equation*}
		\tilde{f}^n_{\alpha,0 } (x,\xi ,I ) = \iiint_{\R_+ \times \R^3 \times \R^3 } j_{\delta_n } (x-y, \xi - \zeta ) h^n_\alpha (y,\zeta , \eta )  \varphi_{\epsilon_n } (I -\eta ) d y d \zeta d \eta ,\ \epsilon_n \in (0,\tfrac{1}{2n}).
	\end{equation*}
	As a result, $\tilde{f}^n_{\alpha,0 } \in \mathscr{S} (x,\xi ,I )_+$. Furthermore, as the similar arguments in \eqref{eq4.46}, one can easily gain
	\begin{equation}\label{fnt-f}
		\lim \limits_{n\rightarrow \infty } \iiint_{\R^3 \times \R^3 \times \R_+ } |\tilde{f}^n_{\alpha,0 } -f_{\alpha,0 } | (1+|x|^2 + |\xi|^2 + I ) dxd\xi dI =0.
	\end{equation}
	Then there exists a constant $C_1 > 0$ such that  
	\begin{equation}\label{fn-bd}
		\sup \limits_{n\geq 1 } \iiint_{\R^3 \times \R^3 \times \R_+ } \tilde{f}^n_{\alpha,0 } (1+|x|^2 + |\xi|^2 + I) dxd\xi dI \leq C_1 .
	\end{equation}
Observe that $\iint_{\R^3 \times \R^3 } j_{\delta_n} dxd\xi =1$, $\int_{\R_+} \varphi_{\epsilon_n} dI =1$ and $ \Phi (t) = t\log ( I^{1 - \delta (\alpha) /2} t ) $ is continuous convex function over $\R_+$. It is then derived from the Jensen's inequality that
	\begin{equation*}{\small
		\begin{aligned}
			&\tilde{f}^n_{\alpha,0 } \log ( I^{1 - \delta (\alpha) /2} \tilde{f}^n_{\alpha,0 } ) \leq \iiint_{\R_+ \times \R^3 \times \R^3 } j_{\delta_n } (x-y, \xi - \zeta ) ( h^n_\alpha \log ( \eta^{1 - \delta (\alpha) /2} h^n_\alpha ) (y,\zeta , \eta )) \varphi_{\epsilon_n } (I -\eta ) d y d \zeta d \eta, \\
			&\tilde{f}^n_{\alpha,0 } \log ^+ ( I^{1 - \delta (\alpha) /2} \tilde{f}^n_{\alpha,0 } ) \leq \iiint_{\R_+ \times \R^3 \times \R^3 } j_{\delta_n } (x-y, \xi - \zeta ) | h^n_\alpha \log ( \eta^{1 - \delta (\alpha) /2} h^n_\alpha ) (y,\zeta , \eta )| \varphi_{\epsilon_n } (I -\eta ) d y d \zeta d \eta.
		\end{aligned}}
	\end{equation*}
	Taking the similar methods in Lemma \ref{Lemma 4.8}, we obtain
	\begin{equation}\label{fn0}{\small
		\begin{aligned}
			\limsup \limits_{n\rightarrow \infty } \iiint_{\R^3 \times \R^3 \times \R_+ } \tilde{f}^n_{\alpha,0 } \log ( I^{1 - \delta (\alpha) /2} \tilde{f}^n_{\alpha,0 } ) dxd\xi dI &\leq \iiint_{\R^3 \times \R^3 \times \R_+ } f_{\alpha,0 } \log ( I^{1 - \delta (\alpha) /2} f_{\alpha,0 } ) dxd\xi dI, \\
		\limsup \limits_{n\rightarrow \infty } \iiint_{\R^3 \times \R^3 \times \R_+ } \tilde{f}^n_{\alpha,0 } \log ^+ ( I^{1 - \delta (\alpha) /2} \tilde{f}^n_{\alpha,0 } ) dxd\xi dI &\leq \iiint_{\R^3 \times \R^3 \times \R_+ } f_{\alpha,0 } |\log ( I^{1 - \delta (\alpha) /2} f_{\alpha,0 } ) | dxd\xi dI.
		\end{aligned}}
	\end{equation}
	Then by the proof of Lemma \ref{Lemma 4.4}, there exists a constant $C_2>0$ such that
\begin{equation*}
\iiint_{\R^3 \times \R^3 \times \R_+ } \tilde{f}^n_{\alpha,0 } \log ^-(I^{1-\delta (\alpha)/2}\tilde{f}^n_{\alpha,0 })dxd\xi dI\leq C_2.
\end{equation*}
Thus, one has
\begin{equation}\label{fn-log-bd}
  \sup \limits_{n \geq 1}   \iiint_{\R^3 \times \R^3 \times \R_+ } \tilde{f}^n_{\alpha,0 } |\log (I^{1-\delta (\alpha)/2}\tilde{f}^n_{\alpha,0 })|dxd\xi dI\leq C.
\end{equation}

The above inequalities show that such $\tilde{f}^n_{\alpha,0}$ satisfies the conditions \eqref{1 of Lemma 4.9}, \eqref{3 of Lemma 4.9}, \eqref{4 of Lemma 4.9} of Lemma \ref{Lemma 4.9}. To obtain a sequence of functions $f^n_{\alpha, 0}$ that satisfies all the conditions of the Lemma \ref{Lemma 4.9}, let
 \begin{equation*}
	f^n_{\alpha,0 } = \tilde{f}^n_{\alpha,0 } + \tfrac{1}{n} \exp (-\tfrac{|x|^2 +|\xi|^2 + I} {2}) \triangleq \tilde{f}^n_{\alpha,0 } + \varphi_n.
\end{equation*}
Then $f^n_{\alpha,0 } \in \mathscr{S} (\R^3 \times \R^3 \times \R_+)_+$. Obviously, $f^n_{\alpha,0 }$ satisfies the relations \eqref{1 of Lemma 4.9}, \eqref{2 of Lemma 4.9} and
\begin{equation*}
  \begin{aligned}
    \iiint_{\R^3 \times \R^3 \times \R_+ } f^n_{\alpha,0 } (1+|x|^2 +|\xi|^2 + I )dxd\xi dI \leq C \,.
  \end{aligned}
\end{equation*} 
By using the elementary inequality $ (u+v) \log (u+v ) \leq u \log u + v \log v + u+ v $ and letting $u=I^{1-\delta(\alpha)/2} \tilde{f}^n_{\alpha,0}, v = I^{1- \delta(\alpha)/2} \varphi_n,$ one has
\begin{equation*}
	\begin{aligned}
		&\iiint_{\R^3 \times \R^3 \times \R_+ } f^n_{\alpha,0 } \log ^+ (I^{1-\delta(\alpha)/2}f^n_{\alpha,0 } )dxd\xi dI \\
		\leq & \iiint_{\R^3 \times \R^3 \times \R_+ } \tilde{f}^n_{\alpha,0 } |\log (I^{ 1-\delta(\alpha)/2} \tilde{f}^n_{\alpha,0 } )| dxd\xi dI 
		+ \iiint_{\R^3 \times \R^3 \times \R_+ } \tilde{f}^n_{\alpha,0 } dxd\xi dI \\
		+& \iiint_{\R^3 \times \R^3 \times \R_+ } (\tfrac{1+\log n}{n} +\tfrac{|x|^2 +|\xi|^2 + I + \delta(\alpha) I} {2n} )
		\exp (-\tfrac{|x|^2 +|\xi|^2 + I} {2}) dxd\xi dI \leq C (\alpha ).
	\end{aligned}
\end{equation*}
Combining the proof of Lemma \ref{Lemma 4.4}, we conclude that 
\begin{equation*}
	\iiint_{\R^3 \times \R^3 \times \R_+ } f^n_{\alpha,0 } |\log  (I^{1-\delta(\alpha)/2} f^n_{\alpha,0 } )| dxd\xi dI \leq C.
\end{equation*}	
Then the relation \eqref{3 of Lemma 4.9} holds. Besides, it is derived from the Differential Mean Value Theorem that
\begin{equation*}
	f^n_{\alpha,0} \log (I^{1-\delta(\alpha)/2} f^n_{\alpha,0}) \leq \tilde{f}^n_{\alpha,0} \log (I^{1-\delta(\alpha)/2} \tilde{f}^n_{\alpha,0}) +\varphi_n ( 1 + \log ( I^{1-\delta(\alpha)/2} {f}^n_{\alpha,0} ) ) .
\end{equation*}
Note that by $\log x < x$ over $x \in \R_+$,
\begin{equation*}
  \begin{aligned}
    \log ( I^{1-\delta(\alpha)/2} {f}^n_{\alpha,0} ) \leq (1 - \tfrac{\delta (\alpha)}{2}) \mathbf{I}_{0 < I < 1} \log I + {f}^n_{\alpha,0} \,.
  \end{aligned}
\end{equation*}
Then there holds
\begin{equation*}
  \begin{aligned}
    \iiint_{\R^3 \times \R^3\times \R_+} \varphi_n ( 1 + \log ( I^{1-\delta(\alpha)/2} {f}^n_{\alpha,0} ) ) d x d \xi d I \leq \tfrac{C}{n} \,.
  \end{aligned}
\end{equation*}
Together with \eqref{fn0}, one knows that the relation \eqref{4 of Lemma 4.9} holds. The proof of Lemma \ref{Lemma 4.9} is therefore finished.
\end{proof}

Together with Lemma \ref{Lemma 4.6}, Lemma \ref{Lemma 4.8} and Lemma \ref{Lemma 4.9} conclude the following lemma.

\begin{lemma}\label{Lmm-Ini-Entpy}
  Let the approximated initial data $f^n_0 = (f^n_{\alpha, 0})_{\alpha \in \{ 1, \cdots , s\}}$ of $f_0$ be constructed in Lemma \ref{Lemma 4.8} and Lemma \ref{Lemma 4.9}. Then
  \begin{equation}{\small
    \begin{aligned}
		& \sum_{\alpha = 1}^{s_0} \iint_{\R^3 \times \R^3} f_{\alpha, 0} \log f_{\alpha, 0} dxd\xi + \sum_{\alpha = s_0 + 1}^{s} \iiint_{\R^3 \times \R^3\times \R_+ } f_{\alpha, 0} \log (I^{1-\delta(\alpha )/2 } f_{\alpha, 0}) dxd\xi dI \\
= & \lim \limits_{n\rightarrow \infty } \Big\{ \sum_{\alpha = 1}^{s_0} \iint_{\R^3 \times \R^3 } f^n_{\alpha, 0} \log f^n_{\alpha, 0} dxd\xi + \sum_{\alpha = s_0 + 1}^{s} \iiint_{\R^3 \times \R^3\times \R_+ } f^n_{\alpha, 0} \log (I^{1-\delta(\alpha)/2}f^n_{\alpha, 0}) dxd\xi dI \Big\} .
\end{aligned}}
	\end{equation} 
Hence $\lim \limits_{n \to \infty} H (f^n_0) = H (f_0)$, where the entropy $H (f_0)$ is defined in \eqref{Hf}.
\end{lemma}

\subsection{Smooth approximation of collision kernels}

In this subsection, the goal is to smoothly approximate the collision kernels $A_{0\alpha \beta },\ A_{1\alpha \beta },\ A_{1\beta \alpha }$ and $ A_{2\alpha \beta }$ defined in \eqref{A}, which will be used to construct the approximation of nonlinear collision operator.

	\begin{lemma}\label{Lemma 4.10}
	Assume collision kernel $A_{2\alpha \beta }$ satisfies the assumption in \eqref{collision}, i.e.,
	\begin{equation*}
		(1+|\xi|^2 + I )^{-1} \int_{|z-\xi |\leq R } \int_{0<\eta -I < R } A_{2\alpha \beta } (z,\eta ) dzd\eta \rightarrow 0, \quad |\xi|, I \rightarrow \infty , \ \forall R <+\infty .
	\end{equation*}	
	Then there exists $B_{2\alpha \beta ,n }\geq 0 $ satisfying
	\begin{enumerate}
		\item {For any fixed $(\mathfrak{R},\mathfrak{r},\omega)\in [0,1]^2 \times \mathbb{S}^2, B_{2\alpha \beta , n } (z,\eta , \mathfrak{R},\mathfrak{r},\omega ) \in C^\infty (\R^3 \times \R_+), supp B_{2\alpha \beta , n } \subset \{(z,\eta , \mathfrak{R},\mathfrak{r},\omega ): \frac{1}{n} \leq |z|,\eta \leq n \}$ and if $ z \cdot \omega \leq \frac{1}{n}, B_{2\alpha \beta , n } =0$;	}\label{1 of Lemma 4.10}
		\item  {For almost all $(z,\eta , \mathfrak{R},\mathfrak{r},\omega )\in \R^3 \times \R_+ \times [0,1]^2 \times \mathbb{S}^2$, 
  $ \lim \limits_{n\rightarrow \infty } B_{2\alpha \beta ,n } =B_{2\alpha \beta } $, where $B_{2\alpha \beta }$ is defined in \eqref{B2-ab};
  } \label{2 of Lemma 4.10}
		\item {For any $R < + \infty$,
		\begin{equation*}
		\sup \limits_{n \geq 1} (1+|\xi|^2 + I )^{-1} \int_{|z-\xi |\leq R } \int_{0<\eta -I < R } A_{2\alpha \beta ,n } (z,\eta ) dzd\eta \rightarrow 0, \ |\xi|, I \rightarrow \infty ,
		\end{equation*}	
  where 
  \begin{equation*}
    A_{2\alpha \beta ,n }=\int_{[0,1]^2 \times \mathbb{S}^2} B_{2\alpha \beta ,n } \mathfrak{r}^{\delta(\alpha)/2-1 } (1-\mathfrak{r})^{\delta(\beta )/2-1} (1-\mathfrak{R})^{(\delta(\alpha)+(\delta(\beta ))/2-1} \mathfrak{R}^{1/2}  d\mathfrak{R}d\mathfrak{r}d\omega.  
  \end{equation*} 
}\label{3 of Lemma 4.10}
	\end{enumerate}
\end{lemma}

\begin{proof}
	Assume that $\chi_n (z,\eta )$ and $\lambda_n(z,\eta ,\mathfrak{R} ,\mathfrak{r}, \omega )$ are respectively the characteristic functions of sets
	\begin{equation*}
		\{(z,\eta )\in \R^3 \times \R_+ :\tfrac{1}{2n}\leq |z|,\eta \leq n-\tfrac{1}{2n} \}
	\end{equation*}
	and 
	\begin{equation*}
		\{(z,\eta ,\mathfrak{R} ,\mathfrak{r}, \omega )\in \R^3 \times \R_+ \times [0,1]^2 \times \mathbb{S}^2: z\cdot \omega \geq \tfrac{1}{2n},\ \tfrac{1}{2n}\leq \mathfrak{R},\mathfrak{r} \leq 1-\tfrac{1}{2n}\}.
	\end{equation*}
	Let
	\begin{equation*}
		\begin{aligned}
		\hat{B}_{2\alpha \beta }=B_{2\alpha \beta } \cdot \mathfrak{r}^{\delta(\alpha)/2-1 }(1-\mathfrak{r})^{\delta(\beta )/2-1}&(1-\mathfrak{R})^{(\delta(\alpha)+\delta(\beta ))/2-1} \mathfrak{R}^{1/2},\\
  \tilde{B}_{2\alpha \beta,n } =\hat{B}_{2\alpha \beta } \cdot \chi_n \cdot \lambda_n,\quad & 
  \hat{B}_{2\alpha \beta ,n }=j_{\delta_n }* \tilde{B}_{2\alpha \beta ,n},\\
  B_{2\alpha \beta ,n }= \hat{B}_{2\alpha \beta,n } \cdot \mathfrak{r}^{1-\delta(\alpha)/2 }(1-\mathfrak{r})^{1-\delta(\beta )/2}&(1-\mathfrak{R})^{1-(\delta(\alpha)+\delta(\beta ))/2} \mathfrak{R}^{-\frac{1}{2}}
		\end{aligned}
	\end{equation*}
	with $0< \delta_n \leq \frac{1}{2n}$, where the convolution is done for the variables $z, \eta$. Then $B_{2\alpha \beta ,n }$ satisfies the condition \eqref{1 of Lemma 4.10}.

For all subset $K \Subset \R^3 \times \R_+,$ there exists $n \geq 1 $ such that $K \subset \{(z,\eta ) \in \R^3 \times \R_+ : 0<|z|,\eta <n-\frac{1}{2n}\}$. Then
	\begin{equation*}{\small
		\begin{aligned}
			&\int_{K} dzd\eta \int_{[0,1]^2 \times \mathbb{S}^2 } |\hat{B}_{2\alpha \beta }-\tilde{B}_{2\alpha \beta ,n}|d\mathfrak{R}d\mathfrak{r}d\omega \\
			\leq & \int_{K \cap \{|z|\leq \frac{1}{2n}\}} dzd\eta \int_{[0,1]^2 \times \mathbb{S}^2 } \hat{B}_{2\alpha \beta } d\mathfrak{R}d\mathfrak{r}d\omega + \int_{K \cap \{\eta \leq \frac{1}{2n}\}} dzd\eta \int_{[0,1]^2 \times \mathbb{S}^2 } \hat{B}_{2\alpha \beta } d\mathfrak{R}d\mathfrak{r}d\omega \\
		+& \int_{K} d z d \eta \int_{ [0,1]^2 } d \mathfrak{r} d \mathfrak{R} \int_{ z \cdot \omega < \frac{1}{2n}} \hat{B}_{ 2 \alpha \beta } d \omega +\int_{K} dzd\eta \int_{[0,1]^2 \times \mathbb{S}^2 } \hat{B}_{2\alpha \beta } (\mathbf{1}_{0<\mathfrak{R}<\frac{1}{2n}}+\mathbf{1}_{0<\mathfrak{r}<\frac{1}{2n}}) d\mathfrak{R}d\mathfrak{r}d\omega\\
  +& \int_{K} dzd\eta \int_{[0,1]^2 \times \mathbb{S}^2 } \hat{B}_{2\alpha \beta } (\mathbf{1}_{1-\frac{1}{2n}<\mathfrak{R}<1}+ \mathbf{1}_{1-\frac{1}{2n}<\mathfrak{r}<1}) d\mathfrak{R}d\mathfrak{r}d\omega.
		\end{aligned}}
	\end{equation*}
	Since $\hat{B}_{2\alpha \beta }=\tilde{\sigma}_{\alpha \beta }|g|(1-\mathfrak{R})E^2_{\alpha \beta } \in L^1_{loc} (\R^3 \times \R_+;L^1 ([0,1]^2 \times \mathbb{S}^2)),$ the Dominated Convergence Theorem concludes that 
	\begin{equation*}
		\lim \limits_{n\rightarrow \infty } \int_{K} dzd\eta \int_{[0,1]^2 \times \mathbb{S}^2 } |\hat{B}_{2\alpha \beta }-\tilde{B}_{2\alpha \beta ,n}|d\mathfrak{R}d\mathfrak{r}d\omega =0.
	\end{equation*}
Moreover, by the property of convolution, we know that for all $n \geq 1$ there exists a $\delta_n^* \in (0,\frac{1}{2n})$ such that for $\delta_n \in (0, \delta_n^*)$
	\begin{equation*}
		\int_{K} dzd\eta \int_{[0,1]^2 \times \mathbb{S}^2 } |\hat{B}_{2\alpha \beta ,n }-\tilde{B}_{2\alpha \beta ,n}|d\mathfrak{R}d\mathfrak{r}d\omega <\tfrac{1}{2n}.
	\end{equation*}
	Therefore $\hat{B}_{2\alpha \beta ,n } \rightarrow \hat{B}_{2\alpha \beta} \enspace $in $L^1_{loc} (\R^3\times \R_+ ; L^1([0,1]^2 \times \mathbb{S}^2 ))$. It therefore follows that \eqref{2 of Lemma 4.10} is true. 

Due to $\tilde{B}_{2\alpha \beta ,n}\leq \hat{B}_{2\alpha \beta},$ one has $A_{2\alpha \beta ,n } \leq j_{\delta_n}* A_{2\alpha \beta }.$ Then
\begin{equation*}{\small
\begin{aligned}
	&\int_{|z-\xi|\leq R} \int_{0< \eta -I <R } A_{2\alpha \beta ,n } (z,\eta )dzd\eta \\
	\leq& \int_{|z-\xi|\leq R}dz \int_{0< \eta -I <R } d\eta \iint_{\R^3\times \R_+} j_{\delta_n}(\zeta ,J) A_{2\alpha \beta } (z-\zeta ,\eta - J ) d\zeta dJ\\
	\leq & \iint_{\R^3\times \R_+} j_{\delta_n}(\zeta ,J) d\zeta dJ \int_{|\psi - \xi |\leq R+1} \int_{0< w -(I-\frac{1}{2}) <R +1} A_{2\alpha \beta }(\psi ,w) d\psi dw\\
	= & \int_{|\psi - \xi |\leq R+1} \int_{0< w -(I-\frac{1}{2}) <R+1 } A_{2\alpha \beta }(\psi ,w) d\psi dw.
\end{aligned}}
\end{equation*}
Together with the assumption on $A_{2 \alpha \beta}$, one then gains that for all $n\geq 1$,
\begin{equation*}{\small
\begin{aligned}
	& (1+|\xi|^2 + I )^{-1} \int_{|z-\xi |\leq R } \int_{0<\eta -I < R } A_{2\alpha \beta ,n } (z,\eta ) dzd\eta \\
	\leq & (1+|\xi|^2 + (I-\tfrac{1}{2}) )^{-1} \int_{|\psi - \xi |\leq R+1} \int_{0< w -(I-\frac{1}{2}) <R +1} A_{2\alpha \beta }(\psi , w) d\psi dw \rightarrow 0
\end{aligned}}
\end{equation*}
as $|\xi|, I \rightarrow \infty$, which means that \eqref{3 of Lemma 4.10} holds. Therefore, the proof of Lemma \ref{Lemma 4.10} is finished.
\end{proof}

\begin{lemma}\label{Lemma 4.11}
	Assume that collision kernels $A_{1\alpha \beta }$ and $A_{1\beta \alpha }$ satisfy the assumptions \eqref{collision}, i.e.,
	\begin{equation*}{\small
		\begin{aligned}
			&(1+|\xi|^2 + \eta )^{-1} \int_{|z-\xi |\leq R } A_{1\alpha \beta } (z,\eta ) dz,  \quad |\xi|,\eta  \rightarrow \infty, \ \forall R <+\infty, \\
			&(1+|\xi|^2 )^{-1}\int_{|z-\xi |\leq R }  \int_{0<\eta < R } A_{1\beta \alpha } (z,\eta ) dzd\eta \rightarrow 0, \quad |\xi| \rightarrow \infty , \ \forall R <+\infty.
		\end{aligned}}
	\end{equation*}	
	Then there exist $B_{1\alpha \beta ,n }(z,\eta , \mathfrak{R},\omega ) ,B_{1 \beta \alpha ,n } (z,\eta , \mathfrak{R},\omega ) \geq 0 $ satisfying 
	\begin{enumerate}
		\item {For any $(\mathfrak{R},\omega)\in [0,1] \times \mathbb{S}^2,\ B_{1\alpha \beta , n } , B_{1 \beta \alpha ,n }  \in C^\infty (\R^3 \times \R_+),\ supp B_{1\alpha \beta , n },supp B_{1 \beta \alpha ,n } \subset \{(z,\eta , \mathfrak{R},\omega ): \frac{1}{n} \leq |z|,\eta \leq n \}$ and if $ (\xi -\xi_* ) \cdot \omega \leq \frac{1}{n} , B_{1\alpha \beta , n } =B_{1 \beta \alpha ,n }=0$;}\label{1 of Lemma 4.11}
		\item  {For almost all $(z,\eta , \mathfrak{R},\omega )\in \R^3 \times \R_+ \times [0,1] \times \mathbb{S}^2$,
		\begin{equation*}
  \lim \limits_{n\rightarrow \infty } B_{1\alpha \beta ,n } =B_{1\alpha \beta },\quad
 \lim \limits_{n\rightarrow \infty } B_{1\beta \alpha ,n } =B_{1\beta \alpha },
\end{equation*} 
where $B_{1 \alpha \beta}$ and $B_{1 \beta \alpha}$ are respectively defined in \eqref{B1-ab} and \eqref{B1-ba};} \label{2 of Lemma 4.11}
		\item {For all $R < + \infty$,
			\begin{equation*}{\small
				\begin{aligned}
					&\sup \limits_{n \geq 1} (1+|\xi|^2 + \eta )^{-1} \int_{|z-\xi |\leq R } A_{1\alpha \beta ,n } (z,\eta ) dz \rightarrow 0, \quad |\xi|, \eta \rightarrow \infty , \\
					&\sup \limits_{n \geq 1} (1+|\xi|^2 )^{-1}\int_{|z-\xi |\leq R }  \int_{0<\eta < R } A_{1\beta \alpha ,n} (z,\eta ) dzd\eta \rightarrow 0, \quad |\xi| \rightarrow \infty \,,
				\end{aligned}}
		\end{equation*}
   where 
{\small
    \begin{align*}
   A_{1\alpha \beta ,n } =\iint_{[0,1]\times \mathbb{S}^2} B_{1\alpha \beta ,n } (1-\mathfrak{R})^{\delta(\beta )/2-1}\mathfrak{R}^{1/2}  d\mathfrak{R}d\omega, A_{1\beta \alpha ,n } =\iint_{[0,1]\times \mathbb{S}^2} B_{1\beta \alpha ,n }   (1-\mathfrak{R})^{\delta(\alpha )/2-1}\mathfrak{R}^{1/2}d\mathfrak{R}d\omega.
   \end{align*}}
  }\label{3 of Lemma 4.11}
	\end{enumerate}
\end{lemma}

\begin{proof}
	Denote by
	\begin{equation*}
    \begin{aligned}
		& \chi_n (z,\eta ) = \mathbf{1}_{ \{(z,\eta )\in \R^3 \times \R_+ :\frac{1}{2n}\leq |z|,\eta \leq n-\frac{1}{2n} \} } \,, \\
        & \lambda_n(z,\eta ,\mathfrak{R} ,\omega ) = \mathbf{1}_{ \{(z,\eta ,\mathfrak{R} ,\omega )\in \R^3 \times \R_+ \times [0,1] \times \mathbb{S}^2: z\cdot \omega \geq \frac{1}{2n},\ \frac{1}{2n}\leq \mathfrak{R}<1-\frac{1}{2n} \} }.
    \end{aligned}
	\end{equation*}
	Let
\begin{equation*}{\small
\begin{aligned}
  & \hat{B}_{1\alpha \beta}(z,\eta ,\mathfrak{R},\omega )=B_{1\alpha \beta} (1-\mathfrak{R})^{\delta(\beta )/2-1}\mathfrak{R}^{1/2}, \ \hat{B}_{1\beta \alpha}(z,\eta ,\mathfrak{R},\omega )=B_{1\beta \alpha}(1-\mathfrak{R})^{\delta(\alpha )/2-1}\mathfrak{R}^{1/2}, \\
  & \tilde{B}_{1\alpha \beta ,n} (z,\eta ,\mathfrak{R},\omega )=\hat{B}_{1\alpha \beta }\cdot \chi_n \cdot \lambda_n, \ \tilde{B}_{1\beta \alpha ,n} (z,\eta ,\mathfrak{R},\omega )=\hat{B}_{1\beta \alpha }\cdot \chi_n \cdot \lambda_n, \\ 
  & \hat{B}_{1\alpha \beta ,n }(z,\eta,\mathfrak{R},\omega ) =\int_{\R^3} j_{\delta_n }(z-y) \tilde{B}_{1\alpha \beta ,n}(y,\eta,\mathfrak{R},\omega )dy,\\ 
  & \hat{B}_{1\beta \alpha ,n }(z,\eta,\mathfrak{R},\omega ) =\iint_{\R^3 \times \R_+} \varphi_{\epsilon_n }(z-y,\eta-I) \tilde{B}_{1\beta \alpha ,n}(y,I,\mathfrak{R},\omega)dydI,\\
  & B_{1\alpha \beta,n}(z,\eta ,\mathfrak{R},\omega )=\hat{B}_{1\alpha \beta,n} (1-\mathfrak{R})^{1-\delta(\beta )/2}\mathfrak{R}^{-\frac{1}{2}},\ B_{1\beta \alpha,n}(z,\eta ,\mathfrak{R},\omega )=\hat{B}_{1\beta \alpha,n}(1-\mathfrak{R})^{1-\delta(\alpha )/2}\mathfrak{R}^{-\frac{1}{2}} \,, 
		\end{aligned}}
	\end{equation*}
	where $0< \delta_n,\epsilon_n \leq \frac{1}{2n} $ and $j_{\delta_n}, \varphi_{\epsilon_n}$ are the standard mollifiers. Then $B_{1\alpha \beta ,n }$ and $B_{1\beta \alpha ,n }$ satisfy the condition \eqref{1 of Lemma 4.11}. 

Taking the similar arguments in Lemma \ref{Lemma 4.10}, we obtain
	\begin{equation*}
		\hat{B}_{1\alpha \beta ,n } \rightarrow \hat{B}_{1\alpha \beta}, \quad
		\hat{B}_{1\beta \alpha ,n} \rightarrow \hat{B}_{1\beta \alpha } \quad \text{in }  L^1_{loc} (\R^3\times \R_+ ; L^1([0,1]\times \mathbb{S}^2 )).
	\end{equation*}
It then follows that \eqref{2 of Lemma 4.11} is true. 

Due to $\tilde{B}_{1\alpha \beta ,n}\leq \hat{B}_{1\alpha \beta}, \tilde{B}_{1\beta \alpha ,n} \leq \hat{B}_{1\beta \alpha }$ we can conclude that $A_{1\alpha \beta ,n } \leq j_{\delta_n}* A_{1\alpha \beta },A_{1\beta \alpha ,n } \leq \varphi_{\epsilon_n}* A_{1\beta \alpha }.$ Then for $\alpha\in \{1,\cdots,s_0\} $ and $ \beta \in \{s_0+1,\cdots,s\}$,
	\begin{equation*}
		\begin{aligned}
			& \int_{|z-\xi|\leq R}  A_{1\alpha \beta ,n } (z,\eta )dz
			\leq \int_{|z-\xi|\leq R} \int_{\R^3} j_{\delta_n}(\zeta ) A_{1\alpha \beta } (z-\zeta ,\eta ) d\zeta dz\\
			= & \int_{\R^3} j_{\delta_n}(\zeta ) \int_{|\psi +\zeta - \xi |\leq R} A_{1\alpha \beta }(\psi ,\eta) d\psi d\zeta \leq \int_{|\psi - \xi |\leq R+1}A_{1\alpha \beta }(\psi,\eta ) d\psi \,.
		\end{aligned}
	\end{equation*}
Moreover, for $\alpha\in \{s_0+1,\cdots,s\} $ and $ \beta \in \{1,\cdots,s_0\}$,
	\begin{equation*}{\small
		\begin{aligned}
			\int_{|z-\xi|\leq R} \int_{0< \eta <R } A_{1\beta\alpha ,n } (z,\eta )dzd\eta \leq& \int_{|z-\xi|\leq R}\int_{0< \eta <R } \iint_{\R^3\times \R_+} \varphi_{\epsilon_n} (\zeta ,J) A_{1\beta\alpha } (z-\zeta ,\eta - J ) d\zeta dJ d\eta dz \\
			= & \iint_{\R^3\times \R_+} \varphi_{\epsilon_n}(\zeta ,J) \int_{|\psi +\zeta - \xi |\leq R} \int_{0< K + J <R } A_{1\beta\alpha }  (\psi , K) d\psi dK dJ d\zeta \\
			\leq & \int_{|\psi - \xi |\leq R+1} \int_{0< K +\frac{1}{2}<R+1} A_{1\beta\alpha }(\psi , K) dK d\psi.
		\end{aligned}}
	\end{equation*}
	Therefore, together with the assumptions on $A_{1\alpha \beta }$ and $A_{1\beta\alpha }$, there hold that for all $ n \geq 1 $ and $0 < R< +\infty$,
		\begin{align*}
			& (1+|\xi|^2 + \eta )^{-1} \int_{|z-\xi |\leq R }  A_{1\alpha \beta ,n } (z,\eta ) dz\rightarrow 0, \quad |\xi|, \eta \rightarrow \infty , \\
			&(1+|\xi|^2 )^{-1} \int_{|z-\xi|\leq R} \int_{0< \eta <R } A_{1\beta\alpha ,n } (z,\eta )dzd\eta\rightarrow 0, \quad |\xi| \rightarrow \infty .
		\end{align*}
	The condition \eqref{3 of Lemma 4.11} is thus true. The proof of Lemma \ref{Lemma 4.11} is then finished.
\end{proof}

	\begin{lemma}\label{Lemma 4.12}
	Assume that collision kernel $A_{0\alpha \beta }$ satisfies the assumption \eqref{A}, hence,
	\begin{equation*}
		(1+|\xi|^2)^{-1} \int_{|z-\xi |\leq R } A_{0\alpha \beta } (z)dz, \quad |\xi|\rightarrow \infty, \ \forall R <+\infty .
	\end{equation*}	
	Then there exists $B_{0\alpha \beta ,n }\geq 0 $ satisfying
	\begin{enumerate}
		\item {For any $\omega\in \mathbb{S}^2, B_{0\alpha \beta ,n } (z,\omega ) \in C^\infty (\R^3 ),supp B_{0\alpha \beta , n} \subset \{(z,\omega ): \frac{1}{n} \leq |z| \leq n \}$ and if $ (\xi -\xi_* ) \cdot \omega \leq \frac{1}{n} , B_{0\alpha \beta ,n } =0$;}\label{1 of Lemma 4.12}
		\item  {For almost all $(z,\omega )\in \R^3 \times \mathbb{S}^2 , \lim \limits_{n\rightarrow \infty } B_{0\alpha \beta ,n } =B_{0\alpha \beta }$, where $B_{0\alpha \beta }$ is defined in \eqref{B0-ab};} \label{2 of Lemma 4.12}
		\item {For any $0 < R < + \infty$,
			\begin{equation*}
				\sup \limits_{n \geq 1} (1+|\xi|^2 )^{-1} \int_{|z-\xi |\leq R } A_{0\alpha \beta ,n } (z) dz \rightarrow 0 ,
		\end{equation*}	
where $A_{0 \alpha \beta , n} (z) = \int_{ \mathbb{S}^2 } B_{0 \alpha \beta, n} (z, \omega) \d \omega$. }\label{3 of Lemma 4.12}
	\end{enumerate}
\end{lemma}	
		\begin{proof}
	Denote by $	\chi_n (z) = \mathbf{1}_{ \{z\in \R^3 :\frac{1}{2n}\leq |z|\leq n-\frac{1}{2n} \} }$ and $\lambda_n(z, \omega ) = \mathbf{1}_{ \{(z, \omega )\in \R^3 \times \mathbb{S}^2: z\cdot \omega \geq \frac{1}{2n}\} }$. Let
	\begin{equation*}
		\tilde{B}_{0\alpha \beta ,n} (z,\omega )=B_{0\alpha \beta }\cdot \chi_n \cdot \lambda_n,\  B_{0\alpha \beta ,n } =j_{\delta_n }* \tilde{B}_{0\alpha \beta ,n},
	\end{equation*}
where $0< \delta_n \leq \frac{1}{2n}.$ Here the convolution is done for the variables $z \in \R^3$. Then the kernel $B_{0\alpha \beta ,n }$ satisfies the condition \eqref{1 of Lemma 4.12}. Taking the similar methods in Lemma \ref{Lemma 4.9}, we obtain
	\begin{equation*}
		B_{0\alpha \beta ,n } \rightarrow B_{0\alpha \beta} \text{ in } L^1_{loc} (\R^3 ; L^1( \mathbb{S}^2 )) \,,
	\end{equation*}
which infers \eqref{2 of Lemma 4.12}. Moreover, the fact $\tilde{B}_{0\alpha \beta ,n}\leq B_{0\alpha \beta}$ implies $A_{0\alpha \beta ,n } \leq j_{\delta_n}* A_{0\alpha \beta }.$ Then one has
	\begin{equation*}
		\begin{aligned}
			\int_{|z-\xi|\leq R}  A_{0\alpha \beta ,n } (z )dz 
			\leq & \int_{|z-\xi|\leq R}\int_{\R^3} j_{\delta_n}(\zeta) A_{0\alpha \beta } (z-\zeta ) d\zeta dz \\
			= & \int_{\R^3} j_{\delta_n}(\zeta )  \int_{|\psi +\zeta - \xi |\leq R} A_{0\alpha \beta }(\psi ) d\psi d\zeta \leq \int_{|\psi - \xi |\leq R+1}  A_{0\alpha \beta }(\psi) d\psi \,,
		\end{aligned}
	\end{equation*}
which, together with the assumption \eqref{A} of $A_{0 \alpha \beta}$, concludes \eqref{3 of Lemma 4.12}. The proof of Lemma \ref{Lemma 4.12} is therefore finished.
\end{proof}

\subsection{Approximation of collision operator}

Now we construct the approximated collision operator  $\tilde{Q}_n=(\tilde{Q}_{1,n},\cdots,\tilde{Q}_{s,n} )$ as follows. Recalling the collision kernels $B_{2 \alpha \beta,n}$, $B_{1 \alpha \beta,n}$, $B_{1 \beta \alpha,n}$ and $B_{0 \alpha \beta,n}$ constructed in Lemma \ref{Lemma 4.10}-\ref{Lemma 4.11}-\ref{Lemma 4.12}, we first define
\begin{equation*}{\small
     \begin{aligned}
 & \hat{B}_{1\alpha \beta,n}(z,\eta ,\mathfrak{R},\omega ) =B_{1\alpha \beta ,n } (1-\mathfrak{R})^{\delta(\beta )/2-1}\mathfrak{R}^{1/2} \,, \hat{B}_{1\beta \alpha,n }(z,\eta ,\mathfrak{R},\omega ) =B_{1\beta \alpha ,n }(1-\mathfrak{R})^{\delta(\alpha )/2-1}\mathfrak{R}^{1/2} \,, \\
  & \hat{B}_{2\alpha \beta ,n}(z,\eta ,\mathfrak{R}, \mathfrak{r} , \omega ) =B_{2\alpha \beta ,n } \mathfrak{r}^{\delta(\alpha)/2-1 }(1-\mathfrak{r})^{\delta(\beta )/2-1}(1-\mathfrak{R})^{(\delta(\alpha)+\delta(\beta ))/2-1} \mathfrak{R}^{1/2} .
     \end{aligned}}
 \end{equation*}
 
We now introduce the operators $Q_{\alpha, n} (f, f)$ for $\alpha \in \{ 1, \cdots, s \}$. 

If $\alpha \in \{1,\cdots , s_0\}$, we define
\begin{equation*}
  \begin{aligned}
    Q_{\alpha, n} (f,f) (\xi) = \sum_{\beta = 1}^{s} Q_{\alpha \beta, n} (f, f) (\xi) \,,
  \end{aligned}
\end{equation*}
where
\begin{equation*}
  \begin{aligned}
    Q_{\alpha \beta, n} (f, f) (\xi) & = \iint_{\R^3 \times \mathbb{S}^2 } B_{0\alpha \beta ,n} (\xi - \xi_*, \omega) (f'_\alpha f'_{\beta_*}- f_\alpha f_{\beta_*})d\omega d\xi_* \\
    & = Q_{\alpha \beta, n}^+ (f, f) (\xi) - Q_{\alpha \beta, n}^- (f, f) (\xi)
  \end{aligned}
\end{equation*}
for $\beta \in \{ 1, \cdots, s_0 \}$, and
\begin{equation*}{\small
  \begin{aligned}
    Q_{\alpha \beta, n} (f,f) (\xi) & = \iiiint \limits_{\R^3\times \R_+ \times [0,1] \times \mathbb{S}^2} \hat{B}_{1\alpha \beta ,n} (\xi - \xi_*, I_*, \mathfrak{R}, \omega) (\tfrac{f'_\alpha f'_{\beta_*}}{(I'_* )^{\delta(\beta)/2-1}} - \tfrac{f_\alpha f_{\beta_*}}{ I_* ^{\delta(\beta)/2-1}}) I_*^{\delta(\beta)/2-1} d\omega d \mathfrak{R} dI_* d\xi_* \\
    & = Q_{\alpha \beta, n}^+ (f, f) (\xi) - Q_{\alpha \beta, n}^- (f, f) (\xi)
  \end{aligned}}
\end{equation*}
for $\beta \in \{ s_0 + 1, \cdots, s \}$. For the loss terms $Q_{\alpha \beta, n}^- (f,f) (\xi)$ with $\alpha \in \{ 1, \cdots, s_0 \}$ and $\beta \in \{ 1 , \cdots , s \}$, one can represent
\begin{equation*}
  \begin{aligned}
    Q_{\alpha \beta, n}^- (f,f) (\xi) = f_\alpha L_{\alpha \beta, n} (f) (\xi) \,,
  \end{aligned}
\end{equation*}
where
\begin{equation*}
  \begin{aligned}
    & L_{\alpha \beta, n} (f) (\xi) = \iint_{\R^3 \times \mathbb{S}^2 } B_{0\alpha \beta ,n} (\xi - \xi_*, \omega) f_{\beta *} d \omega d \xi_* \,, \ \beta \in \{ 1, \cdots, s_0 \} \,, \\
    & L_{\alpha \beta, n} (f) (\xi) = \iiiint \limits_{\R^3\times \R_+ \times [0,1] \times \mathbb{S}^2} \hat{B}_{1\alpha \beta ,n} (\xi - \xi_*, I_*, \mathfrak{R}, \omega) f_{\beta_*} d \omega d \mathfrak{R} dI_* d\xi_* \,, \ \beta \in \{ s_0 + 1, \cdots, s \} \,.
  \end{aligned}
\end{equation*}
Denote by $L_{\alpha, n} (f) (\xi) = \sum_{\beta = 1}^{s} L_{\alpha \beta, n} (f) (\xi)$ for $\alpha \in \{ 1, \cdots, s_0 \}$.

If $\alpha \in \{ s_0 + 1, \cdots, s \}$, we define
\begin{equation*}
  \begin{aligned}
    Q_{\alpha, n} (f,f) (\xi, I) = \sum_{\beta = 1}^{s} Q_{\alpha \beta, n} (f, f) (\xi, I) \,,
  \end{aligned}
\end{equation*}
where
\begin{equation*}{\small
  \begin{aligned}
    Q_{\alpha \beta, n} (f, f) (\xi, I) = & \iiint_{\R^3 \times [0,1] \times \mathbb{S}^2} \hat{B}_{1\beta \alpha ,n} (\xi - \xi_*, I, \mathfrak{R}, \omega) (\tfrac{f'_\alpha f'_{\beta_*}}{(I')^{\delta(\alpha)/2-1} } - \tfrac{f_\alpha f_{\beta_*}}{I^{\delta(\alpha)/2-1}})  I^{\delta(\alpha )/2-1} d\omega d\mathfrak{R} d\xi_* \\
    = & Q_{\alpha \beta, n}^+ (f, f) (\xi, I) - Q_{\alpha \beta, n}^- (f, f) (\xi, I)
  \end{aligned}}
\end{equation*}
for $ \beta \in \{ 1, \cdots, s_0 \} $, and
\begin{equation*}
  \begin{aligned}
    Q_{\alpha \beta, n} (f, f) (\xi, I) = & \iiiint_{\R^3 \times \R_+ \times [0,1]^2 \times \mathbb{S}^2 } (\tfrac{f'_\alpha f'_{\beta_*}}{(I')^{\delta(\alpha)/2-1} (I'_* )^{\delta(\beta)/2-1}} - \tfrac{f_\alpha f_{\beta_*}}{I^{\delta(\alpha)/2-1} I_* ^{\delta(\beta)/2-1}})\\
		& \times \hat{B}_{2 \alpha \beta ,n} (\xi - \xi_*, I + I_*, \mathfrak{R}, \mathfrak{r}, \omega) I^{\delta(\alpha)/2-1 } I_*^{\delta (\beta )/2-1} d \omega d \mathfrak{R} d \mathfrak{r} d I_* d \xi_* \\
= & Q_{\alpha \beta, n}^+ (f, f) (\xi, I) - Q_{\alpha \beta, n}^- (f, f) (\xi, I)
  \end{aligned}
\end{equation*}
for $ \beta \in \{ s_0 + 1, \cdots, s \} $. Moreover, the loss terms $Q_{\alpha \beta, n}^- (f,f) (\xi, I)$ with $\alpha \in \{s_0 + 1, \cdots, s \}$ and $\beta \in \{ 1 , \cdots , s \}$ can be represented by
\begin{equation*}
  \begin{aligned}
    Q_{\alpha \beta, n}^- (f,f) (\xi, I) = f_\alpha L_{\alpha \beta, n} (f) (\xi, I) \,,
  \end{aligned}
\end{equation*}
where
\begin{equation*}{\small
  \begin{aligned}
    & L_{\alpha \beta, n} (f) (\xi, I) = \iiint_{\R^3 \times [0,1] \times \mathbb{S}^2} \hat{B}_{1\beta \alpha ,n} (\xi - \xi_*, I, \mathfrak{R}, \omega) f_{\beta *} d \omega d\mathfrak{R} d\xi_* \,, \ \beta \in \{ 1, \cdots, s_0 \} \,, \\
    & L_{\alpha \beta, n} (f) (\xi, I) = \iiiint\limits_{\R^3 \times \R_+ \times [0,1]^2 \times \mathbb{S}^2 } \hat{B}_{2 \alpha \beta ,n} (\xi - \xi_*, I + I_*, \mathfrak{R}, \mathfrak{r}, \omega) f_{\beta *} d \omega d \mathfrak{R} d \mathfrak{r} d I_* d \xi_* \,, \ \beta \in \{ s_0 + 1, \cdots, s \} \,.
  \end{aligned}}
\end{equation*}
Denote by $L_{\alpha, n} (f) (\xi, I) = \sum_{\beta = 1}^{s} L_{\alpha \beta, n} (f) (\xi, I) $ for $\alpha \in \{ s_0 + 1, \cdots, s \} $. Furthermore, we employ the symbols $Q_{\alpha, n}^\pm (f,f) (\mathbf{Z}) = \sum_{\beta = 1}^{s} Q_{\alpha \beta, n}^\pm (f, f) (\mathbf{Z})$ for $\alpha, \beta \in \{ 1, \cdots, s \}$ and $ \mathbf{Z} = \xi $ or $(\xi, I)$. Then $Q_{\alpha ,n } (f,f) = Q^+_{\alpha ,n } (f,f) -Q^-_{\alpha ,n } (f,f)$.

Applying the above symbols, we construct the smoothly approximated collision operator $\tilde{Q}_{\alpha, n}$ as follows: 
\begin{equation}\label{Q L}
	\begin{aligned}
		\tilde{Q}^\pm _{\alpha ,n } (f,f) &=
		\begin{cases}
			( N_n (f) )^{-1} Q^\pm _{\alpha ,n } (f,f)(\xi ), \quad &\alpha \in \{1,\cdots, s_0 \},\\
			( N_n (f) )^{-1} Q^\pm _{\alpha ,n } (f,f)(\xi ,I ), \quad &\alpha \in \{s_0+1,\cdots, s \},
		\end{cases} \\
			\tilde{Q}_{\alpha ,n } (f,f) &=
		\begin{cases}
			( N_n (f) )^{-1} Q_{\alpha ,n } (f,f)(\xi ), \quad &\alpha \in \{1,\cdots, s_0 \},\\
			( N_n (f) )^{-1} Q_{\alpha ,n } (f,f)(\xi ,I ), \quad &\alpha \in \{s_0+1,\cdots, s \},
		\end{cases}\\
		\tilde{L} _{\alpha ,n } (f) &=
		\begin{cases}
			( N_n (f) )^{-1} L _{\alpha ,n } (f)(\xi ), \quad & \ \ \ \alpha \in \{1,\cdots, s_0 \},\\
			( N_n (f) )^{-1} L _{\alpha ,n } (f)(\xi ,I ), \quad & \ \ \ \alpha \in \{s_0+1,\cdots, s \} \,, 
		\end{cases}
	\end{aligned}
\end{equation}
where
\begin{equation}\label{Nnf}
  \begin{aligned}
    N_n (f) = 1 + \tfrac{1}{n} \sum_{\alpha = 1}^{s_0} \int_{\R^3} | f_\alpha | d \xi + \tfrac{1}{n} \sum_{\alpha = s_0 + 1}^{s} \iint_{\R^3\times \R_+} | f_\alpha | d\xi dI \,. 
  \end{aligned}
\end{equation}

Now we study the $L^1$ and $L^\infty$ properties of the approximated collision operator $\tilde{Q}_{\alpha, n}$, which will be used to investigate the existence of the approximation problem \eqref{approximation}.

\begin{lemma}[$L^1$ estimates on variable $\mathbf{Z}$]\label{Lemma 4.13}
	Denote by the variable $\mathbf{Z}$ by
\begin{equation*}
		\mathbf{Z} =\mathbf{Z}_\alpha :=
		\begin{cases}
			\mathbf{\xi } , \qquad \enspace \, \alpha \in \{1,\cdots , s_0 \} \,, \\
			(\mathbf\xi , I )  , \quad \alpha \in \{s_0 +1,\cdots , s \} \,.
		\end{cases}
	\end{equation*}
There exists a positive constant $C_{s, n}$ depending on $s ,n$ such that for all distribution functions $f (\mathbf{Z}), g(\mathbf{Z} ) $ with the components $f_\alpha , g_\alpha \in L^1(d \mathbf{Z})$, 
	\begin{equation*}
		\begin{aligned}
			& \Vert \tilde{Q}_{\alpha , n} (f,f) (\mathbf{Z}) \Vert_{L^1}\leq C_{s , n} \Vert f_{\alpha} (\mathbf{Z}) \Vert_{L^1}, \\
			& \Vert \tilde{Q}_{\alpha , n} (f,f) (\mathbf{Z}) -\tilde{Q}_{\alpha , n} (g,g ) (\mathbf{Z}) \Vert_{L^1} \leq C_{s , n} \sum_{\beta = 1}^{s} \Vert f_\beta (\mathbf{Z}) -g_\beta (\mathbf{Z}) \Vert_{L^1}.
		\end{aligned}
	\end{equation*}
\end{lemma} 

\begin{proof}
	For all $\alpha \in \{ 1,\cdots , s \},$
	\begin{equation*}
		\begin{aligned}
			\Vert \tilde{Q}_{\alpha , n} (f,f) \Vert_{L^1} 
			&\leq \Vert \tilde{Q}^+_{\alpha , n} (f,f) \Vert_{L^1} + \Vert \tilde{Q}^-_{\alpha , n} (f,f) \Vert_{L^1} ,\\
			\Vert \tilde{Q}_{\alpha , n} (f,f) -\tilde{Q}_{\alpha , n} (g,g ) \Vert_{L^1} 
			&\leq \Vert \tilde{Q}^+_{\alpha , n} (f,f) -\tilde{Q}^+_{\alpha , n} (g,g ) \Vert_{L^1} +\Vert \tilde{Q}^-_{\alpha , n} (f,f) -\tilde{Q}^-_{\alpha , n} (g,g ) \Vert_{L^1}.
		\end{aligned}
	\end{equation*}
   Using invariant properties of $B_{2\alpha \beta }$, that is, $B_{2\alpha \beta }(\xi - \xi_*,I+I_*,\mathfrak{R},\mathfrak{r},\omega) =B_{2\alpha \beta }(\xi' - ,\xi'_*,I'+I'_*,\mathfrak{R}',\mathfrak{r}',\omega')$ for $\alpha, \beta \in \{ s_0 + 1, \cdots, s \}$ (see  \eqref{B-inv}), we easily deduce that $\hat{B}_{2\alpha \beta,n }$ is also invariant for such an exchange of state. Then from using the change variables $(\xi ,I)\rightarrow (\xi',I')$, it follows that for $\alpha, \beta \in \{ s_0 + 1, \cdots, s \}$ and for any $\varphi(\xi,I) \in L^\infty ( \R^3\times \R_+)$ 
	\begin{equation}\label{Q}{\small
	\begin{aligned}
	&\iint_{ \R^3\times \R_+}Q^+_{\alpha \beta ,n}(f,f) \varphi d\xi dI\\
	=&\int_{(\R^3)^2 \times(\R_+)^2 \times [0,1]^2 \times \mathbb{S}^2 } \hat{B}_{2 \alpha \beta ,n} \tfrac{f'_\alpha f'_{\beta_*}}{(I')^{\delta(\alpha)/2-1} (I'_* )^{\delta(\beta)/2-1}}
	 I^{\delta(\alpha)/2-1 } I_*^{\delta (\beta )/2-1}\varphi d\xi d\xi_* d \mathfrak{r} d \mathfrak{R} d \omega dI dI_*\\
	=& \int_{(\R^3)^2 \times(\R_+)^2 \times [0,1]^2 \times \mathbb{S}^2 }\hat{B}_{2 \alpha \beta ,n} f_\alpha f_{\beta_*} 
	\varphi' d\xi d\xi_* d \mathfrak{r} d \mathfrak{R} d \omega dI dI_* \\
	= & \iint_{ \R^3\times \R_+} Q^-_{ \alpha \beta,n }(f,f) \varphi' d \xi dI,
	\end{aligned}}
	\end{equation}
	where $\varphi '= \varphi (\xi',I')$. Moreover, by the similar arguments in \eqref{Q}, we concludes that for the other indices $\alpha,\beta$,
\begin{equation*}
  \begin{aligned}
    \int_{\mathcal{Z}_\alpha }Q^+_{\alpha \beta ,n}(f,f) \varphi d \mathbf{Z}_\alpha = \int_{\mathcal{Z}_\alpha } Q^-_{ \alpha \beta,n }(f,f) \varphi' d \mathbf{Z}_\alpha \,.
  \end{aligned}
\end{equation*}
It then follows that for $\alpha \in \{1,\cdots , s_0\}$,
	\begin{equation*}
		\begin{aligned}
			\Vert	Q^+_{\alpha ,n} (f,f)\Vert_{L^1} 
			\leq &\sum \limits_{\beta =1}^{s_0} \iint_{\R^3 \times \R^3 } A_{0\alpha \beta ,n} | f_\alpha f_{\beta *} | d\xi d\xi_*  + \sum \limits_{\beta =s_0+ 1}^{s} \iiint_{\R^3 \times \R^3 \times \R_+ } A_{1\alpha \beta ,n} | f_\alpha f_{\beta *} | d\xi d\xi_*  dI_*\\
			\leq & C_{\alpha ,n} \Vert f_\alpha \Vert_{L^1} \sum \limits_{\beta =1}^{s}  \Vert f_{\beta } \Vert_{L^1},
		\end{aligned}
	\end{equation*}
	and for $\alpha \in \{s_0+1,\cdots , s\}$,
	\begin{equation*}{\small
		\begin{aligned}
			\Vert	Q^+_{\alpha ,n} (f,f)\Vert_{L^1}
			\leq & \sum \limits_{\beta =1}^{s_0} \iiint_{\R^3 \times \R^3 \times \R_+ } A_{1\beta \alpha ,n} | f_\alpha f_{\beta } | d\xi d\xi_* dI + \sum \limits_{\beta =s_0+ 1}^{s} \iiiint_{\R^3 \times \R^3 \times \R_+ \times \R_+} A_{2 \alpha \beta ,n} | f_\alpha f_{\beta_*} | d\xi d\xi_*dI dI_* \\
			\leq & C_{\alpha ,n} \Vert f_\alpha \Vert_{L^1} \sum \limits_{\beta =1}^{s}  \Vert f_{\beta } \Vert_{L^1}.
		\end{aligned}}
	\end{equation*}
Here we have utilized the facts that the approximated kernels $A_{0 \alpha \beta, n}, A_{1 \alpha \beta, n}, A_{1 \beta \alpha, n}, A_{2 \alpha \beta, n}$ constructed in Lemmas \ref{Lemma 4.10}-\ref{Lemma 4.11}-\ref{Lemma 4.12} belong to $L^\infty$ with upper bounds depending on $n$. Observe that $\sum_{\beta = 1}^{s}\| f_\beta \|_{L^1} ( N_n (f) )^{-1} \leq n$, where $N_n (f)$ is defined in \eqref{Nnf}. As a consequence, it is easy to see that $\Vert \tilde{Q}^+_{\alpha , n} (f,f) \Vert_{L^1}\leq C_{\alpha , n} \Vert f_{\alpha} \Vert_{L^1}$ for all $\alpha \in \{1,\cdots , s\}$. Similarly, one has $\Vert \tilde{Q}^-_{\alpha , n} (f,f) \Vert_{L^1}\leq C_{\alpha , n} \Vert f_{\alpha} \Vert_{L^1}$ for all $\alpha \in \{1,\cdots , s\}$. Then $\Vert \tilde{Q}_{\alpha , n} (f,f) \Vert_{L^1}\leq C_{\alpha , n} \Vert f_{\alpha} \Vert_{L^1}$. In summary, we infer that
	\begin{equation}
	\Vert \tilde{Q}_{\alpha , n} (f,f) \Vert_{L^1}\leq C_{s , n} \Vert f_{\alpha} \Vert_{L^1}.
	\end{equation}

Observe that for $\alpha \in \{1,\cdots ,s_0\}$, 
	\begin{equation*}{\small
		\begin{aligned}
			\Vert \tilde{Q}^+_{\alpha , n} (f,f) - & \tilde{Q}^+_{\alpha , n} (g,g ) \Vert_{L^1} =\sum \limits_{\beta =1}^{s_0}\int_{\R^3} d\xi |\int_{\R^3} \{(N_n (f))^{-1}  f_\alpha f_{\beta *} -(N_n (g))^{-1} g_\alpha  g_{\beta *}\} A_{0\alpha \beta ,n}d\xi_* | \\
			+& \sum \limits_{\beta =s_0+ 1}^{s} \int_{\R^3} d\xi |\iint_{\R^3\times \R_+} \{(N_n (f))^{-1}  f_\alpha f_{\beta *} -(N_n (g))^{-1} g_\alpha  g_{\beta *}\} A_{1\alpha \beta ,n }d\xi_* dI_*| \,.
		\end{aligned}}
	\end{equation*}
Note that
	\begin{equation*}
		\begin{aligned}
			&|(N_n (f))^{-1}  f_\alpha f_{\beta *} -(N_n (g))^{-1}  g_\alpha  g_{\beta_*} | \\
			\leq & (N_n (f))^{-1} | f_{\beta *} | |f_\alpha -g_\alpha | +(N_n (g))^{-1} | g_\alpha | |f_{\beta *}-g_{\beta *}| \\
			+ & | f_{\beta *} g_\alpha | (N_n (f))^{-1} (N_n (g))^{-1} \sum_{\beta = 1}^{s} \| f_\beta - g_\beta \|_{L^1}.
		\end{aligned}
	\end{equation*}
By using the facts $0 \leq A_{0 \alpha \beta, n} , A_{1 \alpha \beta, n} \leq C_{s,n} $ and 
\begin{equation*}
  \begin{aligned}
    \sum_{\beta = 1}^{s} \| f_\beta \|_{L^1} (N_n (f))^{-1} \leq n \,, \ \sum_{\beta = 1}^{s} \| g_\beta \|_{L^1} (N_n (g))^{-1} \leq n \,,
  \end{aligned}
\end{equation*}
one then has
\begin{equation*}
  \begin{aligned}
    \| \tilde{Q}_{\alpha, n}^+ (f, f) - \tilde{Q}_{\alpha, n}^+ (g, g) \|_{L^1} \leq C_{s,n} \sum_{\beta = 1}^{s} \| f_\beta - g_\beta \|_{L^1} 
  \end{aligned}
\end{equation*}
for $\alpha \in \{1,\cdots ,s_0\}$. Similarly, for $\alpha \in \{s_0 + 1,\cdots ,s\}$, there holds
\begin{equation*}
  \begin{aligned}
    \| \tilde{Q}_{\alpha, n}^+ (f, f) - \tilde{Q}_{\alpha, n}^+ (g, g) \|_{L^1} \leq C_{s,n} \sum_{\beta = 1}^{s} \| f_\beta - g_\beta \|_{L^1} \,.
  \end{aligned}
\end{equation*}
Similarly, one infers that
\begin{equation*}
  \begin{aligned}
    \| \tilde{Q}_{\alpha, n}^- (f, f) - \tilde{Q}_{\alpha, n}^- (g, g) \|_{L^1} \leq C_{s,n} \sum_{\beta = 1}^{s} \| f_\beta - g_\beta \|_{L^1} \,.
  \end{aligned}
\end{equation*}
Therefore, the proof of Lemma \ref{Lemma 4.13} is finished.
	\end{proof}

\begin{lemma}[$L^\infty$ esitmates]\label{Lemma 4.13 of infty}
		There exists a constant $C_{s, n} > 0$ depending on $s ,n$ such that for all distribution functions $f=(f_1,\cdots, f_s) $ with the component $f_\alpha \in L^\infty \cap L^1 $, 
		\begin{equation*}
			\Vert \tilde{Q}_{\alpha , n} (f,f) \Vert_{L^\infty } \leq C_{\alpha, n} \Vert f_\alpha \Vert_{L^\infty }
		\end{equation*}
for all $\alpha \in \{ 1, \cdots, s \}$.
	\end{lemma}

\begin{proof}
We first consider the quantity $\tilde{Q}_{\alpha \beta, n} (f, f)$ defined in \eqref{Q L} for the case $\alpha ,\beta \in \{s_0+1,\cdots,s\}$. Recall that $\tilde{Q}_{\alpha \beta, n} (f, f) = ( N_n (f) )^{-1} Q_{\alpha \beta, n} (f,f)$ with $Q_{\alpha \beta, n} (f,f) = Q_{\alpha \beta, n}^+ (f,f) - Q_{\alpha \beta, n}^- (f,f)$. For the quantity $Q_{\alpha \beta, n}^+ (f,f)$,
\begin{equation*}
  \begin{aligned}
    Q_{\alpha \beta, n}^+ (f,f) \leq \| f_\alpha \|_{L^\infty} \int_{\R^3\times \R_+ \times [0,1]^2 \times \mathbb{S}^2 } B_{2\alpha \beta ,n} (\xi-\xi_*,I+I_*,\mathfrak{R},\mathfrak{r},\omega )\tfrac{ | f_{\beta}(\xi'_*,I'_*) | }{(I')^{\delta(\alpha)/2-1}(I'_*)^{\delta(\beta )/2-1}} \\
 \times \mathfrak{r}^{\delta(\alpha)/2-1} (1-\mathfrak{r})^{\delta (\beta )/2-1 } (1-\mathfrak{R} ) ^{(\delta (\alpha) +\delta (\beta ))/2-1} \mathfrak{R}^{1/2}
I^{\delta(\alpha)/2-1}I_*^{\delta(\beta )/2-1} d\xi_* dI_* d\mathfrak{R} d\mathfrak{r}d\omega \,.
  \end{aligned}
\end{equation*}
By using the expressions of $E_{\alpha \beta },\mathfrak{R}$ and $\mathfrak{r}$ in \eqref{Eab}, \eqref{IR} and \eqref{Ir}, respectively, one has
\begin{equation*}
  \begin{aligned}
    (I')^{\delta(\alpha)/2-1}(I'_*)^{\delta(\beta )/2-1} = \mathfrak{r}^{ \delta (\alpha) /2 - 1 } ( 1 - \mathfrak{r} )^{ \delta (\beta) /2 - 1 } ( 1 - \mathfrak{R} )^{ ( \delta (\alpha) + \delta (\beta) ) /2 - 2 } E_{\alpha \beta}^{ ( \delta (\alpha) + \delta (\beta) ) /2 - 2 } \,. 
  \end{aligned}
\end{equation*}
Due to $E_{\alpha \beta} = \tfrac{\mu_{ \alpha \beta }}{2} |g|^2 + I + I_*$, one knows that
\begin{equation*}
  \begin{aligned}
    \tfrac{1}{ E_{\alpha \beta}^{ ( \delta (\alpha) + \delta (\beta) ) /2 - 2 } } I^{ \delta ( \alpha ) / 2 - 1 } I_*^{ \delta ( \beta ) / 2 - 1 } \leq 1 \,.
  \end{aligned}
\end{equation*}
Moreover, Lemma \ref{Lemma 4.10} shows that 
$$supp B_{2\alpha \beta ,n} (z,\eta ,\mathfrak{R},\mathfrak{r},\omega) \subset \{(z,\eta ,\mathfrak{R},\mathfrak{r},\omega): \tfrac{1}{n}\leq |z|,\eta \leq n,\ \tfrac{1}{n} \leq \mathfrak{R},\mathfrak{r} \leq 1-\tfrac{1}{n}\} \,.$$ 
Then
\begin{equation*}
 \begin{aligned}
Q_{\alpha \beta, n}^+ (f,f) \leq \tilde{C}_n \| f_\alpha \|_{L^\infty} \int_{\R^3\times \R_+ \times [\frac{1}{n},1-\frac{1}{n}]^2 \times \mathbb{S}^2 } B_{2 \alpha \beta ,n}(\xi-\xi_*,I+I_*,\mathfrak{R},\mathfrak{r},\omega ) | f_{\beta}(\xi'_*,I'_*) | \\
\times (1-\mathfrak{R}) \mathfrak{R}^{1/2} d\xi_* dI_* d\mathfrak{R} d\mathfrak{r}d\omega.
 \end{aligned}   
\end{equation*}
Furthermore, due to $E_{\alpha \beta} = \frac{\mu_{\alpha \beta}}{2} |g|^2 + I + I_*$ and $I_*' = ( 1 - \mathfrak{r} ) ( 1 - \mathfrak{R} ) E_{\alpha \beta}$, one then has $d I_*' = d [ ( 1 - \mathfrak{r} ) ( 1 - \mathfrak{R} ) ( \frac{\mu_{\alpha \beta}}{2} |g|^2 + I + I_* ) ] = ( 1 - \mathfrak{r} ) ( 1 - \mathfrak{R} ) d I_* $, namely, $ d I_* = ( 1 - \mathfrak{r} )^{-1} ( 1 - \mathfrak{R} )^{-1} d I_*' $. Then
\begin{equation}\label{Q+-Linfty}
 \begin{aligned}
Q_{\alpha \beta, n}^+ (f,f) \leq \tilde{C}_n \| f_\alpha \|_{L^\infty} I_{2 \alpha \beta, n}\,,
 \end{aligned}   
\end{equation}
where
\begin{equation}\label{I-2abn}{\footnotesize
  \begin{aligned}
    I_{2 \alpha \beta, n} = \int_{\R^3\times \R_+ \times [\frac{1}{n},1-\frac{1}{n}]^2 \times \mathbb{S}^2 } B_{2 \alpha \beta ,n}(\xi-\xi_*,I+I_*,\mathfrak{R},\mathfrak{r},\omega ) | f_{\beta}(\xi'_*,I'_*) | (1-\mathfrak{r})^{-1} \mathfrak{R}^{1/2} d\xi_* dI_*' d\mathfrak{R} d\mathfrak{r}d\omega
  \end{aligned}}
\end{equation}

Now we control the integral $I_{2 \alpha \beta, n}$ given in \eqref{I-2abn}. Let $t=|\xi'-\xi'_* | = |g'|$. Then $\xi'_* = \xi' - |\xi'-\xi'_* |\omega = \xi' - t \omega$. We also extend the kernel $B_{2 \alpha \beta, n}$ on the variable $\theta$ associated with the variable $\omega \in \mathbb{S}^2$ from $(0, \frac{\pi}{2})$ to $(0, \pi)$ such that $B_{\alpha \beta ,n}=0 $ if $\theta\in [\frac{\pi}{2} ,\pi )$. The extension is such that $\omega$ goes through the whole $\mathbb{S}^2$ Consider the plane $\Pi (\xi' ,\omega )$ that is perpendicular to vector $\omega $ through point $\xi'$. For any given $\omega$, $  \xi- \xi_*$ can be decomposed as $(s,v)$ variables, where $s=(\xi -\xi_*,\omega ), v\in \Pi (\xi', \omega )$ and $d\xi_* =-dsdv$. Denote by $s=(\xi -\xi_* ,\omega ) = (\xi -\xi_* , \frac{\xi' -\xi'_* }{|\xi'-\xi'_*|})=|g|\cos \theta$. Recalling that $|g'|^2 + \tfrac{2}{\mu_{ \alpha \beta }} \Delta I = |g|^2$, one has $s = \sqrt{t^2+\frac{2}{\mu_{\alpha \beta }}\Delta I } \cos \theta$. Then
	\begin{equation*}
		ds=d(|g| \cos \theta )=\tfrac{t \cos \theta }{\sqrt{t^2+\frac{2}{\mu_{\alpha \beta }}\Delta I }} dt.
	\end{equation*}
As a result, for the case $ \alpha ,\beta \in \{s_0+1,\cdots,s\}$, 
\begin{equation}\label{I-2abn-1}
\begin{aligned}
I_{2 \alpha \beta, n} \leq & \int_{0}^{+\infty} dt \int_{\Pi (\xi' ,\omega )}  dv \int_{\R_+ \times [\frac{1}{n},1-\frac{1}{n} ]^2 \times \mathbb{S}^2 } B_{2\alpha \beta ,n} (\xi-\xi_*,I+I_*,\mathfrak{R},\mathfrak{r},\omega ) \\
&\times | f_{\beta }(\xi'-t\omega ,I'_*) | (1-\mathfrak{r})^{-1}\mathfrak{R}^{1/2} \tfrac{t \cos \theta }{\sqrt{t^2+\frac{2}{\mu_{\alpha \beta }}\Delta I }}dI'_* d\mathfrak{R}d\mathfrak{r}d\omega.    
\end{aligned}
\end{equation}
By Lemma \ref{Lemma 4.10}, one knows $B_{2\alpha \beta ,n} (\xi-\xi_*,I+I_*,\mathfrak{R},\mathfrak{r},\omega ) =0$ if $s=|g|\cos \theta <\frac{1}{n}$.

{\bf Case 1. $\Delta I \leq 0$.}
	
 If $t < \tfrac{1}{n \cos\theta },\ s = |g| \cos\theta = \sqrt{t^2+\tfrac{2}{\mu_{\alpha \beta }}\Delta I } \cos \theta < t \cos \theta <\frac{1}{n}$. That means $B_{2\alpha \beta ,n} (\xi-\xi_*,I+I_*,\mathfrak{R},\mathfrak{r},\omega ) =0$ for $t < \tfrac{1}{n \cos \theta }$. Notices that $|g|^2 = t^2+\frac{2}{\mu_{\alpha \beta }}\Delta I $ and $\tfrac{1}{|g|^2} = \tfrac{\cos^2 \theta}{s^2} \leq n^2 \cos^2 \theta \leq n^2$ for $s \geq \tfrac{1}{n}$. Then $ |g|^2 \tfrac{1}{|g|^2} \frac{t \cos \theta }{\sqrt{t^2+\frac{2}{\mu_{\alpha \beta }}\Delta I }} \leq n^2 t \cos \theta \sqrt{t^2+\frac{2}{\mu_{\alpha \beta }}\Delta I } \leq n^2 t^2 $. Consequently, the quantity $I_{2 \alpha \beta, n}$ can be further bounded by
\begin{equation*}{\small
\begin{aligned}
I_{2 \alpha \beta, n} \leq & \int_{\frac{1}{n \cos \theta} }^{+\infty} dt \int_{\Pi (\xi' ,\omega )} |g|^2 \tfrac{1}{|g|^2} dv \int_{\R_+ \times [\frac{1}{n},1-\frac{1}{n} ]^2 \times \mathbb{S}^2 } B_{2\alpha \beta ,n} (\xi-\xi_*,I+I_*,\mathfrak{R},\mathfrak{r},\omega )\\
& \times | f_{\beta }(\xi'-t\omega ,I'_*) | (1-\mathfrak{r})^{-1} \mathfrak{R}^{1/2} \tfrac{t \cos \theta }{\sqrt{t^2+\frac{2}{\mu_{\alpha \beta }}\Delta I }} d I'_* d \mathfrak{R} d \mathfrak{r} d \omega \\
\leq & n^2 \int_{\frac{1}{n \cos\theta} }^{+\infty}t^2 dt \int_{\Pi (\xi' ,\omega )}  dv \int_{\R_+ \times [\frac{1}{n},1-\frac{1}{n}]^2 \times \mathbb{S}^2 } B_{2\alpha \beta ,n} | f_{\beta }(\xi'-t\omega ,I'_*) | (1-\mathfrak{r})^{-1} \mathfrak{R}^{1/2} dI'_* d \mathfrak{R} d \mathfrak{r} d\omega.
\end{aligned}}
\end{equation*}

{\bf Case 2. $\Delta I > 0$.}

Recalling that $\cos \theta = \tfrac{g \cdot g'}{ |g| |g'| }$ in \eqref{eq2.10}, one knows that as $t \to 0^+$,
\begin{equation*}
  \begin{aligned}
    g \cdot g' = t \sqrt{t^2 + \tfrac{2}{ \mu_{ \alpha \beta } } \Delta I } \cos \theta \to 0 \,,
  \end{aligned}
\end{equation*}
which further means that $\cos \theta \to 0$ as $t \to 0^+$. By the construction of $B_{2 \alpha \beta, n}$ in Lemma \ref{Lemma 4.10}, one infers that there is a small $\kappa_n > 0$ such that $ B_{2 \alpha \beta, n} (\xi-\xi_*,I+I_*,\mathfrak{R},\mathfrak{r},\omega ) = 0 $ as $t < \kappa_n$. Consequently, the integral $I_{2 \alpha \beta, n}$ can be dominated by
\begin{equation*}{\small
\begin{aligned}
I_{2 \alpha \beta, n} \leq & \int_{\kappa_n}^{+\infty} dt \int_{\Pi (\xi' ,\omega )} t^2 \tfrac{1}{t^2} dv \int_{\R_+ \times [\delta_n,1-\delta_n]^2 \times \mathbb{S}^2 } B_{2\alpha \beta ,n} (\xi-\xi_*,I+I_*,\mathfrak{R},\mathfrak{r},\omega ) \\
& \times | f_{\beta }(\xi'-t\omega ,I'_*) | (1-\mathfrak{r})^{-1}\mathfrak{R}^{1/2} \tfrac{t \cos \theta }{\sqrt{t^2+\frac{2}{\mu_{\alpha \beta }}\Delta I }}dI'_* d\mathfrak{R}d\mathfrak{r}d\omega\\
\leq & \tfrac{1}{\kappa_n^2} \int_{\kappa_n }^{+\infty} t^2 dt \int_{\Pi (\xi' ,\omega )}  dv \int_{\R_+ \times [\frac{1}{n},1-\frac{1}{n}]^2 \times \mathbb{S}^2 } B_{2\alpha \beta ,n} | f_{\beta }(\xi'-t\omega ,I'_*) | (1-\mathfrak{r})^{-1}\mathfrak{R}^{1/2} dI'_* d\mathfrak{R}d\mathfrak{r}d\omega.
\end{aligned}}
\end{equation*}	

Therefore, there exists $C'_n > 0$ such that for all $\Delta I \in \R$,
\begin{equation*}
\begin{aligned}
I_{2 \alpha \beta, n} \leq C'_n \int_0^{+\infty} t^2 dt \int_{\Pi (\xi' ,\omega )}  dv \int_{\R_+ \times [\frac{1}{n},1-\frac{1}{n}]^2 \times \mathbb{S}^2 } B_{2\alpha \beta ,n} (\xi-\xi_*,I+I_*,\mathfrak{R},\mathfrak{r},\omega ) \\
\times | f_{\beta }(\xi'-t\omega ,I'_*) | (1-\mathfrak{r})^{-1}\mathfrak{R}^{1/2} dI'_* d\mathfrak{R}d\mathfrak{r}d\omega \,.
\end{aligned}
\end{equation*}	
Moreover, we denote by $C_1$ the maximum area of the intersection between plane $\Pi (\xi',\omega )$ and set $supp B_{2 \alpha \beta ,n}$ and denote by $ C_2$ the maximum value of $B_{2\alpha \beta ,n}$. Then the integral $I_{2 \alpha \beta, n}$ can further controlled by
\begin{equation}\label{B2ab 3}
\begin{aligned}
I_{2 \alpha \beta, n} \leq & C_1 C_2 \int_{[\frac{1}{n},1-\frac{1}{n}]^2} (1-\mathfrak{r})^{-1}\mathfrak{R}^{1/2} d \mathfrak{R} d \mathfrak{r} \int_{0}^{+\infty} t^2 dt\iint_{\R_+  \times \mathbb{S}^2 } | f_{\beta }(\xi'-t\omega ,I'_*) | d I'_* d \omega \\
\leq & C''_n \iint_{\R^3 \times \R_+} | f_{\beta }(\xi'- y ,I'_*) | d y dI'_* = C''_n \| f_\beta \|_{L^1} \,.
\end{aligned}
\end{equation}	
From substituting \eqref{B2ab 3} into \eqref{Q+-Linfty}, it derives that for $\alpha, \beta \in \{ s_0 + 1, \cdots, s \}$,
\begin{equation}\label{Q+1}
\begin{aligned}
Q_{\alpha \beta, n}^+ (f,f) \leq C_n \| f_\alpha \|_{L^\infty} \| f_\beta \|_{L^1} \,.
\end{aligned}
\end{equation}	

Following the similar arguments in \eqref{Q+1}, one can also derive that
\begin{equation}\label{Q+2}
\begin{aligned}
Q_{\alpha \beta, n}^+ (f,f) \leq C_n \| f_\alpha \|_{L^\infty} \| f_\beta \|_{L^1}
\end{aligned}
\end{equation}
for $\alpha, \beta \in \{ 1, \cdots, s_0 \}$, or $\alpha \in \{ 1, \cdots, s_0 \}$, $\beta \in \{ s_0 + 1, \cdots, s \}$, or $\alpha \in \{ s_0 + 1, \cdots, s \}$, $\beta \in \{ 1, \cdots, s_0 \}$. As a consequence, the bounds \eqref{Q+1} and \eqref{Q+2} imply that for all $\alpha \in \{ 1, \cdots, s \}$,
\begin{equation}\label{Q+alpha-n-bnd}
  \begin{aligned}
    \| Q_{\alpha, n}^+ (f,f) \|_{L^\infty} = \| \sum_{\beta = 1}^{s} Q_{\alpha \beta, n}^+ (f,f) \|_{L^\infty} \leq C_{n,s} \| f_\alpha \|_{L^\infty} \sum_{\beta = 1}^{s} \| f_\beta \|_{L^1} \,.
  \end{aligned}
\end{equation}

By the structures of $B_{2 \alpha \beta, n}$, $B_{1 \alpha \beta, n}$, $B_{1 \beta \alpha, n}$ and $B_{0 \alpha \beta, n}$ constructed in Lemmas \ref{Lemma 4.10}-\ref{Lemma 4.11}-\ref{Lemma 4.12}, one easily has $L_{\alpha, n} (f) = \sum_{\beta = 1}^{s} L_{\alpha \beta, n} (f) \leq C_{n,s} \sum_{\beta = 1}^{s} \| f_\beta \|_{L^1}$ for all $\alpha \in \{ 1, \cdots, s \}$, which infer that
\begin{equation}\label{Q-alpha-n-bnd}
  \begin{aligned}
    Q_{\alpha,n}^- (f, f) = f_\alpha L_{\alpha, n} (f) \leq C_{n,s} \| f_\alpha \|_{L^\infty} \sum_{\beta = 1}^{s} \| f_\beta \|_{L^1} \,.
  \end{aligned}
\end{equation}
The bounds \eqref{Q+alpha-n-bnd} and \eqref{Q-alpha-n-bnd} reduce to
\begin{equation*}
  \begin{aligned}
    \| \tilde{Q}_{\alpha, n} (f, f) \|_{L^\infty} = ( N_n (f) )^{-1} \| Q _{\alpha, n} (f, f) \|_{L^\infty} \leq C_{n,s} \| f_\alpha \|_{L^\infty} \sum_{\beta = 1}^{s} \| f_\beta \|_{L^1} \leq C_{n,s} \| f_\alpha \|_{L^\infty}
  \end{aligned}
\end{equation*}
for all $\alpha \in \{ 1, \cdots, s \}$, where $( N_n (f) )^{-1} \sum_{\beta = 1}^{s} \| f_\beta \|_{L^1} \leq n $ is used. The proof of Lemma \ref{Lemma 4.13 of infty} is therefore finished.
\end{proof}

\subsection{Global existence of approximation equation}

In this subsection, based on the Lemma \ref{Lemma 4.13} and \ref{Lemma 4.13 of infty}, we prove the existence and uniqueness of the nonnegative smooth (regularity) solution $f^n$ to the approximate problem \eqref{approximation}. Moreover, we derive the conservation laws (associated with mass, moment, angular moment and energy), the entropy identity and the uniform bounds of the approximated problem \eqref{approximation}.

\begin{lemma}[Existence, Uniqueness and Positivity]\label{Lemma 4.14}
Consider the approximated initial data $f_{\alpha, 0}^n$ constructed in Lemmas \ref{Lemma 4.8}-\ref{Lemma 4.9} and the approximated collision operator $\tilde{Q}_{n} = ( \tilde{Q}_{\alpha, n} )_{\alpha \in \{ 1, \cdots, s \}}$ constructed in \eqref{Q L}. Then the approximated problem \eqref{approximation} admits a unique distributional solution $f^n = ( f^n_\alpha )_{\alpha \in \{ 1, \cdots, s \}}$ satisfying that for any fixed $T>0$,
\begin{enumerate}
\item{for $\alpha \in \{1,\cdots,s_0\}, f^n_\alpha \in C([0,+\infty ); L^1(\R^3 \times \R^3 ))$ and
\begin{equation*}
	f^n_\alpha (t,x,\xi ) \geq \tfrac{1}{n} \exp [ -c_0 (|x|^2+|\xi |^2 ) ] \quad \textrm{for all } \ (t,x,\xi ) \in [0, T] \times \R^3 \times \R^3 ;
	\end{equation*}}
\item{for $\alpha \in \{s_0+1,\cdots,s\}, f^n_\alpha \in C([0,+\infty ); L^1(\R^3 \times \R^3 \times \R_+))$,
	\begin{equation*}
		f^n_\alpha (t,x,\xi , I ) \geq \tfrac{1}{n} \exp [ -c_0 (|x|^2+|\xi |^2 + I ) ] \quad \textrm{for all } \ (t,x,\xi , I ) \in [0, T] \times \R^3 \times \R^3 \times \R_+ \,.
\end{equation*}}	
\end{enumerate}
Here the constant $c_0 = c_0 (n,T)>0 $.
\end{lemma}

\begin{proof}

\underline{\em Step 1. Existence and uniqueness.} For simplicity, we denote by $\mathbf{Z}_\alpha = \xi \in \mathcal{Z}_\alpha = \R^3$ for $\xi \in \{ 1, \cdots, s_0 \}$ and $\mathbf{Z}_\alpha = (\xi, I) \in \mathcal{Z}_\alpha = \R^3 \times \R_+$ for $\alpha \in \{ s_0 + 1, \cdots, s \}$.

For all $f_\alpha (x,\mathbf{Z}_\alpha) \in L^1(\R^3 \times \mathcal{Z}_\alpha ; d x d \mathbf{Z}_\alpha)$ with the norm $\| f_\alpha \|_{L^1} = \int_{\R^3 \times \mathcal{Z}_\alpha} |f_\alpha (x, \mathbf{Z}_\alpha)| d x d \mathbf{Z}_\alpha$, we define
\begin{equation*}
  \begin{aligned}
    U (t) f_\alpha (x, \mathbf{Z}_\alpha) = f_\alpha ( x - t \xi, \mathbf{Z}_\alpha ) \,.
  \end{aligned}
\end{equation*}
It is easy to see that $ U (t + \tau) f_\alpha (x, \mathbf{Z}_\alpha) = U (t) U (\tau) f_\alpha (x, \mathbf{Z}_\alpha) $ and $U (0) f_\alpha (x, \mathbf{Z}_\alpha) = f_\alpha (x, \mathbf{Z}_\alpha)$. Namely, $\{U(t)\}_{t\geq 0}$ is a operator semigroup on $L^1(\R^3 \times \mathcal{Z}_\alpha ;  d x d \mathbf{Z}_\alpha)$. Moreover,
	\begin{equation*}
		A(f_\alpha)=\lim \limits_{t\rightarrow 0^+ } \tfrac{U(t) f_\alpha - f_\alpha }{t} =-\xi \cdot \nabla_x f_\alpha \,,
	\end{equation*}
which means that $A(f_\alpha)$ is the generator of the operator semigroup $U(t)f(x,\mathbf{Z}_\alpha)$. Denote by
\begin{equation*}
  \begin{aligned}
    X = \{ f = (f_\alpha)_{\alpha \in \{ 1, \cdots, s \}} | f_\alpha \in L^1(\R^3 \times \mathcal{Z}_\alpha ; d x d \mathbf{Z}_\alpha) \} = \prod_{\alpha = 1}^{s} L^1(\R^3 \times \mathcal{Z}_\alpha ; d x d \mathbf{Z}_\alpha) 
  \end{aligned}
\end{equation*}
endowed with the norm $ \| f \|_X = \sum_{\alpha = 1}^{s} \| f_\alpha \|_{L^1} < \infty $. By Lemma \ref{Lemma 4.13}, one has
	\begin{equation*}
		\begin{aligned}
			\Vert \tilde{Q}_{\alpha ,n}(f,f) - \tilde{Q}_{\alpha ,n} (g,g) \Vert_{L^1} \leq C_{s,n} \sum_{\beta = 1}^s \Vert f_\beta -g_\beta \Vert_{L^1},\ \Vert \tilde{Q}_{\alpha ,n}(f,f) \Vert_{L^1} \leq C_{s,n} \Vert f_\alpha \Vert_{L^1} \,,
		\end{aligned}
	\end{equation*}
which means that the map $G: X \to X$ such that 
$$f = (f_\alpha)_{\alpha \in \{ 1, \cdots, s \}} \mapsto \tilde{Q}_n (f,f) = ( \tilde{Q}_{\alpha, n} (f,f) )_{\alpha \in \{ 1, \cdots, s \}}$$
is Lipschitz continuous. Then there is a unique $f \in C ( [0, + \infty ) ; X )$ solving the following Cauchy problem \eqref{approximation} with the form
	\begin{equation}\label{G}
			\tfrac{d}{dt} f =A ( f ) + G ( f ) \,, \ f |_{t=0} = f_0 : = ( f^n_{\alpha ,0})_{\alpha \in \{ 1, \cdots, s \}}  
	\end{equation}
with the mild form. Namely, for all $\alpha \in \{ 1, \cdots, s \}$, the components $f_\alpha$ of $f$ satisfy
	\begin{equation*}
		\begin{aligned}
			f^n_\alpha (t, \mathbf{Z}_\alpha) = f^n_{\alpha ,0} (x,\mathbf{Z}_\alpha ) + \int_{0}^{t} G_\alpha ( f^n_\alpha (\tau,x-(t-\tau)\xi ,\mathbf{Z}_\alpha)) d \tau \quad a.e.\ (x,\mathbf{Z}_\alpha) \in \R^3\times \mathcal{Z}_\alpha \,.
		\end{aligned}
	\end{equation*}
Here $A (f) = ( A (f_\alpha) )_{\alpha \in \{ 1, \cdots, s \}}$ and $G (f) = ( G_\alpha (f_\alpha) )_{\alpha \in \{ 1, \cdots, s \}}$ with $G_\alpha (f_\alpha) = \tilde{Q}_{\alpha, n} (f,f)$. By the Lemma \ref{distri-mild}, one infers that $f^n$ is also the unique distributional solution of problem \eqref{approximation}. 

\underline{\em Step 2. Positivity.} We first consider the following auxiliary Cauchy problem 
	\begin{equation} \label{g}
			\tfrac{\partial }{\partial t} g^n  +\xi \cdot \nabla_x g^n = F(g^n) \,, \ g^n|_{t=0} = f^n_{0} \,,
	\end{equation}
	where $g^n = ( g^n_{\alpha} (t, x, \mathbf{Z}_\alpha) )_{\alpha \in \{ 1, \cdots, s \}}$, $F(g^n)=(F_{1 }(g^n),\cdots,F_{s }(g^n))$, $ F_{\alpha }(g^n)=\tilde{Q}_{\alpha ,n}^+ ( g^n_\star , g^n_\star ) - \tilde{L} _{\alpha, n} ( g^n_\star ) g^n_\alpha $ with $g_\star = ( |g^n_\alpha| )_{\alpha \in \{ 1, \cdots, s \}}$. By Lemma \ref{Lemma 4.13}, one infers that $F_\alpha (g^n)$ is a bounded and Lipschitz continuous operator on $X$. Then the Cauchy problem \eqref{g} admits a unique mild solution $g^n \in C([0,+\infty ); X )$. Lemma \ref{Lemma3.3} then implies that for all $\alpha \in \{ 1, \cdots, s \}$,
\begin{equation*}
\begin{aligned}
& g^{n \sharp}_\alpha (t, x, \mathbf{Z}_\alpha) - g^{n\sharp}_\alpha (0, x, \mathbf{Z}_\alpha ) \exp [ - \int_{0}^{t} \tilde{L}_{\alpha ,n } ( g^n_\star)^\sharp (\tau, \mathbf{Z}_\alpha ) d \tau ] \\
= & \int_{0}^{t} \tilde{Q}^+_{\alpha ,n} ( g^n_\star, g^n_\star )^\sharp (\tau, x, \mathbf{Z}_\alpha ) \exp [- \int_{\tau}^{t} \tilde{L}_{\alpha ,n } ( g^n_\star)^\sharp (\tau', x, \mathbf{Z}_\alpha ) d \tau' ] d \tau \,,
\end{aligned}
\end{equation*}
where $g^\sharp (t, x, \mathbf{Z}_\alpha) = g (t, x + t \xi, \mathbf{Z}_\alpha)$. By changing the variable $x+t\xi \rightarrow x$ and using the initial condition $g^n |_{t= 0} = f_0^n$, one gains that for all $\alpha \in \{ 1, \cdots, s \}$,
	\begin{equation*}{\small
		\begin{aligned}
			g^n_\alpha (t,x,\mathbf{Z}_\alpha) = & f_{\alpha, 0}^n (x, \mathbf{Z}_\alpha) \exp [ - \int_{0}^{t} \tilde{L}_{\alpha ,n } ( g^n_\star ) (\tau, x - (t - \tau ) \xi, \mathbf{Z}_\alpha ) d \tau ] \\
			+ & \int_{0}^{t} \tilde{Q}^+_{\alpha ,n} ( g^n_\star, g^n_\star ) (\tau, x - (t - \tau) \xi, \mathbf{Z}_\alpha ) \exp [ - \int_{\tau}^{t} \tilde{L}_{\alpha ,n } ( g^n_\star) (\tau' , x - (t - \tau') \xi, \mathbf{Z}_\alpha ) d \tau' ] d \tau .
		\end{aligned}}
	\end{equation*}
Note that $f_{\alpha, 0}^n (x, \mathbf{Z}_\alpha) > 0$ and $\tilde{Q}^+_{\alpha ,n} ( g^n_\star, g^n_\star ) (\tau, x - (t - \tau) \xi, \mathbf{Z}_\alpha ) \geq 0$. One concludes that the solution $g^n_\alpha (t,x,\mathbf{Z}_\alpha) \geq 0$ for all $(t,x,\mathbf{Z}_\alpha) \in \R_+ \times \R^3 \times \mathcal{Z}_\alpha$ and for all $\alpha \in \{ 1, \cdots, s \}$. Due to $g_\alpha \geq 0$, one has $F (g^n) = G (g^n)$. This means that $g^n$ is also a solution to the problem \eqref{G} with nonnegative components $g_\alpha \geq 0$ ($ \alpha \in \{ 1, \cdots, s \} $). By the uniqueness of the approximate problem \eqref{approximation} (or equivalently \eqref{G}), we know that $f^n_\alpha(t,x,\mathbf{Z}_\alpha) = g^n_\alpha(t,x,\mathbf{Z}_\alpha) \geq 0$ for all $\alpha \in \{ 1, \cdots, s \}$.

At the end, we prove the positive lower bound of the components $f_\alpha (t,x, \mathbf{Z}_\alpha) \geq 0$ for all $\alpha \in \{ 1, \cdots, s \}$. Observe that
	\begin{equation*}
		\tfrac{\partial }{\partial t} f^n_\alpha +\xi \cdot \nabla_x f^n_\alpha = \tilde{Q}^+_{\alpha ,n} (f^n ,f^n)-\tilde{Q}^-_{\alpha ,n} (f^n ,f^n) \geq - \tilde{L}_{\alpha ,n} (f^n)f^n_\alpha \,.
	\end{equation*}
Moreover, by Lemmas \ref{Lemma 4.10}-\ref{Lemma 4.11}-\ref{Lemma 4.12}, one knows that $B_{2 \alpha \beta, n}$, $B_{1 \alpha \beta, n}$, $B_{1 \beta \alpha, n}$ and $B_{0 \alpha \beta, n}$ admit some upper bounds $C_{n,s} > 0$. Then one easily has $L_{\alpha, n} (f^n) = \sum_{\beta = 1}^{s} L_{\alpha \beta, n} (f^n) \leq C_{n,s} \sum_{\beta = 1}^{s} \| f_\beta^n \|_{L^1}$ for all $\alpha \in \{ 1, \cdots, s \}$, which deduces from $ ( N_n (f^n) )^{-1} \sum_{\beta = 1}^{s} \| f_\beta^n \|_{L^1} \leq n $ that
	\begin{equation*}
		\begin{aligned}
			f^n_\alpha \tilde{L}_{\alpha,n} (f^n) = ( N_n (f^n) )^{-1} {L}_{\alpha,n} (f^n) f^n_\alpha \leq C_1 ( n) f^n_\alpha 
		\end{aligned}
	\end{equation*}
for some constant $C_1 (n) > 0$. Consequently,
	\begin{equation*}
		\tfrac{\partial }{\partial t} f^{n\sharp}_\alpha = \tilde{Q}_{\alpha,n} (f^n,f^n)^\sharp \geq - C_1(n) f^{n\sharp}_\alpha.
	\end{equation*}
Then for all $t \in [0, T]$,
\begin{equation*}
  \begin{aligned}
    f^{n\sharp}_\alpha (t, x, \mathbf{Z}_\alpha) \geq e^{-C_1 (n) t} f^{n\sharp}_\alpha (0, x, \mathbf{Z}_\alpha) \geq e^{-C_1 (n) T }f_{\alpha, 0}^n (x, \mathbf{Z}_\alpha) \,,
  \end{aligned}
\end{equation*}
which, together with Lemma \ref{Lemma 4.8} and Lemma \ref{Lemma 4.9}, means that
\begin{equation*}
  \begin{aligned}
    & f^n_\alpha (t,x, \mathbf{Z}_\alpha) \geq e^{-C_1 (n) T } f_{\alpha, 0}^n (x - t \xi, \mathbf{Z}_\alpha) \geq \tfrac{1}{n} \exp ( - \tfrac{|x - t \xi|^2 + |\xi|^2 + 2C_1 (n) T }{2} ) \,, \qquad \alpha \in \{ 1, \cdots, s_0 \} \,, \\ 
    & f^n_\alpha (t,x, \mathbf{Z}_\alpha) \geq e^{-C_1 (n) T }  f_{\alpha, 0}^n (x - t \xi, \mathbf{Z}_\alpha) \geq \tfrac{1}{n} \exp ( - \tfrac{|x - t \xi|^2 + |\xi|^2 + I + 2 C_1 (n) T}{2} ) \,, \quad \alpha \in \{ s_0 + 1, \cdots, s \} \,.
  \end{aligned}
\end{equation*}
Note that $ |x - t \xi|^2 \leq 2 |x|^2 +  2 T^2  |\xi|^2 $ for all $t \in [0, T]$. Take $c_0 = \max \{ 1, T^2 +  \frac{1}{2}, C_1 (n) T \} > 0 $. Then
\begin{equation*}
  \begin{aligned}
    & f^n_\alpha (t,x, \xi) \geq \tfrac{1}{n} \exp ( - c_0 (|x|^2 + |\xi|^2) ) \,, \qquad \quad \ \ \alpha \in \{ 1, \cdots, s_0 \} \,, \\ 
    & f^n_\alpha (t,x, \xi, I) \geq \tfrac{1}{n} \exp ( - c_0 (|x|^2 + |\xi|^2 + I) ) \,, \quad \alpha \in \{ s_0 + 1, \cdots, s \} \,.
  \end{aligned}
\end{equation*}
Therefore, the proof of Lemma \ref{Lemma 4.14} is finished.
\end{proof}

Now we give the regularity results on the approximated problem \eqref{approximation}.

\begin{lemma}[Regularity]\label{Lemma 4.16}
Let $T > 0$ and the integers $m,k\geq 1$ be arbitrarily fixed. Assume that $f^n$ is the distributional solution to the approximated problem \eqref{approximation} constructed in Lemma \ref{Lemma 4.14}. Then there is a constant $ C = C (n, m, k, T) > 0 $ such that the solution $f^n$ further obeys the following statements:

$(1)$ {For all $ \alpha \in \{1,\cdots,s_0\}$ and any multi-index $ \boldsymbol{\zeta} = ( \zeta_0, \zeta_1, \cdots, \zeta_6 ) \in \mathbb{N}^7$ with $| \boldsymbol{\zeta} | = \sum_{i=0}^{6} \zeta_i \leq m$, we have 
$$D^{ \boldsymbol{\zeta} } f^n_\alpha \in L^\infty ((0,T) \times \R^3 \times \R^3 ) \cap L^\infty ( (0, T); L^1 ( \R^3 \times \R^3; (1 + |x|^k + |\xi|^k) d x d \xi ) ) \,, $$
and 
\begin{equation*}
  \sup_{ (t,x,\xi) \in (0,T) \times \R^3 \times \R^3 } | D^{ \boldsymbol{\zeta} } f^n_\alpha | + \sup_{t \in (0, T)} \iint_{\R^3 \times \R^3 } (1+|x|^k+|\xi|^k ) | D^{ \boldsymbol{\zeta} } f^n_\alpha |dxd\xi \leq C \,;
\end{equation*} }

$(2)$ {For all $\alpha \in \{s_0+1,\cdots,s\}$ and any multi-index $\boldsymbol{ \gamma } = (\gamma_0, \gamma_1, \cdots, \gamma_6, \gamma_7) \in \mathbb{N}^8$ with $| \boldsymbol{ \gamma } | = \sum_{i=0}^{7} \gamma_i \leq m$, we have
{\small
$$ D^{ \boldsymbol{ \gamma } } f^n_\alpha \in L^\infty ((0,T)\times \R^3 \times \R^3 \times \R_+) \cap L^\infty ( (0, T); L^1 ( \R^3 \times \R^3 \times \R_+; (1 + |x|^k + |\xi|^k + I^k) d x d \xi d I ) ) \,, $$ }
and
\begin{equation*}
	\sup_{ (t,x,\xi, I) \in (0,T) \times \R^3 \times \R^3 \times \R_+ } | D^{ \boldsymbol{\gamma} } f^n_\alpha | + \sup_{t \in (0, T)} \iiint_{\R^3 \times \R^3 \times \R_+ } (1+|x|^k+|\xi|^k + I^k) | D^{ \boldsymbol{ \gamma } } f^n_\alpha |dxd\xi dI \leq C \,.
	\end{equation*}	}	
\end{lemma}

\begin{proof}
  In order to prove the smoothness of the approximated solution, we mainly focus on the following mild form of \eqref{approximation} constructed in Lemma \ref{Lemma 4.14}:
  \begin{equation}\label{fn-alpha}
	f^n_{\alpha} (t,x,\mathbf{Z}_\alpha) = U(t) f^n_{\alpha, 0} (x, \mathbf{Z}_\alpha) + \int_{0}^{t} U(t-\tau) \tilde{Q}_{\alpha ,n} (f^n ,f^n) (\tau, x, \mathbf{Z}_\alpha) d \tau \,,
\end{equation}
where $U(t) f (x, \mathbf{Z}_\alpha) = f (x - t \xi, \mathbf{Z}_\alpha)$ is the operator semigroup. Here $\mathbf{Z}_\alpha = \xi \in \mathcal{Z}_\alpha = \R^3$ for $\alpha \in \{ 1, \cdots, s_0 \}$ and $\mathbf{Z}_\alpha = (\xi, I) \in \mathcal{Z}_\alpha = \R^3 \times \R_+$ for $\alpha \in \{ s_0 + 1, \cdots, s \}$. Observe that the approximated initial data and collision kernels are all smooth. Then, by applying the derivative operators $D^{ \boldsymbol{\zeta} }$ and $ D^{ \boldsymbol{\gamma} } $ to \eqref{fn-alpha}, employing the Gr\"onwall inequality, combining the Lemmas \ref{Lemma 4.13}-\ref{Lemma 4.13 of infty} and using the Induction Principle for orders of derivatives, one can conclude the results of this lemma. For simplicity, we omit the details here.
\end{proof}

\begin{lemma}\label{Lemma 4.17}
Assume that $f^n$ is the distributional solution to the approximated problem \eqref{approximation} constructed in Lemma \ref{Lemma 4.14}. Then it further enjoys the following properties:
\begin{enumerate}
	\item {Mass, momentum, energy and angular momentum are all conserved. Hence, for all $t \geq 0$ and $\alpha \in \{1,\cdots,s_0\}$, we have
		\begin{equation*}
			\iint_{\R^3 \times \R^3 } f^n_\alpha 
			\left (\begin{array}{c}
				1 \\[2mm]
				\xi \\ [2mm]
				|\xi|^2 \\ [2mm]
				|x-t\xi|^2 \\[2mm]
			\end{array}\right ) dxd\xi 
			= \iint_{\R^3 \times \R^3 } f_{\alpha ,0}^n 
			\left (\begin{array}{c}
				1 \\[2mm]
				\xi \\ [2mm]
				|\xi|^2 \\ [2mm]
				|x|^2 \\[2mm]
			\end{array}\right )dxd\xi ;
	\end{equation*}
For all $t>0$ and $\alpha \in \{s_0+1,\cdots,s\}$, we have
\begin{equation*}
	\iiint_{\R^3 \times \R^3 \times \R_+ } f^n_\alpha 
	\left (\begin{array}{c}
		1 \\[2mm]
		\xi \\ [2mm]
		\frac{1}{2}m_\alpha|\xi|^2+I \\ [2mm]
		|x-t\xi|^2 \\[2mm]
	\end{array}\right ) dxd\xi dI 
	= \iiint_{\R^3 \times \R^3 \times \R_+ } f_{\alpha ,0}^n 
	\left (\begin{array}{c}
		1 \\[2mm]
		\xi \\ [2mm]
		\frac{1}{2}m_\alpha|\xi|^2+I\\ [2mm]
		|x|^2 \\[2mm]
	\end{array}\right )dxd\xi dI .
	\end{equation*}}\label{1 of Lemma 4.17}
	\item{Entropy identity is true. Hence for all $t \geq 0$,
		\begin{equation}\label{entro}
			\begin{aligned}
				H(f^n_{0}) & = H(f^n)(t)+ \tfrac{1}{4}\sum \limits_{\alpha =1}^{s_0} \int_{0}^{t} d \tau \iint_{ \R^3 \times \R^3 } \tilde{e}_\alpha^n(
\tau,x,\xi ) dxd\xi\\
				& +\tfrac{1}{4}\sum \limits_{\alpha =s_0+1}^{s}\int_{0}^{t} d \tau \iiint_{ \R^3 \times \R^3 \times \R_+} \tilde{e}_\alpha^n(\tau,x,\xi ,I) dxd\xi dI ,
			\end{aligned}
		\end{equation}
where the density of entropy dissipation
\begin{equation}\label{e na}
	\begin{aligned}
	\tilde{e}_\alpha^n &= (N_n (f^n))^{-1} \sum \limits_{\beta =1}^s \int_{(\R^3\times \R_+)^3} (\tfrac{I^{\delta(\alpha)/2-1} I_* ^{\delta(\beta)/2-1}{f_\alpha^n}' {f_{\beta_*}^n}'  } {f^n_\alpha f^n_{\beta_* } (I')^{\delta(\alpha)/2-1} (I'_* )^{\delta(\beta)/2-1}} - 1)\\
&\times \log (\tfrac{I^{\delta(\alpha)/2-1} I_* ^{\delta(\beta)/2-1} {f_\alpha^n}' {f_{\beta_*}^n}' }{f^n_\alpha f^n_{\beta_* } (I')^{\delta(\alpha)/2-1} (I'_* )^{\delta(\beta)/2-1} } ) \tfrac{f^n_\alpha f^n_{\beta_* } }{I^{\delta(\alpha)/2-1} I_* ^{\delta(\beta)/2-1} }W_{\alpha\beta,n }d\xi_* d\xi'd\xi'_* dI_* dI' dI'_*,
	\end{aligned}
\end{equation}
the $H$-functional $H(f)$ is defined in \eqref{Hf}, and the transition probabilities $W_{\alpha \beta ,n}$ corresponds to the collision kernels $B_{2 \alpha \beta ,n}$, $B_{1 \alpha \beta, n}$, $B_{1 \beta \alpha, n}$ and $B_{0 \alpha \beta, n}$ constructed in Lemmas \ref{Lemma 4.10}-\ref{Lemma 4.11}-\ref{Lemma 4.12}. }\label{2 of Lemma 4.17}

\item{For all $\alpha \in \{1,\cdots,s_0\}$ and $t \geq 0$, we have
	\begin{equation}\label{FB2}{\small
		\begin{aligned}
  \iint_{ \R^3 \times \R^3} f^n_\alpha (t,x,\xi)(1+|x|^2+|\xi|^2)dxd\xi \leq & \iint_{ \R^3 \times \R^3} f^n_{\alpha,0}(x,\xi)(1+2|x|^2+(2t^2+1)|\xi|^2)dxd\xi,
  \end{aligned}}
	\end{equation}
and
 \begin{equation}\label{FWC2}
		\begin{aligned}
			\sup \limits_{t\geq 0}&\iiint_{ \R^3 \times \R^3} f^n_\alpha(t,x,\xi)|\log f^n_\alpha| dxd\xi +\tfrac{1}{4}\int_{0}^{+\infty} dt \iint_{ \R^3 \times \R^3} \tilde{e}^n_\alpha dxd\xi \\
		   \leq & 2 \iiint_{ \R^3 \times \R^3} f^n_{\alpha,0}(x,\xi) (|\log f^n_{\alpha,0}|+ |x|^2+ |\xi|^2 )dxd\xi +C_1;
		\end{aligned}
	\end{equation}
	For all $\alpha \in \{s_0+1,\cdots,s\}$ and $t \geq 0$, we have
	\begin{equation}\label{FB3}
		\begin{aligned}
   &\iint_{ \R^3 \times \R^3\times \R_+ } f^n_\alpha (t,x,\xi,I)(1+|x|^2+ \tfrac{1}{2} m_\alpha |\xi|^2+I)dxd\xi dI\\
  \leq & \iint_{ \R^3 \times \R^3\times \R_+ } f^n_{\alpha,0}(x,\xi,I ) \{ 1 + 2 |x|^2 + ( 2 t^2 + \tfrac{1}{2} m_\alpha ) |\xi|^2 + (\tfrac{4}{m_\alpha} t^2 + 1) I )\} dxd\xi dI,
\end{aligned}
\end{equation}
and
 \begin{equation}\label{FWC3}{\small
		\begin{aligned}
			\sup \limits_{t\geq 0}& \iiint_{ \R^3 \times \R^3\times \R_+} f^n_\alpha(t,x,\xi,I) |\log (I^{1-\delta(\alpha)/2}f^n_\alpha)| dxd\xi dI +\tfrac{1}{4}\int_{0}^{+\infty} dt \iiint_{ \R^3 \times \R^3\times \R_+} \tilde{e}^n_\alpha dxd\xi dI\\
		\leq & 2 \iiint_{ \R^3 \times \R^3\times \R_+} f^n_{\alpha,0}(x,\xi,I )\{|\log (I^{1-\delta(\alpha)/2}f^n_{\alpha,0})|+ M_\alpha(|x|^2+m_\alpha |\xi|^2 +2I)\}dxd\xi dI+C_1,
		\end{aligned}}
\end{equation}
where $M_\alpha = \max \{1,\frac{1}{m_\alpha}\}$ and the constant $C_1>0$ is independent of $n$.}\label{3 of Lemma 4.17}
\end{enumerate}
\end{lemma}

\begin{proof}
Note that the collision kernels $B_{2 \alpha \beta ,n}$, $B_{1 \alpha \beta, n}$, $B_{1 \beta \alpha, n}$ and $B_{0 \alpha \beta, n}$ constructed in Lemmas \ref{Lemma 4.10}-\ref{Lemma 4.11}-\ref{Lemma 4.12} enjoy the same symmetry properties of $B_{2 \alpha \beta }$, $B_{1 \alpha \beta}$, $B_{1 \beta \alpha}$ and $B_{0 \alpha \beta}$ in \eqref{B0-ab}-\eqref{B1-ab}-\eqref{B1-ba}-\eqref{B2-ab}. Then the all structure properties of the collision operator $Q$ will also be valid for the approximated collision operator $\tilde{Q}_n = ( \tilde{Q}_{\alpha, n} )_{ \alpha \in \{ 1, \cdots, s \} }$ constructed in \eqref{Q L}. Then the collision invariant form of $Q$
	\begin{equation*}
		\varphi_i= ae_i + b \mathbf{m} \xi +c (\mathbf{m} |\xi|^2+2\mathbf{I}),\quad i=1,\cdots,s 
	\end{equation*}
is also that of $\tilde{Q}_n$, where $a,c\in \R $ and $ b\in \R^3$ are arbitrary, see Proposition \ref{prop2.3}. As a result, multiplying the equation \eqref{approximation} by the collision invariants and integrating by parts over $  (t,x,\xi) \in \R_+ \times \R^3 \times \R^3 $ for $\alpha \in \{ 1, \cdots, s_0 \}$ and $(t,x, \xi, I) \in \R_+ \times \R^3 \times \R^3 \times \R_+$ for $\alpha \in \{ s_0 + 1, \cdots, s \}$, one can obtain the conclusion \eqref{1 of Lemma 4.17} in this lemma. The conclusion \eqref{2 of Lemma 4.17} of this lemma can be derived from the same arguments of \eqref{EntropyId} resulted from Section 2.2 of \cite{N.Bernhoff-2023}.

Next we focus on the conclusion \eqref{3 of Lemma 4.17} of this lemma. By the conclusion \eqref{1 of Lemma 4.17}, one infers that for $\alpha \in \{s_0+1,\cdots,s\}$ and $t \geq 0$,
	\begin{equation*}
		\begin{aligned}
			\iiint_{\R^3 \times \R^3 \times \R_+}f^n_\alpha |x|^2 dxd\xi dI
			\leq & 2\iiint_{\R^3 \times \R^3 \times \R_+}f^n_\alpha (|x-t\xi|^2 + t^2 |\xi|^2) dxd\xi dI \\
\leq & 2\iiint_{\R^3 \times \R^3 \times \R_+}f^n_{\alpha,0} (|x|^2 + t^2 |\xi|^2 + \tfrac{2}{m_\alpha} t^2 I) dxd\xi dI.
		\end{aligned}
	\end{equation*}
	Then there holds
	\begin{equation*}
		\begin{aligned}
			&\iiint_{ \R^3 \times \R^3\times \R_+} f^n_\alpha (1+|x|^2+ \tfrac{1}{2} m_\alpha |\xi|^2 +I)dxd\xi dI \\
			\leq &\iiint_{ \R^3 \times \R^3\times \R_+} f^n_{\alpha,0} \{ 1 + 2 |x|^2 + ( 2 t^2 + \tfrac{1}{2} m_\alpha ) |\xi|^2 + (\tfrac{4}{m_\alpha} t^2 + 1) I )\} dxd\xi dI.
		\end{aligned}
	\end{equation*}
	Similarly for $\alpha \in \{1,\cdots,s_0\}$ and $t \geq 0$, one has
	\begin{equation*}
		\iint_{ \R^3 \times \R^3} f^n_\alpha (1+|x|^2+|\xi|^2) dxd\xi \leq \iint_{ \R^3 \times \R^3} f^n_{\alpha,0} (1+2|x|^2+(2t^2+1)|\xi|^2) dxd\xi .
	\end{equation*}

Now we prove the estimate \eqref{FWC2} for $\alpha \in \{1, \cdots, s_0 \}$. Note that 
	\begin{equation*}
		\begin{aligned}
			\iint_{ \R^3 \times \R^3} f^n_\alpha|\log f^n_\alpha| dxd\xi 
			=&\iint_{ \R^3 \times \R^3} f^n_\alpha\log f^n_\alpha dxd\xi -2 \iint_{0<f^n_\alpha\leq 1} f^n_\alpha\log f^n_\alpha dxd\xi .
		\end{aligned}
	\end{equation*}
The entropy identity \eqref{entro} shows that 
$$ \iint_{ \R^3 \times \R^3} f^n_\alpha\log f^n_\alpha dxd\xi \leq \iint_{\R^3 \times \R^3} f^n_{\alpha, 0} \log f^n_{ \alpha, 0 } d x d \xi \leq \iint_{\R^3 \times \R^3} f^n_{\alpha, 0} | \log f^n_{ \alpha, 0 } | d x d \xi $$
for a constant $C > 0$ independent of $n$. Observe that $t\log \frac{1}{t}\leq C_0 \sqrt{t}$ for $t\in (0,1)$ for some $C_0\geq 0$. Then
\begin{equation*}{\small
  \begin{aligned}
    -2 \iint_{0<f^n_\alpha\leq 1} f^n_\alpha\log f^n_\alpha dxd\xi = & -2 \iint_{0<f^n_\alpha\leq e^{ - |x|^2 - |\xi|^2 }} f^n_\alpha\log f^n_\alpha dxd\xi - 2 \iint_{e^{ - |x|^2 - |\xi|^2 }<f^n_\alpha\leq 1} f^n_\alpha\log f^n_\alpha dxd\xi \\
    \leq & 2 C_0 \iint_{\R^3 \times \R^3} e^{ - \frac{|x|^2 + |\xi|^2}{2} } d x d \xi + 2 \iint_{\R^3 \times \R^3} f^n_\alpha(|x|^2 +|\xi|^2) dxd\xi \,.
  \end{aligned}}
\end{equation*}
Notice that $\iint_{\R^3 \times \R^3} f^n_\alpha(|x|^2 +|\xi|^2) dxd\xi = \iint_{\R^3 \times \R^3} f^n_\alpha(|x - t \xi|^2 +|\xi|^2) dxd\xi = \iint_{\R^3 \times \R^3} f^n_{\alpha, 0} (|x|^2 +|\xi|^2) dxd\xi$ by using the conclusion \eqref{1 of Lemma 4.17} of this lemma. As a result, for some constant $C_1 > 0$ independent of $n$,
\begin{equation*}
  \begin{aligned}
    \iint_{ \R^3 \times \R^3} f^n_\alpha|\log f^n_\alpha| dxd\xi \leq \iint_{ \R^3 \times \R^3} f^n_{\alpha,0}(|\log f^n_{\alpha,0}|+2|x|^2+2|\xi|^2 )dxd\xi + C_1 \,.
  \end{aligned}
\end{equation*}
Together with \eqref{entro}, we conclude that
	\begin{equation*}{\small
		\begin{aligned}
			\sup \limits_{t\geq 0}\iint_{ \R^3 \times \R^3} f^n_\alpha|\log f^n_\alpha| dxd\xi +\tfrac{1}{4}\int_{0}^{+\infty} dt \iint_{ \R^3 \times \R^3} \tilde{e}^n_\alpha dxd\xi \leq 2 \iint_{ \R^3 \times \R^3} f^n_{\alpha,0}(|\log f^n_{\alpha,0}|+ |x|^2+ |\xi|^2 )dxd\xi +C_1.
		\end{aligned}}
	\end{equation*}

For the cases $\alpha \in \{s_0+1,\cdots,s\}$, we only deal with the integral 
$$ - \iiint_{0 < I^{1-\delta(\alpha)/2}f^n_\alpha \leq 1 } f^n_\alpha \log (I^{1-\delta(\alpha)/2}f^n_\alpha) dxd\xi dI \,. $$ 
By using the decomposition $\{ 0 < I^{1-\delta(\alpha)/2}f^n_\alpha \leq 1 \} = \{ 0 < I^{1-\delta(\alpha)/2}f^n_\alpha \leq e^{ - |x|^2 - |\xi^2| - I } \} \cup \{ e^{ - |x|^2 - |\xi^2| - I } < I^{1-\delta(\alpha)/2}f^n_\alpha \leq 1 \}$, one infers that
\begin{equation*}{\small
  \begin{aligned}
    - \iiint_{0 < I^{1-\delta(\alpha)/2}f^n_\alpha \leq 1 } f^n_\alpha \log (I^{1-\delta(\alpha)/2}f^n_\alpha) dxd\xi dI \leq & \iiint_{ \R^3 \times \R^3\times \R_+} f^n_\alpha (|x|^2 + |\xi|^2 + I) dxd\xi dI \\
    & + C_0 \iiint_{ \R^3 \times \R^3\times \R_+} I^{\delta(\alpha)/2 - 1} e^{ -\frac{|x|^2 + |\xi^2| + I}{2}  } d x d \xi d I \,.
  \end{aligned}}
\end{equation*}
Following the same arguments in the cases $\alpha \in \{ 1, \cdots, s_0 \}$, one can obtain that
	\begin{equation}\label{f-C}
		\begin{aligned}
			&\sup \limits_{t\geq 0}\iiint_{ \R^3 \times \R^3\times \R_+} f^n_\alpha|\log (I^{1-\delta(\alpha)/2}f^n_\alpha)| dxd\xi dI +\tfrac{1}{4}\int_{0}^{+\infty} dt \iiint_{ R^3 \times R^3\times R_+} \tilde{e}^n_\alpha dxd\xi dI\\
			\leq & 2 \iiint_{ \R^3 \times \R^3\times \R_+} f^n_{\alpha,0}\{|\log (I^{1-\delta(\alpha)/2}f^n_{\alpha,0})|+M_\alpha (|x|^2+m_\alpha |\xi|^2 +2I)\}dxd\xi dI+C_1,
		\end{aligned}
	\end{equation}
	where $M_\alpha = \max \{1,\frac{1}{m_\alpha}\}$. This means that the conclusion \eqref{3 of Lemma 4.17} holds. Therefore, the proof of Lemma \ref{Lemma 4.17} is finished.
\end{proof}

\section{Weak $L^1$ compactness for nonlinear collision operator}\label{sec weak com}

In this section, the main goal is to study the relatively weak $L^1$ compactness of $(1 + f^n_\alpha)^{-1} \tilde{Q}_{\alpha, n} (f^n, f^n)$ for $\alpha \in \{ 1, \cdots, s \}$, where $f^n$ is the mild solution to the approximated problem \eqref{approximation}, and $\tilde{Q}_{\alpha, n} (f^n, f^n)$ is defined in \eqref{Q L}. The ideas of this section are inspired by the DiPerna-Lions' work \cite{Diperna-Lions}. We first consider the loss term part $(1 + f^n_\alpha)^{-1} \tilde{Q}_{\alpha, n}^- (f^n, f^n)$, where $ \tilde{Q}_{\alpha, n}^- (f^n, f^n) = f^n_\alpha \tilde{L}_{\alpha, n} (f^n) $. Here the key point is to study the operators $ \tilde{L}_{\alpha, n} (f^n) $ by using the properties of $A_{2 \alpha \beta, n}$, $A_{1 \alpha \beta, n}$, $A_{1 \beta \alpha, n}$ and $A_{0 \alpha \beta , n}$ mentioned in Lemmas \ref{Lemma 4.10}-\ref{Lemma 4.11}-\ref{Lemma 4.12}. Then we investigate the relatively weak $L^1$ compactness of the gain term part $(1 + f^n_\alpha)^{-1} \tilde{Q}_{\alpha, n}^+ (f^n, f^n)$. The core of approach is to show the Arkeryd-type inequality \cite{Arkeryd-inequality} that for $K > 1$,
\begin{equation}\label{eq5.71}
\tilde{Q}^\pm_{\alpha ,n}(f^n,f^n) \leq \tfrac{1}{\log K}\tilde{e}^n_{\alpha} +K\tilde{Q}^\mp_{\alpha ,n}	(f^n,f^n) \,,
\end{equation}
where $\tilde{e}^n_{\alpha}$ is defined in \eqref{e na}. Then we can use the relatively weak $L^1$ compactness of loss term part to obtain that of gain term part.

\begin{lemma} \label{theo-4.1}
	Let $T, R>0$. Assume that $f^n$ is the unique non-negative distributional solution of \eqref{approximation} constructed in Lemma \ref{Lemma 4.14}. Then the following statements hold:
	\begin{enumerate}
	\item {$ (1+ f^n_\alpha)^{-1} \tilde{Q}^- _{\alpha ,n}(f^n,f^n)$ is
	 relatively weakly compact in $L^\infty (0,T; L^1(\R^3 \times B_R))$ for $\alpha \in \{1,\cdots,s_0\}$;}
	\item {$ (1+ f^n_\alpha)^{-1} \tilde{Q}^- _{\alpha ,n}(f^n,f^n)$ is
 relatively weakly compact in $L^\infty (0,T; L^1(\R^3 \times B_R\times (0,R) ))$ for $\alpha \in \{s_0+1,\cdots s\}$.}
	\end{enumerate}
\end{lemma}

\begin{proof}
	
Recall that by \eqref{Q L}, one has $\tilde{L}_{\alpha, n} (f^n) = N_n (f^n)^{-1} \sum_{\beta = 1}^{s} L_{\alpha \beta, n} (f^n)$. For $\alpha, \beta \in \{ s_0 + 1, \cdots, s \}$, it is implied by \eqref{Q L} and Lemma \ref{Lemma 4.10} that for any $0 < R < \infty$,
\begin{equation*}
  \begin{aligned}
    L_{\alpha \beta, n} (f^n) = \iint\limits_{ \R^3 \times \R_+ } A_{2 \alpha \beta ,n} (\xi - \xi_*, I + I_* ) f^n_\beta (t,x ,\xi_*,I_* ) d I_* d \xi_* \,, \\
    \lim_{ |\xi|, I \to + \infty } \sup \limits_{n \geq 1} (1+|\xi|^2 +  I )^{-1} \int_{|z-\xi |\leq R } \int_{0<\eta -I < R } A_{2\alpha \beta ,n } (z,\eta ) dzd\eta \rightarrow 0 \,.
  \end{aligned}
\end{equation*}
Then for any $\epsilon > 0$ there is a $C_\epsilon > 0$ independent of $n$ such that
\begin{equation}\label{A-bnd}
  \begin{aligned}
    \int_{|z-\xi |\leq R } \int_{0<\eta -I < R } A_{2\alpha \beta ,n } (z,\eta ) dzd\eta \leq \epsilon (1+|\xi|^2 +  I ) + C_\epsilon \,.
  \end{aligned}
\end{equation}
One then infers that for $\alpha, \beta \in \{ s_0 + 1, \cdots , s \}$,
\begin{equation}\label{Trans-fab}
	\begin{aligned}	
		& \iiint_{\R^3 \times B_R \times (0,R)} L_{ \alpha \beta ,n} (f^n) (t,x,\xi, I) dxd\xi dI \\
= & \iiint_{\R^3 \times \R^3 \times \R_+}f^n_{\beta} (t,x,\xi_*, I_*) ( \iint_{B_R \times (0,R)}  A_{2\alpha \beta ,n}(\xi-\xi_*,I+I_*) d\xi dI ) dxd\xi_* dI_* \\
			\leq & \epsilon \iiint_{\R^3 \times \R^3 \times \R_+} f^n_{\beta} (t,x ,\xi_*, I_*) (1+|\xi_*|^2 +I_*) dxd\xi_* dI_* + C_\epsilon \iiint_{\R^3 \times \R^3 \times \R_+} f^n_{\beta} (t,x, \xi_*, I_*) dxd\xi_* dI_*.
		\end{aligned}
	\end{equation}
Together with Lemma \ref{Lemma 4.11} and Lemma \ref{Lemma 4.12}, the similar arguments above imply that for $\alpha,\beta \in \{1,\cdots,s_0\}$,
	\begin{equation*}{\small
		\begin{aligned}	
			\iint_{\R^3 \times B_R } L_{ \alpha \beta ,n} (f^n) (t,x,\xi) dxd\xi \leq \epsilon \iint_{\R^3 \times \R^3 } f^n_{\beta} (t, x, \xi_*) (1+|\xi_*|^2) dxd\xi_* +C_\epsilon \iint_{\R^3 \times \R^3 } f^n_{\beta} (t, x, \xi_*) dxd\xi_*,
		\end{aligned}}
	\end{equation*}
and for $\alpha \in \{1,\cdots,s_0\} , \beta \in \{s_0+1,\cdots,s\}$,
		\begin{align*}	
			\iint_{\R^3 \times B_R } L_{ \alpha \beta ,n} (f^n) (t,x, \xi) dxd\xi \leq \epsilon \iiint_{\R^3 \times \R^3 \times \R_+} f^n_{\beta} (t,x,\xi_*, I_*) (1+|\xi_*|^2 +I_*) dxd\xi_* dI_* \\
+C_\epsilon \iiint_{\R^3 \times \R^3 \times \R_+} f^n_{\beta} (t,x, \xi_*, I_*) dxd\xi_* dI_*,
		\end{align*}
and for $\alpha \in \{s_0+1,\cdots,s\} , \beta \in \{1,\cdots,s_0\}$,
	\begin{equation*}
		\begin{aligned}	
			\iiint_{\R^3 \times B_R \times (0,R)} L_{ \alpha \beta ,n} (f^n) (t,x, \xi, I) dxd\xi dI \leq \epsilon \iint_{\R^3 \times \R^3 } f^n_{\beta} (t, x, \xi_*) (1+|\xi_*|^2) dxd\xi_* \\
+ C_\epsilon \iint_{\R^3 \times \R^3 } f^n_{\beta} (t,x, \xi_*) dxd\xi_* \,.
		\end{aligned}
	\end{equation*}
By the definition of $\tilde{Q}^- _{\alpha ,n}(f^n,f^n)$ in \eqref{Q L}, it is easy to see          
	\begin{equation}\label{Q-L-bnd}
		0 \leq \tfrac{1}{1+ f^n_\alpha } \tilde{Q}^-_{\alpha ,n} (f^n,f^n) \leq L_{\alpha,n}(f^n) =\sum \limits_{\beta =1}^s L_{\alpha \beta ,n} (f^n) \,.
	\end{equation}	
Together with the uniform bounds in Lemma \ref{Lemma 4.17}, it infers that $(1+ f^n_\alpha )^{-1} \tilde{Q}^- _{\alpha ,n}(f^n,f^n)$ is bounded in $L^1$ for all $\alpha \in \{ 1, \cdots, s \}$. 

For $\alpha, \beta \in \{ s_0 + 1, \cdots, s \}$, denote by $L^m_{\alpha \beta ,n} (f^n)$ the operators corresponding to the cutoff kernels $A^m_{\alpha \beta, n} (z, \eta) =  A_{\alpha \beta ,n}(z,\eta ) \mathbf{1}_{|z|\leq m} \mathbf{1}_{0<\eta <m}$ ($m>0$). Hence,
\begin{equation*}
  \begin{aligned}
    L^m_{\alpha \beta ,n} (f^n) = \iint_{\R^3 \times \R_+} A^m_{\alpha \beta, n} (\xi - \xi_*, I + I_*) f_{\beta *} d \xi_* d I_* \,.
  \end{aligned}
\end{equation*}
Let
	\begin{equation*}
		a^m_{2\alpha \beta,n} = \iint_{\R^3 \times \R_+ } A_{2\alpha \beta ,n}(z,\eta ) \mathbf{1}_{|z|\leq m} \mathbf{1}_{0<\eta <m} dz d\eta > 0,\ \Psi(t)= t (\log t)^+.
	\end{equation*}
Observe that $ (\log x)^+ = ( \log \frac{x}{y} y )^+ \leq ( \log \frac{x}{y} )^+ + | \log y | $ for all $x,y > 0$, which means that $\Psi (x) \leq y \Psi ( \frac{x}{y} ) + x | \log y |$. It then holds
		\begin{align*}
			&\iiint_{\R^3 \times B_R\times(0,R)} \Psi (L^m_{ \alpha \beta ,n}(f^n))dxd\xi dI \\
			\leq & a^m_{2\alpha \beta,n} \iiint_{\R^3 \times B_R\times(0,R)} \Psi (\tfrac{L^m_{ \alpha \beta ,n}(f^n)}{a^m_{2\alpha \beta,n}})dxd\xi dI +|\log a^m_{2\alpha \beta,n}|\iiint_{\R^3 \times B_R\times(0,R)} L^m_{ \alpha \beta ,n}(f^n) dxd\xi dI \,.
		\end{align*}
Note that
\begin{equation*}{\small
  \begin{aligned}
    \iiint_{\R^3 \times B_R \times(0,R)} L^m_{ \alpha \beta ,n}(f^n) dxd\xi dI = & \iiint_{\R^3 \times B_R \times(0,R)} \iint_{\R^3 \times \R_+} A^m_{\alpha \beta, n} (\xi - \xi_*, I + I_*) f_{\beta _*} d \xi_* d I_* d x d \xi d I \\
    = & \iiint_{\R^3 \times \R^3 \times \R_+} f_{\beta _*} \iint_{B_R \times(0,R)} A^m_{\alpha \beta, n} (\xi - \xi_*, I + I_*) d \xi d I d x d \xi_* d I_* \\
    \leq & a^m_{2 \alpha \beta, n} \iiint_{\R^3 \times \R^3 \times \R_+} f_{\beta _*} d x d \xi_* d I_* \,.
  \end{aligned}}
\end{equation*}
Owing to the convexity of $\Psi$, the Jensen's inequality indicates that
\begin{equation*}{\small
  \begin{aligned}
    a^m_{2\alpha \beta,n} \iiint_{\R^3 \times B_R\times(0,R)} \Psi (\tfrac{L^m_{ \alpha \beta ,n}(f^n)}{a^m_{2\alpha \beta,n}})dxd\xi dI \leq \iiint_{\R^3 \times B_R \times(0,R) } L^m_{ \alpha \beta ,n} (\Psi(f^n)) dxd\xi dI \\
    \leq a^m_{2\alpha \beta,n} \iiint_{\R^3 \times \R^3 \times \R_+ } \Psi(f^n_{\beta _*}) dxd\xi dI = a^m_{2\alpha \beta,n} \iiint_{\R^3 \times \R^3 \times \R_+} f^n_{\beta _*} (\log f^n_{\beta _*} )^+ dx d\xi_* dI_* \,.
  \end{aligned}}
\end{equation*}
As a result,
\begin{equation*}
  \begin{aligned}
    &\iiint_{\R^3 \times B_R\times(0,R)} \Psi (L^m_{ \alpha \beta ,n}(f^n))dxd\xi dI \\
    \leq & a^m_{2\alpha \beta,n} \iiint_{\R^3 \times \R^3 \times \R_+} f^n_{\beta _*} (\log f^n_{\beta _*} )^+ dx d\xi_* dI_* + a^m_{2\alpha \beta,n} |\log a^m_{2\alpha \beta,n}| \iiint_{\R^3 \times \R^3 \times \R_+} f^n_{\beta _*} dxd\xi_* dI_*.
  \end{aligned}
\end{equation*}
By using the uniform bounds \eqref{FB3} and \eqref{FWC3}, one has
	\begin{equation*}{\small
		\begin{aligned}
			&\iiint_{\R^3 \times \R^3 \times \R_+} f^n_{\beta _*} (\log f^n_{\beta _*} )^+ dx d\xi_* dI_* + \iiint_{\R^3 \times \R^3 \times \R_+} f^n_{\beta _*} dxd\xi_* dI_* \\
			\leq & \iiint_{\R^3 \times \R^3 \times \R_+} f^n_{\beta _*} \log (I^{1-\frac{\delta(\beta  )}{2}}f^n_{\beta _*} )^+ dxd\xi_* dI_* + \tfrac{\delta(\beta )}{2} \iiint_{\R^3 \times \R^3 \times \R_+} f^n_{\beta _*}(1+ I_*) dxd\xi_* dI_* \leq C(T).
		\end{aligned}}
	\end{equation*}
Moreover, taking $\xi = 0$, $I =0$, $\epsilon = 1$ and $R = m$ in \eqref{A-bnd}, one has $0 < a^m_{2 \alpha \beta, n} \leq C_m $ for all $n \geq 1$. Note that $x (\log x)^- \leq C_0 \sqrt{x}$ for some constant $C_0 > 0$. Then for all $n \geq 1$,
\begin{equation*}
  \begin{aligned}
    a^m_{2\alpha \beta,n} |\log a^m_{2\alpha \beta,n}| = a^m_{2\alpha \beta,n} ( \log a^m_{2\alpha \beta,n} )^+ + a^m_{2\alpha \beta,n} ( \log a^m_{2\alpha \beta,n} )^- \leq C_m (\log C_m)^+ + C_0 \sqrt{C_m} \,.
  \end{aligned}
\end{equation*}
Consequently, the above estimates conclude that for $t \in [0, T]$,     
	\begin{equation}\label{Psi-Lnab}
		\iiint_{\R^3 \times B_R \times(0,R)} \Psi (L^m_{ \alpha \beta ,n}(f^n)) (t) dxd\xi dI \leq C(T,m) \,, \ \alpha, \beta \in \{ s_0 + 1, \cdots, s \} \,,
	\end{equation}
where the constant $ C(T,m) > 0 $ only depending on $T$ and $m$. Note that $\frac{\Psi (t)}{t} = (\log t)^+ \to + \infty$ as $t \to + \infty$. Then De La Vall\'ee-Poussin Criterion (Lemma \ref{theo-DE La}) shows that $\{ L^m_{ \alpha \beta ,n}(f^n) \}_{n \geq 1} \subseteq L^\infty (0,T; L^1(\R^3 \times B_R\times (0,R) ))$ is equi-integrable. 

We then prove the tightness of the sequence $ \{ L^m_{ \alpha \beta ,n}(f^n) \}_{n \geq 1} $. Indeed, by the uniform bounds \eqref{FB3},
	\begin{equation*}
		\begin{aligned}
			\int_{|x|>M}dx \iint_{B_R \times (0,R)}L^m_{ \alpha \beta ,n}(f^n)d\xi dI 
			\leq & a^m_{2\alpha \beta,n} \int_{|x|>M}dx \iint_{\R^3 \times \R_+} f^n_{\beta_*} d\xi_* dI_* \\
			\leq & \tfrac{C_m}{M^2} \iiint_{\R^3 \times \R^3 \times \R_+}  f^n_{\beta *}|x|^2dx d\xi_* dI_* \rightarrow 0 \quad (M \rightarrow +\infty).
		\end{aligned}
	\end{equation*}
Consequently, by using Dunford-Pettis Theorem (Lemma \ref{theorem-dunford}), we know that for any fixed $m>0$, $\{ L^m_{ \alpha \beta ,n} (f^n) \}_{n \geq 1} $ is relatively weakly compact in $L^\infty (0,T; L^1(\R^3 \times B_R\times (0,R) ))$ for all $\alpha, \beta \in \{ s_0 + 1, \cdots, s \}$. 

Next we show the relatively weak compactness of $\{ L_{ \alpha \beta ,n} (f^n) \}_{n \geq 1} $ in $L^\infty (0,T; L^1(\R^3 \times B_R\times (0,R) ))$ for all $\alpha, \beta \in \{ s_0 + 1, \cdots, s \}$. Note that by the uniform bounds \eqref{FB3},
	\begin{equation*}
		\begin{aligned}
			&\Vert L^m_{ \alpha \beta ,n} (f^n) (t)-L_{ \alpha \beta ,n} (f^n) (t) \Vert_{L^1(\R^3 \times B_R \times (0,R))} \\
			= & \iiint_{\R^3 \times B_R \times (0,R)}dxd\xi dI \iint_{ \R^3 \times \R_+} A_{2\alpha \beta,n}(\xi-\xi_*,I+I_*) f^n_{\beta *}( \mathbf{1}_{|\xi-\xi_*|> m} + \mathbf{1}_{I+I_*>m})d\xi_* dI_* \\
			\leq & \iiint_{ \R^3 \times \R^3 \times \R_+} f^n_{\beta *}( \mathbf{1}_{|\xi_*|> m-R}+ \mathbf{1}_{I_*>m-R}) dxd\xi_* dI_* \iint_{B_R \times (0,R)} A_{2\alpha \beta,n}(\xi-\xi_*,I+I_*) d\xi dI \\
			\leq &\tfrac{C_\epsilon}{m-R} \iiint_{ \R^3 \times \R^3 \times \R_+} f^n_{\beta *} \cdot I_* dxd\xi_* dI_* + \tfrac{C_\epsilon}{(m-R)^2} \iiint_{ \R^3 \times \R^3 \times \R_+} f^n_{\beta *}|\xi_*|^2 dxd\xi_* dI_* \\
			+ &\epsilon \iiint_{ \R^3 \times \R^3 \times \R_+} f^n_{\beta *} (1+|\xi_*|^2 +I_*) dxd\xi_* dI_* \rightarrow 0 \quad (m \rightarrow \infty \,, \epsilon \rightarrow 0^+)
		\end{aligned}
	\end{equation*}
holds for all $t \in [0, T]$. Therefore, the relatively weak compactness of $\{ L_{ \alpha \beta ,n} (f^n) \}_{n \geq 1} $ in $L^\infty (0,T; L^1(\R^3 \times B_R\times (0,R) ))$ is completed. 

Following the previous arguments, one knows that $\{ L_{ \alpha \beta ,n} (f^n) \}_{n \geq 1} $ is relatively weakly compact in $L^\infty (0,T; L^1(\R^3 \times B_R\times (0,R) ))$ for $\alpha \in \{ s_0 + 1, \cdots, s \}$, $ \beta \in \{ 1, \cdots, s_0 \} $. Then for $\alpha \in \{ s_0 + 1, \cdots, s \}$, 
\begin{equation*}
  \begin{aligned}
    L_{\alpha, n} (f^n) = \sum_{\beta = 1}^{s} L_{ \alpha \beta, n } (f^n)
  \end{aligned}
\end{equation*}
is relatively weakly compact in $L^\infty (0,T; L^1(\R^3 \times B_R\times (0,R) ))$. By \eqref{Q-L-bnd}, we know that the sequence $ \{ \frac{1}{1 + f^n_\alpha} \tilde{Q}^-_{\alpha, n} (f^n, f^n) \}_{n \geq 1} $ is relatively weakly compact in $L^\infty (0,T; L^1(\R^3 \times B_R\times (0,R) ))$ for $\alpha \in \{ s_0 + 1, \cdots, s \}$. Similarly, we can also prove that $ \{ \frac{1}{1 + f^n_\alpha} \tilde{Q}_{\alpha, n}^- (f^n, f^n) \}_{n \geq 1} $ is relatively weakly compact in $L^\infty (0,T; L^1(\R^3 \times B_R ))$ for $\alpha \in \{ 1, \cdots, s_0 \}$. The proof of Lemma \ref{theo-4.1} is therefore finished.
\end{proof}

\begin{lemma} \label{theo-4.2}
		Let $T, R>0$. Assume that $f^n$ is the unique non-negative distributional solution of \eqref{approximation} constructed in Lemma \ref{Lemma 4.14}. Then the following statements hold:
		\begin{enumerate}
			\item {For $\alpha \in \{1,\cdots,s_0\}$, 
				$	(1+ f^n_\alpha)^{-1} \tilde{Q}^+ _{\alpha ,n}(f^n,f^n)$ is
				relatively weakly compact in $L^1 ((0,T)\times \R^3 \times B_R))$;}
			\item { For $\alpha \in \{s_0+1,\cdots s\}$, 
				$	(1+ f^n_\alpha)^{-1} \tilde{Q}^+ _{\alpha ,n}(f^n,f^n)$ is
				relatively weakly compact in $L^1 ((0,T)\times \R^3 \times B_R\times (0,R) ))$.}
		\end{enumerate}
	\end{lemma}

\begin{proof}
We first prove the Arkeryd inequality \eqref{eq5.71}, hence, for all $K>1$ and all $\alpha \in \{1,\cdots,s\}$,
\begin{equation*}
\tilde{Q}^\pm_{\alpha ,n}(f^n,f^n) \leq \tfrac{1}{\log K}\tilde{e}^n_{\alpha} +K\tilde{Q}^\mp_{\alpha ,n}	(f^n,f^n),
\end{equation*}
where $\tilde{e}^n_\alpha$ is defined in \eqref{e na}. Denote by $ \tilde{e}^n_\alpha = N_n (f^n)^{-1} \sum_{\beta = 1}^{s} e^n_{\alpha \beta } $ with $N_n (f^n)$ defined in \eqref{Nnf} and
	\begin{equation*}
		\begin{aligned}
			e^n_{\alpha \beta }(t,x,\mathbf{Z}) 
			= \int_{(\R^3\times \R_+)^3} W_{\alpha \beta ,n} (\tfrac{ f_{\alpha} ^{n}{'} f_{\beta *}^{n}{'} }{(I')^{\delta (\alpha)/2-1} (I'_*)^{\delta(\beta)/2-1} }-\tfrac{f^n_\alpha f^n_{\beta *}}{I^{\delta(\alpha)/2-1} I_*^{\delta(\beta)/2-1}}) \\
\times \log \tfrac{f_\alpha^{n}{'} f_{\beta *}^{n}{'} I^{\delta(\alpha)/2-1} I_*^{\delta(\beta)/2-1} }{f^n_\alpha f^n_{\beta *} (I')^{\delta (\alpha)/2-1} (I'_*)^{\delta(\beta)/2-1}} d\xi_* d\xi' d\xi'_* dI_* dI'dI'_* \,. 
		\end{aligned}
	\end{equation*}
We introduce the compared set
	\begin{equation*}
		A^n_K =  \{(\xi ,\xi_* ,I, I_* ,\omega ):
	\tfrac{f_\alpha^{n}{'} f_{\beta_*}^{n}{'}}{(I')^{\delta (\alpha)/2-1} (I'_*)^{\delta(\beta)/2-1}} \geq K \tfrac{f^n_\alpha f^n_{\beta_*}}{I^{\delta(\alpha)/2-1} I_*^{\delta(\beta)/2-1}}\}, \quad B^n_K = (A^n_K )^c. 		
	\end{equation*}
Then one has
{\small
		\begin{align*}
			&\int_{(\R^3\times \R_+)^3} W_{\alpha \beta,n } \tfrac{ f_{\alpha} ^{n}{'} f_{\beta *}^{n}{'} }{(I')^{\delta (\alpha)/2-1} (I'_*)^{\delta(\beta)/2-1} } \mathbf{1}_{A^n_K} d\xi_* d\xi' d\xi'_* dI_* dI'dI'_*\\
			=& \int_{(\R^3\times \R_+)^3} W_{\alpha \beta,n } (\tfrac{ f_{\alpha} ^{n}{'} f_{\beta *}^{n}{'} }{(I')^{\delta (\alpha)/2-1} (I'_*)^{\delta(\beta)/2-1} }-\tfrac{f^n_\alpha f^n_{\beta *}}{I^{\delta(\alpha)/2-1} I_*^{\delta(\beta)/2-1}}) \mathbf{1}_{A^n_K} d\xi_* d\xi' d\xi'_* dI_* dI'dI'_*\\
			+&\int_{(\R^3\times \R_+)^3} W_{\alpha \beta,n }\tfrac{f^n_\alpha f^n_{\beta *}}{I^{\delta(\alpha)/2-1} I_*^{\delta(\beta)/2-1}} \mathbf{1}_{A^n_K} d\xi_* d\xi' d\xi'_* dI_* dI'dI'_* \\
			\leq & \tfrac{1}{\log K} e^n_{\alpha \beta }  + \int_{(\R^3\times \R_+)^3} W_{\alpha \beta,n }\tfrac{f^n_\alpha f^n_{\beta *}}{I^{\delta(\alpha)/2-1} I_*^{\delta(\beta)/2-1}} \mathbf{1}_{A^n_K} d\xi_* d\xi' d\xi'_* dI_* dI'dI'_* .
		\end{align*}}
It then implies that for $K > 1$,
	\begin{equation*}{\small
		\begin{aligned}
			Q^+_{\alpha \beta ,n} (f^n,f^n) 
			=& \int_{(\R^3\times \R_+)^3} W_{\alpha \beta,n } \tfrac{ f_{\alpha} ^{n}{'} f_{\beta *}^{n}{'} }{(I')^{\delta (\alpha)/2-1} (I'_*)^{\delta(\beta)/2-1} } \mathbf{1}_{A^n_K} d\xi_* d\xi' d\xi'_* dI_* dI'dI'_*\\
			+ & \int_{(\R^3\times \R_+)^3} W_{\alpha \beta,n } \tfrac{ f_{\alpha} ^{n}{'} f_{\beta *}^{n}{'} }{(I')^{\delta (\alpha)/2-1} (I'_*)^{\delta(\beta)/2-1} } \mathbf{1}_{B^n_K} d\xi_* d\xi' d\xi'_* dI_* dI'dI'_* \\
			\leq & \tfrac{1}{\log K} e^n_{\alpha \beta }  + \int_{(\R^3\times \R_+)^3} W_{\alpha \beta,n } \tfrac{f^n_\alpha f^n_{\beta *}}{I^{\delta(\alpha)/2-1} I_*^{\delta(\beta)/2-1}} \mathbf{1}_{A^n_K} d\xi_* d\xi' d\xi'_* dI_* dI'dI'_* \\
			+ & K\int_{(\R^3\times \R_+)^3} W_{\alpha \beta,n }\tfrac{f^n_\alpha f^n_{\beta *}}{I^{\delta(\alpha)/2-1} I_*^{\delta(\beta)/2-1}} \mathbf{1}_{B^n_K} d\xi_* d\xi' d\xi'_* dI_* dI'dI'_* \\
			\leq & \tfrac{1}{\log K} e^n_{\alpha \beta } +K Q^-_{\alpha \beta ,n} (f^n,f^n) .
		\end{aligned}}
	\end{equation*}
	Similarly, we can prove that for $\alpha, \beta \in \{ 1,\cdots,s \}$,
	\begin{equation*}
	\begin{aligned}
	Q^-_{\alpha \beta ,n} (f^n,f^n) &\leq \tfrac{1}{\log K}e^n_{\alpha \beta } +K Q^+_{\alpha \beta ,n} (f^n,f^n) \,.
	\end{aligned}
	\end{equation*}
Therefore, the Arkeryd inequality \eqref{eq5.71} follows from the previous arguments.
	
 By Lemma \ref{Lemma 4.17}, one knows that $\tilde{e}^n_{\alpha}$ is bounded in $L^1((0,T)\times \R^3 \times \R^3)$ for $\alpha \in \{1,\cdots,s_0\}$ and is bounded in $L^1((0,T)\times \R^3 \times \R^3\times \R_+)$ for $\alpha \in \{s_0+1,\cdots,s\}$. Then the Arkeryd inequality \eqref{eq5.71} and Lemma \ref{theo-4.1} imply that for $\alpha \in \{s_0+1,\cdots ,s\}$,
	\begin{equation*}
		\begin{aligned}
			&\int_{0}^{T}dt \iiint_{\R^3 \times B_R \times (0,R) } (\mathbf{1}_A +\mathbf{1}_{|x|\geq M}) \tfrac{1}{1+f^n_\alpha } \tilde{Q}^+_{\alpha ,n} (f^n,f^n) dxd\xi dI \\
			&\leq K\int_{0}^{T}dt \iiint_{\R^3 \times B_R \times (0,R) } (\mathbf{1}_A + \mathbf{1}_{|x|\geq M}) \tfrac{1}{1+f^n_\alpha } \tilde{Q}^-_{\alpha ,n} (f^n,f^n) dxd\xi dI +\tfrac{C'_T}{\log K} \rightarrow  0 
		\end{aligned}
	\end{equation*}
as $|A|\rightarrow 0 $, $ M \rightarrow +\infty$ and $K \to + \infty$. Similarly, for $\alpha \in \{1,\cdots ,s_0\}$,
\begin{equation*}
		\begin{aligned}
			&\int_{0}^{T}dt \iint_{\R^3 \times B_R } (\mathbf{1}_A +\mathbf{1}_{|x|\geq M}) \tfrac{1}{1+f^n_\alpha } \tilde{Q}^+_{\alpha ,n} (f^n,f^n) dxd\xi \rightarrow  0 \quad (|A|\rightarrow 0, M\rightarrow +\infty, K \to + \infty ) \,.
		\end{aligned}
	\end{equation*}
Therefore, the Dunford-Pettis Theorem (Lemma \ref{theorem-dunford}) shows that $(1+ f^n_\alpha )^{-1} \tilde{Q}^+_{\alpha ,n}(f^n,f^n)$ is relatively weakly compact in $L^1((0,T)\times \R^3\times B_R)$  for all $\alpha \in \{1,\cdots, s_0\}$, and that for all $\alpha \in \{s_0+1,\cdots, s\}, (1+ f^n_\alpha )^{-1} \tilde{Q}^+_{\alpha ,n}(f^n,f^n)$ is relatively weakly compact in $L^1((0,T)\times \R^3\times B_R\times (0,R))$.
\end{proof}

\section{Convergence for nonlinear collision operator}\label{Sec:Con}

In this section, we mainly prove the convergence for the quadratic form of the collision operator under the weak convergence $f^n \rightharpoonup f$ in $L^1$ as $n \to \infty$. However, only weak compactness of $f^n$ is insufficient to verify the convergence of nonlinear form. As a result, we first employ the famous averaged velocity lemma, and similar the so-called averaged velocity-internal energy lemma to deal with the internal variable $I \in \R_+$. The goal of these lemmas is to enhance the compactness of the weak convergence sequence $f^n$.

\subsection{Averaged velocity(-internal energy) lemma}

In this subsection, we mainly introduce the important {\em averaged velocity lemma} corresponding to the velocity variable $\xi \in \R^3 $ and the {\em averaged velocity-internal energy} lemma associated velocity variable $\xi \in  \R^3$ and internal energy variable $I \in \R_+$. The previous has been established in works \cite{Diperna-Lions,Golse-velocity,Golse-perthame-velocity}, and the later can be established by using the almost same arguments of \cite{Diperna-Lions,Golse-velocity,Golse-perthame-velocity}. We will omit the detailed proofs here for notational simplicity.

\begin{lemma}[Averaged velocity-internal energy lemma]\label{theo-6.2}
Let $g^n (t,x, \xi, I)$ weakly converge to $g (t,x, \xi, I)$ in $ L^1_{loc} ((0,T)\times \R^3 \times \R^3 \times \R_+) $, and $f^n (t,x, \xi, I)$ weakly converge to $f (t, x, \xi, I)$ in $L^1 ( (0,T)\times \R^3 \times \R^3 \times \R_+ )$. Moreover, they satisfy
	\begin{equation*}
		\tfrac{\partial f^n}{\partial t} +\xi \cdot \nabla_x f^n = g^n \quad \text{in}\ \mathscr{D}' ((0,T) \times \R^3 \times \R^3 \times \R_+) .
	\end{equation*}
Set $\{\varphi^n (t, x, \xi, I) \}_{n\geq 1} \subset L^\infty ((0,T)\times \R^3 \times \R^3 \times \R_+)$ converge to $\varphi (t, x, \xi, I)$ $a.e.$ $(0,T)\times \R^3 \times \R^3 \times \R_+$. Then the sequence $\iint_{\R^3\times \R_+} f^n\varphi^n d\xi dI$ strongly converges to $\iint_{\R^3\times \R_+} f \varphi d\xi dI$ in $L^1((0,T)\times \R^3)$.
\end{lemma}

\begin{lemma}[Averaged velocity lemma]\label{theo-6.2.1}
Let $g^n (t,x, \xi)$ weakly converge to $g (t, x, \xi)$ in $L^1_{loc} ( (0,T) \times \R^3 \times \R^3 )$, and $ f^n (t, x, \xi) $ weakly converge to $f (t, x, \xi)$ in $L^1 ( (0,T) \times \R^3 \times \R^3 )$. Moreover, they satisfy
	\begin{equation*}
		\tfrac{\partial f^n}{\partial t} +\xi \cdot \nabla_x f^n = g^n \quad \text{in}\ \mathscr{D}' ((0,T) \times \R^3 \times \R^3 ) .
	\end{equation*}
Set $\{\phi^n (t,x, \xi)\}_{n\geq 1} \subset L^\infty ((0,T)\times \R^3 \times \R^3 )$ converge to $ \phi (t, x, \xi) $ $a. e.$ $(0,T) \times \R^3 \times \R^3$. Then the sequence $\int_{\R^3} f^n\phi^n d\xi $ strongly converges to $ \int_{\R^3 } f \phi d\xi $ in $L^1((0,T)\times \R^3)$.
\end{lemma}

Based on Lemma \ref{theo-6.2} and Lemma \ref{theo-6.2.1} above, one can generalize the previous results. We denote by $(E,\mu )$ the given measure space.

\begin{lemma}\label{theo-6.4}

Let the bounded sequence $\{ \varphi^n (t, x, \xi, I, e)\}_{n\geq 1}$ in $L^\infty ((0,T)\times \R^3 \times \R^3\times \R_+ ; L^1(E))$ satisfy that there exists $\varphi (t, x, \xi, I, e) \in L^\infty ((0,T)\times \R^3 \times \R^3 \times \R_+; L^1(E))$ such that 
\begin{equation*}
\lim \limits_{n\rightarrow \infty} \Vert \varphi^n -\varphi \Vert_{L^1(E)}(t,x,\xi,I)=0 \ a.e. \ (0,T)\times \R^3 \times \R^3\times \R_+.
\end{equation*}
Besides, $f^n (t, x, \xi, I)$ and $g^n (t, x, \xi, I)$ are given in Lemma \ref{theo-6.2}. Then the sequence $\iint_{\R^3\times \R_+} f^n\varphi^n d\xi dI$ strongly converges to $ \iint_{\R^3\times \R_+} f\varphi d\xi dI $ in $L^1((0,T)\times \R^3\times E)$.
\end{lemma}

\begin{lemma}\label{theo-6.4.1}
Let the bounded sequence $\{	\phi^n (t, x, \xi, e)\}_{n\geq 1}$ in $L^\infty ((0,T)\times \R^3 \times \R^3 ; L^1(E))$ satisfy that there exists $\phi \in L^\infty ((0,T)\times \R^3 \times \R^3 ; L^1(E))$ such that 
\begin{equation*}
	\lim \limits_{n\rightarrow \infty} \Vert \phi^n -\phi \Vert_{L^1(E)}(t,x,\xi)=0 \ a.e. \ (0,T)\times \R^3 \times \R^3.
\end{equation*}
Besides, $f^n (t, x, \xi)$ and $g^n (t, x, \xi)$ are given in Lemma \ref{theo-6.2.1}. Then the sequence $\int_{\R^3} f^n\phi^n d\xi $ strongly converges to $ \int_{\R^3 } f \phi d\xi $ in $L^1((0,T)\times \R^3\times E)$.
\end{lemma}

\subsection{Convergence of approximated solution $f^n_\alpha$}\label{Subsec:Cnv-f}

For the approximated solutions $f^n_\alpha > 0$ ($\alpha \in \{ 1, \cdots, s \}$) to \eqref{approximation} with uniform properties in Lemma \ref{Lemma 4.17}, we will study its convergence in this subsection. We first consider its renormalized form $g^n_{\delta \alpha }=\frac{1}{\delta} \log (1+\delta f^n_\alpha) \ (\delta >0)$, which satisfies
\begin{equation}\label{eq6.75}
	\tfrac{\partial }{\partial t} g^n_{\delta \alpha } + \xi \cdot \nabla_x g^n_{\delta \alpha } = (1+\delta f^n_\alpha)^{-1} \tilde{Q}_{\alpha ,n} (f^n,f^n).
\end{equation}

It is easy to see that $0 < g^n_{\delta \alpha} < f^n_\alpha $. Moreover, Lemma \ref{rwc}, Corollary \ref{rwc1} and Lemma \ref{Lemma 4.17} show that as $n \to + \infty$,
\begin{equation*}
  \begin{aligned}
    & f^n_\alpha (t, x, \xi) \to f_\alpha (t, x, \xi) \quad \textrm{weakly in } \ L^1 ( (0, T) \times \R^3 \times \R^3 ) \,, \alpha \in \{ 1, \cdots, s_0 \} \,, \\
    & f^n_\alpha (t, x, \xi, I) \to f_\alpha (t, x, \xi, I) \quad \textrm{weakly in } \ L^1 ( (0, T) \times \R^3 \times \R^3 \times \R_+ ) \,, \alpha \in \{ s_0 + 1, \cdots, s \} \,,
  \end{aligned}
\end{equation*}
in the sense of subsequence. Then there exist $\tilde{g}_{\delta \alpha} \in L^1$ such that as $n \to + \infty$,
\begin{equation*}
  \begin{aligned}
    & g^n_{\delta \alpha} (t, x, \xi) \to \tilde{g}_{\delta \alpha} (t, x, \xi) \quad \textrm{weakly in } \ L^1 ( (0, T) \times \R^3 \times \R^3 ) \,, \alpha \in \{ 1, \cdots, s_0 \} \,, \\
    & g^n_{\delta \alpha} (t, x, \xi, I) \to \tilde{g}_{\delta \alpha} (t, x, \xi, I) \quad \textrm{weakly in } \ L^1 ( (0, T) \times \R^3 \times \R^3 \times \R_+ ) \,, \alpha \in \{ s_0 + 1, \cdots, s \} \,,
  \end{aligned}
\end{equation*}
in the sense of subsequence. Here $\delta > 0$. We remark that $f_\alpha \geq \tilde{g}_{\delta \alpha} \geq 0$ $a. e. (t,x, \mathbf{Z}_\alpha)$, where $\mathbf{Z}_\alpha = \xi$ for $\alpha \in \{ 1, \cdots, s_0 \}$ and $\mathbf{Z}_\alpha = (\xi, I)$ for $\alpha \in \{ s_0 + 1, \cdots, s \}$.

Besides, it infers by Lemma \ref{theo-4.1} and Lemma \ref{theo-4.2} in Section \ref{sec weak com} that $(1+\delta f^n_\alpha )^{-1}\tilde{Q}^{\pm}_{\alpha ,n}(f^n,f^n)$ is weakly relatively compact in $L^1$. Lemma \ref{theo-6.2.1} indicates that for all $\alpha \in \{1,\cdots ,s_0\}$ and $\phi (t, x, \xi) \in L^\infty ((0,T) \times \R^3\times \R^3)$, $ \int_{\R^3} g^n_{\delta \alpha} \phi d\xi $ is relatively compact in $L^1((0,T)\times \R^3)$. Furthermore, Lemma \ref{theo-6.2} implies that for all $\alpha \in \{s_0+1,\cdots ,s\}$ and $\varphi (t, x, \xi, I) \in L^\infty ((0,T)\times \R^3\times \R^3 \times \R_+)$, $ \iint_{\R^3 \times \R_+} g^n_{\delta \alpha} \varphi d\xi dI$ is relatively compact in $L^1((0,T)\times \R^3)$. Hence, in the sense of subsequence as $n \to + \infty$,
\begin{equation}\label{gn converge}
	\begin{aligned}
		& \| \int_{\R^3} g^n_{\delta \alpha} \phi d\xi - \int_{\R^3} \tilde{g}_{\delta \alpha} \phi d\xi \|_{ L^1 ( (0, T) \times \R^3 ) } \to 0 \quad \text{for}\ \alpha \in \{1,\cdots ,s_0\},\\
		& \| \iint_{\R^3\times \R_+} g^n_{\delta \alpha} \varphi d\xi dI - \iint_{\R^3\times \R_+} \tilde{g}_{\delta \alpha} \varphi d\xi dI \|_{ L^1 ( (0, T) \times \R^3 ) } \to 0 \quad \text{for}\ \alpha \in \{s_0+1,\cdots ,s\}.
	\end{aligned}
\end{equation}

Now we study the limit $\delta \to 0^+$.

\begin{lemma}\label{converge f,g}
Let $T >0$ be fixed. 
\begin{enumerate}
\item {For $\alpha \in \{1,\cdots,s_0\}$, one has
	\begin{equation}\label{eq6.77}
	\begin{aligned}
	\lim \limits_{\delta \rightarrow 0^+} \sup \limits_{n\geq 1 } \sup \limits_{t\in (0,T)} \Vert f^n_\alpha -g^n_{\delta \alpha} \Vert_{L^1(\R^3 \times \R^3 )}=0, \ \lim \limits_{\delta \rightarrow 0^+} \Vert f_\alpha - \tilde{g}_{\delta \alpha} \Vert_{L^1( (0, T) \times \R^3 \times \R^3  )}=0;
	\end{aligned}
	\end{equation}}
\item{For $\alpha \in \{s_0+1,\cdots,s\}$, one has
\begin{equation}\label{eq6.76}
		\begin{aligned}
		\lim \limits_{\delta \rightarrow 0^+} \sup \limits_{n\geq 1 } \sup \limits_{t\in (0,T)} \Vert f^n_\alpha -g^n_{\delta \alpha} \Vert_{L^1(\R^3 \times \R^3\times \R_+ )}=0, \ \lim \limits_{\delta \rightarrow 0^+} \Vert f_\alpha - \tilde{g}_{\delta \alpha} \Vert_{L^1( (0, T) \times \R^3 \times \R^3\times \R_+)}=0.
	\end{aligned}
\end{equation}}
\end{enumerate}
\end{lemma}

\begin{proof}
	Let $\beta_\delta (t)=\frac{1}{\delta} \log (1+\delta t) \ (\delta >0)$. It is easy to see that for all $t\in [0,+\infty)$,
	\begin{equation*}
		0 \leq t-\beta_\delta (t) \leq \epsilon_R (\delta) t + t \mathbf{1}_{t\geq R},
	\end{equation*}
	where $\epsilon_R(\delta )\rightarrow 0$ as $\delta \rightarrow 0_+.$ 
Then, for all $\alpha \in \{s_0+1,\cdots,s\}$, we derive that as $\delta \rightarrow 0^+$ and $R \rightarrow +\infty$,
\begin{equation*}
	\begin{aligned}
		&\sup \limits_{n \geq 1} \sup \limits_{t\in (0,T)} \iiint_{\R^3 \times \R^3 \times \R_+ } |f^n_\alpha -g^n_{\delta \alpha} |dxd\xi dI = \sup \limits_{n \geq 1} \sup \limits_{t\in (0,T)} \iiint_{\R^3 \times \R^3 \times \R_+ } |f^n_\alpha -\tfrac{1}{\delta} \log (1+\delta f^n_\alpha)|dxd\xi dI \\
		\leq & \epsilon_R(\delta ) \sup \limits_{n \geq 1} \sup \limits_{t\in (0,T)} \iiint_{\R^3 \times \R^3 \times \R_+ }f^n_\alpha dxd\xi dI + \sup \limits_{n \geq 1} \sup \limits_{t\in (0,T)} \iiint_{\R^3 \times \R^3 \times \R_+ } f^n_\alpha \mathbf{1}_{f^n_\alpha \geq R}dxd\xi dI \to 0 \,.
	\end{aligned}
\end{equation*}
Note that
\begin{equation*}
	\begin{aligned}
		&\int_0^T dt \iiint_{\R^3 \times \R^3 \times \R_+} | \tilde{g}_{\delta \alpha}-f_\alpha |dxd\xi dI = \int_0^T dt \iiint_{\R^3 \times \R^3 \times \R_+} ( \tilde{g}_{\delta \alpha}-f_\alpha ) sgn(\tilde{g}_{\delta \alpha}-f_\alpha) dxd\xi dI \\
		\leq & \int_0^T dt \iiint_{\R^3 \times \R^3 \times \R_+} (\tilde{g}_{\delta \alpha} - g^n_{\delta \alpha} ) sgn(\tilde{g}_{\delta \alpha}-f_\alpha) dxd\xi dI +\int_0^T dt \iiint_{\R^3 \times \R^3 \times \R_+} |f^n_\alpha -g^n_{\delta \alpha} |dxd\xi dI\\
		+& \int_0^T dt \iiint_{\R^3 \times \R^3 \times \R_+} (f^n_\alpha -f_\alpha) sgn(\tilde{g}_{\delta \alpha}-f_\alpha) dxd\xi dI \,.
	\end{aligned}
\end{equation*}
Since $f^n_\alpha, g^n_{\delta \alpha}$ weakly converge to $f_\alpha, \tilde{g}_{\delta \alpha}$ in $L^1 ( (0, T) \times \R^3 \times \R^3 \times \R_+ )$ as $n \to + \infty$, one has
\begin{equation*}
  \begin{aligned}
    & \int_0^T dt \iiint_{\R^3 \times \R^3 \times \R_+} (\tilde{g}_{\delta \alpha} - g^n_{\delta \alpha} ) sgn(\tilde{g}_{\delta \alpha}-f_\alpha) dxd\xi dI \\
    & + \int_0^T dt \iiint_{\R^3 \times \R^3 \times \R_+} (f^n_\alpha -f_\alpha) sgn(\tilde{g}_{\delta \alpha}-f_\alpha) dxd\xi dI \to 0 \ (n \to + \infty) \,.
  \end{aligned}
\end{equation*}
Moreover, as $\delta \to 0^+$,
\begin{equation*}
  \begin{aligned}
    \int_0^T dt \iiint_{\R^3 \times \R^3 \times \R_+} |f^n_\alpha -g^n_{\delta \alpha} |dxd\xi dI \leq T \sup \limits_{n \geq 1} \sup \limits_{t\in (0,T)} \iiint_{\R^3 \times \R^3 \times \R_+ } |f^n_\alpha -g^n_{\delta \alpha} |dxd\xi dI \to 0 \,.
  \end{aligned}
\end{equation*}
Therefore, for all $\alpha \in \{s_0+1,\cdots,s\}$,
\begin{equation*}
	\Vert \tilde{g}_{\delta \alpha}-f_\alpha \Vert_{L^1((0,T)\times \R^3 \times \R^3 \times \R_+ )} \rightarrow 0 \quad (\delta \rightarrow 0^+).
\end{equation*}
Similarly, for $\alpha \in \{1,\cdots,s_0\}$ and as $\delta \to 0^+$,
\begin{equation*}
	\sup \limits_{n} \sup \limits_{t\in (0,T)} \Vert f^n_\alpha -g^n_{\delta \alpha} \Vert_{L^1(\R^3 \times \R^3 )} \to 0 \,, \ \Vert \tilde{g}_{\delta \alpha}-f_\alpha \Vert_{L^1((0,T)\times \R^3 \times \R^3 )} \rightarrow 0 \,.
\end{equation*}
The proof of Lemma \ref{converge f,g} is then finished.
\end{proof}

Next we give the following enhanced averaged velocity (-internal energy) lemma by using the renormalized approximated solution $g^n_{\delta \alpha} = \frac{1}{\delta} \log ( 1 + \delta f^n_\alpha )$ for $\alpha \in \{ 1, \cdots, s \}$.

\begin{lemma}\label{theo-6.5}
	For all $\varphi (t,x, \xi, I) \in L^\infty ((0,T)\times \R^3 \times \R^3 \times \R_+)$, $\alpha \in \{s_0+1,\cdots,s\},$ and $1\leq p<\infty $,
	\begin{equation*}
		\iint_{ \R^3 \times \R_+} f^n_\alpha \varphi d\xi dI \rightarrow \iint_{ \R^3 \times \R_+} f_\alpha \varphi d\xi dI \quad \text{strongly in } \  L^p((0,T);L^1(\R^3)).
	\end{equation*}
\end{lemma}

\begin{proof}
	Observe that
	\begin{equation*}
		\begin{aligned}
			&\int_{0}^{T}dt \int_{\R^3} |\iint_{ \R^3 \times \R_+} (f^n_\alpha -f_\alpha )\varphi d\xi dI|dx \\
			\leq &\Vert \varphi \Vert_{L^\infty} T \sup_{n \geq 1} \sup_{t \in (0, T)} \iiint_{ \R^3\times \R^3 \times \R_+} |f^n_\alpha -g^n_{\delta \alpha } | dxd\xi dI +\int_{0}^{T}dt \int_{R^3} |\iint_{ R^3 \times R_+} (g^n_{\delta \alpha } -\tilde{g}_{\delta \alpha } )\varphi d\xi dI|dx \\
			+ &	\Vert \varphi \Vert_{L^\infty} \int_{0}^{T}dt \iiint_{ \R^3\times \R^3 \times \R_+} |\tilde{g}_{\delta \alpha } -f_\alpha |dx d\xi dI .
		\end{aligned}
	\end{equation*}
It follows from \eqref{gn converge} that as $ n \to + \infty$,
\begin{equation*}
  \begin{aligned}
    \int_{0}^{T}dt \int_{R^3} |\iint_{ R^3 \times R_+} (g^n_{\delta \alpha } -\tilde{g}_{\delta \alpha } )\varphi d\xi dI|dx \to 0 \,.
  \end{aligned}
\end{equation*}
Lemma \ref{converge f,g} shows that as $\delta \to 0^+$,
\begin{equation*}
  \begin{aligned}
    \sup_{n \geq 1} \sup_{t \in (0, T)} \iiint_{ \R^3\times \R^3 \times \R_+} |f^n_\alpha -g^n_{\delta \alpha } | dxd\xi dI + \int_{0}^{T}dt \iiint_{ \R^3\times \R^3 \times \R_+} |\tilde{g}_{\delta \alpha } -f_\alpha |dx d\xi dI \to 0 \,.
  \end{aligned}
\end{equation*}
As a result, $\iint_{ \R^3 \times \R_+} f^n_\alpha \varphi d\xi dI$ strongly converges to $ \iint_{ \R^3 \times \R_+} f_\alpha \varphi d\xi dI $ in $L^1 ( (0, T) \times \R^3 )$ as $n \to + \infty$.

Furthermore, combining the bound \eqref{FB3}-\eqref{FWC3} in Lemma \ref{Lemma 4.17}, one easily sees that
	\begin{equation*}
		\sup_{t \in (0, T)} \int_{\R^3} |\iint_{ \R^3 \times \R_+} (f^n_\alpha -f_\alpha )\varphi d\xi dI|dx \leq \Vert \varphi \Vert_{L^\infty} \sup_{t \in (0, T)} \iiint_{ \R^3\times \R^3 \times \R_+} (f^n_\alpha +f_\alpha)dxd\xi dI \leq C_T\Vert \varphi \Vert_{L^\infty}.
	\end{equation*}
	Then for all $1\leq p<\infty$:
		\begin{align*}
			&\int_{0}^{T}dt (\int_{\R^3} |\iint_{ \R^3 \times \R_+} (f^n_\alpha -f_\alpha )\varphi d\xi dI|dx)^p \\
			\leq & \sup \limits_{t\in (0,T)} (\int_{\R^3} |\iint_{ \R^3 \times \R_+} (f^n_\alpha -f_\alpha )\varphi d\xi dI|dx)^{p-1} \cdot \int_{0}^{T}dt \int_{\R^3} |\iint_{ \R^3 \times \R_+} (f^n_\alpha -f_\alpha )\varphi d\xi dI|dx \\
			\leq &  C^{p-1}_T\Vert \varphi \Vert_{L^\infty}^{p-1} \int_{0}^{T}dt \int_{\R^3} |\iint_{ \R^3 \times \R_+} (f^n_\alpha -f_\alpha )\varphi d\xi dI|dx 
			\rightarrow  0 \quad (n \rightarrow \infty ).
		\end{align*}
	Thus, the proof of Lemma \ref{theo-6.5} is finished.
\end{proof}

Following the similar arguments above, we establish the following result:

\begin{corollary}\label{cor6.6}
	For all $\phi (t,x, \xi) \in L^\infty ((0,T)\times \R^3 \times \R^3 ) $,  $\alpha \in \{1,\cdots,s_0\},$ and $1\leq p<\infty $,
	\begin{equation*}
		\int_{ \R^3 } f^n_\alpha \phi d\xi  \rightarrow \int_{ \R^3 } f_\alpha \phi d\xi  \quad \text{strongly in }\  L^p((0,T);L^1(\R^3)).
	\end{equation*}
\end{corollary}

Next we study the continuity on the time variable $t$ of the limit functions $f_\alpha$ with $\alpha \in \{ 1, \cdots, s \}$. Recall the notations $f^\sharp (t,x, \xi, I) = f (t, x + t \xi, \xi, I)$ or $f^\sharp (t,x, \xi) = f (t, x + t \xi, \xi)$. 

\begin{lemma}\label{cor6.12}
Let $T > 0$ fixed. There is a constant $C_T > 0$ such that following statements hold:
\begin{enumerate}
\item {For $\alpha \in \{1,\cdots,s_0\},\ f_\alpha, f^\sharp_\alpha \in C([0,\infty); L^1 (\R^3 \times \R^3))_+,$
		\begin{equation}
		\sup_{t \in [0, T]} \iint_{\R^3 \times \R^3 } f_\alpha (1+|x|^2 +|\xi|^2 +|\log f_\alpha|)dxd\xi \leq C_T;
		\end{equation}	}
\item{For $\alpha \in \{s_0+1,\cdots,s\},f_\alpha,f^\sharp_\alpha \in C([0,\infty); L^1 (\R^3 \times \R^3\times \R_+))_+,$
	\begin{equation}\label{falpha-bnd}
	\sup_{t \in [0, T]} \iiint_{\R^3 \times \R^3 \times \R_+} f_\alpha (1+|x|^2 +|\xi|^2 +I + |\log (I^{1-\delta(\alpha)/2}f_\alpha)|)dxd\xi dI \leq C_T.
	\end{equation}}
\end{enumerate}
\end{lemma}
\begin{proof}
	Without loss of generality, we only consider the complicated cases $\alpha \in \{s_0+ 1,\cdots,s\}$. For any $\varphi (x, \xi, I) \in L^\infty (\R^3 \times \R^3 \times \R_+)$, one has
	\begin{equation*}
\iiint_{ \R^3 \times \R^3 \times \R_+} (f^{n\sharp}_\alpha (t) -f^\sharp_\alpha (t))\varphi(x,\xi ,I) dxd\xi dI= \iiint_{ \R^3 \times \R^3 \times \R_+} (f^n_\alpha (t)-f_\alpha (t)) \varphi (x-t\xi ,\xi ,I) dxd\xi dI.
	\end{equation*}
Since $f^n_\alpha$ weakly converges to $f_\alpha$ in $L^1 (\R^3 \times \R^3 \times \R_+)$ by Lemma \ref{rwc} and Lemma \ref{Lemma 4.17}, one know that for any $t \in [0,T],\ f^{n\sharp}_\alpha (t)$ is weakly converges to $f^\sharp_\alpha (t)$ in $L^1(\R^3 \times \R^3 \times \R_+)$. By utilizing the weakly lower semicontinuity of norms, we obtain
	\begin{equation*}
		\Vert f^{\sharp }_\alpha (t)- f^{\sharp }_\alpha (\tau) \Vert_{L^1(\R^3 \times \R^3 \times \R_+)} 
		\leq  \liminf \limits_{n\rightarrow \infty} \Vert f^{n\sharp }_\alpha (t)- f^{n\sharp }_\alpha (\tau) \Vert_{L^1(\R^3 \times \R^3 \times \R_+)} .
	\end{equation*}
Note that for any $R, \delta > 0$,
\begin{equation*}
		\begin{aligned}
			&\Vert f^{n\sharp }_\alpha (t)- f^{n\sharp }_\alpha (\tau) \Vert_{L^1(\R^3 \times \R^3 \times \R_+)} \\
		\leq & \Vert f^{n\sharp }_\alpha (t)- g^{n\sharp }_{\delta \alpha } (t) \Vert_{L^1(\R^3 \times \R^3 \times \R_+)} + \Vert g^{n\sharp }_{\delta \alpha } (\tau)- f^{n\sharp }_\alpha (\tau) \Vert_{L^1(\R^3 \times \R^3 \times \R_+)} \\
& + \Vert g^{n\sharp }_{\delta \alpha } (t)- g^{n\sharp }_{\delta \alpha } (\tau) \Vert_{L^1(\R^3 \times B_R \times (0, R) )} + \Vert g^{n\sharp }_{\delta \alpha } (t)- g^{n\sharp }_{\delta \alpha } (\tau) \Vert_{L^1(\R^3 \times ( B_R \times (0, R) )^c )} \\
		\leq & 2 \sup_{n \geq 1} \sup \limits_{t\in (0,T)} \Vert f^{n\sharp }_\alpha (t)- g^{n\sharp }_{\delta \alpha } (t) \Vert_{L^1(\R^3 \times \R^3 \times \R_+)} + \sup_{n \geq 1} \Vert g^{n\sharp }_{\delta \alpha } (t)- g^{n\sharp }_{\delta \alpha } (\tau) \Vert_{L^1(\R^3 \times B_R \times (0, R) )} \\
		& + \Vert g^{n\sharp }_{\delta \alpha } (t)- g^{n\sharp }_{\delta \alpha } (\tau) \Vert_{L^1(\R^3 \times ( B_R \times (0, R) )^c )} .
		\end{aligned}
	\end{equation*}
Due to $0 < g^n_{\delta \alpha} < f^n_\alpha$, we know that for any $t \in [0,T],\ g^n_{\delta \alpha} (t)$ is weakly converges to $\tilde{g}_{\delta \alpha} (t)$ in $L^1(\R^3 \times \R^3 \times \R_+)$. Then there is a $\epsilon_R > 0$ with $\lim_{R \to + \infty} \epsilon_R = 0$ such that for all $0 \leq t, \tau \leq T$,
\begin{equation*}
  \begin{aligned}
    \Vert g^{n\sharp }_{\delta \alpha } (t)- g^{n\sharp }_{\delta \alpha } (\tau) \Vert_{L^1(\R^3 \times ( B_R \times (0, R) )^c )} \leq \epsilon_R \,.
  \end{aligned}
\end{equation*}
Moreover, by \eqref{eq6.75}, we easily obtain
	\begin{equation*}
		\tfrac{\partial }{\partial t} (g^{n\sharp}_{\delta \alpha}) = (\tfrac{1}{1+\delta f^n_\alpha }\tilde{Q}_{\alpha ,n} (f^n,f^n))^{\sharp}.
	\end{equation*}
	Then for all $0\leq \tau,  t\leq T$ and $\alpha \in \{s_0+1,\cdots,s\}$,
	\begin{equation*}
		\Vert g^{n\sharp }_{\delta \alpha } (t)- g^{n\sharp }_{\delta \alpha } (\tau) \Vert_{L^1(\R^3 \times B_R \times (0,R))} \leq | \int_{\tau}^{t}d\sigma \iiint_{ \R^3 \times B_R \times (0,R)} \tfrac{1}{1+\delta f^n_\alpha } |\tilde{Q}_{\alpha ,n} (f^n,f^n)|dxd\xi dI |.
	\end{equation*}
	Since $(1+\delta f^n_\alpha)^{-1}  \tilde{Q}^\pm_{\alpha ,n} (f^n,f^n)$ is relatively weakly compact in $L^1((0,T)\times \R^3\times B_R \times (0,R))$ by Lemma \ref{theo-4.1} and \ref{theo-4.2}, we derive that
	\begin{equation*}
	\sup \limits_{n\geq 1}	\Vert g^{n\sharp }_{\delta \alpha } (t)- g^{n\sharp }_{\delta \alpha } (\tau) \Vert_{L^1(\R^3 \times B_R \times (0,R))} \leq \omega_\delta(|t-\tau|),
	\end{equation*}
where $\omega_\delta(\lambda) \rightarrow 0^+$ as $\lambda \rightarrow 0^+$. Furthermore, Lemma \ref{converge f,g} shows that as $\delta \to 0^+$,
\begin{equation*}
  \begin{aligned}
    \sup_{n \geq 1} \sup \limits_{t\in (0,T)} \Vert f^{n\sharp }_\alpha (t)- g^{n\sharp }_{\delta \alpha } (t) \Vert_{L^1(\R^3 \times \R^3 \times \R_+)} \to 0 \,.
  \end{aligned}
\end{equation*}
Consequently, one has
\begin{equation*}
  \begin{aligned}
    \Vert f^{n\sharp }_\alpha (t)- f^{n\sharp }_\alpha (\tau) \Vert_{L^1(\R^3 \times \R^3 \times \R_+)} \leq 2 \sup_{n \geq 1} \sup \limits_{t\in (0,T)} \Vert f^{n\sharp }_\alpha (t)- g^{n\sharp }_{\delta \alpha } (t) \Vert_{L^1(\R^3 \times \R^3 \times \R_+)} + \epsilon_R + \omega_\delta(|t-\tau|) 
  \end{aligned}
\end{equation*}
for any $R, \delta > 0$ and $0 \leq \tau, t \leq T$. We first take $\tau \to t$, then take $\delta \to 0^+$, and finally take $R \to + \infty$, so that
\begin{equation*}
  \begin{aligned}
    \lim_{\tau \to t} \Vert f^{n\sharp }_\alpha (t)- f^{n\sharp }_\alpha (\tau) \Vert_{L^1(\R^3 \times \R^3 \times \R_+)} = 0 \,,
  \end{aligned}
\end{equation*}
which means $f^{\sharp }_\alpha (t) \in C([0,+\infty); L^1(\R^3 \times \R^3 \times \R_+))$. 

Next we prove that $f_\alpha (t) \in C([0,+\infty); L^1(\R^3 \times \R^3 \times \R_+))$. Notices that
\begin{equation*}{\small
  \begin{aligned}
    & \Vert f_\alpha (t)- f_\alpha (\tau) \Vert_{L^1(\R^3 \times \R^3 \times \R_+)} = \iiint_{ \R^3 \times \R^3 \times \R_+ } | f_\alpha (t, x, \xi, I) - f_\alpha ( \tau, x, \xi, I ) | d x d \xi d I \\
    = & \iiint_{ \R^3 \times \R^3 \times \R_+ } | f_\alpha (t, x + \tau \xi, \xi, I) - f_\alpha ( \tau, x + \tau \xi, \xi, I ) | d x d \xi d I \\
    \leq & \iiint_{ \R^3 \times \R^3 \times \R_+ } | f_\alpha (t, x + \tau \xi, \xi, I) - f_\alpha ( t, x + t \xi, \xi, I ) | d x d \xi d I + \| f_\alpha^\sharp (t) - f_\alpha^\sharp ( \tau ) \|_{L^1(\R^3 \times \R^3 \times \R_+)} \\
    = & \iiint_{ \R^3 \times \R^3 \times \R_+ } | f_\alpha (t, x + ( \tau - t ) \xi, \xi, I) - f_\alpha ( t, x, \xi, I ) | d x d \xi d I + \| f_\alpha^\sharp (t) - f_\alpha^\sharp ( \tau ) \|_{L^1(\R^3 \times \R^3 \times \R_+)} \,.
  \end{aligned}}
\end{equation*}
The fact $f^{\sharp }_\alpha (t) \in C([0,+\infty); L^1(\R^3 \times \R^3 \times \R_+))$ indicates that $\| f_\alpha^\sharp (t) - f_\alpha^\sharp ( \tau ) \|_{L^1(\R^3 \times \R^3 \times \R_+)} \to 0$ as $\tau \to t$. It is further derived from the absolute continuity of integral shows that
	\begin{equation*}
		\lim\limits_{\tau \rightarrow t} \iiint_{ \R^3 \times \R^3 \times \R_+ } | f_\alpha (t, x + ( \tau - t ) \xi, \xi, I) - f_\alpha ( t, x, \xi, I ) | d x d \xi d I = 0.
	\end{equation*}
Consequently, $f_\alpha (t) \in C([0,+\infty); L^1(\R^3 \times \R^3 \times \R_+))$.

At the end, we prove the bound \eqref{falpha-bnd}. For any $R>0$ and $ t \in (0,T) $, one easily obtains
	\begin{equation*}
		\varphi_R \triangleq (1+|x|^2 +|\xi|^2 +I) \mathbf{1}_{|x|\leq R } \mathbf{1}_{|\xi|\leq R} \mathbf{1}_{0\leq I \leq R} \in L^\infty (\R^3 \times \R^3 \times \R_+).
	\end{equation*}
Recall that $f^n_\alpha (t) $ weakly converges to $ f_\alpha (t)$ in $L^1 (\R^3 \times \R^3 \times \R_+)$ for all $t \in [0, T]$. It is easily derived from Lemma \ref{Lemma 4.17} that for some $C_T > 0$ independent of $R$,
	\begin{equation*}
		\begin{aligned}
			\iiint_{ \R^3 \times \R^3 \times \R_+} f_\alpha (t,x,\xi ,I) \varphi_R dxd\xi dI
			=& \lim \limits_{n\rightarrow \infty } \iiint_{ \R^3 \times \R^3 \times \R_+} f^n_\alpha (t,x,\xi ,I) \varphi_R dxd\xi dI \\
			\leq & \sup \limits_{n \geq 1} \iiint_{ \R^3 \times \R^3 \times \R_+} f^n_\alpha (1+|x|^2 +|\xi|^2 +I) dxd\xi dI \leq C_T \,.
		\end{aligned}
	\end{equation*}
By taking $R \to + \infty$, we obtain that for $\alpha \in \{ s_0 + 1, \cdots , s \}$,
	\begin{equation} \label{eq6.83}
		\sup_{t \in [0, T]} \iiint_{ \R^3 \times \R^3 \times \R_+} f_\alpha (1+|x|^2 +|\xi|^2 +I) dxd\xi dI  \leq C_T.
	\end{equation}
By analogous arguments of \eqref{eq6.83}, one knows that for $\alpha \in \{ 1, \cdots , s_0 \}$,
\begin{equation} \label{eq6.83+1}
		\sup_{t \in [0, T]} \iint_{ \R^3 \times \R^3 } f_\alpha (1+|x|^2 +|\xi|^2 ) dxd\xi \leq C_T.
	\end{equation}

Moreover,  Lemma \ref{Lemma 4.6} implies that for $t \in [0, T]$,
{\small
	\begin{equation}\label{f bound}
		H (f) \leq \liminf \limits_{n\rightarrow \infty } H (f^n) \leq C_T'.
	\end{equation}}
	Furthermore, as the similar arguments in \eqref{f-C}, one sees that for $\alpha \in \{ s_0 + 1, \cdots, s \}$,
	\begin{equation}\label{eq6.84}
		\begin{aligned}
			\iiint_{ \R^3 \times \R^3 \times \R_+} f_\alpha |\log (I^{1-\delta(\alpha)/2}f_\alpha )| dxd\xi dI 
			& \leq 2\iiint_{ \R^3 \times \R^3 \times \R_+} f_\alpha (|x|^2+|\xi|^2+I )dxd\xi dI \\
			& + \iiint_{ \R^3 \times \R^3 \times \R_+} f_\alpha \log (I^{1-\delta(\alpha)/2}f_\alpha ) dxd\xi dI +C_1 ,
		\end{aligned}
	\end{equation}
and for $\alpha \in \{ 1, \cdots, s_0 \}$,
\begin{equation}\label{eq6.84+1}
		\begin{aligned}
			\iint_{ \R^3 \times \R^3 } f_\alpha |\log f_\alpha | dxd\xi  \leq 2\iint_{ \R^3 \times \R^3 } f_\alpha (x|^2+|\xi|^2 )dxd\xi + \iint_{ \R^3 \times \R^3 } f_\alpha \log f_\alpha dxd\xi  + C_1 .
		\end{aligned}
	\end{equation}
As a consequence, the bounds \eqref{eq6.83}, \eqref{f bound}, \eqref{eq6.84} and \eqref{eq6.84+1} conclude the bound \eqref{falpha-bnd}. 

Remark that the cases $\alpha \in \{1,\cdots,s_0\}$ can be proved by the similar arguments above. Therefore, the proof of Lemma \ref{cor6.12} is finished.
\end{proof}

\subsection{Convergence for nonlinear collision operator}\label{Subsec:CFNCO}

In this subsection, we aim at justifying the convergence of the nonlinear collision operator $\tilde{Q}_{\alpha, n}$. We first consider the operator $\tilde{L}_{\alpha, n} (f^n)$, which is the main ingredients of the lose term $\tilde{Q}_{\alpha, n}^- (f^n, f^n)$, (see \eqref{Q L}). Actually, based on $L^1$-weak convergence $f^n_\alpha$ to $f^\alpha$, we will study the convergent relations between $\tilde{L}_{\alpha, n} (f^n)$ and $L_\alpha (f)$ defined in \eqref{L-alpha-def}.

\begin{lemma}\label{converge L}
	Let $T,R>0$ fixed. 
	\begin{enumerate}
		\item {For $\alpha \in \{1,\cdots,s_0\}$, one has
			\begin{equation}
				\lim \limits_{n\rightarrow \infty} \tilde{L}_{\alpha,n} (f^n) = L_\alpha(f) \quad \text{strongly in } L^1 ((0,T)\times \R^3\times B_R);
		\end{equation}}
		\item{For $\alpha \in \{s_0+1,\cdots,s\}$, one has
			\begin{equation}
			\lim \limits_{n\rightarrow \infty} \tilde{L}_{\alpha,n} (f^n) = L_\alpha(f) \quad \text{strongly in } L^1 ((0,T)\times \R^3\times B_R\times (0,R)).
		\end{equation}}
	\end{enumerate}
\end{lemma}

\begin{proof}

Let $E=B_R,\ F=B_R\times (0,R)$ and $r > 0$ fixed. Consider the variables $t \in (0, T)$, $x \in \R^3$, $\xi \in B_R$, $I \in (0, R)$, $\xi_* \in \R^3$ and $I_* \in \R_+$. Denote by
\begin{equation*}
	\begin{aligned}
		\varphi^n_1 (t, x, \xi_*, \xi) & = A_{0\alpha \beta ,n}(\xi -\xi_* ) \mathbf{1}_{|\xi_* |\leq r} \in L^\infty ((0,T)\times \R^3 \times \R^3; L^1(E)), \\
		\varphi^n_2 (t, x, \xi_*, I_*,\xi) & = A_{1\alpha \beta ,n}(\xi -\xi_* ,I_* ) \mathbf{1}_{|\xi_* |\leq r} \mathbf{1}_{0 \leq I_* \leq r} \in L^\infty ((0,T)\times \R^3 \times \R^3 \times \R_+; L^1(E)),\\
		\varphi^n_3 (t, x, \xi_*, \xi, I) & = A_{1\beta \alpha ,n}(\xi -\xi_* ,I) \mathbf{1}_{|\xi_* |\leq r} \in L^\infty ((0,T)\times \R^3 \times \R^3; L^1(F)), \\
		\varphi^n_4 (t, x, \xi_*, I_*, \xi, I) & = A_{2\alpha \beta ,n}(\xi -\xi_* ,I+I_* ) \mathbf{1}_{|\xi_* |\leq r} \mathbf{1}_{0 \leq I_* \leq r} \in L^\infty ((0,T)\times \R^3 \times \R^3\times \R_+; L^1(F)) \,.
	\end{aligned}
\end{equation*}
Similarly, let
\begin{equation*}
	\begin{aligned}
		& \varphi_1 (t, x, \xi_*, \xi) = A_{0\alpha \beta }(\xi -\xi_* ) \mathbf{1}_{|\xi_* |\leq r}, \quad \quad \ \ \, \varphi_2 (t, x, \xi_*, I_*, \xi ) = A_{1\alpha \beta }(\xi -\xi_* ,I_* ) \mathbf{1}_{|\xi_* |\leq r} \mathbf{1}_{0 \leq I_* \leq r}, \\
		&\varphi_3 (t, x, \xi_*, \xi, I) = A_{1\beta \alpha}(\xi -\xi_* ,I) \mathbf{1}_{|\xi_* |\leq r}, \quad \varphi_4 (t, x, \xi_*, I_*, \xi, I ) = A_{2\alpha \beta }(\xi -\xi_* ,I+I_* ) \mathbf{1}_{|\xi_* |\leq r} \mathbf{1}_{0 \leq I_* \leq r}.
	\end{aligned}
\end{equation*}
Together with Lemmas \ref{Lemma 4.10}-\ref{Lemma 4.11}-\ref{Lemma 4.12}, one obtains  
\begin{equation*}
	\begin{aligned}
		& \lim \limits_{n\rightarrow \infty } \Vert \varphi^n_1 -\varphi_1 \Vert_{L^1(B_R)}=0, \quad a.e. \ (t,x,\xi_*) \in (0,T) \times \R^3 \times \R^3, \\
		& \lim \limits_{n\rightarrow \infty } \Vert \varphi^n_2 -\varphi_2 \Vert_{L^1(B_R)}=0, \quad a.e. \ (t,x,\xi_* ,I_*) \in (0,T)\times \R^3 \times \R^3 \times \R_+, \\
		& \lim \limits_{n\rightarrow \infty } \Vert \varphi^n_3 -\varphi_3 \Vert_{L^1(B_R\times (0,R))}=0, \quad a.e. \ (t,x,\xi_*) \in (0,T)\times \R^3 \times \R^3, \\
		& \lim \limits_{n\rightarrow \infty } \Vert \varphi^n_4 -\varphi_4 \Vert_{L^1(B_R\times (0,R) )}=0, \quad a.e. \ (t,x,\xi_* ,I_*) \in (0,T)\times \R^3 \times \R^3 \times \R_+ . 
	\end{aligned}
\end{equation*}
Then Lemma \ref{theo-6.4.1} shows that
\begin{equation*}{\small
	\begin{aligned}
		\lim \limits_{n\rightarrow \infty } \int_{|\xi_* |\leq r} g^n_{\delta \beta_* } A_{0\alpha \beta ,n}(\xi -\xi_* ) d\xi_* &=\int_{|\xi_* |\leq r} \tilde{g}_{\delta \beta_*} A_{0\alpha \beta } (\xi -\xi_* )d\xi_*, \\
		\lim \limits_{n\rightarrow \infty } \int_{|\xi_* |\leq r} \int_{0\leq I_* \leq r} g^n_{\delta \beta_*} A_{1\alpha \beta ,n}(\xi -\xi_* ,I_* ) d\xi_* dI_*&=\int_{|\xi_* |\leq r} \int_{0\leq I_* \leq r} \tilde{g}_{\delta \beta_*} A_{1\alpha \beta } (\xi -\xi_* ,I_* )d\xi_* dI_*
	\end{aligned}}
\end{equation*}
strongly in $L^1((0,T)\times \R^3 \times B_R )$. Moreover, Lemma \ref{theo-6.4} indicates that
\begin{equation*}
	\begin{aligned}
		\lim \limits_{n\rightarrow \infty } &\int_{|\xi_* |\leq r} g^n_{\delta \beta_*} A_{1\beta \alpha ,n}(\xi -\xi_* ,I) d\xi_* =\int_{|\xi_* |\leq r} \tilde{g}_{\delta \beta_*} A_{1\beta \alpha }(\xi -\xi_* ,I) d\xi_*, \\
		\lim \limits_{n\rightarrow \infty } &\int_{|\xi_* |\leq r} \int_{0\leq I_* \leq r} g^n_{\delta \beta_*} A_{2\alpha \beta ,n}(\xi -\xi_* ,I+I_* ) d\xi_* dI_*\\
		=&\int_{|\xi_* |\leq r} \int_{0\leq I_* \leq r} \tilde{g}_{\delta \beta_*} A_{2\alpha \beta }(\xi -\xi_* ,I+I_* ) d\xi_* dI_*
	\end{aligned}
\end{equation*}
strongly in $L^1((0,T)\times \R^3 \times B_R \times (0,R))$.

Recall that $0 < g^n_{\delta \alpha} < f^n_\alpha$ for $\alpha \in \{ 1, \cdots, s \}$. For the cases $\alpha, \beta \in \{s_0+1,\cdots ,s\}$, following the similar arguments in \eqref{Trans-fab}, one knows that for all $\epsilon>0$ there exists $C_\epsilon >0$ such that
\begin{equation*}
	\begin{aligned}
		&\Vert \int_{|\xi_* |\leq r} \int_{0\leq I_* \leq r} (f^n_{ \beta_*} -g^n_{\delta \beta_*}) A_{2\alpha \beta ,n}(\xi -\xi_* ,I+I_* ) d\xi_* dI_* \Vert_{L^1((0,T)\times \R^3 \times B_R \times (0,R))} \\
		\leq & \epsilon \int_{0}^{T} dt \iiint_{\R^3 \times \R^3 \times \R_+} (f^n_{\beta *} -g^n_{\delta \beta * } ) (1+|\xi_*|^2+I_*) dx d\xi_* dI_* \\
		+&C_\epsilon  \int_{0}^{T} dt \iiint_{\R^3 \times \R^3 \times \R_+}| f^n_{\beta *} -g^n_{\delta \beta * } |dx d\xi_* dI_* .
	\end{aligned}
\end{equation*}
By Lemma \ref{converge f,g}, one knows that as $\delta \to 0^+$,
\begin{equation*}{\small
  \begin{aligned}
    C_\epsilon  \int_{0}^{T} dt \iiint_{\R^3 \times \R^3 \times \R_+}| f^n_{\beta *} -g^n_{\delta \beta * } |dx d\xi_* dI_* \leq C_\epsilon T \sup_{n \geq 1} \sup_{ t \in (0, T) } \iiint_{\R^3 \times \R^3 \times \R_+}| f^n_{\beta *} -g^n_{\delta \beta * } |dx d\xi_* dI_* \to 0 \,.
  \end{aligned}}
\end{equation*}
Moreover, by Lemma \ref{Lemma 4.17},
\begin{equation*}
  \begin{aligned}
    & \epsilon \int_{0}^{T} dt \iiint_{\R^3 \times \R^3 \times \R_+} (f^n_{\beta *} -g^n_{\delta \beta * } ) (1+|\xi_*|^2+I_*) dx d\xi_* dI_* \\
    \leq & 2 \epsilon T \sup_{ t \in [0, T] } \iiint_{\R^3 \times \R^3 \times \R_+} f^n_{\beta *} (1+|\xi_*|^2+I_*) dx d\xi_* dI_* \leq C_T \epsilon \to 0 \ (\epsilon \to 0) \,.
  \end{aligned}
\end{equation*}
Consequently, as $\delta \to 0^+$,
\begin{equation*}
  \begin{aligned}
    \Vert \int_{|\xi_* |\leq r} \int_{0\leq I_* \leq r} (f^n_{ \beta_*} -g^n_{\delta \beta_*}) A_{2\alpha \beta ,n}(\xi -\xi_* ,I+I_* ) d\xi_* dI_* \Vert_{L^1((0,T)\times \R^3 \times B_R \times (0,R))} \to 0 \,.
  \end{aligned}
\end{equation*}
Similarly, as $\delta \to 0^+$,
\begin{equation*}
	\Vert \int_{|\xi_* |\leq r} \int_{0\leq I_* \leq r} (f_{ \beta_*} -\tilde{g}_{\delta \beta_*}) A_{2\alpha \beta ,n}(\xi -\xi_* ,I+I_* ) d\xi_* dI_* \Vert_{L^1((0,T)\times \R^3 \times B_R \times (0,R))} \rightarrow  0.
\end{equation*}
Therefore, for the cases $ \alpha, \beta \in \{ s_0 + 1, \cdots, s \} $, as $n \rightarrow + \infty $ and $ \delta \rightarrow 0^+$,
	\begin{align}\label{K1}
		\no &\Vert \int_{|\xi_* |\leq r} \int_{0\leq I_* \leq r} (f^n_{ \beta_*}-f_{\beta_*}) A_{2\alpha \beta ,n}(\xi -\xi_* ,I+I_* ) d\xi_* dI_*  \Vert_{L^1((0,T)\times \R^3 \times B_R \times (0,R))} \\
		\no \leq & \Vert \int_{|\xi_* |\leq r} \int_{0\leq I_* \leq r} (g^n_{\delta \beta_*}-\tilde{g}_{\delta \beta_*}) A_{2\alpha \beta ,n} (\xi -\xi_* ,I+I_* )d\xi_* dI_* \Vert_{L^1((0,T)\times \R^3 \times B_R \times (0,R))} \\
		\no + & \sup_{n \geq 1} \Vert \int_{|\xi_* |\leq r} \int_{0\leq I_* \leq r} (f^n_{ \beta_*} -g^n_{\delta \beta_*}) A_{2\alpha \beta ,n} (\xi -\xi_* ,I+I_* )d\xi_* dI_* \Vert_{L^1((0,T)\times \R^3 \times B_R \times (0,R))} \\
		+ &\Vert \int_{|\xi_* |\leq r} \int_{0\leq I_* \leq r} (f_{\beta_*} -\tilde{g}_{\delta \beta_*}) A_{2\alpha \beta,n }(\xi -\xi_* ,I+I_* ) d\xi_* dI_* \Vert_{L^1((0,T)\times \R^3 \times B_R \times (0,R))} \rightarrow 0.
	\end{align}

On the other hand, following the similar arguments in \eqref{Trans-fab} and combining with the uniform bounds in Lemma \ref{Lemma 4.17}, one knows that for any $\epsilon > 0$ there is a $C_\epsilon > 0$ such that
\begin{equation}\label{K2}
	\begin{aligned}
		&\Vert \iint_{ \R^3 \times \R_+} f^n_{ \beta *} A_{2\alpha \beta ,n}(\xi -\xi_* ,I+I_* ) ( \mathbf{1}_{|\xi_*|>r}+ \mathbf{1}_{I_*>r})d\xi_* dI_* \Vert_{L^1((0,T)\times \R^3 \times B_R \times (0,R))} \\
		\leq & \tfrac{C_\epsilon}{r^2} \int_0^T dt \iiint_{ \R^3 \times \R^3 \times \R_+} f^n_{\beta *} |\xi_*|^2 dxd\xi_* dI_*  + \tfrac{C_\epsilon}{r} \int_0^T dt \iiint_{ \R^3 \times \R^3 \times \R_+} f^n_{\beta *} \cdot I_* dxd\xi_* dI_* \\
		+ & \epsilon  \int_0^T dt \iiint_{ \R^3 \times \R^3 \times \R_+} f^n_{\beta *} (1+|\xi_*|^2+I_*) d\xi_* dI_* \rightarrow 0 \quad (r \rightarrow \infty, \epsilon \rightarrow 0^+).
	\end{aligned}
\end{equation}
Similarly, for $\alpha, \beta \in \{ s_0 + 1, \cdots, s \}$,
\begin{equation}\label{K33}
	\lim_{r \to + \infty} \Vert \iint_{ \R^3 \times \R_+} f_{ \beta *} (A_{2\alpha \beta }+ A_{2\alpha \beta,n }) ( \mathbf{1}_{|\xi_*|>r}+ \mathbf{1}_{I_*>r})d\xi_* dI_* \Vert_{L^1((0,T)\times \R^3 \times B_R \times (0,R))} = 0.
\end{equation}
Therefore, the convergence \eqref{K1}-\eqref{K2}-\eqref{K33} imply that as $n \to + \infty$ and $r \to + \infty$,
\begin{equation}\label{K4}
	\begin{aligned}
		&\Vert L_{ \alpha \beta ,n} (f^n ) -L_{ \alpha \beta} (f) \Vert_{L^1((0,T)\times \R^3 \times B_R \times (0,R))}\\
		\leq &\Vert \int_{|\xi_* |\leq r} \int_{0\leq I_* \leq r} (f^n_{ \beta_*}-f_{\beta_*}) A_{2\alpha \beta ,n}(\xi -\xi_* ,I+I_* ) d\xi_* dI_*  \Vert_{L^1((0,T)\times \R^3 \times B_R \times (0,R))} \\
		+&\Vert \iint_{ \R^3 \times \R_+} f^n_{ \beta_*} A_{2\alpha \beta ,n} ( \mathbf{1}_{|I_*|>r}+ \mathbf{1}_{I_*>r})d\xi_* dI_* \Vert_{L^1((0,T)\times \R^3 \times B_R \times (0,R))} \\
		+& \Vert \iint_{ \R^3 \times \R_+} f_{ \beta_*} (A_{2\alpha \beta }+ A_{2\alpha \beta,n }) ( \mathbf{1}_{|I_*|>r}+ \mathbf{1}_{I_*>r})d\xi_* dI_* \Vert_{L^1((0,T)\times \R^3 \times B_R \times (0,R))} \rightarrow 0 
	\end{aligned}
\end{equation}
for $\alpha, \beta \in \{ s_0 + 1, \cdots, s \}$.

Similarly, for $ \alpha \in \{ s_0 + 1, \cdots, s \} $, $\beta \in \{ 1, \cdots, s_0 \}$, 
\begin{equation}\label{K5}
  \begin{aligned}
    \lim_{n \to + \infty} L_{ \alpha \beta ,n} (f^n ) = L_{ \alpha \beta} (f) \quad \textrm{strongly in } \ L^1 ( (0, T) \times \R^3 \times B_R \times (0, R) ) \,.
  \end{aligned}
\end{equation}
Then, by \eqref{K4}-\eqref{K5}, we have
\begin{equation}\label{K6}
	\lim\limits_{n\rightarrow \infty } L_{\alpha ,n} (f^n) =L_{\alpha } (f) \quad \text{strongly in }\ L^1((0,T)\times \R^3 \times B_R\times (0,R))
\end{equation}
for all $\alpha \in \{s_0+1,\cdots,s\}$, which means that $ \sup_{n \geq 1} \| L_{\alpha ,n} (f^n) \|_{ L^1((0,T)\times \R^3 \times B_R\times (0,R)) } \leq C $. 

Recalling the definition of $N_n (f^n)$ in \eqref{Nnf}, we know that $N_n (f^n)^{-1}$ is bounded in $L^\infty ( (0,T) \times \R^3)$ and converges to $1$ almost everywhere. Together with \eqref{K6},
\begin{equation}\label{eq6.78}{\small
	\begin{aligned}
		& \Vert \tilde{L}_{\alpha ,n}(f^n) -L_\alpha (f) \Vert_{L^1((0,T)\times \R^3 \times B_R \times (0,R))} \\
		\leq & \Vert {L}_{\alpha ,n}(f^n) -L_\alpha (f) \Vert_{L^1((0,T)\times \R^3 \times B_R \times (0,R))} + \Vert \tilde{L}_{\alpha ,n}(f^n) - {L}_{\alpha ,n}(f^n) \Vert_{L^1((0,T)\times \R^3 \times B_R \times (0,R))}  .
	\end{aligned}}
\end{equation}

Then, by applying Egorov's theorem, for any $\epsilon > 0$, there exists a measurable subset $A \subset (0, T) \times B_R \times B_R \times (0, R)$ such that $ N_n (f^n)^{-1} $ converges uniformly to $1$ on $A$, and $\mu \left(((0, T) \times B_R \times B_R \times (0, R)) \setminus A\right) < \epsilon$. Thus,
\begin{equation}\label{Lan-La}
	\lim_{n\to\infty} \Vert \tilde{L}_{\alpha ,n}(f^n) - {L}_{\alpha ,n}(f^n) \Vert_{L^1((0,T)\times \R^3 \times B_R \times (0,R))}=0.
\end{equation}
Equation \eqref{K6} implies that $L_{\alpha,n}(f^n)$ is weakly compact in $L^1((0, T) \times \mathbb{R}^3 \times B_R \times (0, R))$. Then, as $\epsilon \to 0^+$ and $R \to \infty$,
\begin{equation}\label{K7}
	\|L_{\alpha,n}(f^n)\|_{L^1((0,T)\times B_R \times B_R \times (0,R)\setminus A)} + \|L_{\alpha,n}(f^n)\|_{L^1((0,T)\times (B_R)^c \times B_R \times (0,R))} \to 0.  
\end{equation}
Substituting \eqref{K6}, \eqref{Lan-La}, and \eqref{K7} into \eqref{eq6.78}, we obtain that the strong convergence holds.  
Similarly, for $\alpha \in \{1,\cdots,s_0\}$, we have 
\begin{equation*}
	\lim\limits_{n\rightarrow \infty } \tilde{L}_{\alpha ,n} (f^n) =L_{\alpha } (f) \quad \text{in}\ L^1((0,T)\times \R^3 \times B_R).
\end{equation*}
As a result, the proof of Lemma \ref{converge L} is finished.
\end{proof}

Next we investigate the convergence of $\tilde{Q}_{\alpha, n}$.

\begin{lemma}\label{theo-6.7}
	Let $R>0$ and $\alpha \in \{s_0+ 1,\cdots ,s\}$. Suppose $ \varphi (t, x, \xi, I) \in L^\infty ((0,T)\times \R^3 \times \R^3 \times \R_+)$ with $ supp \varphi \subseteq (0,T) \times B_R \times B_R\times (0,R)$. Then for all $\epsilon>0$ there exists a Borel subset $E\subset (0,T)\times B_R \times B_R\times (0,R)$ such that $|E^c| = | (0,T)\times B_R \times B_R\times (0,R) \setminus E | <\epsilon$ and
	\begin{equation*}
		\lim \limits_{n\rightarrow \infty } \iint_{ \R^3 \times \R_+} \tilde{Q}^\pm_{\alpha ,n} (f^n,f^n) \varphi \mathbf{1}_E d\xi dI =  \iint_{ \R^3 \times \R_+} Q^\pm_{\alpha } (f,f) \varphi \mathbf{1}_E d\xi dI \quad \text{strongly in } \ L^1 ((0,T)\times B_R).
	\end{equation*}
\end{lemma}

\begin{proof}
	By Egorov's theorem, for all $\epsilon >0$, there exists a Borel subset $E_0 \subset (0,T)\times B_R \times B_R\times (0,R) $ such that $|E_0^c| = | (0,T)\times B_R \times B_R\times (0,R) \setminus E_0 | <\epsilon$ and $\tilde{L}_{\alpha ,n} (f^n)$ uniformly converges to $L_{\alpha } (f)$ on $E_0$. Besides, it is derived from Lusin theorem that  there exists a closed subset $B \subset (0,T)\times B_R \times B_R\times (0,R)$ such that $|B^c| = | (0,T)\times B_R \times B_R\times (0,R) \setminus B | <\epsilon$ and $L_\alpha (f)$ is bounded on $B$. Setting $E = E_0 \cap B$ and utilizing uniform convergence, we obtain $\tilde{L}_{\alpha ,n} (f^n) $ is uniformly bounded on $E$. Together with the uniform bounds in Lemma \ref{Lemma 4.17} and Lemma \ref{theo-6.5}, one has
	\begin{equation}\label{Qalpha}
		\begin{aligned}		
			&\Vert \iint_{\R^3 \times \R_+} (\tilde{Q}^-_{\alpha ,n}(f^n,f^n)-Q^-_{\alpha } (f,f))  \varphi \mathbf{1}_E d\xi dI \Vert_{L^1((0,T) \times B_R)}\\
			\leq & \Vert \varphi \Vert_{L^\infty }  \sup \limits_{E} |\tilde{L}_{\alpha ,n} (f^n)-L_\alpha (f)|  \int_{0}^{T}dt \iiint_{ \R^3 \times \R^3 \times R_+} f^n_\alpha dxd\xi dI \\
			+& \Vert \iint_{\R^3 \times \R_+} (f^n_\alpha -f_\alpha) L_\alpha(f) \varphi \mathbf{1}_E d\xi dI\Vert_{L^1((0,T) \times B_R)}
			\rightarrow 0 ,
		\end{aligned}
	\end{equation}
as $n\rightarrow \infty $.	Therefore, the situation of $ \tilde{Q}^-_{\alpha ,n}$ is established.

Denote by $f \wedge g = \min \{ f, g \}$. Let 
\begin{equation*}
  \begin{aligned}
    B_{\alpha \beta ,n ;m}=B_{\alpha \beta ,n}\wedge m \mathbf{1}_{|z|\leq m} \mathbf{1}_{0\leq \eta \leq m} \,, \ B^m_{\alpha \beta ,n}=B_{\alpha \beta ,n}(\mathbf{1}_{B_{\alpha \beta ,n}>m}+\mathbf{1}_{|z|>m}+\mathbf{1}_{\eta >m}) \,.
  \end{aligned}
\end{equation*}
Write 
	\begin{equation*}{\small
		\begin{aligned}
		A_{1\beta \alpha  ,n ;m} (z,\eta )&= \int_{[0,1] \times \mathbb{S}^2 } B_{1\beta \alpha  ,n ;m} \mathfrak{R}^{\delta(\alpha)/2-1} \mathfrak{R}^{1/2} d \mathfrak{R} d \omega, \\
		A^m_{1\beta \alpha  ,n } (z,\eta )& = \int_{[0,1] \times \mathbb{S}^2 } B^m_{1\beta \alpha  ,n} \mathfrak{R}^{\delta(\alpha)/2-1} \mathfrak{R}^{1/2} d \mathfrak{R} d \omega, \\
		A_{2\alpha \beta ,n ;m}(z,\eta )&= \int_{[0,1]^2 \times \mathbb{S}^2 } B_{2\alpha \beta ,n ;m} \mathfrak{r}^{\delta(\alpha)/2-1}(1-\mathfrak{r})^{\delta(\beta )/2-1} (1-\mathfrak{R})^{(\delta(\alpha)+\delta(\beta )/2-1)} \mathfrak{R}^{1/2} d \mathfrak{R} d \mathfrak{R} d \omega, \\
		 A^m_{2\alpha \beta ,n }(z,\eta )&= \int_{[0,1]^2 \times \mathbb{S}^2 } B^m_{2\alpha \beta ,n } \mathfrak{r}^{\delta(\alpha)/2-1}(1- \mathfrak{r})^{\delta(\beta )/2-1}(1-\mathfrak{R})^{(\delta(\alpha)+\delta(\beta )/2-1)}\mathfrak{R}^{1/2}d \mathfrak{R} d \mathfrak{r} d\omega.
		\end{aligned}}
	\end{equation*}
Denote by $\tilde{Q}^\pm _{\alpha ,n ;m}(f,f)$ and $\tilde{Q}^{m, \pm} _{\alpha ,n }(f,f)$ the corresponding normalized collision operators. Similarly, let $B_{\alpha \beta ; m}=B_{\alpha \beta }\wedge m \mathbf{1}_{|z|\leq m} \mathbf{1}_{0\leq \eta \leq m}$, $B^m_{\alpha \beta }=B_{\alpha \beta }( \mathbf{1}_{B_{\alpha \beta }>m}+ \mathbf{1}_{|z|>m}+ \mathbf{1}_{\eta >m})$. Similarly we obtain the forms for $A_{\alpha \beta ;m}, A^m_{\alpha \beta}, \tilde{Q}^\pm _{\alpha ;m}(f,f)$ and $\tilde{Q}^{m, \pm} _{\alpha}(f,f)$. Then by the Arkeryd-type inequality \eqref{eq5.71},
	\begin{equation*}{\small
		\begin{aligned}
			&\int_{0}^{T}dt \iiint_{B_R \times B_R \times(0,R)} |(\tilde{Q}^{+}_{\alpha ,n;m}(f^n,f^n) -\tilde{Q}^+_{\alpha,n} (f^n,f^n) )| | \varphi | \mathbf{1}_E dxd\xi dI \\
			\leq & \int_{0}^{T}dt  \iiint_{B_R \times B_R \times(0,R)} \tilde{Q}^{m, +} _{\alpha ,n }(f^n,f^n) | \varphi | \mathbf{1}_E  dxd\xi dI \\
\leq & \tfrac{C}{\log K} + K \int_{0}^{T}dt  \iiint_{B_R\times B_R\times(0,R) } | \varphi | \mathbf{1}_E f^n_\alpha \tilde{L}^m_{\alpha ,n} (f^n) dxd\xi dI,
		\end{aligned}}
	\end{equation*}
	where 
		\begin{align*}
			\tilde{L}^m_{\alpha ,n} (f^n) &= N_n (f^n)^{-1} \sum \limits_{\beta =1}^{s_0} \int_{\R^3} f^n_{\beta *} A^m_{1\beta \alpha ,n }(\xi -\xi_* ,I)d\xi_*\\
			& + N_n (f^n)^{-1} \sum \limits_{\beta =s_0+1}^{s} \iint_{\R^3\times \R_+} f^n_{\beta *} A^m_{2\alpha \beta ,n }(\xi -\xi_* ,I+I_*)d\xi_* dI_*.
		\end{align*}
Notice that $0\leq A_{\alpha \beta ,n;m}\leq A_{\alpha \beta,n}$ and $\tilde{L}_{\alpha ,n} (f^n)$ is uniformly bounded on $E$. It is then easy to see that $\tilde{L}^m_{\alpha ,n} (f^n) \mathbf{1}_E \leq \tilde{L}_{\alpha ,n} (f^n) \mathbf{1}_E$ and $A^m_{\alpha \beta,n} \downarrow 0$ as $m \uparrow +\infty$. Therefore, by the Lebesgue Control Convergence Theorem, we can prove for any fixed $n \geq 1$,
	\begin{equation}\label{eq6.79}
		\int_{0}^{T} dt \iiint_{B_R \times B_R \times(0,R)} | \varphi | \mathbf{1}_E f^n_\alpha \tilde{L}^m_{\alpha ,n} (f^n) dxd\xi dI \downarrow 0 \ \text{as}\ m \uparrow +\infty.
	\end{equation}
	Similarly as in \eqref{eq6.79}, it infers that
	\begin{equation}\label{eq6.80}
		\int_{0}^{T} dt \iiint_{B_R \times B_R \times(0,R)} | \varphi | \mathbf{1}_E f_\alpha L^m_{\alpha } (f) dxd\xi dI \downarrow 0 \ \text{as}\  m \uparrow +\infty,  
	\end{equation}
	where 
	\begin{equation*}
		L^m_{\alpha } (f) =  \sum \limits_{\beta =1}^{s_0} \int_{\R^3} f_{\beta *} A^m_{1\beta \alpha}(\xi -\xi_* ,I)d\xi_* + \sum \limits_{\beta =s_0+1}^{s} \iint_{\R^3\times \R_+} f_{\beta_*} A^m_{2\alpha \beta }(\xi -\xi_* ,I+I_*)d\xi_* dI_*.
	\end{equation*}
Following the same arguments in \eqref{Qalpha}, we get that for any fixed $m\geq 1$ and $n\rightarrow + \infty$, 
	\begin{equation}\label{eq6.81}
		\int_{0}^{T} dt \iiint_{B_R \times B_R \times(0,R)} \varphi \mathbf{1}_E f^n_\alpha \tilde{L}^m_{\alpha ,n} (f^n) dxd\xi dI \rightarrow \int_{0}^{T} dt \iiint_{B_R \times B_R \times(0,R)} \varphi \mathbf{1}_E f_\alpha L^m_{\alpha } (f) dxd\xi dI.
	\end{equation}
	Combining with the relations \eqref{eq6.79}-\eqref{eq6.81}, it can be concluded that
	\begin{equation*}
		\begin{aligned}
			\lim \limits_{m\rightarrow \infty} \sup \limits_{n} \int_{0}^{T} dt \iiint_{B_R \times B_R \times(0,R)} \varphi \mathbf{1}_E f^n_\alpha \tilde{L}^m_{\alpha ,n} (f^n) dxd\xi dI = 0.
		\end{aligned}
	\end{equation*}
	Therefore, by taking $m \to + \infty$ and $K \to + \infty$,
	\begin{equation}\label{Qnm}
		\begin{aligned}
			\lim \limits_{m\rightarrow + \infty} \sup \limits_{n \geq 1}&\int_{0}^{T}dt \iiint_{B_R \times B_R \times(0,R)} dxd\xi dI |(\tilde{Q}^{ +}_{\alpha ,n;m}(f^n,f^n) -\tilde{Q}^+_{\alpha,n} (f^n,f^n) )| |\varphi| \mathbf{1}_E = 0.
		\end{aligned}
	\end{equation}
Similarly, one can also derives that
\begin{equation}\label{Qnm-1}
  \begin{aligned}
    \lim \limits_{m\rightarrow \infty}&\int_{0}^{T}dt \iiint_{B_R \times B_R \times(0,R)} dxd\xi dI |Q^{+}_{\alpha ;m }(f,f) -Q^+_{\alpha} (f,f) | |\varphi| \mathbf{1}_E = 0.
  \end{aligned}
\end{equation}

Moreover, \eqref{Q} and \eqref{Qalpha} show that for any fixed $m\geq 1$,
	\begin{equation}\label{Q+nm}
		\begin{aligned}
			&\lim \limits_{n\rightarrow \infty}\Vert \iint_{\R^3 \times \R_+} (\tilde{Q}^+_{\alpha ,n ;m}(f^n,f^n)-Q^+_{\alpha;m } (f,f))  \varphi \mathbf{1}_E d\xi dI \Vert_{L^1((0,T) \times B_R)}\\
			=&\lim \limits_{n\rightarrow \infty}\Vert \iint_{\R^3 \times \R_+} (\tilde{Q}^-_{\alpha ,n;m}(f^n,f^n)-Q^-_{\alpha;m } (f,f))  \varphi' \mathbf{1}_E d\xi dI \Vert_{L^1((0,T) \times B_R)}=0,
		\end{aligned}
	\end{equation}
where $\varphi ' = \varphi (t,x,\xi ' ,I')$. Together with \eqref{Qnm}, \eqref{Qnm-1} and \eqref{Q+nm}, one then gains
	\begin{equation*}
		\lim \limits_{n\rightarrow \infty } \iint_{ \R^3 \times \R_+} \tilde{Q}^+_{\alpha ,n} (f^n,f^n) \varphi \mathbf{1}_E d\xi dI =  \iint_{ \R^3 \times \R_+} Q^+_{\alpha } (f,f) \varphi \mathbf{1}_E d\xi dI \quad \text{strongly in } \ L^1 ((0,T)\times B_R).
	\end{equation*}
	Therefore, the proof of Lemma \ref{theo-6.7} is finished.
\end{proof}

By employing the similar arguments in Lemma \ref{theo-6.7} above, we can establish the following corollary.

\begin{corollary}\label{cor6.8}
	Let $R>0$ and $\alpha \in \{1,\cdots ,s_0\}$. Let $\varphi   (t, x, \xi) \in L^\infty ((0,T)\times \R^3 \times \R^3 )$ with $ supp \varphi \subset (0,T)\times B_R \times B_R$. Then for all $\epsilon>0$ there exists a Borel subset $E \subset(0,T)\times B_R \times B_R$ such that $|E^c| = | (0,T)\times B_R \times B_R \setminus E | <\epsilon$ and
	\begin{equation*}
		\lim \limits_{n\rightarrow \infty } \int_{\R^3} \tilde{Q}^\pm_{\alpha ,n} (f^n,f^n) \varphi \mathbf{1}_E d\xi  =  \int_{\R^3} Q^\pm_{\alpha } (f,f) \varphi \mathbf{1}_E d\xi \quad \text{strongly in} \ L^1 ((0,T)\times B_R).
	\end{equation*}
\end{corollary}

\begin{lemma}\label{theo-6.9}
Let $T, R > 0$ and $\alpha \in \{s_0+1,\cdots,s\}$. For given $\delta >0$, let $\beta_\delta (t)=t\wedge \frac{1}{\delta} = \min \{ t, \frac{1}{\delta} \} $ or $\beta_\delta (t)=\frac{t}{1+\delta t}$. Then
	\begin{equation*}
		\tilde{Q}^\pm_{\alpha ,n} (\beta_\delta(f^n),\beta_\delta (f^n)) \to Q^\pm _\alpha (\tilde{\beta }_\delta,\tilde{\beta }_\delta) \quad \text{weakly in }\ L^1((0,T)\times \R^3 \times B_R\times (0,R)).
	\end{equation*}
	where $\tilde{\beta}_\delta (t, x, \xi, I)$ is the weak $L^1$ limit of $\beta_\delta(f^n) (t, x, \xi, I)$.
\end{lemma}

\begin{proof}
	It is not hard to verify that $\beta_\delta (t)$ is bounded, Lipschitz on $[0,+\infty)$, $\beta'(t)\leq \frac{C}{1+t}$ for some $C>0$ and 
	\begin{equation}\label{eq h}
		0\leq t -\beta_\delta (t) \leq \mathbf{1}_{t\geq \frac{1}{\delta}} t .
	\end{equation}
Recall that $ f^n_\alpha $ is relatively weakly compact in $L^1((0,T)\times \R^3 \times \R^3 \times \R_+)$, so is $\beta (f^n_\alpha)$. Moreover,
	\begin{equation*}
		0 \leq \beta'_\delta (f^n_\alpha ) \tilde{Q}_{\alpha ,n}^\pm (f^n,f^n) \leq \tfrac{C_\delta}{1+f^n_\alpha} \tilde{Q}_{\alpha ,n}^\pm (f^n,f^n).
	\end{equation*}
Then Lemma \ref{theo-4.1} shows that $\beta'_\delta (f^n_\alpha )\tilde{Q}^\pm _{\alpha ,n}(f^n,f^n)$ is relatively weakly compact in $L^1((0,T)\times \R^3 \times B_R\times (0,R))$. To figure out that the conclusions of $g^n_{\delta \alpha }$ in Subsection \ref{Subsec:Cnv-f} are all valid for $\beta_\delta (f^n_\alpha)$, such as \eqref{gn converge} and Lemma \ref{converge f,g}. Moreover, the similar processes of Lemma \ref{converge L} indicate that 
	\begin{equation}\label{L-beta}
		\lim\limits_{n\rightarrow \infty } \tilde{L}_{\alpha ,n} (\beta_\delta (f^n)) =L_{\alpha } (\tilde{\beta}_\delta)  \quad \text{strongly in } \ L^1((0,T)\times \R^3 \times B_R\times (0,R)).
	\end{equation}
Furthermore, from the similar proof processes of Lemma \ref{theo-6.7}, we can derive that for all $\varphi (t,x, \xi, I) \in L^\infty ((0,T)\times \R^3 \times B_R\times (0,R))$ there exists a Borel subset $E \subset(0,T)\times B_M \times B_R\times (0,R)$ (where $M>R$) such that $|E^c|<\epsilon$ and
	\begin{equation}\label{Qbeta}
		\lim\limits_{n\rightarrow \infty }\Vert \iint_{ \R^3 \times \R_+} (\tilde{Q}^\pm_{\alpha ,n} (\beta_\delta(f^n),\beta_\delta (f^n))-Q^\pm _\alpha (\tilde{\beta }_\delta,\tilde{\beta }_\delta) )\varphi \mathbf{1}_E d\xi dI  \Vert_{L^1((0,T)\times B_R)} = 0.
	\end{equation}
It is also easy to see that by the relatively compactness of $\tilde{L}_{\alpha ,n} (\beta_\delta (f^n))$ in \eqref{L-beta} and the relation \eqref{Q},
	\begin{equation}\label{Qbeta+}{\small
		\begin{aligned}
			&\int_{0}^{T} dt \iiint_{B_M \times B_R \times(0,R)} \tilde{Q}^+_{\alpha ,n} (\beta_\delta(f^n),\beta_\delta (f^n)) \varphi \mathbf{1}_{E^c} dxd\xi dI \\
			=& \int_{0}^{T} dt \iiint_{B_M \times B_R \times(0,R)} \beta_\delta (f^n_\alpha) \tilde{L}_{\alpha ,n} (\beta_\delta (f^n)) \varphi ' 1_{E^c} dxd\xi dI \\
			\leq &C_T \Vert \varphi \Vert_{L^\infty} \sup_{n \geq 1} \int_{E^c} \tilde{L}_{\alpha ,n} (\beta_\delta (f^n)) dtdxd\xi dI \rightarrow  0
		\end{aligned}}
	\end{equation}
as $\epsilon \rightarrow 0^+.$ Similarly, we can obtain the following conclusions:
{\small
		\begin{align}\label{Qbeta C}
			\no \lim\limits_{\epsilon \rightarrow 0^+} \sup  \limits_{n}&\int_{0}^{T} dt \iiint_{B_M \times B_R \times(0,R)} \tilde{Q}^-_{\alpha ,n} (\beta_\delta(f^n),\beta_\delta (f^n))  \mathbf{1}_{E^c} dxd\xi dI = 0 ,\\
			\no \lim\limits_{M \rightarrow \infty}\sup  \limits_{n} &\int_{0}^{T} dt \int_{|x|>M} dx \iint_{B_R \times (0,R)} \tilde{Q}^\pm_{\alpha ,n} (\beta_\delta(f^n),\beta_\delta (f^n))d\xi dI = 0,\\
			\no \lim\limits_{\epsilon \rightarrow 0^+}&\int_{0}^{T} dt \iiint_{B_M \times B_R \times(0,R)} Q^\pm _\alpha (\tilde{\beta }_\delta,\tilde{\beta }_\delta) \mathbf{1}_{E^c} dxd\xi dI = 0, \\
			\lim\limits_{M \rightarrow \infty} &\int_{0}^{T} dt \int_{|x|>M} dx \iint_{B_R \times (0,R)}  Q^\pm _\alpha (\tilde{\beta }_\delta,\tilde{\beta }_\delta)d\xi dI = 0.
		\end{align}}
	Combining with the relations \eqref{Qbeta}-\eqref{Qbeta C}, it can be concluded that 
	\begin{equation}
		\lim\limits_{n \rightarrow \infty}	\int_{0}^{T} dt \iiint_{\R^3 \times B_R \times (0,R)} ( \tilde{Q}^\pm_{\alpha ,n} (\beta_\delta(f^n),\beta_\delta (f^n))-Q^\pm _\alpha (\tilde{\beta }_\delta,\tilde{\beta }_\delta))\varphi dxd\xi dI =0.
	\end{equation}
	Therefore, the proof of Lemma \ref{theo-6.9} is finished.
\end{proof}

By applying the analogous arguments in Lemma \ref{theo-6.9}, one gains the following corollary.

\begin{corollary}\label{cor6.9}
	Let $T, R > 0$ and for $\alpha \in \{1,\cdots,s_0\}$. For given $\delta >0$, let $\beta_\delta (t)=t\wedge \frac{1}{\delta}$ or $\beta_\delta (t)=\frac{t}{1+\delta t}$. Then 
	\begin{equation*}
		\tilde{Q}^\pm_{\alpha ,n} (\beta_\delta(f^n),\beta_\delta(f^n)) \to Q^\pm _\alpha (\tilde{\beta }_\delta,\tilde{\beta }_\delta) \quad \text{weakly in} \ L^1((0,T)\times \R^3 \times B_R)
	\end{equation*}
	where $\tilde{\beta}_\delta (t, x, \xi)$ is the weak $L^1$ limit of $\beta_\delta(f^n) (t, x, \xi)$.
\end{corollary}

\begin{lemma}\label{QL L1}
	Let $T, \delta > 0$, $\alpha \in \{s_0+1,\cdots,s\}$ and $\varphi (t, x, \xi, I) \in L^\infty ((0,T)\times \R^3\times \R^3 \times \R_+)$. Then
	\begin{equation*}
		\lim \limits_{n\rightarrow \infty} \iint_{ \R^3 \times \R_+} \tfrac{\tilde{Q}^\pm _{\alpha,n}(f^n,f^n)}{1+\delta \tilde{L}_{\alpha,n}(f^n)}\varphi d\xi dI = \iint_{ \R^3 \times \R_+} \tfrac{Q^\pm _{\alpha}(f,f)}{1+\delta L_{\alpha}(f)}\varphi d\xi dI \quad \text{strongly in } L^1((0,T)\times \R^3).
	\end{equation*}
\end{lemma}
\begin{proof}
	Without loss of generality, we may assume that $\delta=1$. We derive from the Arkeryd inequality \eqref{eq5.71} and $(1+ \tilde{L}_{\alpha,n}(f^n))^{-1}\tilde{Q}^- _{\alpha,n}(f^n,f^n) \leq f^n_\alpha$ that $(1+ \tilde{L}_{\alpha,n}(f^n))^{-1}\tilde{Q}^\pm _{\alpha,n}(f^n,f^n)$ is relatively weakly compact in $L^1((0,T)\times \R^3 \times \R^3 \times \R_+)$. Let $R > 0$ be arbitrary. For any $\epsilon > 0$, we take the set $ E \subseteq (0,T)\times B_R \times B_R\times (0,R) $ given in Lemma \ref{theo-6.7}. Then $ \tilde{L}_{\alpha,n}(f^n) \rightrightarrows L_{\alpha}(f) $ on $E$ and $ (1+ \tilde{L}_{\alpha,n}(f^n))^{-1} \leq 1 $, $(1+ {L}_{\alpha}(f))^{-1} \leq 1$. Moreover, Lemma \ref{theo-6.7} also indicates that
\begin{equation*}
		\lim \limits_{n\rightarrow \infty } \iint_{ \R^3 \times \R_+} \tilde{Q}^\pm_{\alpha ,n} (f^n,f^n) \varphi \mathbf{1}_E d\xi dI =  \iint_{ \R^3 \times \R_+} Q^\pm_{\alpha } (f,f) \varphi \mathbf{1}_E d\xi dI \quad \text{strongly in } \ L^1 ((0,T)\times B_R).
	\end{equation*}
Then we conclude 
\begin{equation}\label{E}{\small
\begin{aligned}
\lim \limits_{n\rightarrow \infty} \iint_{ \R^3 \times \R_+} \tfrac{\tilde{Q}^\pm _{\alpha,n}(f^n,f^n)}{1+ \tilde{L}_{\alpha,n}(f^n)}\varphi \mathbf{1}_E d\xi dI = \iint_{ \R^3 \times \R_+} \tfrac{Q^\pm _{\alpha}(f,f)}{1+L_{\alpha}(f)}\varphi \mathbf{1}_E d\xi dI \quad \text{strongly in } L^1((0,T)\times B_R).
\end{aligned}}
\end{equation}

On the other hand, the relatively weak tightness of $(1+ \tilde{L}_{\alpha,n}(f^n))^{-1}\tilde{Q}^\pm _{\alpha,n}(f^n,f^n)$ in $L^1((0,T)\times \R^3 \times \R^3 \times \R_+)$ guarantees that
	\begin{equation}\label{Ec}
		\begin{aligned}
			\sup \limits_{n\geq 1} \Vert \iint_{ \R^3 \times \R_+} \tfrac{\tilde{Q}^\pm _{\alpha,n}(f^n,f^n)}{1+ \tilde{L}_{\alpha,n}(f^n)}\varphi \mathbf{1}_{E^c} d\xi dI \Vert_{L^1((0,T)\times B_R)} \rightarrow 0 \ (\epsilon \rightarrow 0^+)
		\end{aligned}
	\end{equation}
and 
\begin{equation}\label{Ec-1}{\small
  \begin{aligned}
    \sup_{n \geq 1} \| \iint_{ \R^3 \times \R_+} [ \tfrac{\tilde{Q}^\pm _{\alpha,n}(f^n,f^n)}{1+ \tilde{L}_{\alpha,n}(f^n)} - \tfrac{Q^\pm _{\alpha}(f,f)}{1+L_{\alpha}(f)} ] \varphi d\xi dI \|_{L^1 ( (0, T) \times ( B_R  )^c )} \to 0 \ (R \to + \infty) \,.
  \end{aligned}}
\end{equation}
Then the convergences \eqref{E}, \eqref{Ec} and \eqref{Ec-1} imply that
	\begin{equation}
	\lim \limits_{n\rightarrow \infty} \iint_{ \R^3 \times \R_+} \tfrac{\tilde{Q}^\pm _{\alpha,n}(f^n,f^n)}{1+ \tilde{L}_{\alpha,n}(f^n)}\varphi d\xi dI = \iint_{ \R^3 \times \R_+} \tfrac{Q^\pm _{\alpha}(f,f)}{1+ L_{\alpha}(f)}\varphi d\xi dI \quad \text{strongly in } L^1((0,T) \times \R^3).
\end{equation}
The proof of Lemma \ref{QL L1} is consequently finished.
\end{proof}

\begin{corollary}\label{corQLL1}
	Let $T, \delta > 0$, $\alpha \in \{1,\cdots,s_0\}$ and $\varphi (t, x, \xi) \in L^\infty ((0,T)\times \R^3\times \R^3 )$. Then
	\begin{equation*}
		\lim \limits_{n\rightarrow \infty} \int_{ \R^3 } \tfrac{\tilde{Q}^\pm _{\alpha,n}(f^n,f^n)}{1+\delta \tilde{L}_{\alpha,n}(f^n)}\varphi d\xi  = \int_{ \R^3} \tfrac{Q^\pm _{\alpha}(f,f)}{1+\delta L_{\alpha}(f)}\varphi d\xi \quad \text{strongly in } L^1((0,T)\times \R^3).
	\end{equation*}
\end{corollary}

At the end, we study the integrability of the limit of nonlinear collision operator.

\begin{lemma}\label{Q in L1}
	For any given $T,R>0$ and all $\alpha \in \{s_0+1,\cdots,s\}$, 
	\begin{equation*}
	\begin{aligned}
	\tfrac{Q^-_\alpha (f,f)}{1+f_\alpha} \in L^\infty(\R_+;L^1( \R^3 \times B_R\times (0,R)),\ \tfrac{Q^+_\alpha (f,f)}{1+f_\alpha} \in L^1((0,T)\times \R^3 \times B_R\times (0,R)).
	\end{aligned}
	\end{equation*}
\end{lemma}

\begin{proof}
	Observe that $(1+f_\alpha)^{-1}Q^-_\alpha (f,f)\leq L_\alpha(f)$. By employing the same arguments in \eqref{Trans-fab}, one knows that for any $\epsilon > 0$ there is a $C_\epsilon > 0$ such that
\begin{equation*}{\small
	\begin{aligned}	
		& \iiint_{\R^3 \times B_R \times (0,R)} L_{ \alpha} (f) (t,x,\xi, I) dxd\xi dI \\
			\leq & \epsilon \sum_{\beta = 1}^{s_0} \iint_{\R^3 \times \R^3} f_{\beta} (t, x, \xi_*) (1+|\xi_*|^2) dxd\xi_* + \epsilon \sum_{\beta = s_0 + 1}^{s} \iiint_{\R^3 \times \R^3 \times \R_+} f_{\beta} (t, x, \xi_*, I_*) (1+|\xi_*|^2 +I_*) dxd\xi_* dI_* \\
& + C_\epsilon \sum_{\beta = 1}^{s_0} \iint_{\R^3 \times \R^3} f_{\beta} (t,x, \xi_*) d x d \xi_*  + C_\epsilon \sum_{\beta = s_0 + 1}^{s} \iiint_{\R^3 \times \R^3 \times \R_+} f_{\beta} (t,x, \xi_*, I_*) dxd\xi_* dI_* \leq C_{T, \epsilon}
		\end{aligned}}
	\end{equation*} 
for all $t \geq 0$. Here Lemma \ref{cor6.12} is utilized. By taking $\epsilon = 1$, we then obtain
	\begin{equation*}
		\tfrac{Q^-_\alpha (f,f)}{1+f_\alpha}\in L^\infty(\R_+ ; L^1( \R^3 \times B_R \times (0,R)).
	\end{equation*}

Next we study the gain term $Q_\alpha^+ (f,f)$. For any $\delta>0$, it infers from the Arkeryd-type inequality \eqref{eq5.71} that for all $K > 1$
	\begin{equation}
		\tfrac{\tilde{Q}^\pm _{\alpha,n}(f^n,f^n)}{1+\delta \tilde{L}_{\alpha,n}(f^n)} \leq K \tfrac{\tilde{Q}^\mp _{\alpha,n}(f^n,f^n)}{1+\delta \tilde{L}_{\alpha,n}(f^n)} +\tfrac{\tilde{e}^n_\alpha}{\log K } \,.
	\end{equation}
By \eqref{FWC3} in Lemma \ref{Lemma 4.17}, one sees that $\tilde{e}^n_\alpha$ is a nonnegative bounded sequence in $L^1((0,T)\times \R^3 \times \R^3 \times \R_+)$. Then there exists a nonnegative finite measure $\mu_\alpha $ such that 
	\begin{equation*}
		\lim \limits_{n\rightarrow \infty} \tilde{e}^n_\alpha =\mu_\alpha \quad \text{in } \ \mathscr{D}' ((0,T)\times \R^3 \times \R^3 \times \R_+).
	\end{equation*}
Moreover, Lemma \ref{QL L1} indicates that
	\begin{equation*}
		\lim \limits_{n\rightarrow \infty} \tfrac{\tilde{Q}^\pm _{\alpha,n}(f^n,f^n)}{1+\delta \tilde{L}_{\alpha,n}(f^n)} = \tfrac{Q^\pm _{\alpha}(f,f)}{1+\delta L_{\alpha}(f)} \quad \text{in } \ \mathscr{D}'((0,T)\times \R^3 \times \R^3 \times \R_+) .
	\end{equation*}
Therefore, in the sense of distribution,
	\begin{equation}\label{KK-1}
		\tfrac{Q^\pm _{\alpha}(f,f)}{1+\delta L_{\alpha}(f)} \leq K \tfrac{Q^\mp _{\alpha}(f,f)}{1+\delta L_{\alpha}(f)} +\tfrac{\mu_\alpha}{\log K }.
	\end{equation}
for $K > 1$.

By the Lebesgue Decomposition Theorem, the measure $\mu_\alpha$ can be decomposed as $\mu_\alpha = \mu_\alpha^1 + \mu_\alpha^2$, where $\mu_\alpha^1 \ll d t d x d \xi d I$ (hence, $\mu_\alpha^1$ is absolutely continuous corresponding to the measure $ d t d x d \xi d I $), and $\mu_\alpha^2 \perp d t d x d \xi d I$ (hence, $\mu_\alpha^2$ is singular with respect to the measure $ d t d x d \xi d I $). Furthermore, the Radon-Nikodym Theorem implies that there is a nonnegative function $e_\alpha^* (t,x,\xi,I) \in L^1 ( (0,T) \times \R^3 \times \R^3 \times \R_+ )$ such that $ \mu_\alpha^1 = e_\alpha^* (t,x,\xi,I) d t d x d \xi d I $. $\mu_\alpha^2 \perp d t d x d \xi d I$ means that the concentrated set $A$ with respect to the measure $\mu_\alpha^2$ is actually a zero-measured set corresponding to the natural measure $d t d x d \xi d I$. Then the bound \eqref{KK-1} indicates that for $K > 1$,
\begin{equation*}{\small
  \begin{aligned}
    \tfrac{Q^\pm _{\alpha}(f,f)}{1+\delta L_{\alpha}(f)} (t,x,\xi,I) \leq K \tfrac{Q^\mp _{\alpha}(f,f)}{1+\delta L_{\alpha}(f)} (t,x,\xi,I) + \tfrac{e_\alpha^*(t,x,\xi,I)}{\log K} \quad a.e. \ (t,x,\xi,I) \in (0, T) \times \R^3 \times \R^3 \times \R_+ \,.
  \end{aligned}}
\end{equation*}
By letting $\delta \to 0^+$, one has
	\begin{equation}\label{AK-bnd}
		Q^\pm_\alpha (f,f) \leq K Q^\mp_\alpha (f,f) + ( \log K )^{-1} e_\alpha^* (t,x,\xi,I) ,
	\end{equation}
which means that
	\begin{equation*}
		\tfrac{Q^+_\alpha (f,f)}{1+f_\alpha}(t,x,\xi,I) \leq K \tfrac{Q^-_\alpha (f,f)}{1+f_\alpha}(t,x,\xi,I)+ \tfrac{e_\alpha^* (t,x,\xi,I)}{\log K}.
	\end{equation*}
Due to $\tfrac{Q^-_\alpha (f,f)}{1+f_\alpha}\in L^\infty(\R_+ ; L^1( \R^3 \times B_R \times (0,R))$ and $e_\alpha^* (t,x,\xi,I) \in L^1 ( (0,T) \times \R^3 \times \R^3 \times \R_+ )$, one gains that $(1+f_\alpha)^{-1}Q^+_\alpha (f,f)\in L^1((0,T)\times \R^3 \times B_R\times (0,R))$. The proof of Lemma \ref{Q in L1} is therefore finished.
\end{proof}

By the analogous arguments in Lemma \ref{Q in L1}, one infers the following corollary.

\begin{corollary}\label{Q-L-loc}
	For any given $T,R>0$ and all $\alpha \in \{1,\cdots,s_0\}$, 
	\begin{equation*}
		\begin{aligned}
		\tfrac{Q^-_\alpha (f,f)}{1+f_\alpha} \in L^\infty(\R_+;L^1( \R^3 \times B_R ),\ \tfrac{Q^+_\alpha (f,f)}{1+f_\alpha} \in L^1((0,T)\times \R^3 \times B_R ).
	\end{aligned}
	\end{equation*}
\end{corollary}

\section{Existence of renormalized solution to \eqref{BE-MP}}\label{Sec:Ext}

In this section, the main goal is to prove that the limit function $f = (f_\alpha)_{ \alpha \in \{ 1, \cdots, s \} }$ constructed in previous section is exactly the renormalized solution to \eqref{BE-MP}. By Lemma \ref{distri-mild}, it is equivalent to show that $f$ is a mild solution of \eqref{BE-MP} and $(1 + f_\alpha)^{-1} Q^\pm_\alpha (f, f) \in L^1_{loc}$. Note that Lemma \ref{Q in L1} and Corollary \ref{Q-L-loc} have been shown that $(1 + f_\alpha)^{-1} Q^\pm_\alpha (f, f) \in L^1_{loc}$ for $\alpha \in \{ 1, \cdots, s \}$. We only need to verify that $f$ is a mild solution of \eqref{BE-MP}. By employing Lemma \ref{Lemma3.3} and Lemma \ref{converge L}, it is equivalent to prove that the components $f_\alpha$ of $f$ satisfy the equality \eqref{3.5} and \eqref{3.6}.

\subsection{Supersolutions of equation \eqref{BE-MP}}
In this subsection, based on Lemma \ref{Lemma3.3}, we consider the solution $f^n$ of \eqref{approximation} with positive components $f^n_\alpha$. After extracting a subsequence we will prove that the weak-$L^1$ limit $f=(f_1,\cdots,f_s)$ is a mild supersolution of \eqref{BE-MP} in the sense that $f$ satisfies the exponential form \eqref{3.5}, \eqref{3.6} with the equality replaced by an inequality. Namely,

\begin{definition}[Supersolution]\label{SupperSolution}
  We say $f = (f_\alpha)_{\alpha \in \{ 1, \cdots, s \}}$ is a supersolution if for all $t \geq 0$ there hold:
  \begin{enumerate}
		\item {for $\alpha \in \{1,\cdots,s_0\}$, one has
	\begin{equation*}
		\begin{aligned}
			&\int_{0}^{t}Q^+_{\alpha } (f,f)^\sharp (\tau ,x,\xi ) \exp(-(F^\sharp _{\alpha } (t,x,\xi )-F^\sharp _{\alpha } (\tau,x,\xi ))) d \tau \\
			\leq &f^{\sharp}_\alpha (t,x,\xi) -f_{\alpha ,0} (x,\xi )\exp(-F^\sharp _{\alpha  } (t,x,\xi )) \quad a.e. \ (x,\xi ) \in \R^3 \times \R^3;
		\end{aligned}
	\end{equation*}}
	\item {for $\alpha \in \{s_0+1,\cdots,s\}$, one has
	\begin{equation*}
		\begin{aligned}
			&\int_{0}^{t}Q^+_{\alpha } (f,f)^\sharp (\tau ,x,\xi ,I) \exp(-(F^\sharp _{\alpha } (t,x,\xi ,I)-F^\sharp _{\alpha } (\tau,x,\xi ,I))) d \tau \\
			\leq &f^{\sharp}_\alpha (t,x,\xi ,I) -f_{\alpha ,0} (x,\xi ,I )\exp(-F^\sharp _{\alpha  } (t,x,\xi ,I)) \quad a.e. \ (x,\xi ,I) \in \R^3 \times \R^3 \times \R_+.
		\end{aligned}
	\end{equation*}}
\end{enumerate}
\end{definition}

Recall that $f^n$ is the mild solution of the approximated problem \eqref{approximation}. By the proof of Lemma \ref{theo-4.1}, one knows that $ \tilde{L}_{\alpha, n} (f^n) \in L^1_{loc} $. Then Lemma \ref{Lemma3.3} tells us that for all $0\leq t <\infty $ and for almost everywhere $(x,\xi)\in \R^3 \times \R^3 $ with $\alpha \in \{1,\cdots,s_0\}$, 
\begin{equation}\label{eq up3}
	\begin{aligned}
		&f^{n\sharp}_\alpha (t,x,\xi) -f^n_{\alpha ,0} (x,\xi )\exp(-\tilde{F}^\sharp _{\alpha ,n } (t,x,\xi ))\\
		=&\int_{0}^{t}\tilde{Q}^+_{\alpha ,n} (f^n,f^n)^\sharp (s ,x,\xi ) \exp(-(\tilde{F}^\sharp _{\alpha ,n} (t,x,\xi )-\tilde{F}^\sharp _{\alpha ,n} (s,x,\xi ))) ds,
	\end{aligned}
\end{equation}
and for almost everywhere $(x,\xi)\in \R^3 \times \R^3 $ with $\alpha \in \{s_0+1,\cdots,s\}$, 
\begin{equation}\label{eq up4}
	\begin{aligned}
		&f^{n\sharp}_\alpha (t,x,\xi ,I) -f^n_{\alpha ,0} (x,\xi ,I )\exp(-\tilde{F}^\sharp _{\alpha ,n } (t,x,\xi ,I))\\
		=&\int_{0}^{t}\tilde{Q}^+_{\alpha ,n} (f^n,f^n)^\sharp (s ,x,\xi ,I) \exp(-(\tilde{F}^\sharp _{\alpha ,n} (t,x,\xi ,I )-\tilde{F}^\sharp _{\alpha ,n} (s,x,\xi ,I))) ds,
	\end{aligned}
\end{equation}
where
\begin{equation*}
	\begin{aligned}
		\tilde{F}^\sharp _{\alpha ,n} (t,x,\xi ) &= \int_{0}^{t} \tilde{L}_{\alpha ,n} (f^n_\alpha )^\sharp (s,x,\xi ) ds, \quad \text{for } \alpha \in \{1,\cdots,s_0\}, \\
		\tilde{F}^\sharp _{\alpha ,n} (t,x,\xi ,I) &= \int_{0}^{t} \tilde{L}_{\alpha ,n} (f^n_\alpha )^\sharp (s,x,\xi ,I) ds, \quad \text{for } \alpha \in \{s_0+1,\cdots,s\}.
	\end{aligned}
\end{equation*}
Here we employ the notations $g^\sharp (t, x, \xi) = g (t, x + t \xi, \xi)$ and $g^\sharp (t, x, \xi, I) = g (t, x + t \xi, \xi, I)$.

\begin{lemma}\label{Lemma 7.1}
	Let $T>0$ and $t\in [0,T]$.
		\begin{enumerate}
		\item {For $\alpha \in \{1,\cdots,s_0\}$, one has
	\begin{equation}\label{eq7.88}
		\begin{aligned}
			& \lim \limits_{n\rightarrow \infty} [ f^{n\sharp}_\alpha (t,x,\xi) -f^n_{\alpha ,0} (x,\xi )\exp(-\tilde{F}^\sharp _{\alpha ,n } (t,x,\xi )) ] \\
			&= f^{\sharp}_\alpha (t,x,\xi) -f_{\alpha ,0} (x,\xi )\exp(-F^\sharp _{\alpha } (t,x,\xi )) \quad \text{weakly in } L^1 (\R^3 \times \R^3);
		\end{aligned}
	\end{equation}}
	\item {For $\alpha \in \{s_0+1,\cdots,s\}$, one has
	\begin{equation}\label{eq7.89}
		\begin{aligned}
			& \lim \limits_{n\rightarrow \infty} [ f^{n\sharp}_\alpha (t,x,\xi ,I) -f^n_{\alpha ,0} (x,\xi ,I )\exp(-\tilde{F}^\sharp _{\alpha ,n } (t,x,\xi ,I)) ] \\
			&= f^{\sharp}_\alpha (t,x,\xi ,I) -f_{\alpha ,0} (x,\xi ,I )\exp(- F^\sharp_{\alpha } (t,x,\xi ,I)) \quad \text{weakly in } L^1 (\R^3 \times \R^3 \times \R_+).
		\end{aligned}
	\end{equation}}
\end{enumerate}
Here $F^\sharp_{\alpha }$ is defined in \eqref{F-sharp}.
\end{lemma}

\begin{proof}
	  By Lemma \ref{converge L}, one infers that for $\alpha \in \{1,\cdots,s_0\}$, 
	\begin{equation*}
	\lim \limits_{n \rightarrow \infty}	\max \limits_{t\in [0,T]} \Vert \tilde{F}^\sharp _{\alpha ,n} (t)- F^\sharp _\alpha (t)\Vert_{L^1(\R^3 \times B_R )} \leq \lim \limits_{n \rightarrow \infty} \Vert \tilde{L}_{\alpha ,n}(f^n) -L_\alpha(f) \Vert_{L^1((0,T)\times \R^3 \times B_R)} = 0 ,
	\end{equation*}
and for $\alpha \in \{s_0+1,\cdots,s\}$,
	\begin{equation*}
		\lim \limits_{n \rightarrow \infty}	\max \limits_{t\in [0,T]} \Vert \tilde{F}^\sharp _{\alpha ,n} (t)- F^\sharp_\alpha(t)\Vert_{L^1(\R^3 \times B_R \times (0,R))} \leq \lim \limits_{n \rightarrow \infty} \Vert \tilde{L}_{\alpha ,n}(f^n) -L_\alpha(f) \Vert_{L^1((0,T)\times \R^3 \times B_R\times (0,R))} = 0.
	\end{equation*}
Therefore, $\exp(-\tilde{F}^\sharp _{\alpha ,n } (t,x,\xi )) \in L^\infty ( \R^3 \times \R^3 )$ converges to $\exp(-F^\sharp _{\alpha} (t,x,\xi ))$ $a.e.$ $(x, \xi) \in \R^3 \times \R^3 $ for $\alpha \in \{1,\cdots,s_0\}$ and $\exp(-\tilde{F}^\sharp _{\alpha ,n } (t,x,\xi ,I)) \in L^\infty (\R^3 \times \R^3 \times \R_+)$ converges to $\exp(-F^\sharp _{\alpha} (t,x,\xi ,I))$ $a.e.$ $(x, \xi, I) \in \R^3 \times \R^3 \times \R_+$ for $\alpha \in \{s_0+1,\cdots,s\}$. Moreover, Lemma \ref{Lemma 4.8} and Lemma \ref{Lemma 4.9} show that $f^n_{\alpha, 0}$ strongly converges to $f_{\alpha, 0}$ in $L^1$. Then the Product Limit Theorem (Lemma \ref{product limit}) tells us that $ f^n_{\alpha ,0} (x,\xi ,I )\exp(-\tilde{F}^\sharp _{\alpha ,n } (t,x,\xi ,I)) $ weakly converges to $f_{\alpha ,0} (x,\xi ,I )\exp(- F^\sharp_{\alpha } (t,x,\xi ,I))$ for $\alpha \in \{ s_0 + 1, \cdots, s \}$, and $ f^n_{\alpha ,0} (x,\xi )\exp(-\tilde{F}^\sharp _{\alpha ,n } (t,x,\xi )) $ weakly converges to $f_{\alpha ,0} (x,\xi )\exp(- F^\sharp_{\alpha } (t,x,\xi))$ for $\alpha \in \{  1, \cdots, s_0 \}$. Together with Lemma \ref{rwc}, Corollary \ref{rwc1} and Lemma \ref{Lemma 4.17}, the proof of Lemma \ref{Lemma 7.1} is finished.
\end{proof}

\begin{lemma}\label{Lemma 7.2}
	Let $T,R>0$ and $t\in [0,T]$.
	\begin{enumerate}
		\item {For $\alpha \in \{1,\cdots,s_0\}$, one has
	\begin{equation}\label{eq7.90}
		\begin{aligned}
			& \lim \limits_{n\rightarrow \infty} \int_{0}^{t}\tilde{Q}^+_{\alpha ,n} (h^n_{M} ,h^n_{M})^\sharp (\tau) \exp(-(\tilde{F}^\sharp _{\alpha ,n} (t)-\tilde{F}^\sharp _{\alpha ,n} (\tau))) d \tau \\
			&= \int_{0}^{t}Q^+_{\alpha } (h_{M} ,h_{M})^\sharp (\tau) \exp(-(F^\sharp _{\alpha } (t)-F^\sharp _{\alpha} (\tau))) d \tau \quad \text{weakly in } L^1 (\R^3 \times B_R);
		\end{aligned}
	\end{equation}}
	\item {For $\alpha \in \{s_0+1,\cdots,s\}$, one has
	\begin{equation}\label{eq7.91}
		\begin{aligned}
			& \lim \limits_{n\rightarrow \infty} \int_{0}^{t}\tilde{Q}^+_{\alpha ,n} (h^n_{M} ,h^n_{M})^\sharp (\tau) \exp(-(\tilde{F}^\sharp _{\alpha ,n} (t)-\tilde{F}^\sharp _{\alpha ,n} (\tau))) d \tau \\
			&= \int_{0}^{t}Q^+_{\alpha } (h_{M} ,h_{M})^\sharp (\tau) \exp(-(F^\sharp _{\alpha } (t)-F^\sharp _{\alpha} (\tau))) d \tau \quad \text{weakly in } L^1 (\R^3 \times B_R \times (0,R)),
		\end{aligned}
	\end{equation}}
\end{enumerate}
	where $h^n_{M} = ( h^n_{\alpha M} )_{ \alpha \in \{ 1, \cdots, s \} }$ with $h^n_{\alpha M} = f^n_\alpha \wedge M = \min \{ f^n_\alpha, M \}$ for $M>0$. Moreover, the components $h_{\alpha M}$ of $h_M = ( h_{\alpha M} )_{ \alpha \in \{ 1, \cdots, s \} }$ is the $L^1$ weak limit of $h^n_{\alpha M}$.
\end{lemma}

\begin{proof}
For $\alpha \in \{s_0+1,\cdots,s\}$, Lemma \ref{theo-6.9} implies
	\begin{equation*}
		\tilde{Q}^\pm_{\alpha ,n} (h^n_{M},h^n_{M}) \to Q^\pm _\alpha (h_{M}, h_{M} ) \quad \text{weakly in}\ L^1((0,T)\times \R^3 \times B_R \times (0,R)).
	\end{equation*}
It is further implied by Lemma \ref{converge L} that as $n\rightarrow \infty$,
	\begin{equation*}
		\begin{aligned}
			&\Vert (\tilde{F}^\sharp _{\alpha ,n} (t)-\tilde{F}^\sharp _{\alpha ,n} (\tau)) - (F^\sharp _{\alpha } (t)-F^\sharp _{\alpha} (\tau)) \Vert_{L^1((0,t)\times \R^3 \times B_R \times (0,R))} \\
			= &\Vert \int_{\tau}^{t}(\tilde{L}_{\alpha ,n} (f^n) ^\sharp (\sigma) -L_{\alpha } (f) ^\sharp (\sigma))  d\sigma \Vert_{L^1((0,t)\times \R^3 \times B_R \times (0,R))} \\
			\leq & T \Vert \tilde{L}_{\alpha ,n} (f^n)  (\sigma) -L_{\alpha } (f) (\sigma) \Vert_{L^1((0,T)\times \R^3 \times B_R \times (0,R))} \rightarrow 0 ,
		\end{aligned}
	\end{equation*}
which means that $\exp(-(\tilde{F}^\sharp _{\alpha ,n} (t)-\tilde{F}^\sharp _{\alpha ,n} (\tau))) \in  L^\infty ( (0, t) \times \R^3 \times B_R \times (0, R) )$ converges to $\exp(-(F^\sharp _{\alpha } (t)-F^\sharp _{\alpha} (\tau)))$ $a.e.$ $(0,t)\times \R^3 \times B_R \times (0,R)$. Then the Product Limit Theorem (Lemma \ref{product limit}) concludes the convergence \eqref{eq7.91}. Analogously, the convergence \eqref{eq7.90} for $\alpha \in \{1,\cdots,s_0\}$ also holds. The proof of Lemma \ref{Lemma 7.2} is then finished.
\end{proof}

\begin{remark}
  For all $t \in [0, T]$ and as $M \to + \infty$, $(h_{\alpha M}^n, h_{\alpha M} ) \to ( f^n_\alpha, f_\alpha )$ strongly in $L^1 (\R^3 \times \mathcal{Z}_\alpha)$ for $\alpha \in \{ 1, \cdots, s \}$. Indeed, for $\alpha \in \{s_0+1,\cdots,s\}$, the inequality \eqref{eq h} and the weak convergence of $f^n_\alpha$ infers that
  \begin{equation*}
\lim \limits_{M\rightarrow \infty}	\sup \limits_{n \geq 1} \sup \limits_{t\in [0,T]} \iiint_{ \R^3 \times \R^3 \times \R_+} |h^n_{\alpha M}-f^n_\alpha | dxd\xi dI \leq \lim \limits_{M\rightarrow \infty} \sup \limits_{n \geq 1} \sup \limits_{t\in [0,T]} \iiint_{f^n_\alpha >M } f^n_\alpha dxd\xi dI = 0.
\end{equation*}
Moreover, as $M\rightarrow \infty$,
\begin{equation}\label{hM to f}
	\begin{aligned}
		&\iiint_{ \R^3 \times \R^3 \times \R_+}|h_{\alpha M} -f_\alpha | dxd\xi \\
		\leq & \iiint_{ \R^3 \times \R^3 \times \R_+} (h_{\alpha M}-h^n_{\alpha M} )sgn (h_{\alpha M} -f_\alpha)dxd\xi dI + \iiint_{ \R^3 \times \R^3 \times \R_+} |h^n_{\alpha M}-f^n_\alpha | dxd\xi dI \\
		+  &\iiint_{ \R^3 \times \R^3 \times \R_+} (f^n_\alpha -f_\alpha) sgn (h_{\alpha M} -f_\alpha)dxd\xi dI \rightarrow 0.
	\end{aligned}
\end{equation}  
The cases $\alpha \in \{1,\cdots,s_0\}$ can be derived analogously.
\end{remark}

\begin{lemma}\label{theo-7.3}
	$f$ is a supersolution to \eqref{BE-MP} given in Definition \ref{SupperSolution}.
\end{lemma}

\begin{proof}
	Without loss of generality, we only prove the cases $\alpha \in \{s_0+1,\cdots,s\}$. Note that $h^n_{\alpha M} \leq f^n_\alpha$, which means that $ \tilde{Q}^+_{\alpha ,n} (h^n_{M} ,h^n_{M}) \leq \tilde{Q}^+_{\alpha ,n} (f^n ,f^n) $. Together with \eqref{eq up4}, we have
	\begin{equation*}
		\begin{aligned}
			&\int_{0}^{t}\tilde{Q}^+_{\alpha ,n} (h^n_{M} ,h^n_{M})^\sharp (\tau ,x,\xi ,I) \exp(-(F^\sharp _{\alpha,n } (t,x,\xi ,I)-F^\sharp _{\alpha ,n} (\tau,x,\xi ,I))) d \tau \\
			\leq &f^{n\sharp}_\alpha (t,x,\xi ,I) -f^n_{\alpha ,0} (x,\xi ,I )\exp(-F^\sharp _{\alpha ,n } (t,x,\xi ,I)) \quad a.e. \ (x,\xi ,I) \in \R^3 \times \R^3 \times \R_+ .
		\end{aligned}
	\end{equation*}
By letting $n \to + \infty$ in the above inequality, Lemma \ref{Lemma 7.1} and \ref{Lemma 7.2} give us
	\begin{equation*}
		\begin{aligned}
			&\int_{0}^{t}Q^+_{\alpha } (h_{\alpha M} ,h_{\alpha M})^\sharp (\tau) \exp(-(F^\sharp _{\alpha } (t)-F^\sharp _{\alpha} (\tau))) d \tau \\
			\leq &f^{\sharp}_\alpha (t,x,\xi ,I) -f_{\alpha ,0} (x,\xi ,I )\exp(-F^\sharp _{\alpha } (t,x,\xi ,I)) \quad a.e. \ (x,\xi ,I) \in \R^3 \times \R^3 \times \R_+.
		\end{aligned}
	\end{equation*}
	Since $h_{\alpha M}\uparrow f_\alpha $ as $M \uparrow +\infty$, the conclusion results from the Levi's Monotonic Convergence Theorem. Then the proof of Lemma \ref{theo-7.3} is finished.
\end{proof}

\subsection{Subsolutions of equation \eqref{BE-MP}}

In this subsection, we will prove that $f$ is a subsolution of \eqref{BE-MP} defined as follows.
\begin{definition}[Subsolution]\label{SubSolution}
  We say $f = (f_\alpha)_{\alpha \in \{ 1, \cdots, s \}}$ is a subsolution if for all $t \geq 0$ there hold:
  \begin{enumerate}
		\item {for $\alpha \in \{1,\cdots,s_0\}$, one has
	\begin{equation}\label{Subs-1}
		\begin{aligned}
			&\int_{0}^{t}Q^+_{\alpha } (f,f)^\sharp (\tau ,x,\xi ) \exp(-(F^\sharp _{\alpha } (t,x,\xi )-F^\sharp _{\alpha } (\tau,x,\xi ))) d \tau \\
			\geq &f^{\sharp}_\alpha (t,x,\xi) -f_{\alpha ,0} (x,\xi )\exp(-F^\sharp _{\alpha  } (t,x,\xi )) \quad a.e. \ (x,\xi ) \in \R^3 \times \R^3;
		\end{aligned}
	\end{equation}}
	\item {for $\alpha \in \{s_0+1,\cdots,s\}$, one has
	\begin{equation}\label{Subs-2}
		\begin{aligned}
			&\int_{0}^{t}Q^+_{\alpha } (f,f)^\sharp (\tau ,x,\xi ,I) \exp(-(F^\sharp _{\alpha } (t,x,\xi ,I)-F^\sharp _{\alpha } (\tau,x,\xi ,I))) d \tau \\
			\geq &f^{\sharp}_\alpha (t,x,\xi ,I) -f_{\alpha ,0} (x,\xi ,I )\exp(-F^\sharp _{\alpha  } (t,x,\xi ,I)) \quad a.e. \ (x,\xi ,I) \in \R^3 \times \R^3 \times \R_+.
		\end{aligned}
	\end{equation}}
\end{enumerate}
\end{definition}

Now we rewrite the approximated renormalized equation \eqref{eq6.75} as
\begin{equation*}
	\tfrac{\partial }{\partial t} g^n_{\delta \alpha } + \xi \cdot \nabla_x g^n_{\delta \alpha } = \tfrac{\tilde{Q}^+_{\alpha ,n} (f^n,f^n)}{1+\delta f^n_\alpha} +\{g^n_{\delta \alpha } -\tfrac{f^n_\alpha}{1+\delta f^n_\alpha}\} \tilde{L}_{\alpha ,n} (f^n) -g^n_{\delta \alpha }\tilde{L}_{\alpha ,n} (f^n) .
\end{equation*}
For $\alpha \in \{s_0 + 1,\cdots,s \}$, by multiplying $g^{n\sharp}_{\delta \alpha } \exp\tilde{F} _{\alpha ,n }$ and integrating the resultant equation over characteristics, we obtain
\begin{equation}\label{eq8.93}{\footnotesize
	\begin{aligned}
		& g^{n\sharp}_{\delta \alpha } (t) -g^n_{\delta \alpha,0}  \exp(-\tilde{F}^\sharp _{\alpha ,n } (t)) = \int_{0}^{t} \{\tfrac{\tilde{Q}^+_{\alpha ,n} (f^n,f^n)^\sharp }{1+\delta f^{n\sharp }_\alpha} +j^{n\sharp}_{\delta \alpha }\tilde{L}_{\alpha ,n}(f^n)^\sharp\} (\tau) \exp(-(\tilde{F}^\sharp _{\alpha ,n } (t)- \tilde{F}^\sharp _{\alpha ,n } (\tau))) d \tau
	\end{aligned}}
\end{equation}
for $t \geq 0 $ and $a.e. \ ( x, \xi, I ) \in ( \R^3 \times \R^3 \times \R_+ )$, where $j^{n}_{\delta \alpha } =g^{n}_{\delta \alpha } -\frac{f^n_\alpha}{1+\delta f^n_\alpha}$. Similarly, the above equation also holds for the cases $\alpha \in \{ 1, \cdots, s_0 \}$.

Obviously, $j^{n}_{\delta \alpha } (t)$ is relatively weakly compact in $L^1 (\R^3 \times \R^3 \times \R_+)$ for any fixed $t \geq 1$. Assume that $j^{n}_{\delta \alpha },\ j^{n}_{\delta \alpha }\wedge K$ weakly converge to $j_{\delta \alpha },\  j^K_{\delta \alpha }$, respectively. Since the function $ t \wedge K $ is concave, one knows $j^K_{\delta \alpha } \leq j_{\delta \alpha }\wedge K$. Using the same argument as in the proof of Lemma \ref{converge f,g}, we obtain
\begin{equation}\label{eq8.94}
	\begin{aligned}
		\lim \limits_{K\rightarrow +\infty }\sup \limits_{t\in (0,T)} \iiint_{ \R^3 \times \R^3 \times \R_+} |j_{\delta \alpha }-j^K_{\delta \alpha }|dxd\xi dI= 0 \quad \text{for } \alpha \in \{s_0+1,\cdots,s\}.
	\end{aligned}
\end{equation}
Similarly for $\alpha \in \{1,\cdots,s_0 \}$,
\begin{equation}\label{eq8.95}
	\begin{aligned}
		\lim \limits_{K\rightarrow +\infty }\sup \limits_{t\in (0,T)} \iint_{ \R^3 \times \R^3} |j_{\delta \alpha }-j^K_{\delta \alpha }|dxd\xi = 0.
	\end{aligned}
\end{equation}

By Lemma \ref{cor6.12} and Lemma \ref{theo-7.3}, one derives that for all $t \geq 0$,
\begin{equation*}
	\begin{aligned}
		&\int_{0}^{t}Q^+_{\alpha } (f,f)^\sharp  \exp(-(F^\sharp _{\alpha } (t)-F^\sharp _{\alpha } (\tau))) d \tau \in L^1(\R^3 \times \R^3) \quad \text{for } \alpha \in \{1,\cdots,s_0\},\\
		&\int_{0}^{t}Q^+_{\alpha } (f,f)^\sharp  \exp(-(F^\sharp _{\alpha } (t)-F^\sharp _{\alpha } (\tau))) d \tau \in L^1(\R^3 \times \R^3\times \R_+) \quad \text{for } \alpha \in \{s_0+1,\cdots,s\}.
	\end{aligned}
\end{equation*}
Furthermore, using the inequality \eqref{AK-bnd} and the fact $e_\alpha^* (t, x, \xi, I) \in L^1 ( (0, T) \times \R^3 \times \R^3 \times \R_+ )$, we get
\begin{equation}\label{eq8.97}{\small
	\begin{aligned}
		&\int_{0}^{t}Q^-_{\alpha } (f,f)^\sharp  \exp(-(F^\sharp _{\alpha } (t)-F^\sharp _{\alpha } (\tau ) ) ) d \tau \in L^1(\R^3 \times \R^3) \quad \text{for } \alpha \in \{1,\cdots,s_0\},\\
		&\int_{0}^{t}Q^-_{\alpha } (f,f)^\sharp  \exp(-(F^\sharp _{\alpha } (t)-F^\sharp _{\alpha } (\tau))) d \tau \in L^1(\R^3 \times \R^3\times \R_+) \quad \text{for } \alpha \in \{s_0+1,\cdots,s\}.
	\end{aligned}}
\end{equation}

\begin{lemma}\label{jnw}
	Let $R >0$ and $t \geq 0$. 
\begin{enumerate}
	\item {For $\alpha \in \{1,\cdots,s_0\}$, one has
	\begin{equation}\label{eq8.101-1}
		\begin{aligned}
			& \lim \limits_{n\rightarrow \infty} \int_{0}^{t} j^{n\sharp}_{\delta \alpha } (\tau) \tilde{L}_{\alpha ,n}(f^n)^\sharp (\tau) \exp(-(\tilde{F}^\sharp _{\alpha ,n } (t) - \tilde{F}^\sharp _{\alpha ,n } (\tau ))) d \tau \\
			&=\int_{0}^{t} j^{\sharp}_{\delta \alpha } (\tau) L_\alpha (f)^\sharp (\tau) \exp(-(F^\sharp _{\alpha } (t)- F^\sharp _{\alpha } (\tau))) d \tau \quad \text{weakly in }\ L^1(\R^3 \times B_R);
		\end{aligned}
	\end{equation}}
	\item {For $\alpha \in \{s_0+1,\cdots,s\}$, one has
	\begin{equation}\label{eq8.101}
		\begin{aligned}
			& \lim \limits_{n\rightarrow \infty} \int_{0}^{t} j^{n\sharp}_{\delta \alpha } (\tau) \tilde{L}_{\alpha ,n}(f^n)^\sharp (\tau) \exp(-(\tilde{F}^\sharp _{\alpha ,n } (t)- \tilde{F}^\sharp _{\alpha ,n } (\tau))) d \tau \\
			&=\int_{0}^{t} j^{\sharp}_{\delta \alpha } (\tau) L_\alpha (f)^\sharp (\tau) \exp(-(F^\sharp _{\alpha } (t)- F^\sharp _{\alpha } (\tau))) d \tau \quad \text{weakly in }\ L^1(\R^3 \times B_R\times (0,R)).
		\end{aligned}
\end{equation}}
\end{enumerate}		
\end{lemma}

\begin{proof}
Without loss of generality, we only consider the cases $\alpha \in \{s_0+1,\cdots,s\}$. Note that $ \lim_{x \to + \infty} [ \frac{1}{\delta} \log (1 + \delta x) - \frac{x}{1 + \delta x} ] / x = 0 $, which means that $j^n_{\delta \alpha} \mathbf{1}_{j^n_{\delta \alpha\geq K}} \leq \epsilon_\delta (K) f^n_\alpha$ for some $\epsilon_\delta (K)\rightarrow 0$ as $K\rightarrow +\infty$. By \eqref{eq up3}-\eqref{eq up4}, Arkeryd inequality \eqref{eq5.71}  and Lemma \ref{Lemma 4.17}, we obtain that for all $\varphi (x, \xi, I) \in L^\infty (\R^3 \times \R^3 \times \R_+)$,
\begin{equation*}
  \begin{aligned}
    \sup \limits_{n \geq 1} \int_{0}^{t} d \tau \iiint_{ \R^3 \times \R^3 \times \R_+} \tilde{Q}^-_{\alpha ,n} (f^n,f^n)^\sharp  \exp(-(\tilde{F}^\sharp _{\alpha ,n } (t)- \tilde{F}^\sharp _{\alpha ,n } (\tau))) |\varphi| dxd\xi dI \leq C_T \| \varphi \|_{ L^\infty ( \R^3 \times \R^3 \times \R_+ ) } \,.
  \end{aligned}
\end{equation*}
Then for all $\alpha \in \{s_0+1,\cdots,s\}$ and
\begin{equation}\label{eq8.98}
	\begin{aligned}
		&\sup \limits_{n \geq 1} \iiint_{ \R^3 \times \R^3 \times \R_+} |\varphi| |\int_{0}^{t} (j^{n\sharp}_{\delta \alpha }-j^{n\sharp}_{\delta \alpha }\wedge K)\tilde{L}_{\alpha ,n}(f^n)^\sharp \exp(-(\tilde{F}^\sharp _{\alpha ,n } (t)- \tilde{F}^\sharp _{\alpha ,n } (\tau)))d \tau |dxd\xi dI \\
		\leq & \epsilon_\delta (K)\sup \limits_{n \geq 1} \int_{0}^{t} d \tau \iiint_{ \R^3 \times \R^3 \times \R_+} \tilde{Q}^-_{\alpha ,n} (f^n,f^n)^\sharp  \exp(-(\tilde{F}^\sharp _{\alpha ,n } (t)- \tilde{F}^\sharp _{\alpha ,n } (\tau))) |\varphi| dxd\xi dI \rightarrow 0 
	\end{aligned}
\end{equation}
as $K\rightarrow \infty$. Due to $j^{n\sharp}_{\delta \alpha }\wedge K,j^{K\sharp}_{\delta \alpha }\in L^\infty ( (0, T) \times \R^3 \times B_R \times (0, R) )$ and $\varphi L_{\alpha}(f)\in L^1 ( (0, T) \times \R^3 \times B_R \times (0, R) )$, we obtain  $j^{n\sharp}_{\delta \alpha }\wedge K $ weakly-* converges to $ j^{K\sharp}_{\delta \alpha }$ and
\begin{equation*}
	\begin{aligned}
		\lim \limits_{n \rightarrow \infty} \int_{0}^{t}d \tau \iiint_{ \R^3 \times B_R \times (0,R)}  (j^{n\sharp}_{\delta \alpha }\wedge K -j^{K\sharp}_{\delta \alpha } ) \varphi L_{\alpha}(f) dxd\xi dI = 0 .
	\end{aligned}
\end{equation*}
As a result, together with Lemma \ref{converge L},
\begin{equation*}
	\begin{aligned}
		&|\int_{0}^{t} d \tau \iiint_{ \R^3 \times B_R \times (0,R)}  (j^{n\sharp}_{\delta \alpha }\wedge K \cdot \tilde{L}_{\alpha ,n}(f^n)^\sharp - j^{K\sharp}_{\delta \alpha } L_\alpha(f)^\sharp) \varphi dxd\xi dI| \\
		\leq &| \int_{0}^{t} d \tau \iiint_{ \R^3 \times B_R \times (0,R)} \varphi (j^{n\sharp}_{\delta \alpha }\wedge K -j^{K\sharp}_{\delta \alpha } ) L_{\alpha}(f) dxd\xi dI | \\
& + K \| \varphi \|_{ L^\infty ( \R^3 \times \R^3 \times \R_+ ) } \int_{0}^{t}d \tau \iiint_{ \R^3 \times B_R \times (0,R)} |\tilde{L}_{\alpha ,n}(f^n)^\sharp-L_\alpha(f)^\sharp|dxd\xi dI \rightarrow 0 
	\end{aligned}
\end{equation*}
as $n\rightarrow +\infty$. That means
\begin{equation*}
	\lim \limits_{n\rightarrow \infty} \int_{0}^{t} j^{n\sharp}_{\delta \alpha }\wedge K \cdot \tilde{L}_{\alpha ,n}(f^n)^\sharp d \tau = \int_{0}^{t} j^{K\sharp}_{\delta \alpha } L_\alpha(f)^\sharp d \tau \quad \text{weakly in}\ L^1( \R^3 \times B_R\times (0,R)).
\end{equation*}
Using the same argument as in the proof of Lemma \ref{Lemma 7.1}, we obtain 
\begin{equation}\label{eq8.99}
	\begin{aligned}
		& \lim \limits_{n\rightarrow \infty} \int_{0}^{t} j^{n\sharp}_{\delta \alpha }\wedge K \cdot \tilde{L}_{\alpha ,n}(f^n)^\sharp \exp(-(\tilde{F}^\sharp _{\alpha ,n } (t)- \tilde{F}^\sharp _{\alpha ,n } (\tau))) d \tau \\
		&=\int_{0}^{t}  j^{K\sharp}_{\delta \alpha } L_\alpha(f)^\sharp \exp(-(F^\sharp _{\alpha } (t)- F^\sharp _{\alpha } (\tau)))d \tau \quad \text{weakly in}\ L^1( \R^3 \times B_R\times (0,R)).
	\end{aligned}
\end{equation}
Then by the strong convergence \eqref{eq8.94} and Egorov's theorem, for all $\epsilon>0 $ there exists a Borel set $E\subset [0,t]\times B_{R'} \times B_R \times (0, R)$ such that $|E^c| = | [0,t]\times B_{R'} \times B_R \times (0, R) \setminus E |<\epsilon$ and $j^{K\sharp}_{\delta \alpha }$ converges uniformly to $j^{\sharp}_{\delta \alpha }$ on $E$. Notices that $0\leq j^n_{\delta \alpha} \leq f^n_\alpha$. We know that the limit $0\leq j_{\delta \alpha } \leq f_\alpha$ in the sense of distribution. It further holds for almost all $(t, x, \xi, I) \in (0, T) \times \R^3 \times B_R \times (0, R)$. Recall that $j_{\delta \alpha }^K \leq j_{\delta \alpha } \wedge K \leq j_{\delta \alpha } \leq f_\alpha$. Therefore, for any given $R, R'>0$ and $t \geq 0$,
\begin{equation}\label{eq8.100}
	\begin{aligned}
		&|\int_{0}^{t} d \tau \iiint_{ B_{R'} \times B_R \times (0,R)} (j^{K\sharp}_{\delta \alpha }-j^{\sharp}_{\delta \alpha }) L_\alpha (f)^\sharp \exp(-(F^\sharp _{\alpha } (t)- F^\sharp _{\alpha } (\tau))) dxd\xi dI |\\
		\leq & \sup \limits_{E} |j^{K\sharp}_{\delta \alpha }-j^{\sharp}_{\delta \alpha }| \int_{0}^{t} d \tau \iiint_{ B_{R'} \times B_R \times (0,R)} L_\alpha (f)^\sharp dxd\xi dI \\
		+ & 2 \int_{E^c} Q^-_\alpha (f,f)^\sharp \exp(-(F^\sharp _{\alpha } (t)- F^\sharp _{\alpha } (\tau))) dxd\xi dI d \tau \rightarrow 0 \ (\text{as } K \rightarrow +\infty, \epsilon \rightarrow 0^+ ) \,.
	\end{aligned}
\end{equation}
Moreover, due to $Q^-_\alpha (f,f)^\sharp \exp(-(F^\sharp _{\alpha } (t)- F^\sharp _{\alpha } (\tau))) \in L^1 ( (0, T) \times \R^3 \times B_R \times (0, R) )$, one has
\begin{equation}\label{eq8.100-1}
  \begin{aligned}
    &|\int_{0}^{t} d \tau \iiint_{ \{  |x| > R' \} \times B_R \times (0,R)} (j^{K\sharp}_{\delta \alpha }-j^{\sharp}_{\delta \alpha }) L_\alpha (f)^\sharp \exp(-(F^\sharp _{\alpha } (t)- F^\sharp _{\alpha } (\tau))) dxd\xi dI |\\
    \leq & 2 \int_{0}^{t} d \tau \iiint_{ \{  |x| > R' \} \times B_R \times (0,R)} Q^-_\alpha (f,f)^\sharp \exp(-(F^\sharp _{\alpha } (t)- F^\sharp _{\alpha } (\tau))) dxd\xi dI \rightarrow 0 \ (R' \to + \infty) \,.
  \end{aligned}
\end{equation}
As a result, the convergences \eqref{eq8.98}, \eqref{eq8.99}, \eqref{eq8.100} and \eqref{eq8.100-1} reduce to the convergence \eqref{eq8.101}. Moreover, the corresponding result \eqref{eq8.101-1} for the cases $\alpha \in \{1,\cdots,s_0\}$ can be analogously proved. Then the proof of Lemma \ref{jnw} is finished.
\end{proof}

\begin{lemma}\label{subsolution}
$f$ is a subsolution of \eqref{BE-MP} given in Definition \ref{SubSolution}.
\end{lemma}

\begin{proof}

Without loss of generality, we only consider the cases $\alpha \in \{ s_0 + 1, \cdots, s \}$. Lemma \ref{theo-4.2} shows that $(1+\delta f^n_\alpha)^{-1}\tilde{Q}^+_{\alpha ,n} (f^n,f^n)$ is relatively weakly compact in $L^1 ( (0, T) \times \R^3 \times B_R \times (0, R) )$ for any given $R, T > 0$. Then for $ t > 0 $ there is $Q^+_{\delta \alpha } \in L^1 ( (0, T) \times \R^3 \times B_R \times (0, R) )$ such that
\begin{equation}\label{eq8.102}
	\lim \limits_{n\rightarrow \infty} \tfrac{\tilde{Q}^+_{\alpha ,n} (f^n,f^n)}{1+\delta f^n_\alpha} = Q^+_{\delta \alpha } \quad \text{weakly in }\ L^1((0,t)\times \R^3 \times B_R\times (0,R)).
\end{equation}

For all $t,R>0$, we take $ \varphi \in L^\infty ((0,t)\times \R^3 \times \R^3 \times \R_+)_+$ with support set contained in $(0,t)\times B_R \times B_R\times (0,R)$. Then Lemma \ref{theo-6.7} implies that there exists a Borel set $E \subset (0,t)\times B_R \times B_R\times (0,R)$ such that $|E^c|<\epsilon$ and
\begin{equation*}
	\begin{aligned}
		& \int_{E}  Q^+_{\delta \alpha } \varphi dtdx  d\xi dI 
		= \lim \limits_{n\rightarrow \infty } \int_{E}  \tfrac{\tilde{Q}^+_{\alpha ,n} (f^n,f^n)}{1+\delta f^n_\alpha} \varphi dtdx  d\xi dI \\
		\leq & \lim \limits_{n\rightarrow \infty } \int_{E}  \tilde{Q}^+_{\alpha ,n} (f^n,f^n) \varphi dtdx d\xi dI = \int_{E} Q^+_{\alpha } (f,f) \varphi dtdx d\xi dI,
	\end{aligned}
\end{equation*}
which means that
\begin{equation}\label{eq8.103}
	Q^+_{\delta \alpha } \leq Q^+_{\alpha } (f,f), \ a.e. \ (t, x, \xi, I) \in \R_+ \times \R^3 \times \R^3 \times \R_+
\end{equation}
by letting $\epsilon \to 0^+$ and $t, R \to + \infty$.

As $n \to + \infty$, from taking $L^1$ weak limit in both sides of \eqref{eq8.93} and using \eqref{eq8.101}, \eqref{eq8.102} and \eqref{eq8.103}, we derive that for all $t \geq 0$ and almost everywhere $(x,\xi ,I)\in \R^3 \times \R^3 \times \R_+$, 
\begin{equation*}
	\begin{aligned}
		&\tilde{g}^{\sharp}_{\delta \alpha } (t,x,\xi ,I) -\tilde{g}_{\delta \alpha,0} (x,\xi ,I) \exp(-F^\sharp _{\alpha} (t,x,\xi ,I)) \\
		\leq& \int_{0}^{t} \{j^{\sharp}_{\delta \alpha }(\tau) L_\alpha (f)^\sharp(\tau)+Q^+_{\alpha } (f,f) ^\sharp (\tau)\} \exp(-(F^\sharp _{\alpha } (t)- F^\sharp _{\alpha } (\tau))) d \tau.
	\end{aligned}
\end{equation*}
Here $ \tilde{g}_{\delta \alpha } $ is the $L^1$ weak limit of $g^n_{\delta \alpha}$.

Observe that for $M > 0$,
\begin{equation*}
  \begin{aligned}
    0 \leq j^n_{\delta \alpha} \leq f^n_\alpha - \tfrac{f^n_\alpha}{1 + \delta f^n_\alpha } \leq \delta M f^n_\alpha + \mathbf{1}_{ f^n_\alpha \geq M } f^n_\alpha \,.
  \end{aligned}
\end{equation*}
We then gain
\begin{equation*}
  \begin{aligned}
    0 \leq j_{\delta \alpha} \leq \delta M f_\alpha + k_{M \alpha} \,,
  \end{aligned}
\end{equation*}
where $ k_{M \alpha} \geq 0 $ is the $L^1$ weak limit of $\mathbf{1}_{ f^n_\alpha \geq M } f^n_\alpha$. Then the $L^1$ weak compactness of $f^n_\alpha$ implies that for all $\varphi \in L^\infty (\R_+ \times \R^3 \times \R^3 \times \R_+)_+$,
\begin{equation*}
	0 \leq \int_{0}^{t} d \tau \iiint_{ \R^3 \times \R^3 \times \R_+} k_{M\alpha} \varphi dxd\xi dI \leq \sup \limits_{n \geq 1}  \int_{0}^{t} d \tau \iiint_{ \R^3 \times \R^3 \times \R_+} \varphi f^n_\alpha \mathbf{1}_{f^n_\alpha \geq M} dxd\xi dI \rightarrow 0 \ (M \rightarrow + \infty).
\end{equation*}
This further indicates that $ k_{M \alpha} \to 0$ for almost all $ (t, x, \xi, I) \in \R_+ \times \R^3 \times \R^3 \times \R_+ $ as $M \to + \infty$. Consequently, for all $\varphi \in L^\infty ( \R^3 \times \R^3 \times \R_+ )_+$, as $\delta \to 0^+$ and $M \to + \infty$,
\begin{equation}\label{j-0}
  \begin{aligned}
    & \int_{0}^{t}d \tau \iiint_{ \R^3 \times \R^3 \times \R_+} \varphi j^{\sharp}_{\delta \alpha }(\tau) L_\alpha (f)^\sharp(\tau) \exp(-(F^\sharp _{\alpha } (t)- F^\sharp _{\alpha } (\tau)))  dxd\xi dI \\
    \leq & \delta M \int_{0}^{t}d \tau \iiint_{ \R^3 \times \R^3 \times \R_+} \varphi Q_\alpha^- (f, f)^\sharp(\tau) \exp(-(F^\sharp _{\alpha } (t)- F^\sharp _{\alpha } (\tau)))  dxd\xi dI \\
    & + \int_{0}^{t}d\tau \iiint_{ \R^3 \times \R^3 \times \R_+} \varphi k^\sharp_{M\alpha} L_\alpha (f)^\sharp(\tau) \exp(-(F^\sharp _{\alpha } (t)- F^\sharp _{\alpha } (\tau))) dxd\xi dI \to 0,
  \end{aligned}
\end{equation}
where the fact \eqref{eq8.97} and the Lebesgue's Control Convergence Theorem have been used. Together with Lemma \ref{converge f,g}, one concludes the inequality \eqref{Subs-2} for the cases $\alpha \in \{ s_0 + 1, \cdots, s \}$ in Definition \ref{SubSolution}. Moreover, the inequality \eqref{Subs-1} for the cases $\alpha \in \{ 1, \cdots, s_0 \}$ in Definition \ref{SubSolution} can be proved analogously above. Therefore, the proof of Lemma \ref{subsolution} is finished.
\end{proof}

\subsection{Proof of existence of the renormalized solution to \eqref{BE-MP}}

Combining with Lemma \ref{theo-7.3} and \ref{subsolution}, we know that for any $t \geq 0$,
	\begin{equation*}
		\begin{aligned}
			&\int_{0}^{t}Q^+_{\alpha } (f,f)^\sharp (\tau ,x,\xi ) \exp(-(F^\sharp _{\alpha } (t,x,\xi )-F^\sharp _{\alpha } (\tau,x,\xi ))) d \tau \\
			= &f^{\sharp}_\alpha (t,x,\xi) -f_{\alpha ,0} (x,\xi )\exp(-F^\sharp _{\alpha  } (t,x,\xi )) \quad a.e. \ (x,\xi ) \in \R^3 \times \R^3 
		\end{aligned}
\end{equation*}
with all $\alpha \in \{1,\cdots,s_0\}$, and 
	\begin{equation*}
		\begin{aligned}
			&\int_{0}^{t}Q^+_{\alpha } (f,f)^\sharp (\tau ,x,\xi ,I) \exp(-(F^\sharp _{\alpha } (t,x,\xi ,I)-F^\sharp _{\alpha } (\tau,x,\xi ,I))) d \tau \\
			= &f^{\sharp}_\alpha (t,x,\xi ,I) -f_{\alpha ,0} (x,\xi ,I )\exp(-F^\sharp _{\alpha  } (t,x,\xi ,I)) \quad a.e. \  (x,\xi ,I) \in \R^3 \times \R^3 \times \R_+ 
		\end{aligned}
	\end{equation*}
with all $\alpha \in \{s_0+1,\cdots,s\}$. Together with Lemma \ref{Lemma3.3} and \ref{converge L}, we see that the $L^1$ weak limit $f$ of $f^n$ is a mild solution to \eqref{BE-MP}. At the end, Lemma \ref{distri-mild}, Lemma \ref{Q in L1} and Corollary \ref{Q-L-loc} show that $f$ is actually a renormalized solution of \eqref{BE-MP}.

\section{Entropy inequality}\label{Sec:Entropy}

In this section, we mainly prove the Part (3) of Theorem \ref{MainThm}. Namely, we will prove the renormalized solution $f$ of \eqref{BE-MP} satisfies the entropy inequality \eqref{entropy-th}, i.e.,
\begin{equation}\label{entropy}
	\begin{aligned}
		&H(f)(t)+ \tfrac{1}{4}\sum \limits_{\alpha =1}^{s_0} \int_{0}^{t} d \tau \iint_{ \R^3 \times \R^3 } e_\alpha(\tau,x,\xi ) dxd\xi\\
		& + \tfrac{1}{4}\sum \limits_{\alpha =s_0+1}^{s}\int_{0}^{t} d \tau \iiint_{ \R^3 \times \R^3 \times \R_+}  e_\alpha(\tau,x,\xi ,I) dxd\xi dI\leq H(f_0),
	\end{aligned}
\end{equation}
where $e_\alpha$ is defined in \eqref{ee}. 

As in \eqref{entro} of Lemma \ref{Lemma 4.17}, the approximated distributional solution $f_n$ to the approximated problem \eqref{approximation} admits the entropy identity. Namely, for all $t \geq 0 $,
\begin{equation*}
	\begin{aligned}
		H(f^n_{0})&=H(f^n)(t)+ \tfrac{1}{4}\sum \limits_{\alpha =1}^{s_0} \int_{0}^{t} d \tau \iint_{ \R^3 \times \R^3 } \tilde{e}_\alpha^n(\tau,x,\xi ) dxd\xi\\
		& +\tfrac{1}{4}\sum \limits_{\alpha =s_0+1}^{s}\int_{0}^{t} d \tau \iiint_{ \R^3 \times \R^3 \times \R_+}  \tilde{e}_\alpha^n(\tau,x,\xi ,I) dxd\xi dI,
	\end{aligned}
\end{equation*}
where $\tilde{e}_\alpha^n$ and $ H(f^n)$ are defined in \eqref{e na} and \eqref{Hf}, respectively. Note that Lemma \ref{Lemma 4.6} and Lemma \ref{Lemma 4.17} imply $ \liminf_{n \to \infty}H(f^n)(t) \geq H (f) (t) $. Moreover, Lemma \ref{Lemma 4.8} and Lemma \ref{Lemma 4.9} indicate that $ \lim_{n \to + \infty} H(f^n_{0}) = H (f_0) $. In order to prove the entropy inequality \eqref{entropy-th} (or \eqref{entropy}), we only need to show that for almost all $(t,x)\in (0,+\infty) \times \R^3$,
\begin{equation}\label{ena}
\begin{aligned}
\liminf_{n\rightarrow \infty} \int_{\R^3 } \tilde{e}_\alpha^n d\xi &\geq \int_{ \R^3 } e_\alpha d\xi \quad \qquad \text{ for } \alpha \in \{1,\cdots,s_0\},\\
\liminf_{n\rightarrow \infty} \iint_{ \R^3 \times \R_+} \tilde{e}_\alpha^n d\xi dI &\geq \iint_{ \R^3 \times \R_+} e_\alpha d\xi dI \quad \, \text{for } \alpha \in \{s_0+1,\cdots,s\}.
\end{aligned}
\end{equation}

For notational simplicity, we introduce the following symbols. 

For $ \alpha , \beta \in \{s_0+1,\cdots,s\}$, let
\begin{equation*}{\small
\begin{aligned}
E_R&=B_R\times B_R \times [0,1]\times [0,1]\times (0,R)\times (0,R)\times \mathbb{S}^2,\\
d \Theta_{\alpha \beta} &=\mathfrak{r}^{\delta(\alpha)/2-1}(1-\mathfrak{r})^{\delta(\beta )/2-1}(1-\mathfrak{R})^{(\delta(\alpha)+\delta(\beta ))/2-1}\mathfrak{R}^{1/2} I^{\delta(\alpha)/2-1} I_*^{\delta(\beta )/2-1}d\xi d\xi_* d \mathfrak{R} d \mathfrak{r} dI dI_* d\omega, \\
d\mu_{\alpha \beta} & = B^R_{2\alpha \beta } d \Theta_{\alpha \beta} \,, d \mu_{\alpha \beta}^n = B^R_{2\alpha \beta, n } d \Theta_{\alpha \beta},
\end{aligned}}
\end{equation*}
where $B^R_{2\alpha \beta }= B_{2\alpha \beta } \wedge R = \min \{ B_{2\alpha \beta } , R \} $ and $B^R_{2\alpha \beta, n }= B_{2\alpha \beta, n } \wedge R = \min \{ B_{2\alpha \beta, n } , R \} $. 

For $\alpha \in  \{s_0+1,\cdots,s\} $ and $ \beta \in \{1,\cdots,s_0\}$, let
\begin{equation*}
	\begin{aligned}
		E_R&=B_R\times B_R \times [0,1]\times (0,R)\times \mathbb{S}^2,\\
		d \Theta_{\alpha \beta} &= (1-\mathfrak{R})^{\delta(\alpha)/2-1}\mathfrak{R}^{1/2} I^{\delta(\alpha)/2-1} d\xi d\xi_* d \mathfrak{R} dI d\omega, d\mu_{\alpha \beta} = B^R_{1\beta \alpha } d \Theta_{\alpha \beta} \,, d \mu_{\alpha \beta}^n = B^R_{1\beta \alpha ,n} d \Theta_{\alpha \beta} , 
	\end{aligned}
\end{equation*}
where $B^R_{1\beta \alpha }= B_{1\beta \alpha }\wedge R = \min \{ B_{1\beta \alpha } , R \} $ and $B^R_{1\beta \alpha, n }= B_{1\beta \alpha, n }\wedge R = \min \{ B_{1\beta \alpha, n } , R \} $. 

For $\alpha \in \{1,\cdots,s_0\} , \beta \in \{s_0+1,\cdots,s\} $, let
\begin{equation*}
	\begin{aligned}
		E_R&=B_R \times B_R \times [0,1]\times (0,R)\times \mathbb{S}^2,\\
		d \Theta_{\alpha \beta} & = (1-\mathfrak{R})^{\delta(\beta)/2-1} \mathfrak{R}^{1/2} I_*^{\delta(\beta )/2-1} d\xi d\xi_* d \mathfrak{R} dI_* d\omega, d \mu_{\alpha \beta} = B^R_{1\alpha \beta } d \Theta_{\alpha \beta} , d \mu_{\alpha \beta}^n = B^R_{1\alpha \beta, n } d \Theta_{\alpha \beta} , 
	\end{aligned}
\end{equation*}
where $B^R_{1\alpha \beta,n }= B_{1\alpha \beta, n }\wedge R = \min \{ B_{1\alpha \beta, n } , R \} $ and $B^R_{1\alpha \beta}= B_{1\alpha \beta}\wedge R = \min \{ B_{1\alpha \beta} , R \} $. 

For $ \alpha , \beta \in \{1,\cdots,s_0\} $, let
\begin{equation*}
	\begin{aligned}
		E_R =B_R\times B_R \times \mathbb{S}^2, \ d \Theta_{\alpha \beta} = d\xi d\xi_* d\omega, d \mu_{\alpha \beta} = B^R_{0\alpha \beta } d \Theta_{\alpha \beta} , d \mu_{\alpha \beta}^n = B^R_{0\alpha \beta, n } d \Theta_{\alpha \beta} , 
	\end{aligned}
\end{equation*}
where $B^R_{0 \alpha \beta }= B_{0 \alpha \beta }\wedge R = \min \{ B_{0 \alpha \beta }, R \} $ and $B^R_{0 \alpha \beta, n }= B_{0 \alpha \beta,n }\wedge R = \min \{ B_{0 \alpha \beta, n }, R \} $.

Moreover, for $ \alpha , \beta \in \{1,\cdots,s\} $ and $R > 0$, denote by
\begin{equation}\label{HR}
\begin{aligned}
H^{n, R}_{\alpha \beta }(f^n)&= \int_{E_R }  \mathfrak{F} (f_\alpha^n, f_\beta^n) d\mu_{\alpha \beta}^n, \ H_{\alpha \beta }^R (f) = \int_{E_R} \mathfrak{F} (f_\alpha, f_\beta) d\mu_{\alpha \beta} , \\
H_{\alpha \beta }^n (f^n)&= \int_{(\R^3\times \R_+)^4} \mathfrak{F} (f_\alpha^n, f_\beta^n) W_{\alpha \beta, n } d\xi d\xi_* d\xi'd\xi_*'dI dI_* dI'dI_*' , \\
H_{\alpha \beta }(f)&= \int_{(\R^3\times \R_+)^4} \mathfrak{F} (f_\alpha, f_\beta) W_{\alpha \beta } d\xi d\xi_* d\xi'd\xi_*'dI dI_* dI'dI_*' ,
\end{aligned}
\end{equation}
where
\begin{equation*}
  \begin{aligned}
    \mathfrak{F} (f_\alpha, f_\beta) = (\tfrac{{f'_\alpha} {f'_{\beta *}}  } {(I')^{\delta(\alpha)/2-1} (I'_* )^{\delta(\beta)/2-1} }  - \tfrac{f_\alpha f_{\beta * } }{I^{\delta(\alpha)/2-1} I_* ^{\delta(\beta)/2-1} }) \log (\tfrac{I^{\delta(\alpha)/2-1} I_* ^{\delta(\beta)/2-1} {f'_\alpha} {f'_{\beta *}} }{f_\alpha f_{\beta * } (I')^{\delta(\alpha)/2-1} (I'_* )^{\delta(\beta)/2-1} } ) \,.
  \end{aligned}
\end{equation*}
We remark that $ \int_{\mathcal{Z}_\alpha} \tilde{e}_\alpha^n d \mathbf{Z}_\alpha = N_n (f^n)^{-1} \sum_{\beta = 1}^{s} H_{\alpha \beta }^n (f^n) $ by \eqref{e na}, where $ N_n (f^n) $ is defined in \eqref{Nnf}. Here $ \mathcal{Z}_\alpha = \R^3 $ and $\mathbb{Z}_\alpha = \xi$ for $\alpha \in \{ 1, \cdots, s_0 \}$, and $ \mathcal{Z}_\alpha = \R^3 \times \R_+ $ and $\mathbb{Z}_\alpha = ( \xi, I )$ for $\alpha \in \{ s_0 + 1, \cdots, s \}$. Denote by $ \tilde{H}_{\alpha \beta }^n (f^n) = N_n (f^n)^{-1} H_{\alpha \beta }^n (f^n) $. In order to show \eqref{ena}, it suffices to prove that for $\alpha, \beta \in \{ 1, \cdots, s \}$ and for almost all $(t, x) \in \R_+ \times \R^3$,
\begin{equation}\label{HH-limit}
  \begin{aligned}
    \liminf_{n \to + \infty} \tilde{H}_{\alpha \beta }^n (f^n) \geq H_{\alpha \beta }(f) \,.
  \end{aligned}
\end{equation}
Without loss of generality, we only prove the cases $ \alpha , \beta \in \{s_0+1,\cdots,s\}$.

To simplify the notation, we write
\begin{align*}
	&A^R_{2\alpha \beta, n } = \int_{[0,1]^2 \times \mathbb{S}^2 } B^R_{2\alpha \beta, n } \mathfrak{r}^{\delta (\alpha)/2-1} (1-\mathfrak{r})^{\delta(\beta)/2-1}(1-\mathfrak{R})^{\delta(\alpha)+\delta(\beta)/2-1} \mathfrak{R}^{1/2}  d \mathfrak{R} d \mathfrak{r} d\omega,\\ 
	&A^R_{2\alpha \beta} = \int_{[0,1]^2 \times \mathbb{S}^2 } B^R_{2\alpha \beta} \mathfrak{r}^{\delta (\alpha)/2-1} (1-\mathfrak{r})^{\delta(\beta)/2-1}(1-\mathfrak{R})^{\delta(\alpha)+\delta(\beta)/2-1} \mathfrak{R}^{1/2}  d \mathfrak{R} d \mathfrak{r} d\omega,\\ 
	& L^R_{2\alpha \beta, n }(f^n)=\iint_{B_R\times(0,R)} f^n_{\beta_*} A^R_{2\alpha \beta, n }d\xi_* dI_* \,, L^R_{2\alpha \beta}(f)=\iint_{B_R\times(0,R)} f_{\beta_*} A^R_{2\alpha \beta}d\xi_* dI_*.
\end{align*}

{\bf Step 1:} {\em The sequences $   \frac{ {f_\alpha^n}' {f_{\beta *}^n}' }{(I')^{\delta(\alpha)/2-1} (I'_* )^{\delta(\beta)/2-1}} B^R_{2 \alpha \beta, n}$ and $  \frac{f^n_\alpha f^n_{\beta_* }}{I^{\delta(\alpha)/2-1} I_* ^{\delta(\beta)/2-1}} B^R_{2 \alpha \beta, n}$ are relatively weakly compact in $L^1(E_R,d \Theta_{\alpha \beta}  )$ for almost all $(t,x) \in (0,R) \times   B_R$.}

We first let
\begin{equation}\label{Nf}
	N (f) = 1 + \sum_{\alpha = 1}^{s_0} \int_{B_R} f_\alpha d \xi + \sum_{\alpha = s_0 + 1}^{s} \iint_{B_R \times (0, R)} f_\alpha d \xi d I \,,
\end{equation}
and
\begin{equation*}
	\begin{aligned}
		P^n_{\alpha \beta}  = N  (f^n)^{-1} \tfrac{f^n_\alpha f^n_{\beta_* }}{I^{\delta(\alpha)/2-1} I_* ^{\delta(\beta)/2-1}} B^R_{2 \alpha \beta, n} \,, &\tilde{P}^n_{\alpha \beta} = N (f^n)^{-1} \tfrac{ {f_\alpha^n}' {f_{\beta *}^n}' }{(I')^{\delta(\alpha)/2-1} (I'_* )^{\delta(\beta)/2-1}} B^R_{2 \alpha \beta, n} \,,\\
		P_{\alpha \beta}  =N  (f )^{-1} \tfrac{f_\alpha f_{\beta_* }}{I^{\delta(\alpha)/2-1} I_* ^{\delta(\beta)/2-1}} B^R_{2 \alpha \beta} \,,& \tilde{P}_{\alpha \beta} = N  (f )^{-1} \tfrac{ {f'_\alpha} f'_{\beta *} }{(I')^{\delta(\alpha)/2-1} (I'_* )^{\delta(\beta)/2-1}} B^R_{2 \alpha \beta} \,.
	\end{aligned}
\end{equation*}
Then
\begin{equation*}{\small
		\begin{aligned}
			\int_{E_R}P^n_{\alpha \beta}  d \Theta_{\alpha \beta} = \frac{1}{N  (f^n)}  \iint_{B_R \times (0, R)} f^n_\alpha L^R_{2\alpha \beta, n }(f^n) d \xi d I \,, \int_{E_R}P_{\alpha \beta}   d \Theta_{\alpha \beta} = \frac{1}{N  (f )}\iint_{B_R \times (0, R)} f_\alpha L^R_{2\alpha \beta }(f) d \xi d I \,.
	\end{aligned}}
\end{equation*}
Now, we take an increasing function $\Psi \in C(\R_+) $ with
\begin{equation}
	\Psi (t) \rightarrow +\infty \text{ as } t\rightarrow +\infty, \quad \Psi (t)(\log t)^{-1}\rightarrow 0 \text{ as } t\rightarrow +\infty.
\end{equation}
This means that for any $\epsilon>0$, there exists constant $M$ such that $t\geq M$, $\Psi (t) \leq \epsilon     \log   t  .$  Observing that
\begin{align*}
&\| \iint_{ \R^3 \times \R_+} f^n_\alpha \Psi (I^{1-\delta(\alpha) / 2}f^n_\alpha ) d\xi dI - \iint_{ \R^3 \times \R_+} f _\alpha \Psi (I^{1-\delta(\alpha) / 2}f_\alpha  ) d\xi dI \|_{L^1((0,T) \times \R^3)} \\
\leq & \|\iint_{ \R^3 \times \R_+} (f^n_\alpha \Psi (I^{1-\delta(\alpha) / 2}f^n_\alpha )\mathbf{1}_{I^{1-\delta(\alpha) / 2}f^n_\alpha \leq M  }  - f _\alpha \Psi (I^{1-\delta(\alpha) / 2}f_\alpha  )\mathbf{1}_{I^{1-\delta(\alpha) / 2}f_\alpha  \leq M  }) d\xi dI\|_{L^1((0,T) \times \R^3)} \\
& + \epsilon \sup_{n \geq 1} \|\iint_{ \R^3 \times \R_+} (f^n_\alpha | \log   I^{1-\delta(\alpha) / 2}f^n_\alpha |   + f _\alpha| \log I^{1-\delta(\alpha) / 2}f_\alpha|   ) d\xi dI\|_{L^1((0,T) \times \R^3)}.
\end{align*}

Since the continuous function $\Psi(t)$ is bounded on $[0, M]$, combined with the bound of $f^n_\alpha \log(I^{1-\delta(\alpha) / 2} f^n_\alpha)$ in $L^\infty((0,T); L^1(\mathbb{R}^3 \times \mathbb{R}^3 \times \mathbb{R}_+))$ and Lemma \ref{theo-6.5}, we conclude that the inequality above tends to $0$ as $n \to \infty$. This implies that the sequence $\iint_{\mathbb{R}^3 \times \mathbb{R}_+} f^n_\alpha \Psi(I^{1-\delta(\alpha) / 2}f^n_\alpha) d\xi dI$ is compact in $L^1((0,T) \times \mathbb{R}^3)$. Thus, we can derive that 
\begin{equation*}
	\begin{aligned}
	N(f^n)^{-1}	\int_{E_R}   \tfrac{f_\alpha^n f_{\beta *}^n}{ I^{\delta(\alpha) /2-1} I_*^{\delta(\beta)/2-1}} &(\Psi (I^{1-\delta(\alpha)/2} f_\alpha^n)+\Psi(I_*^{1-\delta(\beta) / 2} f_{\beta *}^n)) B^R_{2 \alpha \beta, n} d \Theta_{\alpha \beta}  	 \\
	N(f^n)^{-1}	\int_{E_R}   \tfrac{{f_\alpha^n}' {f_{\beta *}^n}'}{(I')^{\delta(\alpha)/2-1} (I'_*)^{\delta(\beta)/2-1 }} &(\Psi ((I')^{1-\delta(\alpha) / 2} {f_\alpha^n}' )+\Psi ((I'_*)^{1-\delta(\beta) / 2} {f_{\beta_*}^n}' )) B^R_{2 \alpha \beta, n} d \Theta_{\alpha \beta} 
	\end{aligned}
\end{equation*}
are compact in $L^1((0,T) \times \mathbb{R}^3)$.

 We claim that if the sequence $f^n(x)\geq 0 $ is compact in $L^1 (  \Omega )$, where $\Omega $ is a bounded set. Then there exists a function $g(x) \in L^1 (  \Omega )$ such that $ 0 \leq f^n ( x ) \leq g ( x ) $ a.e. $\Omega.$ Specifically, since $f^n$ is compact in $L^1 ( \Omega)$, then there exists a function $f \in L^1 ( \Omega ) $ such that $\lim_{n \to + \infty} \|f^n - f\|_{L^1 ( \Omega ) } =0$ ( up to a subsequence if necessary). It also converges almost everywhere, i.e., for any fixed $x \in \Omega$, there exists $N_x \in \mathbb{N},$ if $n \geq N_x, \ f^n (x ) \leq f( x ) + 1 $. Then it is easy to find that $g(x) = \max \{ f(x) +1 , f_1( x ) ,\cdots, f_{N_x} (x ) \} \in L^1 (  \Omega )$ and $ 0 \leq f^n ( x ) \leq g ( x ) $ a.e. $\Omega.$ The claim is therefore holds. As a result, there exist nonnegative functions $\bar{K}  ,\bar{N}  \in L^1 ((0,R) \times B_R )$, independent of $n$, such that 
\begin{equation*}
	\begin{aligned}
		N(f^n)^{-1}	\int_{E_R}   \tfrac{f_\alpha^n f_{\beta *}^n}{ I^{\delta(\alpha) /2-1} I_*^{\delta(\beta)/2-1}} (\Psi (I^{1-\delta(\alpha)/2}& f_\alpha^n)+\Psi(I_*^{1-\delta(\beta) / 2} f_{\beta *}^n)) B^R_{2 \alpha \beta, n} d \Theta_{\alpha \beta}   \leq  \bar{K}    , \\
		 N(f^n)^{-1}	\int_{E_R}   \tfrac{{f_\alpha^n}' {f_{\beta *}^n}'}{(I')^{\delta(\alpha)/2-1} (I'_*)^{\delta(\beta)/2-1 }} (\Psi ((I')^{1-\delta(\alpha) / 2}& {f_\alpha^n}' )+\Psi ((I'_*)^{1-\delta(\beta) / 2} {f_{\beta_*}^n}' )) B^R_{2 \alpha \beta, n} d \Theta_{\alpha \beta}  \leq  \bar{K},   \\
		 N ( f^n) &\leq \bar{N}.
	\end{aligned}
\end{equation*}

Besides, for all $M>0$ and any Borel subset $A$ in $E_R$, one knows that for almost all $(t,x) \in (0, R) \times B_R$,
\begin{equation*}
	\begin{aligned}
			& \int_{A}   \tfrac{f_\alpha^n f_{\beta *}^n}{ I^{\delta(\alpha) /2-1} I_*^{\delta(\beta)/2-1}} B^R_{2 \alpha \beta , n} d \Theta_{\alpha \beta}\\
		\leq& R M^2 \Theta_{\alpha \beta} (A) + \int_{E_R} \tfrac{f_\alpha^n f_{\beta *}^n}{ I^{\delta(\alpha) /2-1} I_*^{\delta(\beta)/2-1}} (\mathbf{1}_{I^{1-\delta(\alpha)/2} f_\alpha^n\geq M}+ \mathbf{1}_{I_*^{1-\delta(\beta) / 2} f_{\beta *}^n \geq M}) B^R_{2 \alpha \beta, n} d \Theta_{\alpha \beta}\\
		\leq & R M^2 \Theta_{\alpha \beta} (A)+\Psi(M)^{-1}\bar{K} \bar{N},
	\end{aligned}
\end{equation*}
which means that the sequence $  \frac{f^n_\alpha f^n_{\beta_* }}{I^{\delta(\alpha)/2-1} I_* ^{\delta(\beta)/2-1}} B^R_{2 \alpha \beta, n}$ is relatively weakly compact in $L^1(E_R,d \Theta_{\alpha \beta} )$ for almost all $(t,x) \in (0,R) \times   B_R$. By employing the analogous arguments above, the sequence  $  \frac{ {f_\alpha^n}' {f_{\beta *}^n}' }{(I')^{\delta(\alpha)/2-1} (I'_* )^{\delta(\beta)/2-1}} B^R_{2 \alpha \beta, n}$ is also relatively weakly compact in $L^1(E_R,d \Theta_{\alpha \beta} )$ for almost all $(t,x) \in (0,R) \times B_R$.

{\bf Step 2:} {\em 
\begin{equation}\label{P con}
\int_{E_R} P^n_{\alpha \beta}\varphi d \Theta_{\alpha \beta} \rightarrow \int_{E_R}  P_{\alpha \beta}\varphi d \Theta_{\alpha \beta} \quad \text{strongly in } L^1((0,R)\times B_R),
\end{equation}
\begin{equation}\label{til P con}
\int_{E_R}   \tilde{P}^n_{\alpha \beta}\varphi d\Theta_{\alpha \beta} \rightarrow \int_{E_R}  \tilde{P}_{\alpha \beta}\varphi d \Theta_{\alpha \beta} \quad \text{strongly in } L^1((0,R)\times B_R)
\end{equation}
for all $\varphi \in L^\infty (E_R)$.}

It is easy to find that $A^R_{2\alpha \beta, n } \in L^\infty ((0,R)\times B_R \times B_R \times (0,R);L^1(B_R \times (0,R)))$, and 
$$\lim_{n \to + \infty} \| A^R_{2\alpha \beta, n } - A^R_{2\alpha \beta} \|_{L^1 (B_R \times (0, R))} = 0$$ 
for almost all $(t, x, \xi_*, I_*) \in (0, R) \times B_R \times B_R \times (0, R)$. Following the analogous arguments in Lemma \ref{converge L}, one has
\begin{equation*}
L^R_{2\alpha \beta, n }(f^n) \rightarrow L^R_{2\alpha \beta}(f) \quad \text{strongly in } L^1((0,R)\times B_R \times B_R \times (0,R)).
\end{equation*}
Then the convergence \eqref{gn converge} and Lemma \ref{converge f,g} indicate that $N (f^n)^{-1}$ strongly converges to $N (f)^{-1}$ in $L^1 ( (0, R) \times B_R )$ as $n \to + \infty$. Then the Egorov's theorem indicates that for any $\epsilon>0$ there exists a Borel $E\subset \Omega = (0,R)\times B_R \times B_R\times (0,R)$ such that $|E^c \cap \Omega|<\epsilon$, $ L^R_{2\alpha \beta, n }(f^n) $ uniformly converges to $ L^R_{2\alpha \beta}(f) $ and $ N (f^n)^{-1} $ uniformly converges to $ N (f)^{-1} $ on $E$. 

Observe that
\begin{equation*}
	\begin{aligned}
		& f^n_\alpha L^R_{2\alpha \beta, n }(f^n) N (f^n)^{-1} - f_\alpha L^R_{2\alpha \beta}(f) N (f)^{-1} \\
        = & f^n_\alpha \{ L^R_{2\alpha \beta, n } (f^n) N (f^n)^{-1} - L^R_{2\alpha \beta}(f) N (f)^{-1} \} + L^R_{2\alpha \beta }(f) N (f)^{-1} (f^n_\alpha -f_\alpha) \,.
	\end{aligned}
\end{equation*}
We then obtain
\begin{equation*}
\begin{aligned}
&\int_{0}^{R}dt \int_{B_R} dx |\int_{E_R}   P^n_{\alpha \beta}\varphi d \Theta_{\alpha \beta } -\int_{E_R} P_{\alpha \beta}\varphi d\Theta_{\alpha \beta }| \\
\leq & \sup_{n \geq 1} \Vert f^n_\alpha \Vert_{L^1 ((0, R) \times B_R \times B_R \times (0, R))} \sup \limits_{E} | L^R_{2\alpha \beta, n }(f^n) N (f^n)^{-1} - L^R_{2\alpha \beta }(f) N (f)^{-1} | \\
+ & ( \sup_{n \geq 1} \Vert A^R_{2\alpha \beta, n } \Vert_{L^\infty} + \Vert A^R_{2\alpha \beta } \Vert_{L^\infty} ) ( \sup_{n \geq 1} \int_{E^c} f^n_\alpha dtdxd\xi dI +\int_{E^c} f_\alpha dtdxd\xi dI) \\
+& \int_{0}^{R}dt \int_{B_R} dx |\iint_{B_R \times (0,R)} f^n_\alpha \tilde{\varphi} d\xi dI -\iint_{B_R \times (0,R)} f_\alpha \tilde{\varphi} d\xi dI |,
\end{aligned}
\end{equation*}
where $\tilde{\varphi}= \mathbf{1}_E L^R_{2\alpha \beta }(f) N (f)^{-1} \in L^\infty$. Note that $\sup_{n \geq 1} \Vert f^n_\alpha \Vert_{L^1 ((0, R) \times B_R \times B_R \times (0, R))} \leq C_R$ by Lemma \ref{Lemma 4.17}, one knows that the first quantity in the right-hand side of above inequality goes to 0 as $n \to + \infty$. Moreover, the convergence \eqref{gn converge} and Lemma \ref{converge f,g} show that the last inequality in the right-hand side of above inequality goes to 0 as $n \to + \infty$. Furthermore, $\sup_{n \geq 1} \Vert A^R_{2\alpha \beta, n } \Vert_{L^\infty} + \Vert A^R_{2\alpha \beta } \Vert_{L^\infty} \leq 2 R$ and $f^n_\alpha $ weakly converges to $f_\alpha$ in $L^1 ( (0, R) \times B_R \times B_R \times (0, R) )$, which imply that the second inequality in the right-hand side of above inequality goes to 0 as $\epsilon \to 0^+$. Thus the assertion \eqref{P con} holds.

We now prove the convergence \eqref{til P con}. Observe that
  \begin{align*}
    \int_{E_R} N (f^n)^{-1} \tilde{P}^n_{\alpha \beta}\varphi d \Theta_{\alpha \beta} = N (f^n)^{-1}\iint_{B_R \times (0,R)} Q^{R, +}_{ \alpha \beta, n}(f^n,f^n)\varphi d\xi dI \,, \\
    \int_{E_R} N (f)^{-1} \tilde{P}_{\alpha \beta}\varphi d \Theta_{\alpha \beta} = N (f)^{-1} \iint_{B_R \times (0,R)} Q^{R, +}_{ \alpha \beta}(f,f)\varphi d\xi dI \,,
  \end{align*}
where $Q^{R, +}_{ \alpha \beta, n}(f^n,f^n)$ is formed by replacing the collision kernel $B_{2 \alpha \beta, n}$ of $Q^+_{\alpha \beta, n} (f^n, f^n)$ by $B_{2 \alpha \beta, n}^R$, and $ Q^{R, +}_{ \alpha \beta}(f,f) $ is defined by replacing the collision kernel $B_{2 \alpha \beta}$ of $Q^+_{\alpha \beta } (f, f)$ by $B_{2 \alpha \beta}^R$. Recalling the identity \eqref{Q}, one similarly has
\begin{equation*}{\small
\begin{aligned}
\iint_{B_R \times (0,R)} Q^{R, +}_{ \alpha \beta, n }(f^n,f^n)\varphi d\xi dI
=&   \int_{(B_R)^2 \times (0,R)^2 \times [0,1]^2 \times \mathbb{S}^2} f^n_\alpha f^n_{\beta *} B^R_{2\alpha \beta, n}  \mathfrak{r}^{\delta (\alpha)/2-1} \\
&\times (1- \mathfrak{r})^{\delta(\beta)/2-1} (1-\mathfrak{R})^{\delta(\alpha)+\delta(\beta)/2-1} \mathfrak{R}^{1/2}\varphi' d \mathfrak{R} d \mathfrak{r} d\omega d\xi d\xi_* dI dI_*, \\
\iint_{B_R \times (0,R)} Q^{R, +}_{ \alpha \beta} (f,f)\varphi d\xi dI
=& \int_{(B_R)^2 \times (0,R)^2 \times [0,1]^2 \times \mathbb{S}^2} f_\alpha f_{\beta *} B^R_{2\alpha \beta} \mathfrak{r}^{\delta (\alpha)/2-1} \\
&\times (1- \mathfrak{r})^{\delta(\beta)/2-1} (1-\mathfrak{R})^{\delta(\alpha)+\delta(\beta)/2-1} \mathfrak{R}^{1/2}\varphi' d \mathfrak{R} d \mathfrak{r} d\omega d\xi d\xi_* dI dI_*,
\end{aligned}}
\end{equation*}
where $\varphi'=\varphi (\xi',I')$. We now introduce
\begin{equation*}{\small
	\begin{aligned}
&\hat{A}^R_{2\alpha \beta, n } (t,x,\xi,\xi_*,I,I_*) =\int_{[0,1]^2 \times \mathbb{S}^2 } B^R_{2\alpha \beta, n} \mathfrak{r}^{\delta (\alpha)/2-1} (1- \mathfrak{r})^{\delta(\beta)/2-1} (1-\mathfrak{R})^{\delta(\alpha)+\delta(\beta)/2-1} \mathfrak{R}^{1/2}\varphi' d \mathfrak{R} d \mathfrak{r} d\omega,\\
&\hat{A}^R_{2\alpha \beta } (t,x,\xi,\xi_*,I,I_*) =\int_{[0,1]^2 \times \mathbb{S}^2 } B^R_{2\alpha \beta } \mathfrak{r}^{\delta (\alpha)/2-1} (1-\mathfrak{r})^{\delta(\beta)/2-1} (1-\mathfrak{R})^{\delta(\alpha)+\delta(\beta)/2-1} \mathfrak{R}^{1/2}\varphi' d \mathfrak{R} d \mathfrak{r}d\omega,\\
&  \bar{L} ^R_{ \alpha \beta, n}(f^n)(t,x,\xi,I) =   \iint_{B_R\times(0,R)} \hat{A}^R_{2\alpha \beta, n } f^n_{\beta *} d \xi_* d I_* \,, \\
& \bar{L}^R_{ \alpha \beta }(f)(t,x,\xi,I) = \iint_{B_R\times(0,R)} \hat{A}^R_{2\alpha \beta } f_{\beta *}d\xi_* dI_*.
\end{aligned}}
\end{equation*}
Then
  \begin{align*}
    \int_{E_R}  \tilde{P}^n_{\alpha \beta}\varphi d \Theta_{\alpha \beta} = N (f^n)^{-1} \iint_{B_R \times (0,R)} f^n_\alpha  \bar{L} ^R_{ \alpha \beta, n}(f^n)  d\xi dI \,, \\
    \int_{E_R}  \tilde{P}_{\alpha \beta}\varphi d \Theta_{\alpha \beta} = N (f)^{-1} \iint_{B_R \times (0,R)} f_\alpha \bar{L}^R_{ \alpha \beta }(f) d\xi dI \,.
  \end{align*}
Then, by following the analogous arguments of proving the convergence \eqref{P con}, one can infer that the convergence \eqref{til P con} holds.

{\bf Step 3:} {\em For almost all $(t,x)\in (0,R)\times B_R$,
	\begin{equation}\label{cnv-fn'*}
	\begin{aligned} 
		N_n  (f^n)^{-1} \tfrac{f^n_\alpha f^n_{\beta_* }}{I^{\delta(\alpha)/2-1} I_* ^{\delta(\beta)/2-1}} B^R_{2 \alpha \beta, n} &\rightarrow  \tfrac{f_\alpha f_{\beta_* }}{I^{\delta(\alpha)/2-1} I_* ^{\delta(\beta)/2-1}} B^R_{2 \alpha \beta} \,, \\
		  N_n (f^n)^{-1} \tfrac{ {f_\alpha^n}' {f_{\beta *}^n}' }{(I')^{\delta(\alpha)/2-1} (I'_* )^{\delta(\beta)/2-1}} B^R_{2 \alpha \beta, n}& \rightarrow \tfrac{ {f'_\alpha} f'_{\beta *} }{(I')^{\delta(\alpha)/2-1} (I'_* )^{\delta(\beta)/2-1}} B^R_{2 \alpha \beta} \,
		  \end{aligned}
	\end{equation}
	weakly in $ L^1(E_R, d \Theta_{\alpha \beta})$.}

The strong convergence of $	\int_{E_R}   P^n_{\alpha \beta} \varphi d \Theta_{\alpha \beta}$ implies that there exists a subsequence converging pointwise a.e. in $(t,x)$ to $\int_{E_R}  P_{\alpha \beta} \varphi d \Theta_{\alpha \beta}$. However, this subsequence may depend on $\varphi$. 

In fact, for any Borel subset $A$ of $E_R$,  there exists a sequence of measurable subsets $\{ A_k\}_{k=1}^\infty $ such that $\mathbf{1}_{A_k} \to \mathbf{1}_A $ in $L^1( E_R)$. Letting $\varphi = \mathbf{1}_{A_k}$ and applying to the strong convergence of $	\int_{E_R}  P^n_{\alpha \beta} \mathbf{1}_{A_k} d\Theta_{\alpha \beta}$, we can extract a subsequence (still denoted by $	\int_{E_R}   P^n_{\alpha \beta} \mathbf{1}_{A_k} d\Theta_{\alpha \beta}$ ) such that
\begin{equation*}
	\int_{E_R}   P^n_{\alpha \beta} \mathbf{1}_{A_k} d\Theta_{\alpha \beta}  \to  	\int_{E_R}   P_{\alpha \beta}  \mathbf{1}_{A_k} d\Theta_{\alpha \beta}  \quad \text{a.e.} \ (t,x) \in (0,R) \times B_R.
\end{equation*} 
Then for any Borel subset $A$ of $E_R$ and any constant $M>0$,
\begin{align*}
	& \left| \int_{E_R} P^n_{\alpha \beta} ( \mathbf{1}_{A_k}  -  \mathbf{1}_{A }   ) d\Theta_{\alpha \beta} \right| \\
	\leq & M^2 \int_{E_R} |\mathbf{1}_{A_k}  -  \mathbf{1}_{A }| d\Theta_{\alpha \beta} +  \int_{E_R}    f^n_\alpha  f^n_{\beta_* } (\mathbf{1}_{   f ^n_\alpha \geq M}+ \mathbf{1}_{   f_{\beta_*}^n \geq M}) B^R_{ 2\alpha \beta, n} d \Theta_{\alpha \beta} \\
	\leq &  M^2 \int_{E_R} |\mathbf{1}_{A_k}  -  \mathbf{1}_{A }| d\Theta_{\alpha \beta} + \Psi(M)^{-1}\bar{K}  \to 0 \quad ( k \to \infty , \  M \to \infty ).
\end{align*}
Therefore,
\begin{align*}
	&\left| 	\int_{E_R}   P^n_{\alpha \beta} \mathbf{1}_{A } d\Theta_{\alpha \beta} - 	\int_{E_R}   P_{\alpha \beta}  \mathbf{1}_{A } d\Theta_{\alpha \beta} \right| \\
	\leq & \left| \int_{E_R} P^n_{\alpha \beta} ( \mathbf{1}_{A_k}  -  \mathbf{1}_{A }   ) d\Theta_{\alpha \beta} \right| + \left| 	\int_{E_R} ( P^n_{\alpha \beta} -P_{\alpha \beta} ) \mathbf{1}_{A_k}  d\Theta_{\alpha \beta} \right| +  \left| \int_{E_R} P_{\alpha \beta}  ( \mathbf{1}_{A_k}  -  \mathbf{1}_{A }   ) d\Theta_{\alpha \beta} \right| \\
	\to & 0 \quad ( k \to \infty , \  n \to \infty ).
\end{align*} 
Since the above holds for all measurable subset $A$, then for almost every $(t,x) \in ( 0,R) \times B_R$ and for all $\varphi \in L^\infty ( E_R)$, 
\begin{align*}
	  P^n_{\alpha \beta} \rightarrow P_{\alpha \beta}, \quad   \tilde{P}^n_{\alpha \beta} \rightarrow \tilde{P}_{\alpha \beta},
\end{align*} 
weakly in $L^1 ( E_R ,d \Theta_{\alpha \beta})$.  Besides, using the facts that $N(f^n) ( t,x)$ and $\left( 1+ \frac{1}{n} \iint_{\R^3 \times \R_+} f^n_\alpha  d \xi dI  \right)^{-1}$ converge a.e. to $N(f)$ and $1$, respectively (see Lemma \ref{theo-6.5}), one easily knows that \eqref{cnv-fn'*} holds.

{\bf Step 4:} {\em The convergence \eqref{HH-limit} holds for $\alpha, \beta \in \{ s_0 + 1, \cdots, s \}$.}

We first define 
\begin{equation*}
j(a,b)=\left\{\begin{array}{rl}
	(a-b) \log \frac{a}{b}, & \quad \text{for } a, b>0 \,, \\
	+\infty, & \quad \text{for } a \text { or } b \leq 0 \,,
\end{array}\right.
\end{equation*}
whose the Hessian matrix 
\begin{equation*}
\begin{aligned}
\tfrac{1}{ab} 
\begin{pmatrix}
	b+\frac{b^2}{a}  &  -(a+b)\\
	-(a+b)   &  a+\frac{a^2}{b} 
\end{pmatrix}, 
\quad a,b>0
\end{aligned}
\end{equation*}
is nonnegative. We then know that $j(a,b)$ is convex. We further consider the integral 
\begin{equation}
J(F,G)=\int_{E_R}j(F,G)d \Theta_{\alpha \beta }, \quad F,G \in L^1(E_R,d \Theta_{\alpha \beta}).
\end{equation}
The convexity of $j (a, b)$ shows that
\begin{equation*}
\begin{aligned}
j(F^n,G^n)\geq j(a,b)+(\log \frac{a}{b}+\frac{a-b}{b})(F^n-a)+(\log \frac{b}{a}+\frac{b-a}{b})(G^n-b) 
\end{aligned}
\end{equation*}
for $F^n,G^n,a,b>0$. We assume that $F^n$ and $G^n$ respectively weakly converge to $F \geq 0$ and $G\geq 0$ in $L^1(E_R,d \Theta_{\alpha \beta })$. For any $\epsilon>0$ and $M > 0$, we take $a=\epsilon +F\wedge M, b=\epsilon +G\wedge M$. It then infers that
\begin{equation*}
\begin{aligned}
\liminf \limits_{n\rightarrow \infty} J(F^n,G^n) 
&\geq \int_{E_R} j(a,b)+(\log \tfrac{a}{b}+\tfrac{a-b}{b})(F-a)+(\log \tfrac{b}{a}+\tfrac{b-a}{b})(G-b) d\Theta_{\alpha \beta}\\
&\geq \int_{E_R} j(a,b)+(\log \tfrac{a}{b}+\tfrac{a-b}{b})(F\wedge M-a)+(\log \tfrac{b}{a}+\tfrac{b-a}{b})(G\wedge M-b) d\Theta_{\alpha \beta}\\
&= \int_{E_R} j(a,b) d \Theta_{\alpha \beta} -\epsilon \int_{E_R} (a-b)(\tfrac{1}{a}-\tfrac{1}{b})d\Theta_{\alpha \beta} \geq \int_{E_R} j(a,b) d\Theta_{\alpha \beta}.
\end{aligned}
\end{equation*}
Hence, for any $\epsilon>0$ and $M > 0$,
\begin{equation*}
\liminf \limits_{n\rightarrow \infty} \int_{E_R}j(F^n,G^n)d\Theta_{\alpha \beta } \geq \int_{E_R}j(\epsilon +F\wedge M ,\epsilon +G\wedge M)d\Theta_{\alpha \beta }.
\end{equation*}
By letting $\epsilon \to 0^+$ and $M \to +\infty$, Fatou's lemma concludes that
\begin{equation}
\liminf \limits_{n\rightarrow \infty} \int_{E_R}j(F^n,G^n)d \Theta_{\alpha \beta } \geq \int_{E_R}j(F ,G) d \Theta_{\alpha \beta }.
\end{equation}
By taking $F^n= N_n (f^n)^{-1} \frac{{f_\alpha^n}' {f_{\beta *}^n}'}{(I')^{\delta(\alpha)/2-1} (I'_*)^{\delta(\beta)/2-1 }} B^R_{2 \alpha \beta, n} $, $G^n= N_n (f^n)^{-1} \frac{f_\alpha^n f_{\beta *}^n}{ I^{\delta(\alpha) /2-1} I_*^{\delta(\beta)/2-1}} B^R_{2 \alpha \beta, n} $, one has 
\begin{equation*}
\liminf \limits_{n\rightarrow \infty} [ N_n (f^n)^{-1} H^{n, R}_{\alpha \beta }(f^n) ] \geq H^R_{\alpha \beta }(f) \quad a.e. \enspace (t, x) \in (0,R) \times B_R.
\end{equation*}
Note that
\begin{equation}
\begin{aligned}
& N_n (f^n)^{-1} H^{n, R}_{\alpha \beta }(f^n) \leq \tilde{H}^n_{\alpha \beta} (f^n) \,, \ H_{\alpha \beta }(f^n)\geq H^R_{\alpha \beta }(f^n) \quad \text{for } R > 0 \,, \\
& H^R_{\alpha \beta }(f) \uparrow H_{\alpha \beta }(f) \quad \text{as } R \uparrow +\infty.
\end{aligned}
\end{equation}
By letting $R \to + \infty$, the convergence \eqref{HH-limit} holds for $\alpha, \beta \in \{ s_0 + 1, \cdots, s \}$.

By the analogous arguments above, one can prove that the convergence \eqref{HH-limit} holds for $\alpha, \beta \in \{1, \cdots, s \}$. Hence the entropy inequality \eqref{entropy-th} holds. The proof of Theorem \ref{MainThm} is therefore finished.

\section{Weak lower semicontinuity of entropy: Proof of Lemma \ref{Lemma 4.6}}\label{subsec-4.6}

In this section, we will prove the weak lower semicontinuity of the entropy $H (f^n)$ defined in \eqref{Hf}, i.e., justifying Lemma \ref{Lemma 4.6}.

\begin{proof}[Proof of Lemma \ref{Lemma 4.6}]
	First, we claim that for any $ r> 1$, 
	\begin{equation}\label{eq4.29}{\small
		\begin{aligned}
  & \sum_{\alpha = 1}^{s_0} \iint_{B_r \times B_r} f_\alpha \log f_\alpha dxd\xi + \sum_{\alpha = s_0 + 1}^{s} \iiint_{B_r \times B_r \times (r^{-1}, r) } f_\alpha \log (I^{1-\delta(\alpha )/2 } f_\alpha) dxd\xi dI \\
\leq & \liminf \limits_{n\rightarrow \infty } \Big\{ \sum_{\alpha = 1}^{s_0} \iint_{B_r \times B_r } f^n_\alpha \log f^n_\alpha dxd\xi + \sum_{\alpha = s_0 + 1}^{s} \iiint_{B_r \times B_r \times (r^{-1},r) } f^n_\alpha \log (I^{1-\delta(\alpha)/2}f^n_\alpha) dxd\xi dI \Big\} .
\end{aligned}}
	\end{equation}
Indeed, denote by $F_m(a,b) = \{ (x,\xi) \in B_r \times B_r : a \leq f_\alpha (x, \xi) <b \}$ for $\alpha \in \{ 1, \cdots, s_0 \}$ and $F_p (a,b) = \{ (x,\xi, I) \in B_r \times B_r\times (r^{-1}, r) : a \leq f_\alpha (x, \xi, I) <b \}$ for $\alpha \in \{ s_0 + 1, \cdots, s \}$. Moreover, for any $\eta > 0$, we divide the interval $[0,b]$ into subintervals with $0 = y_0 < y_1 <\cdots < y_m =b$ such that for $1 \leq i \leq m-1$,
	\begin{equation}\label{eq4.32}
\begin{aligned}
		& \sup \limits_{y\in [y_{i-1} ,y_{i+1}) } |y \log y - y_{i+1} \log y_{i+1} | <\tfrac{\eta}{4 \mathrm{Vol} (B_r \times B_r)} , \\ 
& \sup \limits_{y\in [y_{i-1} ,y_{i+1}) } |y \log ( I^{1 - \delta (\alpha) / 2} y ) - y_{i+1} \log  ( I^{1 - \delta (\alpha) /2} y_{i+1} ) | <\tfrac{\eta}{4 \mathrm{Vol} (B_r \times B_r \times (r^{-1},r ))} \,.
\end{aligned}
	\end{equation}
Note that $f_\alpha \log f_\alpha \in L^1 (\R^3 \times \R^3) $ for $\alpha \in \{ 1, \cdots, s_0 \}$ and $f_\alpha \log ( I^{1 - \delta (\alpha) / 2} f_\alpha ) \in L^1 (\R^3 \times \R^3 \times \R_+) $. Then
	\begin{equation}\label{V1}
		\begin{aligned}
			& |\iint_{F_m(0,b)} f_\alpha \log f_\alpha  dxd\xi  - \sum_{i=0}^{m-1} y_{i+1} \log y_{i+1} \mathrm{Vol} (F(y_i, y_{i+1})) | \\
			\leq & \sum_{i=0}^{m-1} \sup \limits_{(x,\xi) \in F_m (y_i ,y_{i+1})} |f_\alpha \log f_\alpha  - y_{i+1} \log y_{i+1} | \mathrm{Vol} (F_m (y_i, y_{i+1})) < \tfrac{\eta }{4},
		\end{aligned}
	\end{equation}
and
\begin{equation}\label{V2}{\small
		\begin{aligned}
			& |\iint_{F_p(0,b)} f_\alpha \log ( I^{1 - \delta (\alpha) / 2} f_\alpha )  dxd\xi  - \sum_{i=0}^{m-1} y_{i+1} \log ( I^{1 - \delta (\alpha) / 2} y_{i+1} ) \mathrm{Vol} (F_p (y_i, y_{i+1})) | \\
			\leq & \sum_{i=0}^{m-1} \sup \limits_{(x,\xi, I) \in F_p (y_i ,y_{i+1})} |f_\alpha \log ( I^{1 - \delta (\alpha) / 2} f_\alpha )  - y_{i+1} \log ( I^{1 - \delta (\alpha) / 2} y_{i+1} ) | \mathrm{Vol} (F_p (y_i, y_{i+1})) < \tfrac{\eta }{4}.
		\end{aligned}}
	\end{equation}
Furthermore, $f_\alpha^n$ are $L^1$ weakly convergent to $f_\alpha$ for $\alpha \in \{ 1, \cdots, s \}$. It therefore infers that
\begin{equation*}
  \begin{aligned}
    \lim_{n \to \infty} \iint_{F_m (y_j, y_{j+1})} f_\alpha^n dxd\xi = \iint_{F_m (y_j, y_{j+1})} f_\alpha dxd\xi \,, \quad & \alpha \in \{ 1, \cdots, s_0 \} \,, \\
    \lim_{n \to \infty} \iiint_{F_p (y_j, y_{j+1})} f_\alpha^n dxd\xi d I = \iiint_{F_p (y_j, y_{j+1})} f_\alpha dxd\xi d I \,, \quad & \alpha \in \{s_0 + 1, \cdots, s \} \,.
  \end{aligned}
\end{equation*}
Observe that
\begin{equation*}
\begin{aligned}
		y_i \mathrm{Vol} (F_m (y_i ,y_{i+1})) \leq \iint_{F_m (y_j, y_{j+1})} f_\alpha dxd\xi <y_{j+1} \mathrm{Vol} (F_m (y_j, y_{j+1})) , \quad & \alpha \in \{ 1, \cdots, s_0 \} \,, \\
y_i \mathrm{Vol} (F_p (y_i ,y_{i+1})) \leq \iiint_{F_p (y_j, y_{j+1})} f_\alpha dxd\xi d I <y_{j+1} \mathrm{Vol} (F_p (y_j, y_{j+1})), \quad & \alpha \in \{s_0 + 1, \cdots, s \} .
\end{aligned}
\end{equation*}
Consequently, for any $\delta > 0$ there is $n_\delta$ such that for any $n>n_\delta $ and $0\leq i \leq m-1 $, 
	\begin{equation}\label{eq4.34}
\begin{aligned}
		y_i -\delta \leq \tfrac{1}{\mathrm{Vol} (F_m(y_i, y_{i+1}))} \iint_{F_m(y_i, y_{i+1})} f^n_\alpha dxd\xi \leq y_{i+1} + \delta , \\
y_i -\delta \leq \tfrac{1}{\mathrm{Vol} (F_p (y_i, y_{i+1}))} \iiint_{F_p (y_i, y_{i+1})} f^n_\alpha dxd\xi d I \leq y_{i+1} + \delta.
\end{aligned}
	\end{equation}
By employing the Jensen's inequality on $F_m(y_j,y_{j+1})$ for the convex function $\varphi (y) = y \log y ,\ y = f^n_\alpha $, we get that for $n > n_\delta$,
	\begin{equation*}
		\begin{aligned}
			& \iint_{F_m (y_i, y_{i+1})} f^n_\alpha \log f^n_\alpha \tfrac{dxd\xi}{\mathrm{Vol} (F_m(y_i, y_{i+1}))} \\
			\geq & (\iint_{F_m (y_i, y_{i+1})} \tfrac{f^n_\alpha}{\mathrm{Vol} (F_m (y_i, y_{i+1}))} dxd\xi) \log (\iint_{F_m (y_i, y_{i+1})} \tfrac{f^n_\alpha}{\mathrm{Vol} (F_m(y_i, y_{i+1}))} dxd\xi) \geq (y_i - \delta) \log (y_i - \delta) \,.
		\end{aligned}
	\end{equation*}
Taking $0 < \delta < \min \limits_{0 \leq i \leq m-1} |y_{i-1} - y_i | $, one follows from \eqref{eq4.32} that for $n > n_\delta$ and $\alpha \in \{ 1, \cdots, s_0 \}$,
	\begin{equation}\label{V3}
		\begin{aligned}
			\iint_{F_m (y_i, y_{i+1})} f^n_\alpha \log f^n_\alpha dxd\xi 
			&\geq (y_{i+1} \log y_{i+1} - \tfrac{\eta }{4 \mathrm{Vol} (B_r \times B_r)})  \mathrm{Vol} (F_m(y_i, y_{i+1})).
		\end{aligned}
	\end{equation}
Similarly, for $n > n_\delta$ and $\alpha \in \{s_0 + 1, \cdots, s \}$,
\begin{equation}\label{V4}
		\begin{aligned}
			& \iiint_{F_p (y_i, y_{i+1})} f^n_\alpha \log ( I^{1 - \delta (\alpha) /2} f^n_\alpha ) dxd\xi d I \\
			&\geq (y_{i+1} \log ( I^{1 - \delta (\alpha) /2} y_{i+1} ) - \tfrac{\eta }{4 \mathrm{Vol} ( B_r \times B_r \times ( r^{-1}, r ) ) } ) \mathrm{Vol} (F_p(y_i, y_{i+1})).
		\end{aligned}
	\end{equation}
Combining the estimates \eqref{V1}, \eqref{V2}, \eqref{V3} and \eqref{V4}, one obtains that for $n>n_\delta$, $\eta > 0$ and large enough $b > 0$,
\begin{equation*}
\begin{aligned}
	& \sum_{\alpha = 1}^{s_0} \iint_{F_m (0,b)} f^n_\alpha \log f^n_\alpha dxd\xi + \sum_{\alpha = s_0 + 1}^{s} \iiint_{F_p (0, b)} f^n_\alpha \log ( I^{1 - \delta (\alpha) /2} f^n_\alpha ) dxd\xi d I \\
\geq & \sum_{\alpha = 1}^{s_0} \iint_{F_m (0,b)} f_\alpha \log f_\alpha dxd\xi + \sum_{\alpha = s_0 + 1}^{s} \iiint_{F_p (0, b)} f_\alpha \log ( I^{1 - \delta (\alpha) /2} f_\alpha ) dxd\xi d I - \tfrac{s}{2} \eta \,,
\end{aligned}
\end{equation*}
which implies the claim \eqref{eq4.29} by letting $b \to + \infty$, $n \to + \infty$ and $\eta \to 0^+$.

Denote by 
  \begin{align*}
    a (r) = & \sum_{\alpha = 1}^{s_0} \iint_{B_r \times B_r} f_\alpha \log f_\alpha dxd\xi + \sum_{\alpha = s_0 + 1}^{s} \iiint_{B_r \times B_r \times (r^{-1}, r) } f_\alpha \log (I^{1-\delta(\alpha )/2 } f_\alpha) dxd\xi dI \,, \\
    b_n (r) = & \sum_{\alpha = 1}^{s_0} \iint_{B_r \times B_r } f^n_\alpha \log f^n_\alpha dxd\xi + \sum_{\alpha = s_0 + 1}^{s} \iiint_{B_r \times B_r \times (r^{-1},r) } f^n_\alpha \log (I^{1-\delta(\alpha)/2}f^n_\alpha) dxd\xi dI \,.
  \end{align*}
Then the claim \eqref{eq4.29} indicates that $a (r) \leq \liminf_{n \to \infty} b_n (r)$ for any fixed $r > 1$. Taking a subsequence $\{ b_{n_k} (r) \}$ of $\{ b_n (r) \}$ such that $\liminf_{n \to \infty} b_n (r) = \lim_{k \to \infty} b_{n_k} (r)$. Then for any $\epsilon > 0$, there is a $K = K (\epsilon, r) > 0$ such that for all $k > K $,
\begin{equation*}
  \begin{aligned}
    a (r) - \epsilon < b_{n_k} (r) \,.
  \end{aligned}
\end{equation*}
Due to $f^n_\alpha \log f^n_\alpha \in L^1 (\R^3 \times \R^3)$ for $\alpha \in \{ 1, \cdots, s_0 \}$ and $f^n_\alpha \log (I^{1-\delta(\alpha)/2}f^n_\alpha) \in L^1 (\R^3 \times \R^3 \times \R_+)$ for $\alpha \in \{ s_0 + 1, \cdots, s \}$, $\lim_{r \to + \infty} b_{n_k} (r) = b_{n_k} (\infty)$ exists. Consequently, by letting $r \to + \infty$, $k \to \infty$ and $\epsilon \to 0^+$, one has $a (\infty) \leq \lim_{k \to \infty} b_{n_k} (\infty) = \liminf_{n \to \infty} b_n (\infty)$. The proof of Lemma \ref{Lemma 4.6} is finished.
\end{proof}


\section*{Acknowledgments}
This work was supported by National Key R\&D Program of China under the grant 2023YFA 1010300, the National Natural Science Foundation of China under contract No. 12201220 and the Guang Dong Basic and Applied Basic Research Foundation under contract No. 2024A1515012358.

\bigskip


\begin{thebibliography}{99}




\bibitem{AV-CPAM-2002} R. Alexandre and C. Villani. On the Boltzmann equation for long-range interactions. {\em Comm. Pure Appl. Math.} {\bf 55} (2002), no. 1, 30-70.

\bibitem{ARKERYD-Boltz} L. Arkeryd. On the Boltzmann equation. I. Existence. {\em Arch. Rat. Mech. Anal.} {\bf 45} (1972), 1-16.

\bibitem{Arkeryd-inequality} L. Arkeryd. Loeb solutions of the Boltzmann equation. {\em Arch. Rational Mech. Anal.} {\bf 86} (1984), no. 1, 85-97. 
    
\bibitem{DSR-2019-BOOK} D. Ars\'enio and L. Saint-Raymond. {\em From the Vlasov-Maxwell-Boltzmann system to incompressible viscous electro-magneto-hydrodynamics. Vol. 1.} { EMS Monographs in Mathematics. European Mathematical Society (EMS)}, Z\"urich, 2019. xii+406 pp.
    
\bibitem{BGL-JSP-1991} C. Bardos, F. Golse and C. D. Levermore. Fluid dynamic limits of kinetic equations. I. Formal derivations. {\em J. Statist. Phys.} {\bf 63} (1991), no. 1-2, 323-344.
    
\bibitem{BGL-CPAM-1993} C. Bardos, F. Golse and C. D. Levermore. Fluid dynamic limits of kinetic equations. II. Convergence proofs for the Boltzmann equation. {\em Comm. Pure Appl. Math.} {\bf 46} (1993), no. 5, 667-753.
    
\bibitem{BGL-ARMA-2000} C. Bardos, F. Golse and C. D. Levermore. The acoustic limit for the Boltzmann equation. {\em Arch.Ration. Mech. Anal.} {\bf 153} (2000), no. 3, 177-204. 
    
\bibitem{BGL-1993-CPAM} C. Bardos, F. Golse and C.D. Levermore. Fluid Dynamic Limits of the Boltzmann Equation II: Convergence Proofs, {\em Comm. Pure Appl. Math.}, {\bf 46} (1993), 667-753.

\bibitem{Bernhoff-AAM-2023} N. Bernhoff. Linearized Boltzmann collision operator: I. Polyatomic molecules modeled by a discrete internal energy variable and multicomponent mixtures. {\em Acta Appl. Math.} {\bf 183} (2023), no. 3, 1-45.

\bibitem{Bernhoff-KRM-2023} N. Bernhoff. Linearized Boltzmann collision operator: II. Polyatomic molecules modeled by a continuous internal energy variable. {\em Kinet. Relat. Models} {\bf 16} (2023), 828-849.

\bibitem{N.Bernhoff-2023} N. Bernhoff. Compactness property of the linearized Boltzmann collision operator for a mixture of monatomic and polyatomic species. {\em J. Stat. Phys.} {\bf 191} (2024), no. 32, 1-11.
	
\bibitem{Borgnakke-Larsen-JCP-1975} C. Borgnakke and P.S. Larsen. Statistical collision model for Monte-Carlo simulation of polyatomic mixtures. {\em J. Comput. Phys.} {\bf 18} (1975), 405-420.

\bibitem{Chapman-Cowling-1960} S. Chapman and T. G. Cowling. {\em The Mathematical Theory of Non-uniform Gases: an Account of the Kinetic Theory of Viscosity, Thermal Conduction, and Diffusion in Gases}. Cambridge University Press, New York, 1960.


\bibitem{R-de-la-Vallée-Poussin}R. del Campo, A. Fernández, F. Mayoral and F. Naranjo, The de la Vallée-Poussin theorem and Orlicz spaces associated to a vector measure. {\em J. Math. Anal. Appl.} {\bf 470} (2019), no. 1,  279-291.
    
\bibitem{DiPerna-Lions-Inv} R. J. DiPerna and P.-L. Lions. Ordinary differential equations, transport theory and Sobolev spaces. {\em Invent. Math.} {\bf 98} (1989), no. 3, 511-547.

\bibitem{Diperna-Lions} R. J. Diperna and P. L. Lions. On the Cauchy problem for the Boltzmann equation: Global existence and weak stability. {\em Ann. Math.} {\bf 130} (1989), 312-366.
    
\bibitem{DL-ARMA-1991} R. J. Diperna and P. L. Lions. Global solutions of Boltzmann's equation and the entropy inequality. {\em Arch. Rational Mech. Anal.} {\bf 114} (1991), no. 1, 47-55.

\bibitem{Dunford-Schwartz-1958} N. Dunford and J. T. Schwartz. {\em Linear Operators. I. General Theory}. With the assistance of W. G. Bade and R. G. Bartle. Pure and Applied Mathematics, Vol. {\bf 7}. Interscience Publishers, Inc., New York; Interscience Publishers Ltd., London, 1958. xiv+858 pp.

\bibitem{Evans-1990-AMS} L. C. Evans, Weak convergence methods for nonlinear partial differential equations. {\em Amer. Math. Soc.} {\bf 74} (1990), 4-6.

\bibitem{Gamba-Pavic-JMP-2023} I. Gamba and M. Pavic-Colic. On the Cauchy problem for Boltzmann equation modelling polyatomic gas. {\em J. Math. Phys.} {\bf 64 } (2023), 1-51.

    
\bibitem{GL-CPAM-2002} F. Golse and C. D. Levermore. Stokes-Fourier and acoustic limits for the Boltzmann equation:convergence proofs. {\em Comm. Pure Appl. Math.} {\bf 55} (2002), no. 3, 336-393.

\bibitem{Golse-velocity}F. Golse, P. L. Lions, B. Perthame and R. Sentis. Regularity of the moments of the solution of a transport equation. {\em J. Funct. Anal.} {\bf 76} (1988), 110-125.

\bibitem{Golse-perthame-velocity}F. Golse, B. Perthame and R. Sentis. Un résultat pour les équations de transport et application au calcul de la limite de la valeur propre principale d'un opérateur de transport. {\em C.R. Acad. Sci.} Paris {\bf 301} (1985), 341-344. 
    
\bibitem{GSR-IM-2004} F. Golse and L. Saint-Raymond. The Navier-Stokes limit of the Boltzmann equation for bounded collision kernels. {\em Invent. Math.} {\bf 155} (2004), no. 1, 81-161.
    
\bibitem{GSR-JMPA-2009} F. Golse and L. Saint-Raymond. The incompressible Navier-Stokes limit of the Boltzmann equation for hard cutoff potentials. {\em J. Math. Pures Appl. (9)} {\bf 91} (2009), no. 5, 508-552.

\bibitem{Grad-cut-off} H. Grad. Asymptotic theory of the Boltzmann equation. II. {\em Rarefied Gas Dynamics} (Proc. 3rd Internat. Sympos., Palais de l'UNESCO, Paris, 1962), Vol. I (1963) pp. 26-59, Academic Press, New York.
    
\bibitem{Guo-CPAM-2006} Y. Guo. Boltzmann diffusive limit beyond the Navier-Stokes approximation. {\em Comm. Pure Appl. Math.} {\bf 59} (2006), no. 5, 626-687.
    
\bibitem{Guo-ARMA-2010} Y. Guo. Decay and continuity of the Boltzmann equation in bounded domains. {\em Arch. Ration. Mech. Anal.} {\bf 197} (2010), no. 3, 713-809.

\bibitem{GJJ-CPAM-2010} Y. Guo, J. Jang and N. Jiang. Acoustic limit for the Boltzmann equation in optimal scaling. {\em Comm. Pure Appl. Math.} {\bf 63} (2010), no. 3, 337-361.
    
\bibitem{JLM-CPDE-2010} N. Jiang, C. D. Levermore and N. Masmoudi. Remarks on the acoustic limit for theBoltzmann equation. {\em Comm. Partial Differential Equations} {\bf 35} (2010), no. 9, 1590-1609.
    
\bibitem{JL-APDE-2022} N. Jiang and Y.-L. Luo. From Vlasov-Maxwell-Boltzmann system to two-fluid incompressible Navier-Stokes-Fourier-Maxwell system with Ohm's law: convergence for classical solutions. {\em Ann. PDE} {\bf 8} (2022), no. 1, Paper No. 4, 126 pp.
    
\bibitem{JLT-TAMS-2024} N. Jiang, Y.-L. Luo and S.J. Tang. Compressible Euler limit from Boltzmann equation with complete diffusive boundary condition in half-space. {\em Trans. Amer. Math. Soc.} {\bf 377} (2024), no. 8, 5323-5359.
    
\bibitem{JLZ-ARMA-2023} N. Jiang, Y.-L. Luo and T.-F. Zhang. Hydrodynamic limit of the incompressible Navier-Stokes-Fourier-Maxwell system with Ohm's law from the Vlasov-Maxwell-Boltzmann system: Hilbert expansion approach. {\em Arch. Ration. Mech. Anal.} {\bf 247} (2023), no. 3, Paper No. 55, 85 pp.
    
\bibitem{JM-CPAM-2017} N. Jiang and N. Masmoudi. Boundary layers and incompressible Navier-Stokes-Fourier limit of the Boltzmann equation in bounded domain I. {\em Comm. Pure Appl. Math.} {\bf 70} (2017), no. 1, 90-171.
    
\bibitem{JXZ-IUMJ-2018} N. Jiang, C.-J. Xu and H.J. Zhao. Incompressible Navier-Stokes-Fourier limit from the Boltzmann equation: classical solutions. {\em Indiana Univ. Math. J.} {\bf 67} (2018), no. 5, 1817-1855.
    
\bibitem{JZ-SIMA-2019} N. Jiang and X. Zhang. The Boltzmann equation with incoming boundary condition: global solutions and Navier-Stokes limit. {\em SIAM J. Math. Anal.} {\bf 51} (2019), no. 3, 2504-2534. 
    
\bibitem{JZ-JDE-2019} N. Jiang and X. Zhang. Global renormalized solutions and Navier-Stokes limit of the Boltzmann equation with incoming boundary condition for long range interaction. {\em J. Differential Equations} {\bf 266} (2019), no. 5, 2597-2637. 
    
\bibitem{LM-ARMA-2010} C. D. Levermore and N. Masmoudi. From the Boltzmann equation to an incompressible Navier-Stokes-Fourier system. {\em Arch. Ration. Mech. Anal.} {\bf 196} (2010), no. 3, 753-809.
    
\bibitem{LM-ARMA-2001-I} P.-L. Lions and N. Masmoudi. From the Boltzmann equations to the equations of incompressible ﬂuid mechanics, I. {\em Arch. Ration. Mech. Anal.} {\bf 158} (2001), no. 3, 173-193.
     
\bibitem{LM-ARMA-2001-II} P.-L. Lions and N. Masmoudi. From Boltzmann equation to Navier-Stokes and Euler equations,II. {\em Arch. Ration. Mech. Anal.} {\bf 158} (2001), 195-211.
    
\bibitem{MSR-CPAM-2003} N. Masmoudi and L. Saint-Raymond. From the Boltzmann equation to the Stokes-Fourier system in a bounded domain. {\em Comm. Pure Appl. Math.} {\bf 56} (2003), no. 9, 1263-1293.
    
\bibitem{Mischler-ASENS-2010} S. Mischler. Kinetic equations with Maxwell boundary conditions. {\em Ann. Sci. \'Ec. Norm. Sup\'er. (4)} {\bf 43} (2010), no. 5, 719-760.

\bibitem{Morse-PF-1964} T. F. Morse. Kinetic model for gases with internal degrees of freedom. {\em Phys. Fluids}, {\bf 7} (1964), 159-169.
    
\bibitem{SR-ARMA-2003} L. Saint-Raymond. Convergence of solutions to the Boltzmann equation in the incompressible Euler limit. {\em Arch. Ration. Mech. Anal.} {\bf 166} (2003), no. 1, 47-80.
    
\bibitem{SR-2009-BOOK} L. Saint-Raymond. {\em Hydrodynamic limits of the Boltzmann equation}. Lecture Notes in Mathe-matics, 1971. Springer, Berlin, 2009. 
    
\bibitem{SG-ARMA-2008} R. M. Strain and Y. Guo. Exponential decay for soft potentials near Maxwellian. {\em Arch. Ration. Mech. Anal.} {\bf 187} (2008), no. 2, 287-339.

\bibitem{Wang-Uhlenbeck-1951} C. S. Wang Chang and G. E. Uhlenbeck. Transport phenomena in polyatomic molecules. {\em Univ. Mich. Eng. Res. Rep.} {\bf CM-681} (1951).
    
\bibitem{Villani-ADE-1966} C. Villani. On the Cauchy problem for Landau equation: sequential stability, global existence. {\em Adv. Differential Equations} {\bf 1} (1996), no. 5, 793-816. 

\end{thebibliography}

\end{document}